\documentclass[11pt]{amsart}
\usepackage{amsmath,amssymb,times,color,dsfont,mathrsfs}
\usepackage{graphicx}
\usepackage{dsfont}
\definecolor{webred}{rgb}{0.75,0,0}
\definecolor{webgreen}{rgb}{0,0.75,0}
\usepackage[citecolor=webgreen,colorlinks=true,linkcolor=webred]{hyperref}

\makeatletter
\def\section{\@startsection{section}{1}\z@{.9\linespacing\@plus\linespacing}%
  {.7\linespacing} {\fontsize{12}{14}\selectfont\scshape\centering}}
\makeatother

\theoremstyle{plain}
\newtheorem{thrm}{Theorem}[section]
\newtheorem{lmm}[thrm]{Lemma}
\newtheorem{crllr}[thrm]{Corollary}
\newtheorem{prpstn}[thrm]{Proposition}

\theoremstyle{definition}
\newtheorem{dfntn}[thrm]{Definition}
\newtheorem{xmpl}[thrm]{Exemple}

\theoremstyle{remark}
\newtheorem{rmrk}[thrm]{Remark}
{}

\newcommand{\la}{\lambda}
\newcommand{\eps}{\varepsilon}

\newcommand{\unA}{\underline{A}}
\newcommand{\bx}{\mathbf{x}}
\newcommand{\re}{\mathrm{e}}
\newcommand{\cB}{\mathcal{B}}
\newcommand{\rH}{\mathrm{H}}
\newcommand{\cI}{\mathcal{I}}
\newcommand{\cH}{\mathscr{H}}
\newcommand{\cL}{\mathscr{L}}
\newcommand{\cM}{\mathscr{N}}

\newcommand{\uu}{u}
\newcommand{\vv}{v}
\newcommand{\inff}{\mathop{\operatorname{\vphantom{p}inf}}}
\newcommand{\iti}[1]{\mbox{{\it (\romannumeral #1)\/}}}

\newcommand{\BO}{\mathsf{BO}}

\newcommand{\ess}{\mathsf{ess}}
\newcommand{\dis}{\mathsf{dis}}
\newcommand{\sing}{\mathsf{sing}}

\newcommand{\comp}{\mathsf{cpt}}

\newcommand{\Tri}{\mathsf{Tri}}

\newcommand{\Dir}{\mathsf{Dir}}
\newcommand{\Mix}{\mathsf{Mix}}
\newcommand{\Neu}{\mathsf{Neu}}

\newcommand{\cN}{\mathcal{N}}

\newcommand{\di}{\displaystyle}

\newcommand{\A}{\mathsf{A}}
\newcommand{\Ai}{\mathsf{Ai}}

\newcommand{\Dom}{\mathsf{Dom}}

\newcommand{\odd}{\mathsf{odd}}
\newcommand{\even}{\mathsf{even}}

\newcommand{\Id} {\mathrm{Id}}

\def\cf{\emph{cf.\/}}
\def\ie{\emph{i.e.\/}}

\def\xC{\mathbb{C}} \def\xR{\mathbb{R}}

\def\xCone{{\rm C}^{1}}

\def\xHone{{\rm H}^{1}}

\def\xLtwo{{\rm L}^{2}} 
\def\xLinfty{{\rm L}^{\infty}} 

\def\xdif{\,{\rm d}}
\begin{document}

%%%%%%%%%%%%%%%--fig4tex--%%%%%%%%%%%%%%%
% fig4tex.tex  Version 1.8.4, May 5, 2007
% Copyright 2001-2007 Yvon Lafranche
%
% This material is subject to the LaTeX Project Public License.
% See http://www.ctan.org/tex-archive/help/Catalogue/licenses.lppl.html
% for the details of that license.
%
% Authors:  Yvon Lafranche, Daniel Martin
%           IRMAR, Universite de Rennes 1 - France
%           Yvon.Lafranche@univ-rennes1.fr
%           Daniel.Martin@univ-rennes1.fr
%
% Web Site: http://perso.univ-rennes1.fr/yvon.lafranche/fig4tex
%
%%%%%%%%%%%%%%%%%%%%%%%%%%%%%%%%%%%%%%%%%%%%%%%%%%%%%%%%%%%%%%%%%%%%%%%%%%%%%%%
\ifx\figforTeXisloaded\relax \else\global\let\figforTeXisloaded=\relax\fi
\message{version 1.8.4}
\catcode`\@=11
\ifx\ctr@ln@m\undefined\else%
    \immediate\write16{*** Fig4TeX WARNING : \string\ctr@ln@m\space already defined.}\fi
% Warns if \ControlSequence is already defined. The test rely on the control sequence
% \undefined which is supposed not to be defined.
% Nota: The test fills TeX's memory since the \ControlSequence passed as argument is
%       put into the tables strings, string characters and multiletter control sequences.
%       This not a problem since we need those control sequences.
%       Outside a macro, we could use the fact that "TeX does not put undefined control
%       sequences into its internal tables if they follow \ifx" (cf The TeXBook, p 384),
%       which does not seem possible here, in the context of a macro call.
% Appel (interne) : \ctr@ln@m{\ControlSequence}
\def\ctr@ln@m#1{\ifx#1\undefined\else%
    \immediate\write16{*** Fig4TeX WARNING : \string#1 already defined.}\fi}
% Apply \Command to \ControlSequence after having tested if \ControlSequence
% is already defined, in which case a message is issued.
% \Command is typically one of the commands used to define a control sequence
% namely \def, \edef, \gdef, \xdef, \chardef, \mathchardef, \let.
% Appel (interne) : \ctr@ld@f{\Command}{\ControlSequence}
\ctr@ln@m\ctr@ld@f
\def\ctr@ld@f#1#2{\ctr@ln@m#2#1#2}
% Apply \csname Command\endcsname to \ControlSequence after having tested if
% \ControlSequence is already defined, in which case a message is issued.
% Command is typically the name of the commands (without the leading \) used
% to reserve a register, namely newbox, newcount, newdimen, newif, newtoks,
% newread, newwrite.
% It is analogous to \ctr@ld@f, but \ctr@ld@f cannot be used for this
% purpose since the commands \newbox, \newcount... cannot be passed as argument.
% Appel (interne) : \ctr@ln@w{Command}{\ControlSequence}
\ctr@ld@f\def\ctr@ln@w#1#2{\ctr@ln@m#2\csname#1\endcsname#2}
{\catcode`\/=0 \catcode`/\=12 /ctr@ld@f/gdef/BS@{\}}
% Warns if \csname Text\endcsname is already defined. Used to test internal macros
% like points and associated text.
% Nota: The test is based on the fact that the expansion of \csname ...\endcsname
%       is "defined to be like \relax if its meaning is currently undefined" (cf.
%       The TeXBook, p 40 and 213). However, by itself, the expansion of \csname
%       ...\endcsname defines the macro and fills TeX's memory, so that after
%       \ctr@lcsn@m{Text}, the command \ctr@ln@m{\Text} will warn.
% Appel (interne) : \ctr@lcsn@m{Text}
%                   \ctr@lcsn@m{\ControlSequence} with \ControlSequence expands to Text
\ctr@ld@f\def\ctr@lcsn@m#1{\expandafter\ifx\csname#1\endcsname\relax\else%
    \immediate\write16{*** Fig4TeX WARNING : \BS@\expandafter\string#1\space already defined.}\fi}
\ctr@ld@f\edef\colonc@tcode{\the\catcode`\:}
\ctr@ld@f\edef\semicolonc@tcode{\the\catcode`\;}
\ctr@ld@f\def\t@stc@tcodech@nge{{\let\c@tcodech@nged=\z@%
    \ifnum\colonc@tcode=\the\catcode`\:\else\let\c@tcodech@nged=\@ne\fi%
    \ifnum\semicolonc@tcode=\the\catcode`\;\else\let\c@tcodech@nged=\@ne\fi%
    \ifx\c@tcodech@nged\@ne%
    \immediate\write16{}
    \immediate\write16{!!!=============================================================!!!}
    \immediate\write16{ Fig4TeX WARNING :}
    \immediate\write16{ The category code of some characters has been changed, which will}
    \immediate\write16{ result in an error (message "Runaway argument?").}
    \immediate\write16{ This probably comes from another package that changed the category}
    \immediate\write16{ code after Fig4TeX was loaded. If that proves to be exact, the}
    \immediate\write16{ solution is to exchange the loading commands on top of your file}
    \immediate\write16{ so that Fig4TeX is loaded last. For example, in LaTeX, we should}
    \immediate\write16{ say :}
    \immediate\write16{\BS@ usepackage[french]{babel}}
    \immediate\write16{\BS@ usepackage{fig4tex}}
    \immediate\write16{!!!=============================================================!!!}
    \immediate\write16{}
    \fi}}
% Fig4TeX logo
\ctr@ld@f\def\FigforTeX{F\kern-.05em i\kern-.05em g\kern-.1em\raise-.14em\hbox{4}\kern-.19em\TeX}
%%%%%%%%%%%%%%%%%%%%%%%%%%%%%%%%%%%%%%%%%%%%%%%%%%%%%%%%%%%%%%%%%%%%%%%%%%%%%%%
% Points with numbers >= 0 are devoted to the user.
% Points with numbers <  0 are reserved to internal use.
%%%%%%%%%%%%%%%%%%%%%%%%%%%%%%%%%%%%%%%%%%%%%%%%%%%%%%%%%%%%%%%%%%%%%%%%%%%%%%%
\ctr@ln@w{newdimen}\epsil@n\epsil@n=0.00005pt
\ctr@ln@w{newdimen}\Cepsil@n\Cepsil@n=0.005pt
\ctr@ln@w{newdimen}\dcq@\dcq@=254pt
\ctr@ln@w{newdimen}\PI@\PI@=3.141592pt
\ctr@ln@w{newdimen}\DemiPI@deg\DemiPI@deg=90pt
\ctr@ln@w{newdimen}\PI@deg\PI@deg=180pt
\ctr@ln@w{newdimen}\DePI@deg\DePI@deg=360pt
\ctr@ld@f\chardef\t@n=10
\ctr@ld@f\chardef\c@nt=100
\ctr@ld@f\chardef\@lxxiv=74
\ctr@ld@f\chardef\@xci=91
\ctr@ld@f\mathchardef\@nMnCQn=9949
\ctr@ld@f\chardef\@vi=6
\ctr@ld@f\chardef\@xxx=30
\ctr@ld@f\chardef\@lvi=56
\ctr@ld@f\chardef\@@lxxi=71
\ctr@ld@f\chardef\@lxxxv=85
\ctr@ld@f\mathchardef\@@mmmmlxviii=4068
\ctr@ld@f\mathchardef\@ccclx=360
\ctr@ld@f\mathchardef\@dccxx=720
\ctr@ln@w{newcount}\p@rtent \ctr@ln@w{newcount}\f@ctech \ctr@ln@w{newcount}\result@tent
\ctr@ln@w{newdimen}\v@lmin \ctr@ln@w{newdimen}\v@lmax \ctr@ln@w{newdimen}\v@leur
\ctr@ln@w{newdimen}\result@t\ctr@ln@w{newdimen}\result@@t
\ctr@ln@w{newdimen}\mili@u \ctr@ln@w{newdimen}\c@rre \ctr@ln@w{newdimen}\delt@
\ctr@ld@f\def\degT@rd{0.017453 }  % pi/180
\ctr@ld@f\def\rdT@deg{57.295779 } % 180/pi
\ctr@ln@m\v@leurseule
{\catcode`p=12 \catcode`t=12 \gdef\v@leurseule#1pt{#1}}
\ctr@ld@f\def\repdecn@mb#1{\expandafter\v@leurseule\the#1\space}
\ctr@ld@f\def\arct@n#1(#2,#3){{\v@lmin=#2\v@lmax=#3%
    \maxim@m{\mili@u}{-\v@lmin}{\v@lmin}\maxim@m{\c@rre}{-\v@lmax}{\v@lmax}%
    \delt@=\mili@u\m@ech\mili@u%
    \ifdim\c@rre>\@nMnCQn\mili@u\divide\v@lmax\tw@\c@lATAN\v@leur(\z@,\v@lmax)% DY > 9949 DX
    \else%
    \maxim@m{\mili@u}{-\v@lmin}{\v@lmin}\maxim@m{\c@rre}{-\v@lmax}{\v@lmax}%
    \m@ech\c@rre%
    \ifdim\mili@u>\@nMnCQn\c@rre\divide\v@lmin\tw@% DX > 9949 DY
    \maxim@m{\mili@u}{-\v@lmin}{\v@lmin}\c@lATAN\v@leur(\mili@u,\z@)%
    \else\c@lATAN\v@leur(\delt@,\v@lmax)\fi\fi%
    \ifdim\v@lmin<\z@\v@leur=-\v@leur\ifdim\v@lmax<\z@\advance\v@leur-\PI@%
    \else\advance\v@leur\PI@\fi\fi%
    \global\result@t=\v@leur}#1=\result@t}
\ctr@ld@f\def\m@ech#1{\ifdim#1>1.646pt\divide\mili@u\t@n\divide\c@rre\t@n\m@ech#1\fi}
\ctr@ld@f\def\c@lATAN#1(#2,#3){{\v@lmin=#2\v@lmax=#3\v@leur=\z@\delt@=\tw@ pt%
    \un@iter{0.785398}{\v@lmax<}%
    \un@iter{0.463648}{\v@lmax<}%
    \un@iter{0.244979}{\v@lmax<}%
    \un@iter{0.124355}{\v@lmax<}%
    \un@iter{0.062419}{\v@lmax<}%
    \un@iter{0.031240}{\v@lmax<}%
    \un@iter{0.015624}{\v@lmax<}%
    \un@iter{0.007812}{\v@lmax<}%
    \un@iter{0.003906}{\v@lmax<}%
    \un@iter{0.001953}{\v@lmax<}%
    \un@iter{0.000976}{\v@lmax<}%
    \un@iter{0.000488}{\v@lmax<}%
    \un@iter{0.000244}{\v@lmax<}%
    \un@iter{0.000122}{\v@lmax<}%
    \un@iter{0.000061}{\v@lmax<}%
    \un@iter{0.000030}{\v@lmax<}%
    \un@iter{0.000015}{\v@lmax<}%
    \global\result@t=\v@leur}#1=\result@t}
\ctr@ld@f\def\un@iter#1#2{%
    \divide\delt@\tw@\edef\dpmn@{\repdecn@mb{\delt@}}%
    \mili@u=\v@lmin%
    \ifdim#2\z@%
      \advance\v@lmin-\dpmn@\v@lmax\advance\v@lmax\dpmn@\mili@u%
      \advance\v@leur-#1pt%
    \else%
      \advance\v@lmin\dpmn@\v@lmax\advance\v@lmax-\dpmn@\mili@u%
      \advance\v@leur#1pt%
    \fi}
\ctr@ld@f\def\c@ssin#1#2#3{\expandafter\ifx\csname COS@\number#3\endcsname\relax\c@lCS{#3pt}%
    \expandafter\xdef\csname COS@\number#3\endcsname{\repdecn@mb\result@t}%
    \expandafter\xdef\csname SIN@\number#3\endcsname{\repdecn@mb\result@@t}\fi%
    \edef#1{\csname COS@\number#3\endcsname}\edef#2{\csname SIN@\number#3\endcsname}}
\ctr@ld@f\def\c@lCS#1{{\mili@u=#1\p@rtent=\@ne%
    \relax\ifdim\mili@u<\z@\red@ng<-\else\red@ng>+\fi\f@ctech=\p@rtent%
    \relax\ifdim\mili@u<\z@\mili@u=-\mili@u\f@ctech=-\f@ctech\fi\c@@lCS}}
\ctr@ld@f\def\c@@lCS{\v@lmin=\mili@u\c@rre=-\mili@u\advance\c@rre\DemiPI@deg\v@lmax=\c@rre%
    \mili@u\@@lxxi\mili@u\divide\mili@u\@@mmmmlxviii%
    \edef\v@larg{\repdecn@mb{\mili@u}}\mili@u=-\v@larg\mili@u%
    \edef\v@lmxde{\repdecn@mb{\mili@u}}%
    \c@rre\@@lxxi\c@rre\divide\c@rre\@@mmmmlxviii%
    \edef\v@largC{\repdecn@mb{\c@rre}}\c@rre=-\v@largC\c@rre%
    \edef\v@lmxdeC{\repdecn@mb{\c@rre}}%
    \fctc@s\mili@u\v@lmin\global\result@t\p@rtent\v@leur%
    \let\t@mp=\v@larg\let\v@larg=\v@largC\let\v@largC=\t@mp%
    \let\t@mp=\v@lmxde\let\v@lmxde=\v@lmxdeC\let\v@lmxdeC=\t@mp%
    \fctc@s\c@rre\v@lmax\global\result@@t\f@ctech\v@leur}
\ctr@ld@f\def\fctc@s#1#2{\v@leur=#1\relax\ifdim#2<\@lxxxv\p@\cosser@h\else\sinser@t\fi}
\ctr@ld@f\def\cosser@h{\advance\v@leur\@lvi\p@\divide\v@leur\@lvi%
    \v@leur=\v@lmxde\v@leur\advance\v@leur\@xxx\p@%
    \v@leur=\v@lmxde\v@leur\advance\v@leur\@ccclx\p@%
    \v@leur=\v@lmxde\v@leur\advance\v@leur\@dccxx\p@\divide\v@leur\@dccxx}
\ctr@ld@f\def\sinser@t{\v@leur=\v@lmxdeC\p@\advance\v@leur\@vi\p@%
    \v@leur=\v@largC\v@leur\divide\v@leur\@vi}
\ctr@ld@f\def\red@ng#1#2{\relax\ifdim\mili@u#1#2\DemiPI@deg\advance\mili@u#2-\PI@deg%
    \p@rtent=-\p@rtent\red@ng#1#2\fi}
\ctr@ld@f\def\pr@c@lCS#1#2#3{\ctr@lcsn@m{COS@\number#3 }%
    \expandafter\xdef\csname COS@\number#3\endcsname{#1}%
    \expandafter\xdef\csname SIN@\number#3\endcsname{#2}}
\pr@c@lCS{1}{0}{0}
\pr@c@lCS{0.7071}{0.7071}{45}\pr@c@lCS{0.7071}{-0.7071}{-45}
\pr@c@lCS{0}{1}{90}          \pr@c@lCS{0}{-1}{-90}
\pr@c@lCS{-1}{0}{180}        \pr@c@lCS{-1}{0}{-180}
\pr@c@lCS{0}{-1}{270}        \pr@c@lCS{0}{1}{-270}
\ctr@ld@f\def\invers@#1#2{{\v@leur=#2\maxim@m{\v@lmax}{-\v@leur}{\v@leur}%
    \f@ctech=\@ne\m@inv@rs%
    \multiply\v@leur\f@ctech\edef\v@lv@leur{\repdecn@mb{\v@leur}}%
    \p@rtentiere{\p@rtent}{\v@leur}\v@lmin=\p@\divide\v@lmin\p@rtent%
    \inv@rs@\multiply\v@lmax\f@ctech\global\result@t=\v@lmax}#1=\result@t}
\ctr@ld@f\def\m@inv@rs{\ifdim\v@lmax<\p@\multiply\v@lmax\t@n\multiply\f@ctech\t@n\m@inv@rs\fi}
\ctr@ld@f\def\inv@rs@{\v@lmax=-\v@lmin\v@lmax=\v@lv@leur\v@lmax%
    \advance\v@lmax\tw@ pt\v@lmax=\repdecn@mb{\v@lmin}\v@lmax%
    \delt@=\v@lmax\advance\delt@-\v@lmin\ifdim\delt@<\z@\delt@=-\delt@\fi%
    \ifdim\delt@>\epsil@n\v@lmin=\v@lmax\inv@rs@\fi}
\ctr@ld@f\def\minim@m#1#2#3{\relax\ifdim#2<#3#1=#2\else#1=#3\fi}
\ctr@ld@f\def\maxim@m#1#2#3{\relax\ifdim#2>#3#1=#2\else#1=#3\fi}
\ctr@ld@f\def\p@rtentiere#1#2{#1=#2\divide#1by65536 }
\ctr@ld@f\def\r@undint#1#2{{\v@leur=#2\divide\v@leur\t@n\p@rtentiere{\p@rtent}{\v@leur}%
    \v@leur=\p@rtent pt\global\result@t=\t@n\v@leur}#1=\result@t}
\ctr@ld@f\def\sqrt@#1#2{{\v@leur=#2%
    \minim@m{\v@lmin}{\p@}{\v@leur}\maxim@m{\v@lmax}{\p@}{\v@leur}%
    \f@ctech=\@ne\m@sqrt@\sqrt@@%
    \mili@u=\v@lmin\advance\mili@u\v@lmax\divide\mili@u\tw@\multiply\mili@u\f@ctech%
    \global\result@t=\mili@u}#1=\result@t}
\ctr@ld@f\def\m@sqrt@{\ifdim\v@leur>\dcq@\divide\v@leur\c@nt\v@lmax=\v@leur%
    \multiply\f@ctech\t@n\m@sqrt@\fi}
\ctr@ld@f\def\sqrt@@{\mili@u=\v@lmin\advance\mili@u\v@lmax\divide\mili@u\tw@%
    \c@rre=\repdecn@mb{\mili@u}\mili@u%
    \ifdim\c@rre<\v@leur\v@lmin=\mili@u\else\v@lmax=\mili@u\fi%
    \delt@=\v@lmax\advance\delt@-\v@lmin\ifdim\delt@>\epsil@n\sqrt@@\fi}
\ctr@ld@f\def\extrairelepremi@r#1\de#2{\expandafter\lepremi@r#2@#1#2}
\ctr@ld@f\def\lepremi@r#1,#2@#3#4{\def#3{#1}\def#4{#2}\ignorespaces}
\ctr@ld@f\def\@cfor#1:=#2\do#3{%
  \edef\@fortemp{#2}%
  \ifx\@fortemp\empty\else\@cforloop#2,\@nil,\@nil\@@#1{#3}\fi}
\ctr@ln@m\@nextwhile
\ctr@ld@f\def\@cforloop#1,#2\@@#3#4{%
  \def#3{#1}%
  \ifx#3\Fig@nnil\let\@nextwhile=\Fig@fornoop\else#4\relax\let\@nextwhile=\@cforloop\fi%
  \@nextwhile#2\@@#3{#4}}

\ctr@ld@f\def\@ecfor#1:=#2\do#3{%
  \def\@@cfor{\@cfor#1:=}%
  \edef\@@@cfor{#2}%
  \expandafter\@@cfor\@@@cfor\do{#3}}
\ctr@ld@f\def\Fig@nnil{\@nil}
\ctr@ld@f\def\Fig@fornoop#1\@@#2#3{}
\ctr@ln@m\list@@rg
\ctr@ld@f\def\trtlis@rg#1#2{\def\list@@rg{#1}%
    \@ecfor\p@rv@l:=\list@@rg\do{\expandafter#2\p@rv@l|}}
\ctr@ln@w{newbox}\b@xvisu
\ctr@ln@w{newtoks}\let@xte
\ctr@ln@w{newif}\ifitis@K
\ctr@ln@w{newcount}\s@mme
\ctr@ln@w{newcount}\l@mbd@un \ctr@ln@w{newcount}\l@mbd@de
\ctr@ln@w{newcount}\superc@ntr@l\superc@ntr@l=\@ne        % Controle impose
\ctr@ln@w{newcount}\typec@ntr@l\typec@ntr@l=\superc@ntr@l % Controle souhaite
\ctr@ln@w{newdimen}\v@lX  \ctr@ln@w{newdimen}\v@lY  \ctr@ln@w{newdimen}\v@lZ
\ctr@ln@w{newdimen}\v@lXa \ctr@ln@w{newdimen}\v@lYa \ctr@ln@w{newdimen}\v@lZa
\ctr@ln@w{newdimen}\unit@\unit@=\p@ % Initialisation a la valeur par defaut.
\ctr@ld@f\def\unit@util{pt}
\ctr@ld@f\def\ptT@ptps{0.996264 }
\ctr@ld@f\def\ptpsT@pt{1.00375 }
\ctr@ld@f\def\ptT@unit@{1} % Initialisation correspondant a la valeur par defaut de \unit@
\ctr@ld@f\def\setunit@#1{\def\unit@util{#1}\setunit@@#1:\invers@{\result@t}{\unit@}%
    \edef\ptT@unit@{\repdecn@mb\result@t}}
\ctr@ld@f\def\setunit@@#1#2:{\ifcat#1a\unit@=\@ne#1#2\else\unit@=#1#2\fi}
\ctr@ld@f\def\d@fm@cdim#1#2{{\v@leur=#2\v@leur=\ptT@unit@\v@leur\xdef#1{\repdecn@mb\v@leur}}}
\ctr@ln@w{newif}\ifBdingB@x\BdingB@xtrue
\ctr@ln@w{newdimen}\c@@rdXmin \ctr@ln@w{newdimen}\c@@rdYmin  % Dimensions de la BoundingBox
\ctr@ln@w{newdimen}\c@@rdXmax \ctr@ln@w{newdimen}\c@@rdYmax
\ctr@ld@f\def\b@undb@x#1#2{\ifBdingB@x%
    \relax\ifdim#1<\c@@rdXmin\global\c@@rdXmin=#1\fi%
    \relax\ifdim#2<\c@@rdYmin\global\c@@rdYmin=#2\fi%
    \relax\ifdim#1>\c@@rdXmax\global\c@@rdXmax=#1\fi%
    \relax\ifdim#2>\c@@rdYmax\global\c@@rdYmax=#2\fi\fi}
\ctr@ld@f\def\b@undb@xP#1{{\Figg@tXY{#1}\b@undb@x{\v@lX}{\v@lY}}}
\ctr@ld@f\def\ellBB@x#1;#2,#3(#4,#5,#6){{\s@uvc@ntr@l\et@tellBB@x%
    \setc@ntr@l{2}\figptell-2::#1;#2,#3(#4,#6)\b@undb@xP{-2}%
    \figptell-2::#1;#2,#3(#5,#6)\b@undb@xP{-2}%
    \c@ssin{\C@}{\S@}{#6}\v@lmin=\C@ pt\v@lmax=\S@ pt%
    \mili@u=#3\v@lmin\delt@=#2\v@lmax\arct@n\v@leur(\delt@,\mili@u)%
    \mili@u=-#3\v@lmax\delt@=#2\v@lmin\arct@n\c@rre(\delt@,\mili@u)%
    \v@leur=\rdT@deg\v@leur\advance\v@leur-\DePI@deg%
    \c@rre=\rdT@deg\c@rre\advance\c@rre-\DePI@deg%
    \v@lmin=#4pt\v@lmax=#5pt%
    \loop\ifdim\v@leur<\v@lmax\ifdim\v@leur>\v@lmin%
    \edef\@ngle{\repdecn@mb\v@leur}\figptell-2::#1;#2,#3(\@ngle,#6)%
    \b@undb@xP{-2}\fi\advance\v@leur\PI@deg\repeat%
    \loop\ifdim\c@rre<\v@lmax\ifdim\c@rre>\v@lmin%
    \edef\@ngle{\repdecn@mb\c@rre}\figptell-2::#1;#2,#3(\@ngle,#6)%
    \b@undb@xP{-2}\fi\advance\c@rre\PI@deg\repeat%
    \resetc@ntr@l\et@tellBB@x}\ignorespaces}
\ctr@ld@f\def\initb@undb@x{\c@@rdXmin=\maxdimen\c@@rdYmin=\maxdimen%
    \c@@rdXmax=-\maxdimen\c@@rdYmax=-\maxdimen}
\ctr@ld@f\def\c@ntr@lnum#1{%
    \relax\ifnum\typec@ntr@l=\@ne%
    \ifnum#1<\z@%
    \immediate\write16{*** Forbidden point number (#1). Abort.}\end\fi\fi%
    \set@bjc@de{#1}}
\ctr@ln@m\objc@de
\ctr@ld@f\def\set@bjc@de#1{\edef\objc@de{@BJ\ifnum#1<\z@ M\romannumeral-#1\else\romannumeral#1\fi}}
\s@mme=\m@ne\loop\ifnum\s@mme>-19
  \set@bjc@de{\s@mme}\ctr@lcsn@m\objc@de\ctr@lcsn@m{\objc@de T}
\advance\s@mme\m@ne\repeat
\s@mme=\@ne\loop\ifnum\s@mme<6
  \set@bjc@de{\s@mme}\ctr@lcsn@m\objc@de\ctr@lcsn@m{\objc@de T}
\advance\s@mme\@ne\repeat
\ctr@ld@f\def\setc@ntr@l#1{\ifnum\superc@ntr@l>#1\typec@ntr@l=\superc@ntr@l%
    \else\typec@ntr@l=#1\fi}
\ctr@ld@f\def\resetc@ntr@l#1{\global\superc@ntr@l=#1\setc@ntr@l{#1}}
\ctr@ld@f\def\s@uvc@ntr@l#1{\edef#1{\the\superc@ntr@l}}
\ctr@ln@m\c@lproscal
\ctr@ld@f\def\c@lproscalDD#1[#2,#3]{{\Figg@tXY{#2}%
    \edef\Xu@{\repdecn@mb{\v@lX}}\edef\Yu@{\repdecn@mb{\v@lY}}\Figg@tXY{#3}%
    \global\result@t=\Xu@\v@lX\global\advance\result@t\Yu@\v@lY}#1=\result@t}
\ctr@ld@f\def\c@lproscalTD#1[#2,#3]{{\Figg@tXY{#2}\edef\Xu@{\repdecn@mb{\v@lX}}%
    \edef\Yu@{\repdecn@mb{\v@lY}}\edef\Zu@{\repdecn@mb{\v@lZ}}%
    \Figg@tXY{#3}\global\result@t=\Xu@\v@lX\global\advance\result@t\Yu@\v@lY%
    \global\advance\result@t\Zu@\v@lZ}#1=\result@t}
\ctr@ld@f\def\c@lprovec#1{%
    \det@rmC\v@lZa(\v@lX,\v@lY,\v@lmin,\v@lmax)%
    \det@rmC\v@lXa(\v@lY,\v@lZ,\v@lmax,\v@leur)%
    \det@rmC\v@lYa(\v@lZ,\v@lX,\v@leur,\v@lmin)%
    \Figv@ctCreg#1(\v@lXa,\v@lYa,\v@lZa)}
\ctr@ld@f\def\det@rm#1[#2,#3]{{\Figg@tXY{#2}\Figg@tXYa{#3}%
    \delt@=\repdecn@mb{\v@lX}\v@lYa\advance\delt@-\repdecn@mb{\v@lY}\v@lXa%
    \global\result@t=\delt@}#1=\result@t}
\ctr@ld@f\def\det@rmC#1(#2,#3,#4,#5){{\global\result@t=\repdecn@mb{#2}#5%
    \global\advance\result@t-\repdecn@mb{#3}#4}#1=\result@t}
\ctr@ld@f\def\getredf@ctDD#1(#2,#3){{\maxim@m{\v@lXa}{-#2}{#2}\maxim@m{\v@lYa}{-#3}{#3}%
    \maxim@m{\v@lXa}{\v@lXa}{\v@lYa}% \v@lXa = ||X||inf
    \ifdim\v@lXa>\@xci pt\divide\v@lXa\@xci%
    \p@rtentiere{\p@rtent}{\v@lXa}\advance\p@rtent\@ne\else\p@rtent=\@ne\fi%
    \global\result@tent=\p@rtent}#1=\result@tent\ignorespaces}
\ctr@ld@f\def\getredf@ctTD#1(#2,#3,#4){{\maxim@m{\v@lXa}{-#2}{#2}\maxim@m{\v@lYa}{-#3}{#3}%
    \maxim@m{\v@lZa}{-#4}{#4}\maxim@m{\v@lXa}{\v@lXa}{\v@lYa}%
    \maxim@m{\v@lXa}{\v@lXa}{\v@lZa}% \v@lXa = ||X||inf
    \ifdim\v@lXa>\@lxxiv pt\divide\v@lXa\@lxxiv%
    \p@rtentiere{\p@rtent}{\v@lXa}\advance\p@rtent\@ne\else\p@rtent=\@ne\fi%
    \global\result@tent=\p@rtent}#1=\result@tent\ignorespaces}
\ctr@ld@f\def\FigptintercircB@zDD#1:#2:#3,#4[#5,#6,#7,#8]{{\s@uvc@ntr@l\et@tfigptintercircB@zDD%
    \setc@ntr@l{2}\figvectPDD-1[#5,#8]\Figg@tXY{-1}\getredf@ctDD\f@ctech(\v@lX,\v@lY)%
    \mili@u=#4\unit@\divide\mili@u\f@ctech\c@rre=\repdecn@mb{\mili@u}\mili@u%
    \figptBezierDD-5::#3[#5,#6,#7,#8]%
    \v@lmin=#3\p@\v@lmax=\v@lmin\advance\v@lmax0.1\p@%
    \loop\edef\T@{\repdecn@mb{\v@lmax}}\figptBezierDD-2::\T@[#5,#6,#7,#8]%
    \figvectPDD-1[-5,-2]\n@rmeucCDD{\delt@}{-1}\ifdim\delt@<\c@rre\v@lmin=\v@lmax%
    \advance\v@lmax0.1\p@\repeat%
    \loop\mili@u=\v@lmin\advance\mili@u\v@lmax%
    \divide\mili@u\tw@\edef\T@{\repdecn@mb{\mili@u}}\figptBezierDD-2::\T@[#5,#6,#7,#8]%
    \figvectPDD-1[-5,-2]\n@rmeucCDD{\delt@}{-1}\ifdim\delt@>\c@rre\v@lmax=\mili@u%
    \else\v@lmin=\mili@u\fi\v@leur=\v@lmax\advance\v@leur-\v@lmin%
    \ifdim\v@leur>\epsil@n\repeat\figptcopyDD#1:#2/-2/%
    \resetc@ntr@l\et@tfigptintercircB@zDD}\ignorespaces}
\ctr@ln@m\figptinterlines
\ctr@ld@f\def\inters@cDD#1:#2[#3,#4;#5,#6]{{\s@uvc@ntr@l\et@tinters@cDD%
    \setc@ntr@l{2}\vecunit@{-1}{#4}\vecunit@{-2}{#6}%
    \Figg@tXY{-1}\setc@ntr@l{1}\Figg@tXYa{#3}%
    \edef\A@{\repdecn@mb{\v@lX}}\edef\B@{\repdecn@mb{\v@lY}}%
    \v@lmin=\B@\v@lXa\advance\v@lmin-\A@\v@lYa%
    \Figg@tXYa{#5}\setc@ntr@l{2}\Figg@tXY{-2}%
    \edef\C@{\repdecn@mb{\v@lX}}\edef\D@{\repdecn@mb{\v@lY}}%
    \v@lmax=\D@\v@lXa\advance\v@lmax-\C@\v@lYa%
    \delt@=\A@\v@lY\advance\delt@-\B@\v@lX%
    \invers@{\v@leur}{\delt@}\edef\v@ldelta{\repdecn@mb{\v@leur}}%
    \v@lXa=\A@\v@lmax\advance\v@lXa-\C@\v@lmin%
    \v@lYa=\B@\v@lmax\advance\v@lYa-\D@\v@lmin%
    \v@lXa=\v@ldelta\v@lXa\v@lYa=\v@ldelta\v@lYa%
    \setc@ntr@l{1}\Figp@intregDD#1:{#2}(\v@lXa,\v@lYa)%
    \resetc@ntr@l\et@tinters@cDD}\ignorespaces}
\ctr@ld@f\def\inters@cTD#1:#2[#3,#4;#5,#6]{{\s@uvc@ntr@l\et@tinters@cTD%
    \setc@ntr@l{2}\figvectNVTD-1[#4,#6]\figvectNVTD-2[#6,-1]\figvectPTD-1[#3,#5]%
    \r@pPSTD\v@leur[-2,-1,#4]\edef\v@lcoef{\repdecn@mb{\v@leur}}%
    \figpttraTD#1:{#2}=#3/\v@lcoef,#4/\resetc@ntr@l\et@tinters@cTD}\ignorespaces}
\ctr@ld@f\def\r@pPSTD#1[#2,#3,#4]{{\Figg@tXY{#2}\edef\Xu@{\repdecn@mb{\v@lX}}%
    \edef\Yu@{\repdecn@mb{\v@lY}}\edef\Zu@{\repdecn@mb{\v@lZ}}%
    \Figg@tXY{#3}\v@lmin=\Xu@\v@lX\advance\v@lmin\Yu@\v@lY\advance\v@lmin\Zu@\v@lZ%
    \Figg@tXY{#4}\v@lmax=\Xu@\v@lX\advance\v@lmax\Yu@\v@lY\advance\v@lmax\Zu@\v@lZ%
    \invers@{\v@leur}{\v@lmax}\global\result@t=\repdecn@mb{\v@leur}\v@lmin}%
    #1=\result@t}
\ctr@ln@m\n@rminf
\ctr@ld@f\def\n@rminfDD#1#2{{\Figg@tXY{#2}\maxim@m{\v@lX}{\v@lX}{-\v@lX}%
    \maxim@m{\v@lY}{\v@lY}{-\v@lY}\maxim@m{\global\result@t}{\v@lX}{\v@lY}}%
    #1=\result@t}
\ctr@ld@f\def\n@rminfTD#1#2{{\Figg@tXY{#2}\maxim@m{\v@lX}{\v@lX}{-\v@lX}%
    \maxim@m{\v@lY}{\v@lY}{-\v@lY}\maxim@m{\v@lZ}{\v@lZ}{-\v@lZ}%
    \maxim@m{\v@lX}{\v@lX}{\v@lY}\maxim@m{\global\result@t}{\v@lX}{\v@lZ}}%
    #1=\result@t}
\ctr@ld@f\def\n@rmeucCDD#1#2{\Figg@tXY{#2}\divide\v@lX\f@ctech\divide\v@lY\f@ctech%
    #1=\repdecn@mb{\v@lX}\v@lX\v@lX=\repdecn@mb{\v@lY}\v@lY\advance#1\v@lX}
\ctr@ld@f\def\n@rmeucCTD#1#2{\Figg@tXY{#2}%
    \divide\v@lX\f@ctech\divide\v@lY\f@ctech\divide\v@lZ\f@ctech%
    #1=\repdecn@mb{\v@lX}\v@lX\v@lX=\repdecn@mb{\v@lY}\v@lY\advance#1\v@lX%
    \v@lX=\repdecn@mb{\v@lZ}\v@lZ\advance#1\v@lX}
\ctr@ln@m\n@rmeucSV
\ctr@ld@f\def\n@rmeucSVDD#1#2{{\Figg@tXY{#2}%
    \v@lXa=\repdecn@mb{\v@lX}\v@lX\v@lYa=\repdecn@mb{\v@lY}\v@lY%
    \advance\v@lXa\v@lYa\sqrt@{\global\result@t}{\v@lXa}}#1=\result@t}
\ctr@ld@f\def\n@rmeucSVTD#1#2{{\Figg@tXY{#2}\v@lXa=\repdecn@mb{\v@lX}\v@lX%
    \v@lYa=\repdecn@mb{\v@lY}\v@lY\v@lZa=\repdecn@mb{\v@lZ}\v@lZ%
    \advance\v@lXa\v@lYa\advance\v@lXa\v@lZa\sqrt@{\global\result@t}{\v@lXa}}#1=\result@t}
\ctr@ln@m\n@rmeuc
\ctr@ld@f\def\n@rmeucDD#1#2{{\Figg@tXY{#2}\getredf@ctDD\f@ctech(\v@lX,\v@lY)%
    \divide\v@lX\f@ctech\divide\v@lY\f@ctech%
    \v@lXa=\repdecn@mb{\v@lX}\v@lX\v@lYa=\repdecn@mb{\v@lY}\v@lY%
    \advance\v@lXa\v@lYa\sqrt@{\global\result@t}{\v@lXa}%
    \global\multiply\result@t\f@ctech}#1=\result@t}
\ctr@ld@f\def\n@rmeucTD#1#2{{\Figg@tXY{#2}\getredf@ctTD\f@ctech(\v@lX,\v@lY,\v@lZ)%
    \divide\v@lX\f@ctech\divide\v@lY\f@ctech\divide\v@lZ\f@ctech%
    \v@lXa=\repdecn@mb{\v@lX}\v@lX%
    \v@lYa=\repdecn@mb{\v@lY}\v@lY\v@lZa=\repdecn@mb{\v@lZ}\v@lZ%
    \advance\v@lXa\v@lYa\advance\v@lXa\v@lZa\sqrt@{\global\result@t}{\v@lXa}%
    \global\multiply\result@t\f@ctech}#1=\result@t}
\ctr@ln@m\vecunit@
\ctr@ld@f\def\vecunit@DD#1#2{{\Figg@tXY{#2}\getredf@ctDD\f@ctech(\v@lX,\v@lY)%
    \divide\v@lX\f@ctech\divide\v@lY\f@ctech%
    \Figv@ctCreg#1(\v@lX,\v@lY)\n@rmeucSV{\v@lYa}{#1}%
    \invers@{\v@lXa}{\v@lYa}\edef\v@lv@lXa{\repdecn@mb{\v@lXa}}%
    \v@lX=\v@lv@lXa\v@lX\v@lY=\v@lv@lXa\v@lY%
    \Figv@ctCreg#1(\v@lX,\v@lY)\multiply\v@lYa\f@ctech\global\result@t=\v@lYa}}
\ctr@ld@f\def\vecunit@TD#1#2{{\Figg@tXY{#2}\getredf@ctTD\f@ctech(\v@lX,\v@lY,\v@lZ)%
    \divide\v@lX\f@ctech\divide\v@lY\f@ctech\divide\v@lZ\f@ctech%
    \Figv@ctCreg#1(\v@lX,\v@lY,\v@lZ)\n@rmeucSV{\v@lYa}{#1}%
    \invers@{\v@lXa}{\v@lYa}\edef\v@lv@lXa{\repdecn@mb{\v@lXa}}%
    \v@lX=\v@lv@lXa\v@lX\v@lY=\v@lv@lXa\v@lY\v@lZ=\v@lv@lXa\v@lZ%
    \Figv@ctCreg#1(\v@lX,\v@lY,\v@lZ)\multiply\v@lYa\f@ctech\global\result@t=\v@lYa}}
\ctr@ld@f\def\vecunitC@TD[#1,#2]{\Figg@tXYa{#1}\Figg@tXY{#2}%
    \advance\v@lX-\v@lXa\advance\v@lY-\v@lYa\advance\v@lZ-\v@lZa\c@lvecunitTD}
\ctr@ld@f\def\vecunitCV@TD#1{\Figg@tXY{#1}\c@lvecunitTD}
\ctr@ld@f\def\c@lvecunitTD{\getredf@ctTD\f@ctech(\v@lX,\v@lY,\v@lZ)%
    \divide\v@lX\f@ctech\divide\v@lY\f@ctech\divide\v@lZ\f@ctech%
    \v@lXa=\repdecn@mb{\v@lX}\v@lX%
    \v@lYa=\repdecn@mb{\v@lY}\v@lY\v@lZa=\repdecn@mb{\v@lZ}\v@lZ%
    \advance\v@lXa\v@lYa\advance\v@lXa\v@lZa\sqrt@{\v@lYa}{\v@lXa}%
    \invers@{\v@lXa}{\v@lYa}\edef\v@lv@lXa{\repdecn@mb{\v@lXa}}%
    \v@lX=\v@lv@lXa\v@lX\v@lY=\v@lv@lXa\v@lY\v@lZ=\v@lv@lXa\v@lZ}
\ctr@ln@m\figgetangle
\ctr@ld@f\def\figgetangleDD#1[#2,#3,#4]{\ifps@cri{\s@uvc@ntr@l\et@tfiggetangleDD\setc@ntr@l{2}%
    \figvectPDD-1[#2,#3]\figvectPDD-2[#2,#4]\vecunit@{-1}{-1}%
    \c@lproscalDD\delt@[-2,-1]\figvectNVDD-1[-1]\c@lproscalDD\v@leur[-2,-1]%
    \arct@n\v@lmax(\delt@,\v@leur)\v@lmax=\rdT@deg\v@lmax%
    \ifdim\v@lmax<\z@\advance\v@lmax\DePI@deg\fi\xdef#1{\repdecn@mb{\v@lmax}}%
    \resetc@ntr@l\et@tfiggetangleDD}\ignorespaces\fi}
\ctr@ld@f\def\figgetangleTD#1[#2,#3,#4,#5]{\ifps@cri{\s@uvc@ntr@l\et@tfiggetangleTD\setc@ntr@l{2}%
    \figvectPTD-1[#2,#3]\figvectPTD-2[#2,#5]\figvectNVTD-3[-1,-2]%
    \figvectPTD-2[#2,#4]\figvectNVTD-4[-3,-1]%
    \vecunit@{-1}{-1}\c@lproscalTD\delt@[-2,-1]\c@lproscalTD\v@leur[-2,-4]%
    \arct@n\v@lmax(\delt@,\v@leur)\v@lmax=\rdT@deg\v@lmax%
    \ifdim\v@lmax<\z@\advance\v@lmax\DePI@deg\fi\xdef#1{\repdecn@mb{\v@lmax}}%
    \resetc@ntr@l\et@tfiggetangleTD}\ignorespaces\fi}    
\ctr@ld@f\def\figgetdist#1[#2,#3]{\ifps@cri{\s@uvc@ntr@l\et@tfiggetdist\setc@ntr@l{2}%
    \figvectP-1[#2,#3]\n@rmeuc{\v@lX}{-1}\v@lX=\ptT@unit@\v@lX\xdef#1{\repdecn@mb{\v@lX}}%
    \resetc@ntr@l\et@tfiggetdist}\ignorespaces\fi}
\ctr@ld@f\def\Figg@tT#1{\c@ntr@lnum{#1}%
    {\expandafter\expandafter\expandafter\extr@ctT\csname\objc@de\endcsname:%
     \ifnum\B@@ltxt=\z@\ptn@me{#1}\else\csname\objc@de T\endcsname\fi}}
\ctr@ld@f\def\extr@ctT#1,#2,#3/#4:{\def\B@@ltxt{#3}}
\ctr@ld@f\def\Figg@tXY#1{\c@ntr@lnum{#1}%
    \expandafter\expandafter\expandafter\extr@ctC\csname\objc@de\endcsname:}
\ctr@ln@m\extr@ctC
\ctr@ld@f\def\extr@ctCDD#1/#2,#3,#4:{\v@lX=#2\v@lY=#3}
\ctr@ld@f\def\extr@ctCTD#1/#2,#3,#4:{\v@lX=#2\v@lY=#3\v@lZ=#4}
\ctr@ld@f\def\Figg@tXYa#1{\c@ntr@lnum{#1}%
    \expandafter\expandafter\expandafter\extr@ctCa\csname\objc@de\endcsname:}
\ctr@ln@m\extr@ctCa
\ctr@ld@f\def\extr@ctCaDD#1/#2,#3,#4:{\v@lXa=#2\v@lYa=#3}
\ctr@ld@f\def\extr@ctCaTD#1/#2,#3,#4:{\v@lXa=#2\v@lYa=#3\v@lZa=#4}
\ctr@ln@m\t@xt@
\ctr@ld@f\def\figinit#1{\t@stc@tcodech@nge\initpr@lim\Figinit@#1,:\initpss@ttings\ignorespaces}
\ctr@ld@f\def\Figinit@#1,#2:{\setunit@{#1}\def\t@xt@{#2}\ifx\t@xt@\empty\else\Figinit@@#2:\fi}
\ctr@ld@f\def\Figinit@@#1#2:{\if#12 \else\Figs@tproj{#1}\initTD@\fi}
\ctr@ln@w{newif}\ifTr@isDim
\ctr@ld@f\def\UnD@fined{UNDEFINED}
\ctr@ld@f\def\ifundefined#1{\expandafter\ifx\csname#1\endcsname\relax}
\ctr@ln@m\@utoFN
\ctr@ln@m\@utoFInDone
\ctr@ln@m\disob@unit
\ctr@ld@f\def\initpr@lim{\initb@undb@x\figsetmark{}\figsetptname{$A_{##1}$}\def\Sc@leFact{1}%
    \initDD@\figsetroundcoord{yes}\ps@critrue\expandafter\setupd@te\defaultupdate:%
    \edef\disob@unit{\UnD@fined}\edef\t@rgetpt{\UnD@fined}\gdef\@utoFInDone{1}\gdef\@utoFN{0}}
\ctr@ld@f\def\initDD@{\Tr@isDimfalse%
    \ifPDFm@ke%
     \let\Ps@rcerc=\Ps@rcercBz%
     \let\Ps@rell=\Ps@rellBz%
    \fi
    \let\c@lDCUn=\c@lDCUnDD%
    \let\c@lDCDeux=\c@lDCDeuxDD%
    \let\c@ldefproj=\relax%
    \let\c@lproscal=\c@lproscalDD%
    \let\c@lprojSP=\relax%
    \let\extr@ctC=\extr@ctCDD%
    \let\extr@ctCa=\extr@ctCaDD%
    \let\extr@ctCF=\extr@ctCFDD%
    \let\Figp@intreg=\Figp@intregDD%
    \let\Figpts@xes=\Figpts@xesDD%
    \let\n@rmeucSV=\n@rmeucSVDD\let\n@rmeuc=\n@rmeucDD\let\n@rminf=\n@rminfDD%
    \let\pr@dMatV=\pr@dMatVDD%
    \let\ps@xes=\ps@xesDD%
    \let\vecunit@=\vecunit@DD%
    \let\figcoord=\figcoordDD%
    \let\figgetangle=\figgetangleDD%
    \let\figpt=\figptDD%
    \let\figptBezier=\figptBezierDD%
    \let\figptbary=\figptbaryDD%
    \let\figptcirc=\figptcircDD%
    \let\figptcircumcenter=\figptcircumcenterDD%
    \let\figptcopy=\figptcopyDD%
    \let\figptcurvcenter=\figptcurvcenterDD%
    \let\figptell=\figptellDD%
    \let\figptendnormal=\figptendnormalDD%
    \let\figptinterlineplane=\figptinterlineplaneDD%
    \let\figptinterlines=\inters@cDD%
    \let\figptorthocenter=\figptorthocenterDD%
    \let\figptorthoprojline=\figptorthoprojlineDD%
    \let\figptorthoprojplane=\figptorthoprojplaneDD%
    \let\figptrot=\figptrotDD%
    \let\figptscontrol=\figptscontrolDD%
    \let\figptsintercirc=\figptsintercircDD%
    \let\figptsinterlinell=\figptsinterlinellDD%
    \let\figptsorthoprojline=\figptsorthoprojlineDD%
    \let\figptorthoprojplane=\figptorthoprojplaneDD%
    \let\figptsrot=\figptsrotDD%
    \let\figptssym=\figptssymDD%
    \let\figptstra=\figptstraDD%
    \let\figptsym=\figptsymDD%
    \let\figpttraC=\figpttraCDD%
    \let\figpttra=\figpttraDD%
    \let\figptvisilimSL=\figptvisilimSLDD%
    \let\figsetobdist=\figsetobdistDD%
    \let\figsettarget=\figsettargetDD%
    \let\figsetview=\figsetviewDD%
    \let\figvectDBezier=\figvectDBezierDD%
    \let\figvectN=\figvectNDD%
    \let\figvectNV=\figvectNVDD%
    \let\figvectP=\figvectPDD%
    \let\figvectU=\figvectUDD%
    \let\psarccircP=\psarccircPDD%
    \let\psarccirc=\psarccircDD%
    \let\psarcell=\psarcellDD%
    \let\psarcellPA=\psarcellPADD%
    \let\psarrowBezier=\psarrowBezierDD%
    \let\psarrowcircP=\psarrowcircPDD%
    \let\psarrowcirc=\psarrowcircDD%
    \let\psarrowhead=\psarrowheadDD%
    \let\psarrow=\psarrowDD%
    \let\psBezier=\psBezierDD%
    \let\pscirc=\pscircDD%
    \let\pscurve=\pscurveDD%
    \let\psnormal=\psnormalDD%
    }
\ctr@ld@f\def\initTD@{\Tr@isDimtrue\initb@undb@xTD\newt@rgetptfalse\newdis@bfalse%
    \let\c@lDCUn=\c@lDCUnTD%
    \let\c@lDCDeux=\c@lDCDeuxTD%
    \let\c@ldefproj=\c@ldefprojTD%
    \let\c@lproscal=\c@lproscalTD%
    \let\extr@ctC=\extr@ctCTD%
    \let\extr@ctCa=\extr@ctCaTD%
    \let\extr@ctCF=\extr@ctCFTD%
    \let\Figp@intreg=\Figp@intregTD%
    \let\Figpts@xes=\Figpts@xesTD%
    \let\n@rmeucSV=\n@rmeucSVTD\let\n@rmeuc=\n@rmeucTD\let\n@rminf=\n@rminfTD%
    \let\pr@dMatV=\pr@dMatVTD%
    \let\ps@xes=\ps@xesTD%
    \let\vecunit@=\vecunit@TD%
    \let\figcoord=\figcoordTD%
    \let\figgetangle=\figgetangleTD%
    \let\figpt=\figptTD%
    \let\figptBezier=\figptBezierTD%
    \let\figptbary=\figptbaryTD%
    \let\figptcirc=\figptcircTD%
    \let\figptcircumcenter=\figptcircumcenterTD%
    \let\figptcopy=\figptcopyTD%
    \let\figptcurvcenter=\figptcurvcenterTD%
    \let\figptinterlineplane=\figptinterlineplaneTD%
    \let\figptinterlines=\inters@cTD%
    \let\figptorthocenter=\figptorthocenterTD%
    \let\figptorthoprojline=\figptorthoprojlineTD%
    \let\figptorthoprojplane=\figptorthoprojplaneTD%
    \let\figptrot=\figptrotTD%
    \let\figptscontrol=\figptscontrolTD%
    \let\figptsintercirc=\figptsintercircTD%
    \let\figptsorthoprojline=\figptsorthoprojlineTD%
    \let\figptsorthoprojplane=\figptsorthoprojplaneTD%
    \let\figptsrot=\figptsrotTD%
    \let\figptssym=\figptssymTD%
    \let\figptstra=\figptstraTD%
    \let\figptsym=\figptsymTD%
    \let\figpttraC=\figpttraCTD%
    \let\figpttra=\figpttraTD%
    \let\figptvisilimSL=\figptvisilimSLTD%
    \let\figsetobdist=\figsetobdistTD%
    \let\figsettarget=\figsettargetTD%
    \let\figsetview=\figsetviewTD%
    \let\figvectDBezier=\figvectDBezierTD%
    \let\figvectN=\figvectNTD%
    \let\figvectNV=\figvectNVTD%
    \let\figvectP=\figvectPTD%
    \let\figvectU=\figvectUTD%
    \let\psarccircP=\psarccircPTD%
    \let\psarccirc=\psarccircTD%
    \let\psarcell=\psarcellTD%
    \let\psarcellPA=\psarcellPATD%
    \let\psarrowBezier=\psarrowBezierTD%
    \let\psarrowcircP=\psarrowcircPTD%
    \let\psarrowcirc=\psarrowcircTD%
    \let\psarrowhead=\psarrowheadTD%
    \let\psarrow=\psarrowTD%
    \let\psBezier=\psBezierTD%
    \let\pscirc=\pscircTD%
    \let\pscurve=\pscurveTD%
    }
\ctr@ld@f\def\un@v@ilable#1{\immediate\write16{*** The macro #1 is not available in the current context.}}
\ctr@ld@f\def\figinsert#1{{\def\t@xt@{#1}\relax%
    \ifx\t@xt@\empty\ifnum\@utoFInDone>\z@\Figinsert@\DefGIfilen@me,:\fi%
    \else\expandafter\FiginsertNu@#1 :\fi}\ignorespaces}
\ctr@ld@f\def\FiginsertNu@#1 #2:{\def\t@xt@{#1}\relax\ifx\t@xt@\empty\def\t@xt@{#2}%
    \ifx\t@xt@\empty\ifnum\@utoFInDone>\z@\Figinsert@\DefGIfilen@me,:\fi%
    \else\FiginsertNu@#2:\fi\else\expandafter\FiginsertNd@#1 #2:\fi}
\ctr@ld@f\def\FiginsertNd@#1#2:{\ifcat#1a\Figinsert@#1#2,:\else%
    \ifnum\@utoFInDone>\z@\Figinsert@\DefGIfilen@me,#1#2,:\fi\fi}
\ctr@ln@m\Sc@leFact
\ctr@ld@f\def\Figinsert@#1,#2:{\def\t@xt@{#2}\ifx\t@xt@\empty\xdef\Sc@leFact{1}\else%
    \X@rgdeux@#2\xdef\Sc@leFact{\@rgdeux}\fi%
    \Figdisc@rdLTS{#1}{\t@xt@}\@psfgetbb{\t@xt@}%
    \v@lX=\@psfllx\p@\v@lX=\ptpsT@pt\v@lX\v@lX=\Sc@leFact\v@lX%
    \v@lY=\@psflly\p@\v@lY=\ptpsT@pt\v@lY\v@lY=\Sc@leFact\v@lY%
    \b@undb@x{\v@lX}{\v@lY}%
    \v@lX=\@psfurx\p@\v@lX=\ptpsT@pt\v@lX\v@lX=\Sc@leFact\v@lX%
    \v@lY=\@psfury\p@\v@lY=\ptpsT@pt\v@lY\v@lY=\Sc@leFact\v@lY%
    \b@undb@x{\v@lX}{\v@lY}%
    \ifPDFm@ke\Figinclud@PDF{\t@xt@}{\Sc@leFact}\else%
    \v@lX=\c@nt pt\v@lX=\Sc@leFact\v@lX\edef\F@ct{\repdecn@mb{\v@lX}}%
    \ifx\TeXturesonMacOSltX\special{postscriptfile #1 vscale=\F@ct\space hscale=\F@ct}%
    \else\includegraphics{#1}\fi\fi%
    \message{[\t@xt@]}\ignorespaces}
\ctr@ld@f\def\Figdisc@rdLTS#1#2{\expandafter\Figdisc@rdLTS@#1 :#2}
\ctr@ld@f\def\Figdisc@rdLTS@#1 #2:#3{\def#3{#1}\relax\ifx#3\empty\expandafter\Figdisc@rdLTS@#2:#3\fi}
\ctr@ld@f\def\figinsertE#1{\FiginsertE@#1,:\ignorespaces}
\ctr@ld@f\def\FiginsertE@#1,#2:{{\def\t@xt@{#2}\ifx\t@xt@\empty\xdef\Sc@leFact{1}\else%
    \X@rgdeux@#2\xdef\Sc@leFact{\@rgdeux}\fi%
    \Figdisc@rdLTS{#1}{\t@xt@}\pdfximage{\t@xt@}%
    \setbox\Gb@x=\hbox{\pdfrefximage\pdflastximage}%
    \v@lX=\z@\v@lY=-\Sc@leFact\dp\Gb@x\b@undb@x{\v@lX}{\v@lY}%
    \advance\v@lX\Sc@leFact\wd\Gb@x\advance\v@lY\Sc@leFact\dp\Gb@x%
    \advance\v@lY\Sc@leFact\ht\Gb@x\b@undb@x{\v@lX}{\v@lY}%
    \v@lX=\Sc@leFact\wd\Gb@x\pdfximage width \v@lX {\t@xt@}%
    \rlap{\pdfrefximage\pdflastximage}\message{[\t@xt@]}}\ignorespaces}
\ctr@ld@f\def\X@rgdeux@#1,{\edef\@rgdeux{#1}}
\ctr@ln@m\figpt
\ctr@ld@f\def\figptDD#1:#2(#3,#4){\ifps@cri\c@ntr@lnum{#1}%
    {\v@lX=#3\unit@\v@lY=#4\unit@\Fig@dmpt{#2}{\z@}}\ignorespaces\fi}
\ctr@ld@f\def\Fig@dmpt#1#2{\def\t@xt@{#1}\ifx\t@xt@\empty\def\B@@ltxt{\z@}%
    \else\expandafter\gdef\csname\objc@de T\endcsname{#1}\def\B@@ltxt{\@ne}\fi%
    \expandafter\xdef\csname\objc@de\endcsname{\ifitis@vect@r\C@dCl@svect%
    \else\C@dCl@spt\fi,\z@,\B@@ltxt/\the\v@lX,\the\v@lY,#2}}
\ctr@ld@f\def\C@dCl@spt{P}
\ctr@ld@f\def\C@dCl@svect{V}
\ctr@ln@m\c@@rdYZ
\ctr@ln@m\c@@rdY
\ctr@ld@f\def\figptTD#1:#2(#3,#4){\ifps@cri\c@ntr@lnum{#1}%
    \def\c@@rdYZ{#4,0,0}\extrairelepremi@r\c@@rdY\de\c@@rdYZ%
    \extrairelepremi@r\c@@rdZ\de\c@@rdYZ%
    {\v@lX=#3\unit@\v@lY=\c@@rdY\unit@\v@lZ=\c@@rdZ\unit@\Fig@dmpt{#2}{\the\v@lZ}%
    \b@undb@xTD{\v@lX}{\v@lY}{\v@lZ}}\ignorespaces\fi}
\ctr@ln@m\Figp@intreg
\ctr@ld@f\def\Figp@intregDD#1:#2(#3,#4){\c@ntr@lnum{#1}%
    {\result@t=#4\v@lX=#3\v@lY=\result@t\Fig@dmpt{#2}{\z@}}\ignorespaces}
\ctr@ld@f\def\Figp@intregTD#1:#2(#3,#4){\c@ntr@lnum{#1}%
    \def\c@@rdYZ{#4,\z@,\z@}\extrairelepremi@r\c@@rdY\de\c@@rdYZ%
    \extrairelepremi@r\c@@rdZ\de\c@@rdYZ%
    {\v@lX=#3\v@lY=\c@@rdY\v@lZ=\c@@rdZ\Fig@dmpt{#2}{\the\v@lZ}%
    \b@undb@xTD{\v@lX}{\v@lY}{\v@lZ}}\ignorespaces}
\ctr@ln@m\figptBezier
\ctr@ld@f\def\figptBezierDD#1:#2:#3[#4,#5,#6,#7]{\ifps@cri{\s@uvc@ntr@l\et@tfigptBezierDD%
    \FigptBezier@#3[#4,#5,#6,#7]\Figp@intregDD#1:{#2}(\v@lX,\v@lY)%
    \resetc@ntr@l\et@tfigptBezierDD}\ignorespaces\fi}
\ctr@ld@f\def\figptBezierTD#1:#2:#3[#4,#5,#6,#7]{\ifps@cri{\s@uvc@ntr@l\et@tfigptBezierTD%
    \FigptBezier@#3[#4,#5,#6,#7]\Figp@intregTD#1:{#2}(\v@lX,\v@lY,\v@lZ)%
    \resetc@ntr@l\et@tfigptBezierTD}\ignorespaces\fi}
\ctr@ld@f\def\FigptBezier@#1[#2,#3,#4,#5]{\setc@ntr@l{2}%
    \edef\T@{#1}\v@leur=\p@\advance\v@leur-#1pt\edef\UNmT@{\repdecn@mb{\v@leur}}%
    \figptcopy-4:/#2/\figptcopy-3:/#3/\figptcopy-2:/#4/\figptcopy-1:/#5/%
    \l@mbd@un=-4 \l@mbd@de=-\thr@@\p@rtent=\m@ne\c@lDecast%
    \l@mbd@un=-4 \l@mbd@de=-\thr@@\p@rtent=-\tw@\c@lDecast%
    \l@mbd@un=-4 \l@mbd@de=-\thr@@\p@rtent=-\thr@@\c@lDecast\Figg@tXY{-4}}
\ctr@ln@m\c@lDCUn
\ctr@ld@f\def\c@lDCUnDD#1#2{\Figg@tXY{#1}\v@lX=\UNmT@\v@lX\v@lY=\UNmT@\v@lY%
    \Figg@tXYa{#2}\advance\v@lX\T@\v@lXa\advance\v@lY\T@\v@lYa%
    \Figp@intregDD#1:(\v@lX,\v@lY)}
\ctr@ld@f\def\c@lDCUnTD#1#2{\Figg@tXY{#1}\v@lX=\UNmT@\v@lX\v@lY=\UNmT@\v@lY\v@lZ=\UNmT@\v@lZ%
    \Figg@tXYa{#2}\advance\v@lX\T@\v@lXa\advance\v@lY\T@\v@lYa\advance\v@lZ\T@\v@lZa%
    \Figp@intregTD#1:(\v@lX,\v@lY,\v@lZ)}
\ctr@ld@f\def\c@lDecast{\relax\ifnum\l@mbd@un<\p@rtent\c@lDCUn{\l@mbd@un}{\l@mbd@de}%
    \advance\l@mbd@un\@ne\advance\l@mbd@de\@ne\c@lDecast\fi}
\ctr@ld@f\def\figptmap#1:#2=#3/#4/#5/{\ifps@cri{\s@uvc@ntr@l\et@tfigptmap%
    \setc@ntr@l{2}\figvectP-1[#4,#3]\Figg@tXY{-1}%
    \pr@dMatV/#5/\figpttra#1:{#2}=#4/1,-1/%
    \resetc@ntr@l\et@tfigptmap}\ignorespaces\fi}
\ctr@ln@m\pr@dMatV
\ctr@ld@f\def\pr@dMatVDD/#1,#2;#3,#4/{\v@lXa=#1\v@lX\advance\v@lXa#2\v@lY%
    \v@lYa=#3\v@lX\advance\v@lYa#4\v@lY\Figv@ctCreg-1(\v@lXa,\v@lYa)}
\ctr@ld@f\def\pr@dMatVTD/#1,#2,#3;#4,#5,#6;#7,#8,#9/{%
    \v@lXa=#1\v@lX\advance\v@lXa#2\v@lY\advance\v@lXa#3\v@lZ%
    \v@lYa=#4\v@lX\advance\v@lYa#5\v@lY\advance\v@lYa#6\v@lZ%
    \v@lZa=#7\v@lX\advance\v@lZa#8\v@lY\advance\v@lZa#9\v@lZ%
    \Figv@ctCreg-1(\v@lXa,\v@lYa,\v@lZa)}
\ctr@ln@m\figptbary
\ctr@ld@f\def\figptbaryDD#1:#2[#3;#4]{\ifps@cri{\edef\list@num{#3}\extrairelepremi@r\p@int\de\list@num%
    \s@mme=\z@\@ecfor\c@ef:=#4\do{\advance\s@mme\c@ef}%
    \edef\listec@ef{#4,0}\extrairelepremi@r\c@ef\de\listec@ef%
    \Figg@tXY{\p@int}\divide\v@lX\s@mme\divide\v@lY\s@mme%
    \multiply\v@lX\c@ef\multiply\v@lY\c@ef%
    \@ecfor\p@int:=\list@num\do{\extrairelepremi@r\c@ef\de\listec@ef%
           \Figg@tXYa{\p@int}\divide\v@lXa\s@mme\divide\v@lYa\s@mme%
           \multiply\v@lXa\c@ef\multiply\v@lYa\c@ef%
           \advance\v@lX\v@lXa\advance\v@lY\v@lYa}%
    \Figp@intregDD#1:{#2}(\v@lX,\v@lY)}\ignorespaces\fi}
\ctr@ld@f\def\figptbaryTD#1:#2[#3;#4]{\ifps@cri{\edef\list@num{#3}\extrairelepremi@r\p@int\de\list@num%
    \s@mme=\z@\@ecfor\c@ef:=#4\do{\advance\s@mme\c@ef}%
    \edef\listec@ef{#4,0}\extrairelepremi@r\c@ef\de\listec@ef%
    \Figg@tXY{\p@int}\divide\v@lX\s@mme\divide\v@lY\s@mme\divide\v@lZ\s@mme%
    \multiply\v@lX\c@ef\multiply\v@lY\c@ef\multiply\v@lZ\c@ef%
    \@ecfor\p@int:=\list@num\do{\extrairelepremi@r\c@ef\de\listec@ef%
           \Figg@tXYa{\p@int}\divide\v@lXa\s@mme\divide\v@lYa\s@mme\divide\v@lZa\s@mme%
           \multiply\v@lXa\c@ef\multiply\v@lYa\c@ef\multiply\v@lZa\c@ef%
           \advance\v@lX\v@lXa\advance\v@lY\v@lYa\advance\v@lZ\v@lZa}%
    \Figp@intregTD#1:{#2}(\v@lX,\v@lY,\v@lZ)}\ignorespaces\fi}
\ctr@ld@f\def\figptbaryR#1:#2[#3;#4]{\ifps@cri{%
    \v@leur=\z@\@ecfor\c@ef:=#4\do{\maxim@m{\v@lmax}{\c@ef pt}{-\c@ef pt}%
    \ifdim\v@lmax>\v@leur\v@leur=\v@lmax\fi}%
    \ifdim\v@leur<\p@\f@ctech=\@M\else\ifdim\v@leur<\t@n\p@\f@ctech=\@m\else%
    \ifdim\v@leur<\c@nt\p@\f@ctech=\c@nt\else\ifdim\v@leur<\@m\p@\f@ctech=\t@n\else%
    \f@ctech=\@ne\fi\fi\fi\fi%
    \def\listec@ef{0}%
    \@ecfor\c@ef:=#4\do{\sc@lec@nvRI{\c@ef pt}\edef\listec@ef{\listec@ef,\the\s@mme}}%
    \extrairelepremi@r\c@ef\de\listec@ef\figptbary#1:#2[#3;\listec@ef]}\ignorespaces\fi}
\ctr@ld@f\def\sc@lec@nvRI#1{\v@leur=#1\p@rtentiere{\s@mme}{\v@leur}\advance\v@leur-\s@mme\p@%
    \multiply\v@leur\f@ctech\p@rtentiere{\p@rtent}{\v@leur}%
    \multiply\s@mme\f@ctech\advance\s@mme\p@rtent}
\ctr@ln@m\figptcirc
\ctr@ld@f\def\figptcircDD#1:#2:#3;#4(#5){\ifps@cri{\s@uvc@ntr@l\et@tfigptcircDD%
    \c@lptellDD#1:{#2}:#3;#4,#4(#5)\resetc@ntr@l\et@tfigptcircDD}\ignorespaces\fi}
\ctr@ld@f\def\figptcircTD#1:#2:#3,#4,#5;#6(#7){\ifps@cri{\s@uvc@ntr@l\et@tfigptcircTD%
    \setc@ntr@l{2}\c@lExtAxes#3,#4,#5(#6)\figptellP#1:{#2}:#3,-4,-5(#7)%
    \resetc@ntr@l\et@tfigptcircTD}\ignorespaces\fi}
\ctr@ln@m\figptcircumcenter
\ctr@ld@f\def\figptcircumcenterDD#1:#2[#3,#4,#5]{\ifps@cri{\s@uvc@ntr@l\et@tfigptcircumcenterDD%
    \setc@ntr@l{2}\figvectNDD-5[#3,#4]\figptbaryDD-3:[#3,#4;1,1]%
                  \figvectNDD-6[#4,#5]\figptbaryDD-4:[#4,#5;1,1]%
    \resetc@ntr@l{2}\inters@cDD#1:{#2}[-3,-5;-4,-6]%
    \resetc@ntr@l\et@tfigptcircumcenterDD}\ignorespaces\fi}
\ctr@ld@f\def\figptcircumcenterTD#1:#2[#3,#4,#5]{\ifps@cri{\s@uvc@ntr@l\et@tfigptcircumcenterTD%
    \setc@ntr@l{2}\figvectNTD-1[#3,#4,#5]%
    \figvectPTD-3[#3,#4]\figvectNVTD-5[-1,-3]\figptbaryTD-3:[#3,#4;1,1]%
    \figvectPTD-4[#4,#5]\figvectNVTD-6[-1,-4]\figptbaryTD-4:[#4,#5;1,1]%
    \resetc@ntr@l{2}\inters@cTD#1:{#2}[-3,-5;-4,-6]%
    \resetc@ntr@l\et@tfigptcircumcenterTD}\ignorespaces\fi}
\ctr@ln@m\figptcopy
\ctr@ld@f\def\figptcopyDD#1:#2/#3/{\ifps@cri{\Figg@tXY{#3}%
    \Figp@intregDD#1:{#2}(\v@lX,\v@lY)}\ignorespaces\fi}
\ctr@ld@f\def\figptcopyTD#1:#2/#3/{\ifps@cri{\Figg@tXY{#3}%
    \Figp@intregTD#1:{#2}(\v@lX,\v@lY,\v@lZ)}\ignorespaces\fi}
\ctr@ln@m\figptcurvcenter
\ctr@ld@f\def\figptcurvcenterDD#1:#2:#3[#4,#5,#6,#7]{\ifps@cri{\s@uvc@ntr@l\et@tfigptcurvcenterDD%
    \setc@ntr@l{2}\c@lcurvradDD#3[#4,#5,#6,#7]\edef\Sprim@{\repdecn@mb{\result@t}}%
    \figptBezierDD-1::#3[#4,#5,#6,#7]\figpttraDD#1:{#2}=-1/\Sprim@,-5/%
    \resetc@ntr@l\et@tfigptcurvcenterDD}\ignorespaces\fi}
\ctr@ld@f\def\figptcurvcenterTD#1:#2:#3[#4,#5,#6,#7]{\ifps@cri{\s@uvc@ntr@l\et@tfigptcurvcenterTD%
    \setc@ntr@l{2}\figvectDBezierTD -5:1,#3[#4,#5,#6,#7]%
    \figvectDBezierTD -6:2,#3[#4,#5,#6,#7]\vecunit@TD{-5}{-5}%
    \edef\Sprim@{\repdecn@mb{\result@t}}\figvectNVTD-1[-6,-5]%
    \figvectNVTD-5[-5,-1]\c@lproscalTD\v@leur[-6,-5]%
    \invers@{\v@leur}{\v@leur}\v@leur=\Sprim@\v@leur\v@leur=\Sprim@\v@leur%
    \figptBezierTD-1::#3[#4,#5,#6,#7]\edef\Sprim@{\repdecn@mb{\v@leur}}%
    \figpttraTD#1:{#2}=-1/\Sprim@,-5/\resetc@ntr@l\et@tfigptcurvcenterTD}\ignorespaces\fi}
\ctr@ld@f\def\c@lcurvradDD#1[#2,#3,#4,#5]{{\figvectDBezierDD -5:1,#1[#2,#3,#4,#5]%
    \figvectDBezierDD -6:2,#1[#2,#3,#4,#5]\vecunit@DD{-5}{-5}%
    \edef\Sprim@{\repdecn@mb{\result@t}}\figvectNVDD-5[-5]\c@lproscalDD\v@leur[-6,-5]%
    \invers@{\v@leur}{\v@leur}\v@leur=\Sprim@\v@leur\v@leur=\Sprim@\v@leur%
    \global\result@t=\v@leur}}
\ctr@ln@m\figptell
\ctr@ld@f\def\figptellDD#1:#2:#3;#4,#5(#6,#7){\ifps@cri{\s@uvc@ntr@l\et@tfigptell%
    \c@lptellDD#1::#3;#4,#5(#6)\figptrotDD#1:{#2}=#1/#3,#7/%
    \resetc@ntr@l\et@tfigptell}\ignorespaces\fi}
\ctr@ld@f\def\c@lptellDD#1:#2:#3;#4,#5(#6){\c@ssin{\C@}{\S@}{#6}\v@lmin=\C@ pt\v@lmax=\S@ pt%
    \v@lmin=#4\v@lmin\v@lmax=#5\v@lmax%
    \edef\Xc@mp{\repdecn@mb{\v@lmin}}\edef\Yc@mp{\repdecn@mb{\v@lmax}}%
    \setc@ntr@l{2}\figvectC-1(\Xc@mp,\Yc@mp)\figpttraDD#1:{#2}=#3/1,-1/}
\ctr@ld@f\def\figptellP#1:#2:#3,#4,#5(#6){\ifps@cri{\s@uvc@ntr@l\et@tfigptellP%
    \setc@ntr@l{2}\figvectP-1[#3,#4]\figvectP-2[#3,#5]%
    \v@leur=#6pt\c@lptellP{#3}{-1}{-2}\figptcopy#1:{#2}/-3/%
    \resetc@ntr@l\et@tfigptellP}\ignorespaces\fi}
\ctr@ln@m\@ngle
\ctr@ld@f\def\c@lptellP#1#2#3{\edef\@ngle{\repdecn@mb\v@leur}\c@ssin{\C@}{\S@}{\@ngle}%
    \figpttra-3:=#1/\C@,#2/\figpttra-3:=-3/\S@,#3/}
\ctr@ln@m\figptendnormal
\ctr@ld@f\def\figptendnormalDD#1:#2:#3,#4[#5,#6]{\ifps@cri{\s@uvc@ntr@l\et@tfigptendnormal%
    \Figg@tXYa{#5}\Figg@tXY{#6}%
    \advance\v@lX-\v@lXa\advance\v@lY-\v@lYa%
    \setc@ntr@l{2}\Figv@ctCreg-1(\v@lX,\v@lY)\vecunit@{-1}{-1}\Figg@tXY{-1}%
    \delt@=#3\unit@\maxim@m{\delt@}{\delt@}{-\delt@}\edef\l@ngueur{\repdecn@mb{\delt@}}%
    \v@lX=\l@ngueur\v@lX\v@lY=\l@ngueur\v@lY%
    \delt@=\p@\advance\delt@-#4pt\edef\l@ngueur{\repdecn@mb{\delt@}}%
    \figptbaryR-1:[#5,#6;#4,\l@ngueur]\Figg@tXYa{-1}%
    \advance\v@lXa\v@lY\advance\v@lYa-\v@lX%
    \setc@ntr@l{1}\Figp@intregDD#1:{#2}(\v@lXa,\v@lYa)\resetc@ntr@l\et@tfigptendnormal}%
    \ignorespaces\fi}
\ctr@ld@f\def\figptexcenter#1:#2[#3,#4,#5]{\ifps@cri{\let@xte={-}%
    \Figptexinsc@nter#1:#2[#3,#4,#5]}\ignorespaces\fi}
\ctr@ld@f\def\figptincenter#1:#2[#3,#4,#5]{\ifps@cri{\let@xte={}%
    \Figptexinsc@nter#1:#2[#3,#4,#5]}\ignorespaces\fi}
\ctr@ld@f\let\figptinscribedcenter=\figptincenter% pour compatibilite avec anciennes versions
\ctr@ld@f\def\Figptexinsc@nter#1:#2[#3,#4,#5]{%
    \figgetdist\LA@[#4,#5]\figgetdist\LB@[#3,#5]\figgetdist\LC@[#3,#4]%
    \figptbaryR#1:{#2}[#3,#4,#5;\the\let@xte\LA@,\LB@,\LC@]}
\ctr@ln@m\figptinterlineplane
\ctr@ld@f\def\figptinterlineplaneDD{\un@v@ilable{figptinterlineplane}}
\ctr@ld@f\def\figptinterlineplaneTD#1:#2[#3,#4;#5,#6]{\ifps@cri{\s@uvc@ntr@l\et@tfigptinterlineplane%
    \setc@ntr@l{2}\figvectPTD-1[#3,#5]\vecunit@TD{-2}{#6}%
    \r@pPSTD\v@leur[-2,-1,#4]\edef\v@lcoef{\repdecn@mb{\v@leur}}%
    \figpttraTD#1:{#2}=#3/\v@lcoef,#4/\resetc@ntr@l\et@tfigptinterlineplane}\ignorespaces\fi}
\ctr@ln@m\figptorthocenter
\ctr@ld@f\def\figptorthocenterDD#1:#2[#3,#4,#5]{\ifps@cri{\s@uvc@ntr@l\et@tfigptorthocenterDD%
    \setc@ntr@l{2}\figvectNDD-3[#3,#4]\figvectNDD-4[#4,#5]%
    \resetc@ntr@l{2}\inters@cDD#1:{#2}[#5,-3;#3,-4]%
    \resetc@ntr@l\et@tfigptorthocenterDD}\ignorespaces\fi}
\ctr@ld@f\def\figptorthocenterTD#1:#2[#3,#4,#5]{\ifps@cri{\s@uvc@ntr@l\et@tfigptorthocenterTD%
    \setc@ntr@l{2}\figvectNTD-1[#3,#4,#5]%
    \figvectPTD-2[#3,#4]\figvectNVTD-3[-1,-2]%
    \figvectPTD-2[#4,#5]\figvectNVTD-4[-1,-2]%
    \resetc@ntr@l{2}\inters@cTD#1:{#2}[#5,-3;#3,-4]%
    \resetc@ntr@l\et@tfigptorthocenterTD}\ignorespaces\fi}
\ctr@ln@m\figptorthoprojline
\ctr@ld@f\def\figptorthoprojlineDD#1:#2=#3/#4,#5/{\ifps@cri{\s@uvc@ntr@l\et@tfigptorthoprojlineDD%
    \setc@ntr@l{2}\figvectPDD-3[#4,#5]\figvectNVDD-4[-3]\resetc@ntr@l{2}%
    \inters@cDD#1:{#2}[#3,-4;#4,-3]\resetc@ntr@l\et@tfigptorthoprojlineDD}\ignorespaces\fi}
\ctr@ld@f\def\figptorthoprojlineTD#1:#2=#3/#4,#5/{\ifps@cri{\s@uvc@ntr@l\et@tfigptorthoprojlineTD%
    \setc@ntr@l{2}\figvectPTD-1[#4,#3]\figvectPTD-2[#4,#5]\vecunit@TD{-2}{-2}%
    \c@lproscalTD\v@leur[-1,-2]\edef\v@lcoef{\repdecn@mb{\v@leur}}%
    \figpttraTD#1:{#2}=#4/\v@lcoef,-2/\resetc@ntr@l\et@tfigptorthoprojlineTD}\ignorespaces\fi}
\ctr@ln@m\figptorthoprojplane
\ctr@ld@f\def\figptorthoprojplaneDD{\un@v@ilable{figptorthoprojplane}}
\ctr@ld@f\def\figptorthoprojplaneTD#1:#2=#3/#4,#5/{\ifps@cri{\s@uvc@ntr@l\et@tfigptorthoprojplane%
    \setc@ntr@l{2}\figvectPTD-1[#3,#4]\vecunit@TD{-2}{#5}%
    \c@lproscalTD\v@leur[-1,-2]\edef\v@lcoef{\repdecn@mb{\v@leur}}%
    \figpttraTD#1:{#2}=#3/\v@lcoef,-2/\resetc@ntr@l\et@tfigptorthoprojplane}\ignorespaces\fi}
\ctr@ld@f\def\figpthom#1:#2=#3/#4,#5/{\ifps@cri{\s@uvc@ntr@l\et@tfigpthom%
    \setc@ntr@l{2}\figvectP-1[#4,#3]\figpttra#1:{#2}=#4/#5,-1/%
    \resetc@ntr@l\et@tfigpthom}\ignorespaces\fi}
\ctr@ln@m\figptrot
\ctr@ld@f\def\figptrotDD#1:#2=#3/#4,#5/{\ifps@cri{\s@uvc@ntr@l\et@tfigptrotDD%
    \c@ssin{\C@}{\S@}{#5}\setc@ntr@l{2}\figvectPDD-1[#4,#3]\Figg@tXY{-1}%
    \v@lXa=\C@\v@lX\advance\v@lXa-\S@\v@lY%
    \v@lYa=\S@\v@lX\advance\v@lYa\C@\v@lY%
    \Figv@ctCreg-1(\v@lXa,\v@lYa)\figpttraDD#1:{#2}=#4/1,-1/%
    \resetc@ntr@l\et@tfigptrotDD}\ignorespaces\fi}
\ctr@ld@f\def\figptrotTD#1:#2=#3/#4,#5,#6/{\ifps@cri{\s@uvc@ntr@l\et@tfigptrotTD%
    \c@ssin{\C@}{\S@}{#5}%
    \setc@ntr@l{2}\figptorthoprojplaneTD-3:=#4/#3,#6/\figvectPTD-2[-3,#3]%
    \n@rmeucTD\v@leur{-2}\ifdim\v@leur<\Cepsil@n\Figg@tXYa{#3}\else%
    \edef\v@lcoef{\repdecn@mb{\v@leur}}\figvectNVTD-1[#6,-2]%
    \Figg@tXYa{-1}\v@lXa=\v@lcoef\v@lXa\v@lYa=\v@lcoef\v@lYa\v@lZa=\v@lcoef\v@lZa%
    \v@lXa=\S@\v@lXa\v@lYa=\S@\v@lYa\v@lZa=\S@\v@lZa\Figg@tXY{-2}%
    \advance\v@lXa\C@\v@lX\advance\v@lYa\C@\v@lY\advance\v@lZa\C@\v@lZ%
    \Figg@tXY{-3}\advance\v@lXa\v@lX\advance\v@lYa\v@lY\advance\v@lZa\v@lZ\fi%
    \Figp@intregTD#1:{#2}(\v@lXa,\v@lYa,\v@lZa)\resetc@ntr@l\et@tfigptrotTD}\ignorespaces\fi}
\ctr@ln@m\figptsym
\ctr@ld@f\def\figptsymDD#1:#2=#3/#4,#5/{\ifps@cri{\s@uvc@ntr@l\et@tfigptsymDD%
    \resetc@ntr@l{2}\figptorthoprojlineDD-5:=#3/#4,#5/\figvectPDD-2[#3,-5]%
    \figpttraDD#1:{#2}=#3/2,-2/\resetc@ntr@l\et@tfigptsymDD}\ignorespaces\fi}
\ctr@ld@f\def\figptsymTD#1:#2=#3/#4,#5/{\ifps@cri{\s@uvc@ntr@l\et@tfigptsymTD%
    \resetc@ntr@l{2}\figptorthoprojplaneTD-3:=#3/#4,#5/\figvectPTD-2[#3,-3]%
    \figpttraTD#1:{#2}=#3/2,-2/\resetc@ntr@l\et@tfigptsymTD}\ignorespaces\fi}
\ctr@ln@m\figpttra
\ctr@ld@f\def\figpttraDD#1:#2=#3/#4,#5/{\ifps@cri{\Figg@tXYa{#5}\v@lXa=#4\v@lXa\v@lYa=#4\v@lYa%
    \Figg@tXY{#3}\advance\v@lX\v@lXa\advance\v@lY\v@lYa%
    \Figp@intregDD#1:{#2}(\v@lX,\v@lY)}\ignorespaces\fi}
\ctr@ld@f\def\figpttraTD#1:#2=#3/#4,#5/{\ifps@cri{\Figg@tXYa{#5}\v@lXa=#4\v@lXa\v@lYa=#4\v@lYa%
    \v@lZa=#4\v@lZa\Figg@tXY{#3}\advance\v@lX\v@lXa\advance\v@lY\v@lYa%
    \advance\v@lZ\v@lZa\Figp@intregTD#1:{#2}(\v@lX,\v@lY,\v@lZ)}\ignorespaces\fi}
\ctr@ln@m\figpttraC
\ctr@ld@f\def\figpttraCDD#1:#2=#3/#4,#5/{\ifps@cri{\v@lXa=#4\unit@\v@lYa=#5\unit@%
    \Figg@tXY{#3}\advance\v@lX\v@lXa\advance\v@lY\v@lYa%
    \Figp@intregDD#1:{#2}(\v@lX,\v@lY)}\ignorespaces\fi}
\ctr@ld@f\def\figpttraCTD#1:#2=#3/#4,#5,#6/{\ifps@cri{\v@lXa=#4\unit@\v@lYa=#5\unit@\v@lZa=#6\unit@%
    \Figg@tXY{#3}\advance\v@lX\v@lXa\advance\v@lY\v@lYa\advance\v@lZ\v@lZa%
    \Figp@intregTD#1:{#2}(\v@lX,\v@lY,\v@lZ)}\ignorespaces\fi}
\ctr@ld@f\def\figptsaxes#1:#2(#3){\ifps@cri{\an@lys@xes#3,:\ifx\t@xt@\empty%
    \ifTr@isDim\Figpts@xes#1:#2(0,#3,0,#3,0,#3)\else\Figpts@xes#1:#2(0,#3,0,#3)\fi%
    \else\Figpts@xes#1:#2(#3)\fi}\ignorespaces\fi}
\ctr@ln@m\Figpts@xes
\ctr@ld@f\def\Figpts@xesDD#1:#2(#3,#4,#5,#6){%
    \s@mme=#1\figpttraC\the\s@mme:$x$=#2/#4,0/%
    \advance\s@mme\@ne\figpttraC\the\s@mme:$y$=#2/0,#6/}
\ctr@ld@f\def\Figpts@xesTD#1:#2(#3,#4,#5,#6,#7,#8){%
    \s@mme=#1\figpttraC\the\s@mme:$x$=#2/#4,0,0/%
    \advance\s@mme\@ne\figpttraC\the\s@mme:$y$=#2/0,#6,0/%
    \advance\s@mme\@ne\figpttraC\the\s@mme:$z$=#2/0,0,#8/}
\ctr@ld@f\def\figptsmap#1=#2/#3/#4/{\ifps@cri{\s@uvc@ntr@l\et@tfigptsmap%
    \setc@ntr@l{2}\def\list@num{#2}\s@mme=#1%
    \@ecfor\p@int:=\list@num\do{\figvectP-1[#3,\p@int]\Figg@tXY{-1}%
    \pr@dMatV/#4/\figpttra\the\s@mme:=#3/1,-1/\advance\s@mme\@ne}%
    \resetc@ntr@l\et@tfigptsmap}\ignorespaces\fi}
\ctr@ln@m\figptscontrol
\ctr@ld@f\def\figptscontrolDD#1[#2,#3,#4,#5]{\ifps@cri{\s@uvc@ntr@l\et@tfigptscontrolDD\setc@ntr@l{2}%
    \v@lX=\z@\v@lY=\z@\Figtr@nptDD{-5}{#2}\Figtr@nptDD{2}{#5}%
    \divide\v@lX\@vi\divide\v@lY\@vi%
    \Figtr@nptDD{3}{#3}\Figtr@nptDD{-1.5}{#4}\Figp@intregDD-1:(\v@lX,\v@lY)%
    \v@lX=\z@\v@lY=\z@\Figtr@nptDD{2}{#2}\Figtr@nptDD{-5}{#5}%
    \divide\v@lX\@vi\divide\v@lY\@vi\Figtr@nptDD{-1.5}{#3}\Figtr@nptDD{3}{#4}%
    \s@mme=#1\advance\s@mme\@ne\Figp@intregDD\the\s@mme:(\v@lX,\v@lY)%
    \figptcopyDD#1:/-1/\resetc@ntr@l\et@tfigptscontrolDD}\ignorespaces\fi}
\ctr@ld@f\def\figptscontrolTD#1[#2,#3,#4,#5]{\ifps@cri{\s@uvc@ntr@l\et@tfigptscontrolTD\setc@ntr@l{2}%
    \v@lX=\z@\v@lY=\z@\v@lZ=\z@\Figtr@nptTD{-5}{#2}\Figtr@nptTD{2}{#5}%
    \divide\v@lX\@vi\divide\v@lY\@vi\divide\v@lZ\@vi%
    \Figtr@nptTD{3}{#3}\Figtr@nptTD{-1.5}{#4}\Figp@intregTD-1:(\v@lX,\v@lY,\v@lZ)%
    \v@lX=\z@\v@lY=\z@\v@lZ=\z@\Figtr@nptTD{2}{#2}\Figtr@nptTD{-5}{#5}%
    \divide\v@lX\@vi\divide\v@lY\@vi\divide\v@lZ\@vi\Figtr@nptTD{-1.5}{#3}\Figtr@nptTD{3}{#4}%
    \s@mme=#1\advance\s@mme\@ne\Figp@intregTD\the\s@mme:(\v@lX,\v@lY,\v@lZ)%
    \figptcopyTD#1:/-1/\resetc@ntr@l\et@tfigptscontrolTD}\ignorespaces\fi}
\ctr@ld@f\def\Figtr@nptDD#1#2{\Figg@tXYa{#2}\v@lXa=#1\v@lXa\v@lYa=#1\v@lYa%
    \advance\v@lX\v@lXa\advance\v@lY\v@lYa}
\ctr@ld@f\def\Figtr@nptTD#1#2{\Figg@tXYa{#2}\v@lXa=#1\v@lXa\v@lYa=#1\v@lYa\v@lZa=#1\v@lZa%
    \advance\v@lX\v@lXa\advance\v@lY\v@lYa\advance\v@lZ\v@lZa}
\ctr@ld@f\def\figptscontrolcurve#1,#2[#3]{\ifps@cri{\s@uvc@ntr@l\et@tfigptscontrolcurve%
    \def\list@num{#3}\extrairelepremi@r\Ak@\de\list@num%
    \extrairelepremi@r\Ai@\de\list@num\extrairelepremi@r\Aj@\de\list@num%
    \s@mme=#1\figptcopy\the\s@mme:/\Ai@/%
    \setc@ntr@l{2}\figvectP -1[\Ak@,\Aj@]%
    \@ecfor\Ak@:=\list@num\do{\advance\s@mme\@ne\figpttra\the\s@mme:=\Ai@/\curv@roundness,-1/%
       \figvectP -1[\Ai@,\Ak@]\advance\s@mme\@ne\figpttra\the\s@mme:=\Aj@/-\curv@roundness,-1/%
       \advance\s@mme\@ne\figptcopy\the\s@mme:/\Aj@/%
       \edef\Ai@{\Aj@}\edef\Aj@{\Ak@}}\advance\s@mme-#1\divide\s@mme\thr@@%
       \xdef#2{\the\s@mme}%
    \resetc@ntr@l\et@tfigptscontrolcurve}\ignorespaces\fi}
\ctr@ln@m\figptsintercirc
\ctr@ld@f\def\figptsintercircDD#1[#2,#3;#4,#5]{\ifps@cri{\s@uvc@ntr@l\et@tfigptsintercircDD%
    \setc@ntr@l{2}\let\c@lNVintc=\c@lNVintcDD\Figptsintercirc@#1[#2,#3;#4,#5]%    
    \resetc@ntr@l\et@tfigptsintercircDD}\ignorespaces\fi}
\ctr@ld@f\def\figptsintercircTD#1[#2,#3;#4,#5;#6]{\ifps@cri{\s@uvc@ntr@l\et@tfigptsintercircTD%
    \setc@ntr@l{2}\let\c@lNVintc=\c@lNVintcTD\vecunitC@TD[#2,#6]%
    \Figv@ctCreg-3(\v@lX,\v@lY,\v@lZ)\Figptsintercirc@#1[#2,#3;#4,#5]%
    \resetc@ntr@l\et@tfigptsintercircTD}\ignorespaces\fi}
\ctr@ld@f\def\Figptsintercirc@#1[#2,#3;#4,#5]{\figvectP-1[#2,#4]%
    \vecunit@{-1}{-1}\delt@=\result@t\f@ctech=\result@tent%
    \s@mme=#1\advance\s@mme\@ne\figptcopy#1:/#2/\figptcopy\the\s@mme:/#4/%
    \ifdim\delt@=\z@\else%
    \v@lmin=#3\unit@\v@lmax=#5\unit@\v@leur=\v@lmin\advance\v@leur\v@lmax%
    \ifdim\v@leur>\delt@%
    \v@leur=\v@lmin\advance\v@leur-\v@lmax\maxim@m{\v@leur}{\v@leur}{-\v@leur}%
    \ifdim\v@leur<\delt@%
    \divide\v@lmin\f@ctech\divide\v@lmax\f@ctech\divide\delt@\f@ctech%
    \v@lmin=\repdecn@mb{\v@lmin}\v@lmin\v@lmax=\repdecn@mb{\v@lmax}\v@lmax%
    \invers@{\v@leur}{\delt@}\advance\v@lmax-\v@lmin%
    \v@lmax=-\repdecn@mb{\v@leur}\v@lmax\advance\delt@\v@lmax\delt@=.5\delt@%
    \v@lmax=\delt@\multiply\v@lmax\f@ctech%
    \edef\t@ille{\repdecn@mb{\v@lmax}}\figpttra-2:=#2/\t@ille,-1/%
    \delt@=\repdecn@mb{\delt@}\delt@\advance\v@lmin-\delt@%
    \sqrt@{\v@leur}{\v@lmin}\multiply\v@leur\f@ctech\edef\t@ille{\repdecn@mb{\v@leur}}%
    \c@lNVintc\figpttra#1:=-2/-\t@ille,-1/\figpttra\the\s@mme:=-2/\t@ille,-1/\fi\fi\fi}
\ctr@ld@f\def\c@lNVintcDD{\Figg@tXY{-1}\Figv@ctCreg-1(-\v@lY,\v@lX)} % <=> \figvectNVDD-1[-1]
\ctr@ld@f\def\c@lNVintcTD{{\Figg@tXY{-3}\v@lmin=\v@lX\v@lmax=\v@lY\v@leur=\v@lZ%
    \Figg@tXY{-1}\c@lprovec{-3}\vecunit@{-3}{-3}% <=> \figvectNVTD-3[-1,-3]\vecunit@{-3}{-3}
    \Figg@tXY{-1}\v@lmin=\v@lX\v@lmax=\v@lY%
    \v@leur=\v@lZ\Figg@tXY{-3}\c@lprovec{-1}}} % <=> \figvectNVTD-1[-3,-1]
\ctr@ln@m\figptsinterlinell
\ctr@ld@f\def\figptsinterlinellDD#1[#2,#3,#4,#5;#6,#7]{\ifps@cri{\s@uvc@ntr@l\et@tfigptsinterlinellDD%
    \figptcopy#1:/#6/\s@mme=#1\advance\s@mme\@ne\figptcopy\the\s@mme:/#7/%
    \v@lmin=#3\unit@\v@lmax=#4\unit@% a, b
    \setc@ntr@l{2}\figptbaryDD-4:[#6,#7;1,1]\figptsrotDD-3=-4,#7/#2,-#5/% D et rotation
    \Figg@tXY{-3}\Figg@tXYa{#2}\advance\v@lX-\v@lXa\advance\v@lY-\v@lYa% alpha, beta
    \figvectP-1[-3,-2]\Figg@tXYa{-1}\figvectP-3[-4,#7]\Figptsint@rLE{#1}% u1, u2
    \resetc@ntr@l\et@tfigptsinterlinellDD}\ignorespaces\fi}
\ctr@ld@f\def\figptsinterlinellP#1[#2,#3,#4;#5,#6]{\ifps@cri{\s@uvc@ntr@l\et@tfigptsinterlinellP%
    \figptcopy#1:/#5/\s@mme=#1\advance\s@mme\@ne\figptcopy\the\s@mme:/#6/\setc@ntr@l{2}%
    \figvectP-1[#2,#3]\vecunit@{-1}{-1}\v@lmin=\result@t% a
    \figvectP-2[#2,#4]\vecunit@{-2}{-2}\v@lmax=\result@t% b
    \figptbary-4:[#5,#6;1,1]% D
    \figvectP-3[#2,-4]\c@lproscal\v@lX[-3,-1]\c@lproscal\v@lY[-3,-2]% alpha, beta
    \figvectP-3[-4,#6]\c@lproscal\v@lXa[-3,-1]\c@lproscal\v@lYa[-3,-2]% u1, u2
    \Figptsint@rLE{#1}\resetc@ntr@l\et@tfigptsinterlinellP}\ignorespaces\fi}
\ctr@ld@f\def\Figptsint@rLE#1{%
    \getredf@ctDD\f@ctech(\v@lmin,\v@lmax)%
    \getredf@ctDD\p@rtent(\v@lX,\v@lY)\ifnum\p@rtent>\f@ctech\f@ctech=\p@rtent\fi%
    \getredf@ctDD\p@rtent(\v@lXa,\v@lYa)\ifnum\p@rtent>\f@ctech\f@ctech=\p@rtent\fi%
    \divide\v@lmin\f@ctech\divide\v@lmax\f@ctech\divide\v@lX\f@ctech\divide\v@lY\f@ctech%
    \divide\v@lXa\f@ctech\divide\v@lYa\f@ctech%
    \c@rre=\repdecn@mb\v@lXa\v@lmax\mili@u=\repdecn@mb\v@lYa\v@lmin%
    \getredf@ctDD\f@ctech(\c@rre,\mili@u)%
    \c@rre=\repdecn@mb\v@lX\v@lmax\mili@u=\repdecn@mb\v@lY\v@lmin%
    \getredf@ctDD\p@rtent(\c@rre,\mili@u)\ifnum\p@rtent>\f@ctech\f@ctech=\p@rtent\fi%
    \divide\v@lmin\f@ctech\divide\v@lmax\f@ctech\divide\v@lX\f@ctech\divide\v@lY\f@ctech%
    \divide\v@lXa\f@ctech\divide\v@lYa\f@ctech%
    \v@lmin=\repdecn@mb{\v@lmin}\v@lmin\v@lmax=\repdecn@mb{\v@lmax}\v@lmax%
    \edef\G@xde{\repdecn@mb\v@lmin}\edef\P@xde{\repdecn@mb\v@lmax}%
    \c@rre=-\v@lmax\v@leur=\repdecn@mb\v@lY\v@lY\advance\c@rre\v@leur\c@rre=\G@xde\c@rre%
    \v@leur=\repdecn@mb\v@lX\v@lX\v@leur=\P@xde\v@leur\advance\c@rre\v@leur% C
    \v@lmin=\repdecn@mb\v@lYa\v@lmin\v@lmax=\repdecn@mb\v@lXa\v@lmax%
    \mili@u=\repdecn@mb\v@lX\v@lmax\advance\mili@u\repdecn@mb\v@lY\v@lmin% B
    \v@lmax=\repdecn@mb\v@lXa\v@lmax\advance\v@lmax\repdecn@mb\v@lYa\v@lmin% A
    \ifdim\v@lmax>\epsil@n%
    \maxim@m{\v@leur}{\c@rre}{-\c@rre}\maxim@m{\v@lmin}{\mili@u}{-\mili@u}%
    \maxim@m{\v@leur}{\v@leur}{\v@lmin}\maxim@m{\v@lmin}{\v@lmax}{-\v@lmax}%
    \maxim@m{\v@leur}{\v@leur}{\v@lmin}\p@rtentiere{\p@rtent}{\v@leur}\advance\p@rtent\@ne%
    \divide\c@rre\p@rtent\divide\mili@u\p@rtent\divide\v@lmax\p@rtent%
    \delt@=\repdecn@mb{\mili@u}\mili@u\v@leur=\repdecn@mb{\v@lmax}\c@rre%
    \advance\delt@-\v@leur\ifdim\delt@<\z@\else\sqrt@\delt@\delt@%
    \invers@\v@lmax\v@lmax\edef\Uns@rAp{\repdecn@mb\v@lmax}%
    \v@leur=-\mili@u\advance\v@leur-\delt@\v@leur=\Uns@rAp\v@leur%
    \edef\t@ille{\repdecn@mb\v@leur}\figpttra#1:=-4/\t@ille,-3/\s@mme=#1\advance\s@mme\@ne%
    \v@leur=-\mili@u\advance\v@leur\delt@\v@leur=\Uns@rAp\v@leur%
    \edef\t@ille{\repdecn@mb\v@leur}\figpttra\the\s@mme:=-4/\t@ille,-3/\fi\fi}
\ctr@ln@m\figptsorthoprojline
\ctr@ld@f\def\figptsorthoprojlineDD#1=#2/#3,#4/{\ifps@cri{\s@uvc@ntr@l\et@tfigptsorthoprojlineDD%
    \setc@ntr@l{2}\figvectPDD-3[#3,#4]\figvectNVDD-4[-3]\resetc@ntr@l{2}%
    \def\list@num{#2}\s@mme=#1\@ecfor\p@int:=\list@num\do{%
    \inters@cDD\the\s@mme:[\p@int,-4;#3,-3]\advance\s@mme\@ne}%
    \resetc@ntr@l\et@tfigptsorthoprojlineDD}\ignorespaces\fi}
\ctr@ld@f\def\figptsorthoprojlineTD#1=#2/#3,#4/{\ifps@cri{\s@uvc@ntr@l\et@tfigptsorthoprojlineTD%
    \setc@ntr@l{2}\figvectPTD-2[#3,#4]\vecunit@TD{-2}{-2}%
    \def\list@num{#2}\s@mme=#1\@ecfor\p@int:=\list@num\do{%
    \figvectPTD-1[#3,\p@int]\c@lproscalTD\v@leur[-1,-2]%
    \edef\v@lcoef{\repdecn@mb{\v@leur}}\figpttraTD\the\s@mme:=#3/\v@lcoef,-2/%
    \advance\s@mme\@ne}\resetc@ntr@l\et@tfigptsorthoprojlineTD}\ignorespaces\fi}
\ctr@ln@m\figptsorthoprojplane
\ctr@ld@f\def\figptsorthoprojplaneDD{\un@v@ilable{figptsorthoprojplane}}
\ctr@ld@f\def\figptsorthoprojplaneTD#1=#2/#3,#4/{\ifps@cri{\s@uvc@ntr@l\et@tfigptsorthoprojplane%
    \setc@ntr@l{2}\vecunit@TD{-2}{#4}%
    \def\list@num{#2}\s@mme=#1\@ecfor\p@int:=\list@num\do{\figvectPTD-1[\p@int,#3]%
    \c@lproscalTD\v@leur[-1,-2]\edef\v@lcoef{\repdecn@mb{\v@leur}}%
    \figpttraTD\the\s@mme:=\p@int/\v@lcoef,-2/\advance\s@mme\@ne}%
    \resetc@ntr@l\et@tfigptsorthoprojplane}\ignorespaces\fi}
\ctr@ld@f\def\figptshom#1=#2/#3,#4/{\ifps@cri{\s@uvc@ntr@l\et@tfigptshom%
    \setc@ntr@l{2}\def\list@num{#2}\s@mme=#1%
    \@ecfor\p@int:=\list@num\do{\figvectP-1[#3,\p@int]%
    \figpttra\the\s@mme:=#3/#4,-1/\advance\s@mme\@ne}%
    \resetc@ntr@l\et@tfigptshom}\ignorespaces\fi}
\ctr@ln@m\figptsrot
\ctr@ld@f\def\figptsrotDD#1=#2/#3,#4/{\ifps@cri{\s@uvc@ntr@l\et@tfigptsrotDD%
    \c@ssin{\C@}{\S@}{#4}\setc@ntr@l{2}\def\list@num{#2}\s@mme=#1%
    \@ecfor\p@int:=\list@num\do{\figvectPDD-1[#3,\p@int]\Figg@tXY{-1}%
    \v@lXa=\C@\v@lX\advance\v@lXa-\S@\v@lY%
    \v@lYa=\S@\v@lX\advance\v@lYa\C@\v@lY%
    \Figv@ctCreg-1(\v@lXa,\v@lYa)\figpttraDD\the\s@mme:=#3/1,-1/\advance\s@mme\@ne}%
    \resetc@ntr@l\et@tfigptsrotDD}\ignorespaces\fi}
\ctr@ld@f\def\figptsrotTD#1=#2/#3,#4,#5/{\ifps@cri{\s@uvc@ntr@l\et@tfigptsrotTD%
    \c@ssin{\C@}{\S@}{#4}%
    \setc@ntr@l{2}\def\list@num{#2}\s@mme=#1%
    \@ecfor\p@int:=\list@num\do{\figptorthoprojplaneTD-3:=#3/\p@int,#5/%
    \figvectPTD-2[-3,\p@int]%
    \figvectNVTD-1[#5,-2]\n@rmeucTD\v@leur{-2}\edef\v@lcoef{\repdecn@mb{\v@leur}}%
    \Figg@tXYa{-1}\v@lXa=\v@lcoef\v@lXa\v@lYa=\v@lcoef\v@lYa\v@lZa=\v@lcoef\v@lZa%
    \v@lXa=\S@\v@lXa\v@lYa=\S@\v@lYa\v@lZa=\S@\v@lZa\Figg@tXY{-2}%
    \advance\v@lXa\C@\v@lX\advance\v@lYa\C@\v@lY\advance\v@lZa\C@\v@lZ%
    \Figg@tXY{-3}\advance\v@lXa\v@lX\advance\v@lYa\v@lY\advance\v@lZa\v@lZ%
    \Figp@intregTD\the\s@mme:(\v@lXa,\v@lYa,\v@lZa)\advance\s@mme\@ne}%
    \resetc@ntr@l\et@tfigptsrotTD}\ignorespaces\fi}
\ctr@ln@m\figptssym
\ctr@ld@f\def\figptssymDD#1=#2/#3,#4/{\ifps@cri{\s@uvc@ntr@l\et@tfigptssymDD%
    \setc@ntr@l{2}\figvectPDD-3[#3,#4]\Figg@tXY{-3}\Figv@ctCreg-4(-\v@lY,\v@lX)%
    \resetc@ntr@l{2}\def\list@num{#2}\s@mme=#1%
    \@ecfor\p@int:=\list@num\do{\inters@cDD-5:[#3,-3;\p@int,-4]\figvectPDD-2[\p@int,-5]%
    \figpttraDD\the\s@mme:=\p@int/2,-2/\advance\s@mme\@ne}%
    \resetc@ntr@l\et@tfigptssymDD}\ignorespaces\fi}
\ctr@ld@f\def\figptssymTD#1=#2/#3,#4/{\ifps@cri{\s@uvc@ntr@l\et@tfigptssymTD%
    \setc@ntr@l{2}\vecunit@TD{-2}{#4}\def\list@num{#2}\s@mme=#1%
    \@ecfor\p@int:=\list@num\do{\figvectPTD-1[\p@int,#3]%
    \c@lproscalTD\v@leur[-1,-2]\v@leur=2\v@leur\edef\v@lcoef{\repdecn@mb{\v@leur}}%
    \figpttraTD\the\s@mme:=\p@int/\v@lcoef,-2/\advance\s@mme\@ne}%
    \resetc@ntr@l\et@tfigptssymTD}\ignorespaces\fi}
\ctr@ln@m\figptstra
\ctr@ld@f\def\figptstraDD#1=#2/#3,#4/{\ifps@cri{\Figg@tXYa{#4}\v@lXa=#3\v@lXa\v@lYa=#3\v@lYa%
    \def\list@num{#2}\s@mme=#1\@ecfor\p@int:=\list@num\do{\Figg@tXY{\p@int}%
    \advance\v@lX\v@lXa\advance\v@lY\v@lYa%
    \Figp@intregDD\the\s@mme:(\v@lX,\v@lY)\advance\s@mme\@ne}}\ignorespaces\fi}
\ctr@ld@f\def\figptstraTD#1=#2/#3,#4/{\ifps@cri{\Figg@tXYa{#4}\v@lXa=#3\v@lXa\v@lYa=#3\v@lYa%
    \v@lZa=#3\v@lZa\def\list@num{#2}\s@mme=#1\@ecfor\p@int:=\list@num\do{\Figg@tXY{\p@int}%
    \advance\v@lX\v@lXa\advance\v@lY\v@lYa\advance\v@lZ\v@lZa%
    \Figp@intregTD\the\s@mme:(\v@lX,\v@lY,\v@lZ)\advance\s@mme\@ne}}\ignorespaces\fi}
\ctr@ln@m\figptvisilimSL
\ctr@ld@f\def\figptvisilimSLDD{\un@v@ilable{figptvisilimSL}}
\ctr@ld@f\def\figptvisilimSLTD#1:#2[#3,#4;#5,#6]{\ifps@cri{\s@uvc@ntr@l\et@tfigptvisilimSLTD%
    \setc@ntr@l{2}\figvectP-1[#3,#4]\n@rminf{\delt@}{-1}%
    \ifcase\curr@ntproj\v@lX=\cxa@\p@\v@lY=-\p@\v@lZ=\cxb@\p@% Proj cav
    \Figv@ctCreg-2(\v@lX,\v@lY,\v@lZ)\figvectP-3[#5,#6]\figvectNV-1[-2,-3]%
    \or\figvectP-1[#5,#6]\vecunitCV@TD{-1}\v@lmin=\v@lX\v@lmax=\v@lY% Proj ortho
    \v@leur=\v@lZ\v@lX=\cza@\p@\v@lY=\czb@\p@\v@lZ=\czc@\p@\c@lprovec{-1}%
    \or\c@ley@pt{-2}\figvectN-1[#5,#6,-2]\fi% Proj rea
    \edef\Ai@{#3}\edef\Aj@{#4}\figvectP-2[#5,\Ai@]\c@lproscal\v@leur[-1,-2]%
    \ifdim\v@leur>\z@\p@rtent=\@ne\else\p@rtent=\m@ne\fi%
    \figvectP-2[#5,\Aj@]\c@lproscal\v@leur[-1,-2]%
    \ifdim\p@rtent\v@leur>\z@\figptcopy#1:#2/#3/%
    \message{*** \BS@ figptvisilimSL: points are on the same side.}\else%
    \figptcopy-3:/#3/\figptcopy-4:/#4/%
    \loop\figptbary-5:[-3,-4;1,1]\figvectP-2[#5,-5]\c@lproscal\v@leur[-1,-2]%
    \ifdim\p@rtent\v@leur>\z@\figptcopy-3:/-5/\else\figptcopy-4:/-5/\fi%
    \divide\delt@\tw@\ifdim\delt@>\epsil@n\repeat%
    \figptbary#1:#2[-3,-4;1,1]\fi\resetc@ntr@l\et@tfigptvisilimSLTD}\ignorespaces\fi}
\ctr@ld@f\def\c@ley@pt#1{\t@stp@r\ifitis@K\v@lX=\cza@\p@\v@lY=\czb@\p@\v@lZ=\czc@\p@%
    \Figv@ctCreg-1(\v@lX,\v@lY,\v@lZ)\Figp@intreg-2:(\wd\Bt@rget,\ht\Bt@rget,\dp\Bt@rget)%
    \figpttra#1:=-2/-\disob@intern,-1/\else\end\fi}
\ctr@ld@f\def\t@stp@r{\itis@Ktrue\ifnewt@rgetpt\else\itis@Kfalse%
    \message{*** \BS@ figptvisilimXX: target point undefined.}\fi\ifnewdis@b\else%
    \itis@Kfalse\message{*** \BS@ figptvisilimXX: observation distance undefined.}\fi%
    \ifitis@K\else\message{*** This macro must be called after \BS@ psbeginfig or after
    having set the missing parameter(s) with \BS@ figset proj()}\fi}
\ctr@ld@f\def\figscan#1(#2,#3){{\s@uvc@ntr@l\et@tfigscan\@psfgetbb{#1}\if@psfbbfound\else%
    \def\@psfllx{0}\def\@psflly{20}\def\@psfurx{540}\def\@psfury{640}\fi\figscan@{#2}{#3}%
    \resetc@ntr@l\et@tfigscan}\ignorespaces}
\ctr@ld@f\def\figscan@#1#2{%
    \unit@=\@ne bp\setc@ntr@l{2}\figsetmark{}%
    \def\minst@p{20pt}%
    \v@lX=\@psfllx\p@\v@lX=\Sc@leFact\v@lX\r@undint\v@lX\v@lX%
    \v@lY=\@psflly\p@\v@lY=\Sc@leFact\v@lY\ifdim\v@lY>\z@\r@undint\v@lY\v@lY\fi%
    \delt@=\@psfury\p@\delt@=\Sc@leFact\delt@%
    \advance\delt@-\v@lY\v@lXa=\@psfurx\p@\v@lXa=\Sc@leFact\v@lXa\v@leur=\minst@p%
    \edef\valv@lY{\repdecn@mb{\v@lY}}\edef\LgTr@it{\the\delt@}%
    \loop\ifdim\v@lX<\v@lXa\edef\valv@lX{\repdecn@mb{\v@lX}}%
    \figptDD -1:(\valv@lX,\valv@lY)\figwriten -1:\hbox{\vrule height\LgTr@it}(0)%
    \ifdim\v@leur<\minst@p\else\figsetmark{\raise-8bp\hbox{$\scriptscriptstyle\triangle$}}%
    \figwrites -1:\@ffichnb{0}{\valv@lX}(6)\v@leur=\z@\figsetmark{}\fi%
    \advance\v@leur#1pt\advance\v@lX#1pt\repeat%
    \def\minst@p{10pt}%
    \v@lX=\@psfllx\p@\v@lX=\Sc@leFact\v@lX\ifdim\v@lX>\z@\r@undint\v@lX\v@lX\fi%
    \v@lY=\@psflly\p@\v@lY=\Sc@leFact\v@lY\r@undint\v@lY\v@lY%
    \delt@=\@psfurx\p@\delt@=\Sc@leFact\delt@%
    \advance\delt@-\v@lX\v@lYa=\@psfury\p@\v@lYa=\Sc@leFact\v@lYa\v@leur=\minst@p%
    \edef\valv@lX{\repdecn@mb{\v@lX}}\edef\LgTr@it{\the\delt@}%
    \loop\ifdim\v@lY<\v@lYa\edef\valv@lY{\repdecn@mb{\v@lY}}%
    \figptDD -1:(\valv@lX,\valv@lY)\figwritee -1:\vbox{\hrule width\LgTr@it}(0)%
    \ifdim\v@leur<\minst@p\else\figsetmark{$\triangleright$\kern4bp}%
    \figwritew -1:\@ffichnb{0}{\valv@lY}(6)\v@leur=\z@\figsetmark{}\fi%
    \advance\v@leur#2pt\advance\v@lY#2pt\repeat}
\ctr@ld@f\let\figscanI=\figscan
\ctr@ld@f\def\figscan@E#1(#2,#3){{\s@uvc@ntr@l\et@tfigscan@E%
    \Figdisc@rdLTS{#1}{\t@xt@}\pdfximage{\t@xt@}%
    \setbox\Gb@x=\hbox{\pdfrefximage\pdflastximage}%
    \edef\@psfllx{0}\v@lY=-\dp\Gb@x\edef\@psflly{\repdecn@mb{\v@lY}}%
    \edef\@psfurx{\repdecn@mb{\wd\Gb@x}}%
    \v@lY=\dp\Gb@x\advance\v@lY\ht\Gb@x\edef\@psfury{\repdecn@mb{\v@lY}}%
    \figscan@{#2}{#3}\resetc@ntr@l\et@tfigscan@E}\ignorespaces}
\ctr@ld@f\def\figshowpts[#1,#2]{{\figsetmark{$\bullet$}\figsetptname{\bf ##1}%
    \p@rtent=#2\relax\ifnum\p@rtent<\z@\p@rtent=\z@\fi%
    \s@mme=#1\relax\ifnum\s@mme<\z@\s@mme=\z@\fi%
    \loop\ifnum\s@mme<\p@rtent\pt@rvect{\s@mme}%
    \ifitis@K\figwriten{\the\s@mme}:(4pt)\fi\advance\s@mme\@ne\repeat%
    \pt@rvect{\s@mme}\ifitis@K\figwriten{\the\s@mme}:(4pt)\fi}\ignorespaces}
\ctr@ld@f\def\pt@rvect#1{\set@bjc@de{#1}%
    \expandafter\expandafter\expandafter\inqpt@rvec\csname\objc@de\endcsname:}
\ctr@ld@f\def\inqpt@rvec#1#2:{\if#1\C@dCl@spt\itis@Ktrue\else\itis@Kfalse\fi}
\ctr@ld@f\def\figshowsettings{{%
    \immediate\write16{====================================================================}%
    \immediate\write16{ Current settings about:}%
    \immediate\write16{ --- GENERAL ---}%
    \immediate\write16{Scale factor and Unit = \unit@util\space (\the\unit@)
     \space -> \BS@ figinit{ScaleFactorUnit}}%
    \immediate\write16{Update mode = \ifpsupdatem@de yes\else no\fi
     \space-> \BS@ psset(update=yes/no) or \BS@ pssetdefault(update=yes/no)}%
    \immediate\write16{ --- PRINTING ---}%
    \immediate\write16{Implicit point name = \ptn@me{i} \space-> \BS@ figsetptname{Name}}%
    \immediate\write16{Point marker = \the\c@nsymb \space -> \BS@ figsetmark{Mark}}%
    \immediate\write16{Print rounded coordinates = \ifr@undcoord yes\else no\fi
     \space-> \BS@ figsetroundcoord{yes/no}}%
    \immediate\write16{ --- GRAPHICAL (general) ---}%
    \immediate\write16{First-level (or primary) settings:}%
    \immediate\write16{ Color = \curr@ntcolor \space-> \BS@ psset(color=ColorDefinition)}%
    \immediate\write16{ Filling mode = \iffillm@de yes\else no\fi
     \space-> \BS@ psset(fillmode=yes/no)}%
    \immediate\write16{ Line join = \curr@ntjoin \space-> \BS@ psset(join=miter/round/bevel)}%
    \immediate\write16{ Line style = \curr@ntdash \space-> \BS@ psset(dash=Index/Pattern)}%
    \immediate\write16{ Line width = \curr@ntwidth
     \space-> \BS@ psset(width=real in PostScript units)}%
    \immediate\write16{Second-level (or secondary) settings:}%
    \immediate\write16{ Color = \sec@ndcolor \space-> \BS@ psset second(color=ColorDefinition)}%
    \immediate\write16{ Line style = \curr@ntseconddash
     \space-> \BS@ psset second(dash=Index/Pattern)}%
    \immediate\write16{ Line width = \curr@ntsecondwidth
     \space-> \BS@ psset second(width=real in PostScript units)}%
    \immediate\write16{Third-level (or ternary) settings:}%
    \immediate\write16{ Color = \th@rdcolor \space-> \BS@ psset third(color=ColorDefinition)}%
    \immediate\write16{ --- GRAPHICAL (specific) ---}%
    \immediate\write16{Arrow-head:}%
    \immediate\write16{ (half-)Angle = \@rrowheadangle
     \space-> \BS@ psset arrowhead(angle=real in degrees)}%
    \immediate\write16{ Filling mode = \if@rrowhfill yes\else no\fi
     \space-> \BS@ psset arrowhead(fillmode=yes/no)}%
    \immediate\write16{ "Outside" = \if@rrowhout yes\else no\fi
     \space-> \BS@ psset arrowhead(out=yes/no)}%
    \immediate\write16{ Length = \@rrowheadlength
     \if@rrowratio\space(not active)\else\space(active)\fi
     \space-> \BS@ psset arrowhead(length=real in user coord.)}%
    \immediate\write16{ Ratio = \@rrowheadratio
     \if@rrowratio\space(active)\else\space(not active)\fi
     \space-> \BS@ psset arrowhead(ratio=real in [0,1])}%
    \immediate\write16{Curve: Roundness = \curv@roundness
     \space-> \BS@ psset curve(roundness=real in [0,0.5])}%
    \immediate\write16{Mesh: Diagonal = \c@ntrolmesh
     \space-> \BS@ psset mesh(diag=integer in {-1,0,1})}%
    \immediate\write16{Flow chart:}%
    \immediate\write16{ Arrow position = \@rrowp@s
     \space-> \BS@ psset flowchart(arrowposition=real in [0,1])}%
    \immediate\write16{ Arrow reference point = \ifcase\@rrowr@fpt start\else end\fi
     \space-> \BS@ psset flowchart(arrowrefpt = start/end)}%     
    \immediate\write16{ Line type = \ifcase\fclin@typ@ curve\else polygon\fi
     \space-> \BS@ psset flowchart(line=polygon/curve)}%
    \immediate\write16{ Padding = (\Xp@dd, \Yp@dd)
     \space-> \BS@ psset flowchart(padding = real in user coord.)}%
    \immediate\write16{\space\space\space\space(or
     \BS@ psset flowchart(xpadding=real, ypadding=real) )}%
    \immediate\write16{ Radius = \fclin@r@d
     \space-> \BS@ psset flowchart(radius=positive real in user coord.)}%
    \immediate\write16{ Shape = \fcsh@pe
     \space-> \BS@ psset flowchart(shape = rectangle, ellipse or lozenge)}%
    \immediate\write16{ Thickness = \thickn@ss
     \space-> \BS@ psset flowchart(thickness = real in user coord.)}%
    \ifTr@isDim%
    \immediate\write16{ --- 3D to 2D PROJECTION ---}%
    \immediate\write16{Projection : \typ@proj \space-> \BS@ figinit{ScaleFactorUnit, ProjType}}%
    \immediate\write16{Longitude (psi) = \v@lPsi \space-> \BS@ figset proj(psi=real in degrees)}%
    \ifcase\curr@ntproj\immediate\write16{Depth coeff. (Lambda)
     \space = \v@lTheta \space-> \BS@ figset proj(lambda=real in [0,1])}%
    \else\immediate\write16{Latitude (theta)
     \space = \v@lTheta \space-> \BS@ figset proj(theta=real in degrees)}%
    \fi%
    \ifnum\curr@ntproj=\tw@%
    \immediate\write16{Observation distance = \disob@unit
     \space-> \BS@ figset proj(dist=real in user coord.)}%
    \immediate\write16{Target point = \t@rgetpt \space-> \BS@ figset proj(targetpt=pt number)}%
     \v@lX=\ptT@unit@\wd\Bt@rget\v@lY=\ptT@unit@\ht\Bt@rget\v@lZ=\ptT@unit@\dp\Bt@rget%
    \immediate\write16{ Its coordinates are
     (\repdecn@mb{\v@lX}, \repdecn@mb{\v@lY}, \repdecn@mb{\v@lZ})}%
    \fi%
    \fi%
    \immediate\write16{====================================================================}%
    \ignorespaces}}
\ctr@ln@w{newif}\ifitis@vect@r
\ctr@ld@f\def\figvectC#1(#2,#3){{\itis@vect@rtrue\figpt#1:(#2,#3)}\ignorespaces}
\ctr@ld@f\def\Figv@ctCreg#1(#2,#3){{\itis@vect@rtrue\Figp@intreg#1:(#2,#3)}\ignorespaces}
\ctr@ln@m\figvectDBezier
\ctr@ld@f\def\figvectDBezierDD#1:#2,#3[#4,#5,#6,#7]{\ifps@cri{\s@uvc@ntr@l\et@tfigvectDBezierDD%
    \FigvectDBezier@#2,#3[#4,#5,#6,#7]\v@lX=\c@ef\v@lX\v@lY=\c@ef\v@lY%
    \Figv@ctCreg#1(\v@lX,\v@lY)\resetc@ntr@l\et@tfigvectDBezierDD}\ignorespaces\fi}
\ctr@ld@f\def\figvectDBezierTD#1:#2,#3[#4,#5,#6,#7]{\ifps@cri{\s@uvc@ntr@l\et@tfigvectDBezierTD%
    \FigvectDBezier@#2,#3[#4,#5,#6,#7]\v@lX=\c@ef\v@lX\v@lY=\c@ef\v@lY\v@lZ=\c@ef\v@lZ%
    \Figv@ctCreg#1(\v@lX,\v@lY,\v@lZ)\resetc@ntr@l\et@tfigvectDBezierTD}\ignorespaces\fi}
\ctr@ld@f\def\FigvectDBezier@#1,#2[#3,#4,#5,#6]{\setc@ntr@l{2}%
    \edef\T@{#2}\v@leur=\p@\advance\v@leur-#2pt\edef\UNmT@{\repdecn@mb{\v@leur}}%
    \ifnum#1=\tw@\def\c@ef{6}\else\def\c@ef{3}\fi%
    \figptcopy-4:/#3/\figptcopy-3:/#4/\figptcopy-2:/#5/\figptcopy-1:/#6/%
    \l@mbd@un=-4 \l@mbd@de=-\thr@@\p@rtent=\m@ne\c@lDecast%
    \ifnum#1=\tw@\c@lDCDeux{-4}{-3}\c@lDCDeux{-3}{-2}\c@lDCDeux{-4}{-3}\else%
    \l@mbd@un=-4 \l@mbd@de=-\thr@@\p@rtent=-\tw@\c@lDecast%
    \c@lDCDeux{-4}{-3}\fi\Figg@tXY{-4}}
\ctr@ln@m\c@lDCDeux
\ctr@ld@f\def\c@lDCDeuxDD#1#2{\Figg@tXY{#2}\Figg@tXYa{#1}%
    \advance\v@lX-\v@lXa\advance\v@lY-\v@lYa\Figp@intregDD#1:(\v@lX,\v@lY)}
\ctr@ld@f\def\c@lDCDeuxTD#1#2{\Figg@tXY{#2}\Figg@tXYa{#1}\advance\v@lX-\v@lXa%
    \advance\v@lY-\v@lYa\advance\v@lZ-\v@lZa\Figp@intregTD#1:(\v@lX,\v@lY,\v@lZ)}
\ctr@ln@m\figvectN
\ctr@ld@f\def\figvectNDD#1[#2,#3]{\ifps@cri{\Figg@tXYa{#2}\Figg@tXY{#3}%
    \advance\v@lX-\v@lXa\advance\v@lY-\v@lYa%
    \Figv@ctCreg#1(-\v@lY,\v@lX)}\ignorespaces\fi}
\ctr@ld@f\def\figvectNTD#1[#2,#3,#4]{\ifps@cri{\vecunitC@TD[#2,#4]\v@lmin=\v@lX\v@lmax=\v@lY%
    \v@leur=\v@lZ\vecunitC@TD[#2,#3]\c@lprovec{#1}}\ignorespaces\fi}
\ctr@ln@m\figvectNV
\ctr@ld@f\def\figvectNVDD#1[#2]{\ifps@cri{\Figg@tXY{#2}\Figv@ctCreg#1(-\v@lY,\v@lX)}\ignorespaces\fi}
\ctr@ld@f\def\figvectNVTD#1[#2,#3]{\ifps@cri{\vecunitCV@TD{#3}\v@lmin=\v@lX\v@lmax=\v@lY%
    \v@leur=\v@lZ\vecunitCV@TD{#2}\c@lprovec{#1}}\ignorespaces\fi}
\ctr@ln@m\figvectP
\ctr@ld@f\def\figvectPDD#1[#2,#3]{\ifps@cri{\Figg@tXYa{#2}\Figg@tXY{#3}%
    \advance\v@lX-\v@lXa\advance\v@lY-\v@lYa%
    \Figv@ctCreg#1(\v@lX,\v@lY)}\ignorespaces\fi}
\ctr@ld@f\def\figvectPTD#1[#2,#3]{\ifps@cri{\Figg@tXYa{#2}\Figg@tXY{#3}%
    \advance\v@lX-\v@lXa\advance\v@lY-\v@lYa\advance\v@lZ-\v@lZa%
    \Figv@ctCreg#1(\v@lX,\v@lY,\v@lZ)}\ignorespaces\fi}
\ctr@ln@m\figvectU
\ctr@ld@f\def\figvectUDD#1[#2]{\ifps@cri{\n@rmeuc\v@leur{#2}\invers@\v@leur\v@leur%
    \delt@=\repdecn@mb{\v@leur}\unit@\edef\v@ldelt@{\repdecn@mb{\delt@}}%
    \Figg@tXY{#2}\v@lX=\v@ldelt@\v@lX\v@lY=\v@ldelt@\v@lY%
    \Figv@ctCreg#1(\v@lX,\v@lY)}\ignorespaces\fi}
\ctr@ld@f\def\figvectUTD#1[#2]{\ifps@cri{\n@rmeuc\v@leur{#2}\invers@\v@leur\v@leur%
    \delt@=\repdecn@mb{\v@leur}\unit@\edef\v@ldelt@{\repdecn@mb{\delt@}}%
    \Figg@tXY{#2}\v@lX=\v@ldelt@\v@lX\v@lY=\v@ldelt@\v@lY\v@lZ=\v@ldelt@\v@lZ%
    \Figv@ctCreg#1(\v@lX,\v@lY,\v@lZ)}\ignorespaces\fi}
\ctr@ld@f\def\figvisu#1#2#3{\c@ldefproj\initb@undb@x\xdef\figforTeXFigno{\figforTeXnextFigno}%
    \s@mme=\figforTeXnextFigno\advance\s@mme\@ne\xdef\figforTeXnextFigno{\number\s@mme}%
    \setbox\b@xvisu=\hbox{\ifnum\@utoFN>\z@\figinsert{}\gdef\@utoFInDone{0}\fi\ignorespaces#3}%
    \gdef\@utoFInDone{1}\gdef\@utoFN{0}%
    \v@lXa=-\c@@rdYmin\v@lYa=\c@@rdYmax\advance\v@lYa-\c@@rdYmin%
    \v@lX=\c@@rdXmax\advance\v@lX-\c@@rdXmin%
    \setbox#1=\hbox{#2}\v@lY=-\v@lX\maxim@m{\v@lX}{\v@lX}{\wd#1}%
    \advance\v@lY\v@lX\divide\v@lY\tw@\advance\v@lY-\c@@rdXmin%
    \setbox#1=\vbox{\parindent0mm\hsize=\v@lX\vskip\v@lYa%
    \rlap{\hskip\v@lY\smash{\raise\v@lXa\box\b@xvisu}}%
    \def\t@xt@{#2}\ifx\t@xt@\empty\else\medskip\centerline{#2}\fi}\wd#1=\v@lX}
\ctr@ld@f\def\figDecrementFigno{{\xdef\figforTeXnextFigno{\figforTeXFigno}%
    \s@mme=\figforTeXFigno\advance\s@mme\m@ne\xdef\figforTeXFigno{\number\s@mme}}}
\ctr@ln@w{newbox}\Bt@rget\setbox\Bt@rget=\null
\ctr@ln@w{newbox}\BminTD@\setbox\BminTD@=\null
\ctr@ln@w{newbox}\BmaxTD@\setbox\BmaxTD@=\null
\ctr@ln@w{newif}\ifnewt@rgetpt\ctr@ln@w{newif}\ifnewdis@b
\ctr@ld@f\def\b@undb@xTD#1#2#3{%
    \relax\ifdim#1<\wd\BminTD@\global\wd\BminTD@=#1\fi%
    \relax\ifdim#2<\ht\BminTD@\global\ht\BminTD@=#2\fi%
    \relax\ifdim#3<\dp\BminTD@\global\dp\BminTD@=#3\fi%
    \relax\ifdim#1>\wd\BmaxTD@\global\wd\BmaxTD@=#1\fi%
    \relax\ifdim#2>\ht\BmaxTD@\global\ht\BmaxTD@=#2\fi%
    \relax\ifdim#3>\dp\BmaxTD@\global\dp\BmaxTD@=#3\fi}
\ctr@ld@f\def\c@ldefdisob{{\ifdim\wd\BminTD@<\maxdimen\v@leur=\wd\BmaxTD@\advance\v@leur-\wd\BminTD@%
    \delt@=\ht\BmaxTD@\advance\delt@-\ht\BminTD@\maxim@m{\v@leur}{\v@leur}{\delt@}%
    \delt@=\dp\BmaxTD@\advance\delt@-\dp\BminTD@\maxim@m{\v@leur}{\v@leur}{\delt@}%
    \v@leur=5\v@leur\else\v@leur=800pt\fi\c@ldefdisob@{\v@leur}}}
\ctr@ln@m\disob@intern
\ctr@ln@m\disob@
\ctr@ln@m\divf@ctproj
\ctr@ld@f\def\c@ldefdisob@#1{{\v@leur=#1\ifdim\v@leur<\p@\v@leur=800pt\fi%
    \xdef\disob@intern{\repdecn@mb{\v@leur}}%
    \delt@=\ptT@unit@\v@leur\xdef\disob@unit{\repdecn@mb{\delt@}}%
    \f@ctech=\@ne\loop\ifdim\v@leur>\t@n pt\divide\v@leur\t@n\multiply\f@ctech\t@n\repeat%
    \xdef\disob@{\repdecn@mb{\v@leur}}\xdef\divf@ctproj{\the\f@ctech}}%
    \global\newdis@btrue}
\ctr@ln@m\t@rgetpt
\ctr@ld@f\def\c@ldeft@rgetpt{\newt@rgetpttrue\def\t@rgetpt{CenterBoundBox}{%
    \delt@=\wd\BmaxTD@\advance\delt@-\wd\BminTD@\divide\delt@\tw@%
    \v@leur=\wd\BminTD@\advance\v@leur\delt@\global\wd\Bt@rget=\v@leur%
    \delt@=\ht\BmaxTD@\advance\delt@-\ht\BminTD@\divide\delt@\tw@%
    \v@leur=\ht\BminTD@\advance\v@leur\delt@\global\ht\Bt@rget=\v@leur%
    \delt@=\dp\BmaxTD@\advance\delt@-\dp\BminTD@\divide\delt@\tw@%
    \v@leur=\dp\BminTD@\advance\v@leur\delt@\global\dp\Bt@rget=\v@leur}}
\ctr@ln@m\c@ldefproj
\ctr@ld@f\def\c@ldefprojTD{\ifnewt@rgetpt\else\c@ldeft@rgetpt\fi\ifnewdis@b\else\c@ldefdisob\fi}
\ctr@ld@f\def\c@lprojcav{% Projection cavaliere : X = x + y L cos t, Y = z + y L sin t
    \v@lZa=\cxa@\v@lY\advance\v@lX\v@lZa%
    \v@lZa=\cxb@\v@lY\v@lY=\v@lZ\advance\v@lY\v@lZa\ignorespaces}
\ctr@ln@m\v@lcoef
\ctr@ld@f\def\c@lprojrea{% Projection realiste
    \advance\v@lX-\wd\Bt@rget\advance\v@lY-\ht\Bt@rget\advance\v@lZ-\dp\Bt@rget%
    \v@lZa=\cza@\v@lX\advance\v@lZa\czb@\v@lY\advance\v@lZa\czc@\v@lZ%
    \divide\v@lZa\divf@ctproj\advance\v@lZa\disob@ pt\invers@{\v@lZa}{\v@lZa}%
    \v@lZa=\disob@\v@lZa\edef\v@lcoef{\repdecn@mb{\v@lZa}}%
    \v@lXa=\cxa@\v@lX\advance\v@lXa\cxb@\v@lY\v@lXa=\v@lcoef\v@lXa%
    \v@lY=\cyb@\v@lY\advance\v@lY\cya@\v@lX\advance\v@lY\cyc@\v@lZ%
    \v@lY=\v@lcoef\v@lY\v@lX=\v@lXa\ignorespaces}
\ctr@ld@f\def\c@lprojort{% Projection orthogonale
    \v@lXa=\cxa@\v@lX\advance\v@lXa\cxb@\v@lY%
    \v@lY=\cyb@\v@lY\advance\v@lY\cya@\v@lX\advance\v@lY\cyc@\v@lZ%
    \v@lX=\v@lXa\ignorespaces}
\ctr@ld@f\def\Figptpr@j#1:#2/#3/{{\Figg@tXY{#3}\superc@lprojSP%
    \Figp@intregDD#1:{#2}(\v@lX,\v@lY)}\ignorespaces}
\ctr@ln@m\figsetobdist
\ctr@ld@f\def\figsetobdistDD{\un@v@ilable{figsetobdist}}
\ctr@ld@f\def\figsetobdistTD(#1){{\ifcurr@ntPS%
    \immediate\write16{*** \BS@ figsetobdist is ignored inside a
     \BS@ psbeginfig-\BS@ psendfig block.}%
    \else\v@leur=#1\unit@\c@ldefdisob@{\v@leur}\fi}\ignorespaces}
\ctr@ln@m\c@lprojSP
\ctr@ln@m\curr@ntproj
\ctr@ln@m\typ@proj
\ctr@ln@m\superc@lprojSP
\ctr@ld@f\def\Figs@tproj#1{%
    \if#13 \d@faultproj\else\if#1c\d@faultproj%
    \else\if#1o\xdef\curr@ntproj{1}\xdef\typ@proj{orthogonal}%
         \figsetviewTD(\def@ultpsi,\def@ulttheta)%
         \global\let\c@lprojSP=\c@lprojort\global\let\superc@lprojSP=\c@lprojort%
    \else\if#1r\xdef\curr@ntproj{2}\xdef\typ@proj{realistic}%
         \figsetviewTD(\def@ultpsi,\def@ulttheta)%
         \global\let\c@lprojSP=\c@lprojrea\global\let\superc@lprojSP=\c@lprojrea%
    \else\d@faultproj\message{*** Unknown projection. Cavalier projection assumed.}%
    \fi\fi\fi\fi}
\ctr@ld@f\def\d@faultproj{\xdef\curr@ntproj{0}\xdef\typ@proj{cavalier}\figsetviewTD(\def@ultpsi,0.5)%
         \global\let\c@lprojSP=\c@lprojcav\global\let\superc@lprojSP=\c@lprojcav}
\ctr@ln@m\figsettarget
\ctr@ld@f\def\figsettargetDD{\un@v@ilable{figsettarget}}
\ctr@ld@f\def\figsettargetTD[#1]{{\ifcurr@ntPS%
    \immediate\write16{*** \BS@ figsettarget is ignored inside a
     \BS@ psbeginfig-\BS@ psendfig block.}%
    \else\global\newt@rgetpttrue\xdef\t@rgetpt{#1}\Figg@tXY{#1}\global\wd\Bt@rget=\v@lX%
    \global\ht\Bt@rget=\v@lY\global\dp\Bt@rget=\v@lZ\fi}\ignorespaces}
\ctr@ln@m\figsetview
\ctr@ld@f\def\figsetviewDD{\un@v@ilable{figsetview}}
\ctr@ld@f\def\figsetviewTD(#1){\ifcurr@ntPS%
     \immediate\write16{*** \BS@ figsetview is ignored inside a
     \BS@ psbeginfig-\BS@ psendfig block.}\else\Figsetview@#1,:\fi\ignorespaces}
\ctr@ld@f\def\Figsetview@#1,#2:{{\xdef\v@lPsi{#1}\def\t@xt@{#2}%
    \ifx\t@xt@\empty\def\@rgdeux{\v@lTheta}\else\X@rgdeux@#2\fi%
    \c@ssin{\costhet@}{\sinthet@}{#1}\v@lmin=\costhet@ pt\v@lmax=\sinthet@ pt%
    \ifcase\curr@ntproj%
    \v@leur=\@rgdeux\v@lmin\xdef\cxa@{\repdecn@mb{\v@leur}}%
    \v@leur=\@rgdeux\v@lmax\xdef\cxb@{\repdecn@mb{\v@leur}}\v@leur=\@rgdeux pt%
    \relax\ifdim\v@leur>\p@\message{*** Lambda too large ! See \BS@ figset proj() !}\fi%
    \else%
    \v@lmax=-\v@lmax\xdef\cxa@{\repdecn@mb{\v@lmax}}\xdef\cxb@{\costhet@}%
    \ifx\t@xt@\empty\edef\@rgdeux{\def@ulttheta}\fi\c@ssin{\C@}{\S@}{\@rgdeux}%
    \v@lmax=-\S@ pt%
    \v@leur=\v@lmax\v@leur=\costhet@\v@leur\xdef\cya@{\repdecn@mb{\v@leur}}%
    \v@leur=\v@lmax\v@leur=\sinthet@\v@leur\xdef\cyb@{\repdecn@mb{\v@leur}}%
    \xdef\cyc@{\C@}\v@lmin=-\C@ pt%
    \v@leur=\v@lmin\v@leur=\costhet@\v@leur\xdef\cza@{\repdecn@mb{\v@leur}}%
    \v@leur=\v@lmin\v@leur=\sinthet@\v@leur\xdef\czb@{\repdecn@mb{\v@leur}}%
    \xdef\czc@{\repdecn@mb{\v@lmax}}\fi%
    \xdef\v@lTheta{\@rgdeux}}}
\ctr@ld@f\def\def@ultpsi{40}
\ctr@ld@f\def\def@ulttheta{25}
\ctr@ln@m\l@debut
\ctr@ln@m\n@mref
\ctr@ld@f\def\figset#1(#2){\def\t@xt@{#1}\ifx\t@xt@\empty\trtlis@rg{#2}{\Figsetwr@te}% write
    \else\keln@mde#1|%
    \def\n@mref{pr}\ifx\l@debut\n@mref\ifcurr@ntPS% projection
     \immediate\write16{*** \BS@ figset proj(...) is ignored inside a
     \BS@ psbeginfig-\BS@ psendfig block.}\else\trtlis@rg{#2}{\Figsetpr@j}\fi\else%
    \def\n@mref{wr}\ifx\l@debut\n@mref\trtlis@rg{#2}{\Figsetwr@te}\else% write
    \immediate\write16{*** Unknown keyword: \BS@ figset #1(...)}%
    \fi\fi\fi\ignorespaces}
\ctr@ld@f\def\Figsetpr@j#1=#2|{\keln@mtr#1|%
    \def\n@mref{dep}\ifx\l@debut\n@mref\Figsetd@p{#2}\else% depth (lambda)
    \def\n@mref{dis}\ifx\l@debut\n@mref%
     \ifnum\curr@ntproj=\tw@\figsetobdist(#2)\else\Figset@rr\fi\else% dist
    \def\n@mref{lam}\ifx\l@debut\n@mref\Figsetd@p{#2}\else% depth (lambda)
    \def\n@mref{lat}\ifx\l@debut\n@mref\Figsetth@{#2}\else% latitude (theta)
    \def\n@mref{lon}\ifx\l@debut\n@mref\figsetview(#2)\else% longitude (psi)
    \def\n@mref{psi}\ifx\l@debut\n@mref\figsetview(#2)\else% longitude (psi)
    \def\n@mref{tar}\ifx\l@debut\n@mref%
     \ifnum\curr@ntproj=\tw@\figsettarget[#2]\else\Figset@rr\fi\else% target point
    \def\n@mref{the}\ifx\l@debut\n@mref\Figsetth@{#2}\else% latitude (theta)
    \immediate\write16{*** Unknown attribute: \BS@ figset proj(..., #1=...).}%
    \fi\fi\fi\fi\fi\fi\fi\fi}
\ctr@ld@f\def\Figsetd@p#1{\ifnum\curr@ntproj=\z@\figsetview(\v@lPsi,#1)\else\Figset@rr\fi}
\ctr@ld@f\def\Figsetth@#1{\ifnum\curr@ntproj=\z@\Figset@rr\else\figsetview(\v@lPsi,#1)\fi}
\ctr@ld@f\def\Figset@rr{\message{*** \BS@ figset proj(): Attribute "\n@mref" ignored, incompatible
    with current projection}}
\ctr@ld@f\def\initb@undb@xTD{\wd\BminTD@=\maxdimen\ht\BminTD@=\maxdimen\dp\BminTD@=\maxdimen%
    \wd\BmaxTD@=-\maxdimen\ht\BmaxTD@=-\maxdimen\dp\BmaxTD@=-\maxdimen}
\ctr@ln@w{newbox}\Gb@x      % boite a tout faire
\ctr@ln@w{newbox}\Gb@xSC    % boite qui contient le point marker
\ctr@ln@w{newtoks}\c@nsymb  % the point marker
\ctr@ln@w{newif}\ifr@undcoord\ctr@ln@w{newif}\ifunitpr@sent
\ctr@ld@f\def\unssqrttw@{0.707106 }
\ctr@ld@f\def\figAst{\raise-1.15ex\hbox{$\ast$}}
\ctr@ld@f\def\figBullet{\raise-1.15ex\hbox{$\bullet$}}
\ctr@ld@f\def\figCirc{\raise-1.15ex\hbox{$\circ$}}
\ctr@ld@f\def\figDiamond{\raise-1.15ex\hbox{$\diamond$}}%
\ctr@ld@f\def\boxit#1#2{\leavevmode\hbox{\vrule\vbox{\hrule\vglue#1%
    \vtop{\hbox{\kern#1{#2}\kern#1}\vglue#1\hrule}}\vrule}}
\ctr@ld@f\def\centertext#1#2{\vbox{\hsize#1\parindent0cm%
    \leftskip=0pt plus 1fil\rightskip=0pt plus 1fil\parfillskip=0pt{#2}}}
\ctr@ld@f\def\lefttext#1#2{\vbox{\hsize#1\parindent0cm\rightskip=0pt plus 1fil#2}}
\ctr@ld@f\def\c@nterpt{\ignorespaces%
    \kern-.5\wd\Gb@xSC%
    \raise-.5\ht\Gb@xSC\rlap{\hbox{\raise.5\dp\Gb@xSC\hbox{\copy\Gb@xSC}}}%
    \kern .5\wd\Gb@xSC\ignorespaces}
\ctr@ld@f\def\b@undb@xSC#1#2{{\v@lXa=#1\v@lYa=#2%
    \v@leur=\ht\Gb@xSC\advance\v@leur\dp\Gb@xSC%
    \advance\v@lXa-.5\wd\Gb@xSC\advance\v@lYa-.5\v@leur\b@undb@x{\v@lXa}{\v@lYa}%
    \advance\v@lXa\wd\Gb@xSC\advance\v@lYa\v@leur\b@undb@x{\v@lXa}{\v@lYa}}}
\ctr@ln@m\Dist@n
\ctr@ln@m\l@suite
\ctr@ld@f\def\@keldist#1#2{\edef\Dist@n{#2}\y@tiunit{\Dist@n}%
    \ifunitpr@sent#1=\Dist@n\else#1=\Dist@n\unit@\fi}
\ctr@ld@f\def\y@tiunit#1{\unitpr@sentfalse\expandafter\y@tiunit@#1:}
\ctr@ld@f\def\y@tiunit@#1#2:{\ifcat#1a\unitpr@senttrue\else\def\l@suite{#2}%
    \ifx\l@suite\empty\else\y@tiunit@#2:\fi\fi}
\ctr@ln@m\figcoord
\ctr@ld@f\def\figcoordDD#1{{\v@lX=\ptT@unit@\v@lX\v@lY=\ptT@unit@\v@lY%
    \ifr@undcoord\ifcase#1\v@leur=0.5pt\or\v@leur=0.05pt\or\v@leur=0.005pt%
    \or\v@leur=0.0005pt\else\v@leur=\z@\fi%
    \ifdim\v@lX<\z@\advance\v@lX-\v@leur\else\advance\v@lX\v@leur\fi%
    \ifdim\v@lY<\z@\advance\v@lY-\v@leur\else\advance\v@lY\v@leur\fi\fi%
    (\@ffichnb{#1}{\repdecn@mb{\v@lX}},\ifmmode\else\thinspace\fi%
    \@ffichnb{#1}{\repdecn@mb{\v@lY}})}}
\ctr@ld@f\def\@ffichnb#1#2{{\def\@@ffich{\@ffich#1(}\edef\n@mbre{#2}%
    \expandafter\@@ffich\n@mbre)}}
\ctr@ld@f\def\@ffich#1(#2.#3){{#2\ifnum#1>\z@.\fi\def\dig@ts{#3}\s@mme=\z@%
    \loop\ifnum\s@mme<#1\expandafter\@ffichdec\dig@ts:\advance\s@mme\@ne\repeat}}
\ctr@ld@f\def\@ffichdec#1#2:{\relax#1\def\dig@ts{#20}}
\ctr@ld@f\def\figcoordTD#1{{\v@lX=\ptT@unit@\v@lX\v@lY=\ptT@unit@\v@lY\v@lZ=\ptT@unit@\v@lZ%
    \ifr@undcoord\ifcase#1\v@leur=0.5pt\or\v@leur=0.05pt\or\v@leur=0.005pt%
    \or\v@leur=0.0005pt\else\v@leur=\z@\fi%
    \ifdim\v@lX<\z@\advance\v@lX-\v@leur\else\advance\v@lX\v@leur\fi%
    \ifdim\v@lY<\z@\advance\v@lY-\v@leur\else\advance\v@lY\v@leur\fi%
    \ifdim\v@lZ<\z@\advance\v@lZ-\v@leur\else\advance\v@lZ\v@leur\fi\fi%
    (\@ffichnb{#1}{\repdecn@mb{\v@lX}},\ifmmode\else\thinspace\fi%
     \@ffichnb{#1}{\repdecn@mb{\v@lY}},\ifmmode\else\thinspace\fi%
     \@ffichnb{#1}{\repdecn@mb{\v@lZ}})}}
\ctr@ld@f\def\figsetroundcoord#1{\expandafter\Figsetr@undcoord#1:\ignorespaces}
\ctr@ld@f\def\Figsetr@undcoord#1#2:{\if#1n\r@undcoordfalse\else\r@undcoordtrue\fi}
\ctr@ld@f\def\Figsetwr@te#1=#2|{\keln@mun#1|%
    \def\n@mref{m}\ifx\l@debut\n@mref\figsetmark{#2}\else% mark
    \immediate\write16{*** Unknown attribute: \BS@ figset (..., #1=...)}%
    \fi}
\ctr@ld@f\def\figsetmark#1{\c@nsymb={#1}\setbox\Gb@xSC=\hbox{\the\c@nsymb}\ignorespaces}
\ctr@ln@m\ptn@me
\ctr@ld@f\def\figsetptname#1{\def\ptn@me##1{#1}\ignorespaces}
\ctr@ld@f\def\FigWrit@L#1:#2(#3,#4){\ignorespaces\@keldist\v@leur{#3}\@keldist\delt@{#4}%
    \C@rp@r@m\def\list@num{#1}\@ecfor\p@int:=\list@num\do{\FigWrit@pt{\p@int}{#2}}}
\ctr@ld@f\def\FigWrit@pt#1#2{\FigWp@r@m{#1}{#2}\Vc@rrect\figWp@si%
    \ifdim\wd\Gb@xSC>\z@\b@undb@xSC{\v@lX}{\v@lY}\fi\figWBB@x}
\ctr@ld@f\def\FigWp@r@m#1#2{\Figg@tXY{#1}%
    \setbox\Gb@x=\hbox{\def\t@xt@{#2}\ifx\t@xt@\empty\Figg@tT{#1}\else#2\fi}\c@lprojSP}
\ctr@ld@f\let\Vc@rrect=\relax
\ctr@ld@f\let\C@rp@r@m=\relax
\ctr@ld@f\def\figwrite[#1]#2{{\ignorespaces\def\list@num{#1}\@ecfor\p@int:=\list@num\do{%
    \setbox\Gb@x=\hbox{\def\t@xt@{#2}\ifx\t@xt@\empty\Figg@tT{\p@int}\else#2\fi}%
    \Figwrit@{\p@int}}}\ignorespaces}
\ctr@ld@f\def\Figwrit@#1{\Figg@tXY{#1}\c@lprojSP%
    \rlap{\kern\v@lX\raise\v@lY\hbox{\unhcopy\Gb@x}}\v@leur=\v@lY%
    \advance\v@lY\ht\Gb@x\b@undb@x{\v@lX}{\v@lY}\advance\v@lX\wd\Gb@x%
    \v@lY=\v@leur\advance\v@lY-\dp\Gb@x\b@undb@x{\v@lX}{\v@lY}}
\ctr@ld@f\def\figwritec[#1]#2{{\ignorespaces\def\list@num{#1}%
    \@ecfor\p@int:=\list@num\do{\Figwrit@c{\p@int}{#2}}}\ignorespaces}
\ctr@ld@f\def\Figwrit@c#1#2{\FigWp@r@m{#1}{#2}%
    \rlap{\kern\v@lX\raise\v@lY\hbox{\rlap{\kern-.5\wd\Gb@x%
    \raise-.5\ht\Gb@x\hbox{\raise.5\dp\Gb@x\hbox{\unhcopy\Gb@x}}}}}%
    \v@leur=\ht\Gb@x\advance\v@leur\dp\Gb@x%
    \advance\v@lX-.5\wd\Gb@x\advance\v@lY-.5\v@leur\b@undb@x{\v@lX}{\v@lY}%
    \advance\v@lX\wd\Gb@x\advance\v@lY\v@leur\b@undb@x{\v@lX}{\v@lY}}
\ctr@ld@f\def\figwritep[#1]{{\ignorespaces\def\list@num{#1}\setbox\Gb@x=\hbox{\c@nterpt}%
    \@ecfor\p@int:=\list@num\do{\Figwrit@{\p@int}}}\ignorespaces}
\ctr@ld@f\def\figwritew#1:#2(#3){\figwritegcw#1:{#2}(#3,0pt)}
\ctr@ld@f\def\figwritee#1:#2(#3){\figwritegce#1:{#2}(#3,0pt)}
\ctr@ld@f\def\figwriten#1:#2(#3){{\def\Vc@rrect{\v@lZ=\v@leur\advance\v@lZ\dp\Gb@x}%
    \Figwrit@NS#1:{#2}(#3)}\ignorespaces}
\ctr@ld@f\def\figwrites#1:#2(#3){{\def\Vc@rrect{\v@lZ=-\v@leur\advance\v@lZ-\ht\Gb@x}%
    \Figwrit@NS#1:{#2}(#3)}\ignorespaces}
\ctr@ld@f\def\Figwrit@NS#1:#2(#3){\let\figWp@si=\FigWp@siNS\let\figWBB@x=\FigWBB@xNS%
    \FigWrit@L#1:{#2}(#3,0pt)}
\ctr@ld@f\def\FigWp@siNS{\rlap{\kern\v@lX\raise\v@lY\hbox{\rlap{\kern-.5\wd\Gb@x%
    \raise\v@lZ\hbox{\unhcopy\Gb@x}}\c@nterpt}}}
\ctr@ld@f\def\FigWBB@xNS{\advance\v@lY\v@lZ%
    \advance\v@lY-\dp\Gb@x\advance\v@lX-.5\wd\Gb@x\b@undb@x{\v@lX}{\v@lY}%
    \advance\v@lY\ht\Gb@x\advance\v@lY\dp\Gb@x%
    \advance\v@lX\wd\Gb@x\b@undb@x{\v@lX}{\v@lY}}
\ctr@ld@f\def\figwritenw#1:#2(#3){{\let\figWp@si=\FigWp@sigW\let\figWBB@x=\FigWBB@xgWE%
    \def\C@rp@r@m{\v@leur=\unssqrttw@\v@leur\delt@=\v@leur%
    \ifdim\delt@=\z@\delt@=\epsil@n\fi}\let@xte={-}\FigWrit@L#1:{#2}(#3,0pt)}\ignorespaces}
\ctr@ld@f\def\figwritesw#1:#2(#3){{\let\figWp@si=\FigWp@sigW\let\figWBB@x=\FigWBB@xgWE%
    \def\C@rp@r@m{\v@leur=\unssqrttw@\v@leur\delt@=-\v@leur%
    \ifdim\delt@=\z@\delt@=-\epsil@n\fi}\let@xte={-}\FigWrit@L#1:{#2}(#3,0pt)}\ignorespaces}
\ctr@ld@f\def\figwritene#1:#2(#3){{\let\figWp@si=\FigWp@sigE\let\figWBB@x=\FigWBB@xgWE%
    \def\C@rp@r@m{\v@leur=\unssqrttw@\v@leur\delt@=\v@leur%
    \ifdim\delt@=\z@\delt@=\epsil@n\fi}\let@xte={}\FigWrit@L#1:{#2}(#3,0pt)}\ignorespaces}
\ctr@ld@f\def\figwritese#1:#2(#3){{\let\figWp@si=\FigWp@sigE\let\figWBB@x=\FigWBB@xgWE%
    \def\C@rp@r@m{\v@leur=\unssqrttw@\v@leur\delt@=-\v@leur%
    \ifdim\delt@=\z@\delt@=-\epsil@n\fi}\let@xte={}\FigWrit@L#1:{#2}(#3,0pt)}\ignorespaces}
\ctr@ld@f\def\figwritegw#1:#2(#3,#4){{\let\figWp@si=\FigWp@sigW\let\figWBB@x=\FigWBB@xgWE%
    \let@xte={-}\FigWrit@L#1:{#2}(#3,#4)}\ignorespaces}
\ctr@ld@f\def\figwritege#1:#2(#3,#4){{\let\figWp@si=\FigWp@sigE\let\figWBB@x=\FigWBB@xgWE%
    \let@xte={}\FigWrit@L#1:{#2}(#3,#4)}\ignorespaces}
\ctr@ld@f\def\FigWp@sigW{\v@lXa=\z@\v@lYa=\ht\Gb@x\advance\v@lYa\dp\Gb@x%
    \ifdim\delt@>\z@\relax%
    \rlap{\kern\v@lX\raise\v@lY\hbox{\rlap{\kern-\wd\Gb@x\kern-\v@leur%
          \raise\delt@\hbox{\raise\dp\Gb@x\hbox{\unhcopy\Gb@x}}}\c@nterpt}}%
    \else\ifdim\delt@<\z@\relax\v@lYa=-\v@lYa%
    \rlap{\kern\v@lX\raise\v@lY\hbox{\rlap{\kern-\wd\Gb@x\kern-\v@leur%
          \raise\delt@\hbox{\raise-\ht\Gb@x\hbox{\unhcopy\Gb@x}}}\c@nterpt}}%
    \else\v@lXa=-.5\v@lYa%
    \rlap{\kern\v@lX\raise\v@lY\hbox{\rlap{\kern-\wd\Gb@x\kern-\v@leur%
          \raise-.5\ht\Gb@x\hbox{\raise.5\dp\Gb@x\hbox{\unhcopy\Gb@x}}}\c@nterpt}}%
    \fi\fi}
\ctr@ld@f\def\FigWp@sigE{\v@lXa=\z@\v@lYa=\ht\Gb@x\advance\v@lYa\dp\Gb@x%
    \ifdim\delt@>\z@\relax%
    \rlap{\kern\v@lX\raise\v@lY\hbox{\c@nterpt\kern\v@leur%
          \raise\delt@\hbox{\raise\dp\Gb@x\hbox{\unhcopy\Gb@x}}}}%
    \else\ifdim\delt@<\z@\relax\v@lYa=-\v@lYa%
    \rlap{\kern\v@lX\raise\v@lY\hbox{\c@nterpt\kern\v@leur%
          \raise\delt@\hbox{\raise-\ht\Gb@x\hbox{\unhcopy\Gb@x}}}}%
    \else\v@lXa=-.5\v@lYa%
    \rlap{\kern\v@lX\raise\v@lY\hbox{\c@nterpt\kern\v@leur%
          \raise-.5\ht\Gb@x\hbox{\raise.5\dp\Gb@x\hbox{\unhcopy\Gb@x}}}}%
    \fi\fi}
\ctr@ld@f\def\FigWBB@xgWE{\advance\v@lY\delt@%
    \advance\v@lX\the\let@xte\v@leur\advance\v@lY\v@lXa\b@undb@x{\v@lX}{\v@lY}%
    \advance\v@lX\the\let@xte\wd\Gb@x\advance\v@lY\v@lYa\b@undb@x{\v@lX}{\v@lY}}
\ctr@ld@f\def\figwritegcw#1:#2(#3,#4){{\let\figWp@si=\FigWp@sigcW\let\figWBB@x=\FigWBB@xgcWE%
    \let@xte={-}\FigWrit@L#1:{#2}(#3,#4)}\ignorespaces}
\ctr@ld@f\def\figwritegce#1:#2(#3,#4){{\let\figWp@si=\FigWp@sigcE\let\figWBB@x=\FigWBB@xgcWE%
    \let@xte={}\FigWrit@L#1:{#2}(#3,#4)}\ignorespaces}
\ctr@ld@f\def\FigWp@sigcW{\rlap{\kern\v@lX\raise\v@lY\hbox{\rlap{\kern-\wd\Gb@x\kern-\v@leur%
     \raise-.5\ht\Gb@x\hbox{\raise\delt@\hbox{\raise.5\dp\Gb@x\hbox{\unhcopy\Gb@x}}}}%
     \c@nterpt}}}
\ctr@ld@f\def\FigWp@sigcE{\rlap{\kern\v@lX\raise\v@lY\hbox{\c@nterpt\kern\v@leur%
    \raise-.5\ht\Gb@x\hbox{\raise\delt@\hbox{\raise.5\dp\Gb@x\hbox{\unhcopy\Gb@x}}}}}}
\ctr@ld@f\def\FigWBB@xgcWE{\v@lZ=\ht\Gb@x\advance\v@lZ\dp\Gb@x%
    \advance\v@lX\the\let@xte\v@leur\advance\v@lY\delt@\advance\v@lY.5\v@lZ%
    \b@undb@x{\v@lX}{\v@lY}%
    \advance\v@lX\the\let@xte\wd\Gb@x\advance\v@lY-\v@lZ\b@undb@x{\v@lX}{\v@lY}}
\ctr@ld@f\def\figwritebn#1:#2(#3){{\def\Vc@rrect{\v@lZ=\v@leur}\Figwrit@NS#1:{#2}(#3)}\ignorespaces}
\ctr@ld@f\def\figwritebs#1:#2(#3){{\def\Vc@rrect{\v@lZ=-\v@leur}\Figwrit@NS#1:{#2}(#3)}\ignorespaces}
\ctr@ld@f\def\figwritebw#1:#2(#3){{\let\figWp@si=\FigWp@sibW\let\figWBB@x=\FigWBB@xbWE%
    \let@xte={-}\FigWrit@L#1:{#2}(#3,0pt)}\ignorespaces}
\ctr@ld@f\def\figwritebe#1:#2(#3){{\let\figWp@si=\FigWp@sibE\let\figWBB@x=\FigWBB@xbWE%
    \let@xte={}\FigWrit@L#1:{#2}(#3,0pt)}\ignorespaces}
\ctr@ld@f\def\FigWp@sibW{\rlap{\kern\v@lX\raise\v@lY\hbox{\rlap{\kern-\wd\Gb@x\kern-\v@leur%
          \hbox{\unhcopy\Gb@x}}\c@nterpt}}}
\ctr@ld@f\def\FigWp@sibE{\rlap{\kern\v@lX\raise\v@lY\hbox{\c@nterpt\kern\v@leur%
          \hbox{\unhcopy\Gb@x}}}}
\ctr@ld@f\def\FigWBB@xbWE{\v@lZ=\ht\Gb@x\advance\v@lZ\dp\Gb@x%
    \advance\v@lX\the\let@xte\v@leur\advance\v@lY\ht\Gb@x\b@undb@x{\v@lX}{\v@lY}%
    \advance\v@lX\the\let@xte\wd\Gb@x\advance\v@lY-\v@lZ\b@undb@x{\v@lX}{\v@lY}}
\ctr@ln@w{newread}\frf@g  \ctr@ln@w{newwrite}\fwf@g
\ctr@ln@w{newif}\ifcurr@ntPS
\ctr@ln@w{newif}\ifps@cri
\ctr@ln@w{newif}\ifUse@llipse
\ctr@ln@w{newif}\ifpsdebugmode \psdebugmodefalse 
\ctr@ln@w{newif}\ifPDFm@ke
\ifx\pdfliteral\undefined\else\ifnum\pdfoutput>\z@\PDFm@ketrue\fi\fi
\ctr@ld@f\def\initPDF@rDVI{%
\ifPDFm@ke
 \let\figscan=\figscan@E
 \let\newGr@FN=\newGr@FNPDF
 \ctr@ld@f\def\c@mcurveto{c}
 \ctr@ld@f\def\c@mfill{f}
 \ctr@ld@f\def\c@mgsave{q}
 \ctr@ld@f\def\c@mgrestore{Q}
 \ctr@ld@f\def\c@mlineto{l}
 \ctr@ld@f\def\c@mmoveto{m}
 \ctr@ld@f\def\c@msetgray{g}     \ctr@ld@f\def\c@msetgrayStroke{G}
 \ctr@ld@f\def\c@msetcmykcolor{k}\ctr@ld@f\def\c@msetcmykcolorStroke{K}
 \ctr@ld@f\def\c@msetrgbcolor{rg}\ctr@ld@f\def\c@msetrgbcolorStroke{RG}
 \ctr@ld@f\def\d@fprimarC@lor{\curr@ntcolor\space\curr@ntcolorc@md%
               \space\curr@ntcolor\space\curr@ntcolorc@mdStroke}
 \ctr@ld@f\def\d@fsecondC@lor{\sec@ndcolor\space\sec@ndcolorc@md%
               \space\sec@ndcolor\space\sec@ndcolorc@mdStroke}
 \ctr@ld@f\def\d@fthirdC@lor{\th@rdcolor\space\th@rdcolorc@md%
              \space\th@rdcolor\space\th@rdcolorc@mdStroke}
 \ctr@ld@f\def\c@msetdash{d}
 \ctr@ld@f\def\c@msetlinejoin{j}
 \ctr@ld@f\def\c@msetlinewidth{w}
 \ctr@ld@f\def\f@gclosestroke{\immediate\write\fwf@g{s}}
 \ctr@ld@f\def\f@gfill{\immediate\write\fwf@g{\fillc@md}}% Voir la def de \fillc@md ****
 \ctr@ld@f\def\f@gnewpath{}
 \ctr@ld@f\def\f@gstroke{\immediate\write\fwf@g{S}}
\else
 \let\figinsertE=\figinsert
 \let\newGr@FN=\newGr@FNDVI
 \ctr@ld@f\def\c@mcurveto{curveto}
 \ctr@ld@f\def\c@mfill{fill}
 \ctr@ld@f\def\c@mgsave{gsave}
 \ctr@ld@f\def\c@mgrestore{grestore}
 \ctr@ld@f\def\c@mlineto{lineto}
 \ctr@ld@f\def\c@mmoveto{moveto}
 \ctr@ld@f\def\c@msetgray{setgray}          \ctr@ld@f\def\c@msetgrayStroke{}
 \ctr@ld@f\def\c@msetcmykcolor{setcmykcolor}\ctr@ld@f\def\c@msetcmykcolorStroke{}
 \ctr@ld@f\def\c@msetrgbcolor{setrgbcolor}  \ctr@ld@f\def\c@msetrgbcolorStroke{}
 \ctr@ld@f\def\d@fprimarC@lor{\curr@ntcolor\space\curr@ntcolorc@md}
 \ctr@ld@f\def\d@fsecondC@lor{\sec@ndcolor\space\sec@ndcolorc@md}
 \ctr@ld@f\def\d@fthirdC@lor{\th@rdcolor\space\th@rdcolorc@md}
 \ctr@ld@f\def\c@msetdash{setdash}
 \ctr@ld@f\def\c@msetlinejoin{setlinejoin}
 \ctr@ld@f\def\c@msetlinewidth{setlinewidth}
 \ctr@ld@f\def\f@gclosestroke{\immediate\write\fwf@g{closepath\space stroke}}
 \ctr@ld@f\def\f@gfill{\immediate\write\fwf@g{\fillc@md}}
 \ctr@ld@f\def\f@gnewpath{\immediate\write\fwf@g{newpath}}
 \ctr@ld@f\def\f@gstroke{\immediate\write\fwf@g{stroke}}
\fi}
\ctr@ld@f\def\c@pypsfile#1#2{\c@pyfil@{\immediate\write#1}{#2}}
\ctr@ld@f\def\Figinclud@PDF#1#2{\openin\frf@g=#1\pdfliteral{q #2 0 0 #2 0 0 cm}%
    \c@pyfil@{\pdfliteral}{\frf@g}\pdfliteral{Q}\closein\frf@g}
\ctr@ln@w{newif}\ifmored@ta
\ctr@ln@m\bl@nkline
\ctr@ld@f\def\c@pyfil@#1#2{\def\bl@nkline{\par}{\catcode`\%=12
    \loop\ifeof#2\mored@tafalse\else\mored@tatrue\immediate\read#2 to\tr@c
    \ifx\tr@c\bl@nkline\else#1{\tr@c}\fi\fi\ifmored@ta\repeat}}
\ctr@ld@f\def\keln@mun#1#2|{\def\l@debut{#1}\def\l@suite{#2}}
\ctr@ld@f\def\keln@mde#1#2#3|{\def\l@debut{#1#2}\def\l@suite{#3}}
\ctr@ld@f\def\keln@mtr#1#2#3#4|{\def\l@debut{#1#2#3}\def\l@suite{#4}}
\ctr@ld@f\def\keln@mqu#1#2#3#4#5|{\def\l@debut{#1#2#3#4}\def\l@suite{#5}}
\ctr@ld@f\let\@psffilein=\frf@g % file to \read
\ctr@ln@w{newif}\if@psffileok    % continue looking for the bounding box?
\ctr@ln@w{newif}\if@psfbbfound   % success?
\ctr@ln@w{newif}\if@psfverbose   % report what you're making?
\@psfverbosetrue
\ctr@ln@m\@psfllx \ctr@ln@m\@psflly
\ctr@ln@m\@psfurx \ctr@ln@m\@psfury
\ctr@ln@m\resetcolonc@tcode
\ctr@ld@f\def\@psfgetbb#1{\global\@psfbbfoundfalse%
\global\def\@psfllx{0}\global\def\@psflly{0}%
\global\def\@psfurx{30}\global\def\@psfury{30}%
\openin\@psffilein=#1\relax
\ifeof\@psffilein\errmessage{I couldn't open #1, will ignore it}\else
   \edef\resetcolonc@tcode{\catcode`\noexpand\:\the\catcode`\:\relax}%
   {\@psffileoktrue \chardef\other=12
    \def\do##1{\catcode`##1=\other}\dospecials \catcode`\ =10 \resetcolonc@tcode
    \loop
       \read\@psffilein to \@psffileline
       \ifeof\@psffilein\@psffileokfalse\else
          \expandafter\@psfaux\@psffileline:. \\%
       \fi
   \if@psffileok\repeat
   \if@psfbbfound\else
    \if@psfverbose\message{No bounding box comment in #1; using defaults}\fi\fi
   }\closein\@psffilein\fi}%
\ctr@ln@m\@psfbblit
\ctr@ln@m\@psfpercent
{\catcode`\%=12 \global\let\@psfpercent=%\global\def\@psfbblit{%BoundingBox}}%
\ctr@ln@m\@psfaux
\long\def\@psfaux#1#2:#3\\{\ifx#1\@psfpercent
   \def\testit{#2}\ifx\testit\@psfbblit
      \@psfgrab #3 . . . \\%
      \@psffileokfalse
      \global\@psfbbfoundtrue
   \fi\else\ifx#1\par\else\@psffileokfalse\fi\fi}%
\ctr@ld@f\def\@psfempty{}%
\ctr@ld@f\def\@psfgrab #1 #2 #3 #4 #5\\{%
\global\def\@psfllx{#1}\ifx\@psfllx\@psfempty
      \@psfgrab #2 #3 #4 #5 .\\\else
   \global\def\@psflly{#2}%
   \global\def\@psfurx{#3}\global\def\@psfury{#4}\fi}%
\ctr@ld@f\def\PSwrit@cmd#1#2#3{{\Figg@tXY{#1}\c@lprojSP\b@undb@x{\v@lX}{\v@lY}%
    \v@lX=\ptT@ptps\v@lX\v@lY=\ptT@ptps\v@lY%
    \immediate\write#3{\repdecn@mb{\v@lX}\space\repdecn@mb{\v@lY}\space#2}}}
\ctr@ld@f\def\PSwrit@cmdS#1#2#3#4#5{{\Figg@tXY{#1}\c@lprojSP\b@undb@x{\v@lX}{\v@lY}%
    \global\result@t=\v@lX\global\result@@t=\v@lY%
    \v@lX=\ptT@ptps\v@lX\v@lY=\ptT@ptps\v@lY%
    \immediate\write#3{\repdecn@mb{\v@lX}\space\repdecn@mb{\v@lY}\space#2}}%
    \edef#4{\the\result@t}\edef#5{\the\result@@t}}
\ctr@ld@f\def\psaltitude#1[#2,#3,#4]{{\ifcurr@ntPS\ifps@cri%
    \PSc@mment{psaltitude Square Dim=#1, Triangle=[#2 / #3,#4]}%
    \s@uvc@ntr@l\et@tpsaltitude\resetc@ntr@l{2}\figptorthoprojline-5:=#2/#3,#4/%
    \figvectP -1[#3,#4]\n@rminf{\v@leur}{-1}\vecunit@{-3}{-1}%
    \figvectP -1[-5,#3]\n@rminf{\v@lmin}{-1}\figvectP -2[-5,#4]\n@rminf{\v@lmax}{-2}%
    \ifdim\v@lmin<\v@lmax\s@mme=#3\else\v@lmax=\v@lmin\s@mme=#4\fi%
    \figvectP -4[-5,#2]\vecunit@{-4}{-4}\delt@=#1\unit@%
    \edef\t@ille{\repdecn@mb{\delt@}}\figpttra-1:=-5/\t@ille,-3/%
    \figptstra-3=-5,-1/\t@ille,-4/\psline[#2,-5]\s@uvdash{\typ@dash}%
    \pssetdash{\defaultdash}\psline[-1,-2,-3]\pssetdash{\typ@dash}%
    \ifdim\v@leur<\v@lmax\Pss@tsecondSt\psline[-5,\the\s@mme]\Psrest@reSt\fi%
    \PSc@mment{End psaltitude}\resetc@ntr@l\et@tpsaltitude\fi\fi}}
\ctr@ld@f\def\Ps@rcerc#1;#2(#3,#4){\ellBB@x#1;#2,#2(#3,#4,0)%
    \f@gnewpath{\delt@=#2\unit@\delt@=\ptT@ptps\delt@%
    \BdingB@xfalse%
    \PSwrit@cmd{#1}{\repdecn@mb{\delt@}\space #3\space #4\space arc}{\fwf@g}}}
\ctr@ln@m\psarccirc
\ctr@ld@f\def\psarccircDD#1;#2(#3,#4){\ifcurr@ntPS\ifps@cri%
    \PSc@mment{psarccircDD Center=#1 ; Radius=#2 (Ang1=#3, Ang2=#4)}%
    \iffillm@de\Ps@rcerc#1;#2(#3,#4)%
    \f@gfill%
    \else\Ps@rcerc#1;#2(#3,#4)\f@gstroke\fi%
    \PSc@mment{End psarccircDD}\fi\fi}
\ctr@ld@f\def\psarccircTD#1,#2,#3;#4(#5,#6){{\ifcurr@ntPS\ifps@cri\s@uvc@ntr@l\et@tpsarccircTD%
    \PSc@mment{psarccircTD Center=#1,P1=#2,P2=#3 ; Radius=#4 (Ang1=#5, Ang2=#6)}%
    \setc@ntr@l{2}\c@lExtAxes#1,#2,#3(#4)\psarcellPATD#1,-4,-5(#5,#6)%
    \PSc@mment{End psarccircTD}\resetc@ntr@l\et@tpsarccircTD\fi\fi}}
\ctr@ld@f\def\c@lExtAxes#1,#2,#3(#4){%
    \figvectPTD-5[#1,#2]\vecunit@{-5}{-5}\figvectNTD-4[#1,#2,#3]\vecunit@{-4}{-4}%
    \figvectNVTD-3[-4,-5]\delt@=#4\unit@\edef\r@yon{\repdecn@mb{\delt@}}%
    \figpttra-4:=#1/\r@yon,-5/\figpttra-5:=#1/\r@yon,-3/}
\ctr@ln@m\psarccircP
\ctr@ld@f\def\psarccircPDD#1;#2[#3,#4]{{\ifcurr@ntPS\ifps@cri\s@uvc@ntr@l\et@tpsarccircPDD%
    \PSc@mment{psarccircPDD Center=#1; Radius=#2, [P1=#3, P2=#4]}%
    \Ps@ngleparam#1;#2[#3,#4]\ifdim\v@lmin>\v@lmax\advance\v@lmax\DePI@deg\fi%
    \edef\@ngdeb{\repdecn@mb{\v@lmin}}\edef\@ngfin{\repdecn@mb{\v@lmax}}%
    \psarccirc#1;\r@dius(\@ngdeb,\@ngfin)%
    \PSc@mment{End psarccircPDD}\resetc@ntr@l\et@tpsarccircPDD\fi\fi}}
\ctr@ld@f\def\psarccircPTD#1;#2[#3,#4,#5]{{\ifcurr@ntPS\ifps@cri\s@uvc@ntr@l\et@tpsarccircPTD%
    \PSc@mment{psarccircPTD Center=#1; Radius=#2, [P1=#3, P2=#4, P3=#5]}%
    \setc@ntr@l{2}\c@lExtAxes#1,#3,#5(#2)\psarcellPP#1,-4,-5[#3,#4]%
    \PSc@mment{End psarccircPTD}\resetc@ntr@l\et@tpsarccircPTD\fi\fi}}
\ctr@ld@f\def\Ps@ngleparam#1;#2[#3,#4]{\setc@ntr@l{2}%
    \figvectPDD-1[#1,#3]\vecunit@{-1}{-1}\Figg@tXY{-1}\arct@n\v@lmin(\v@lX,\v@lY)%
    \figvectPDD-2[#1,#4]\vecunit@{-2}{-2}\Figg@tXY{-2}\arct@n\v@lmax(\v@lX,\v@lY)%
    \v@lmin=\rdT@deg\v@lmin\v@lmax=\rdT@deg\v@lmax%
    \v@leur=#2pt\maxim@m{\mili@u}{-\v@leur}{\v@leur}%
    \edef\r@dius{\repdecn@mb{\mili@u}}}
\ctr@ld@f\def\Ps@rcercBz#1;#2(#3,#4){\Ps@rellBz#1;#2,#2(#3,#4,0)}
\ctr@ld@f\def\Ps@rellBz#1;#2,#3(#4,#5,#6){%
    \ellBB@x#1;#2,#3(#4,#5,#6)\BdingB@xfalse%
    \c@lNbarcs{#4}{#5}\v@leur=#4pt\setc@ntr@l{2}\figptell-13::#1;#2,#3(#4,#6)%
    \f@gnewpath\PSwrit@cmd{-13}{\c@mmoveto}{\fwf@g}%
    \s@mme=\z@\bcl@rellBz#1;#2,#3(#6)\BdingB@xtrue}
\ctr@ld@f\def\bcl@rellBz#1;#2,#3(#4){\relax%
    \ifnum\s@mme<\p@rtent\advance\s@mme\@ne%
    \advance\v@leur\delt@\edef\@ngle{\repdecn@mb\v@leur}\figptell-14::#1;#2,#3(\@ngle,#4)%
    \advance\v@leur\delt@\edef\@ngle{\repdecn@mb\v@leur}\figptell-15::#1;#2,#3(\@ngle,#4)%
    \advance\v@leur\delt@\edef\@ngle{\repdecn@mb\v@leur}\figptell-16::#1;#2,#3(\@ngle,#4)%
    \figptscontrolDD-18[-13,-14,-15,-16]%
    \PSwrit@cmd{-18}{}{\fwf@g}\PSwrit@cmd{-17}{}{\fwf@g}%
    \PSwrit@cmd{-16}{\c@mcurveto}{\fwf@g}%
    \figptcopyDD-13:/-16/\bcl@rellBz#1;#2,#3(#4)\fi}
\ctr@ld@f\def\Ps@rell#1;#2,#3(#4,#5,#6){\ellBB@x#1;#2,#3(#4,#5,#6)%
    \f@gnewpath{\v@lmin=#2\unit@\v@lmin=\ptT@ptps\v@lmin%
    \v@lmax=#3\unit@\v@lmax=\ptT@ptps\v@lmax\BdingB@xfalse%
    \PSwrit@cmd{#1}%
    {#6\space\repdecn@mb{\v@lmin}\space\repdecn@mb{\v@lmax}\space #4\space #5\space ellipse}{\fwf@g}}%
    \global\Use@llipsetrue}
\ctr@ln@m\psarcell
\ctr@ld@f\def\psarcellDD#1;#2,#3(#4,#5,#6){{\ifcurr@ntPS\ifps@cri%
    \PSc@mment{psarcellDD Center=#1 ; XRad=#2, YRad=#3 (Ang1=#4, Ang2=#5, Inclination=#6)}%
    \iffillm@de\Ps@rell#1;#2,#3(#4,#5,#6)%
    \f@gfill%
    \else\Ps@rell#1;#2,#3(#4,#5,#6)\f@gstroke\fi%
    \PSc@mment{End psarcellDD}\fi\fi}}
\ctr@ld@f\def\psarcellTD#1;#2,#3(#4,#5,#6){{\ifcurr@ntPS\ifps@cri\s@uvc@ntr@l\et@tpsarcellTD%
    \PSc@mment{psarcellTD Center=#1 ; XRad=#2, YRad=#3 (Ang1=#4, Ang2=#5, Inclination=#6)}%
    \setc@ntr@l{2}\figpttraC -8:=#1/#2,0,0/\figpttraC -7:=#1/0,#3,0/%
    \figvectC -4(0,0,1)\figptsrot -8=-8,-7/#1,#6,-4/\psarcellPATD#1,-8,-7(#4,#5)%
    \PSc@mment{End psarcellTD}\resetc@ntr@l\et@tpsarcellTD\fi\fi}}
\ctr@ln@m\psarcellPA
\ctr@ld@f\def\psarcellPADD#1,#2,#3(#4,#5){{\ifcurr@ntPS\ifps@cri\s@uvc@ntr@l\et@tpsarcellPADD%
    \PSc@mment{psarcellPADD Center=#1,PtAxis1=#2,PtAxis2=#3 (Ang1=#4, Ang2=#5)}%
    \setc@ntr@l{2}\figvectPDD-1[#1,#2]\vecunit@DD{-1}{-1}\v@lX=\ptT@unit@\result@t%
    \edef\XR@d{\repdecn@mb{\v@lX}}\Figg@tXY{-1}\arct@n\v@lmin(\v@lX,\v@lY)%
    \v@lmin=\rdT@deg\v@lmin\edef\Inclin@{\repdecn@mb{\v@lmin}}%
    \figgetdist\YR@d[#1,#3]\psarcellDD#1;\XR@d,\YR@d(#4,#5,\Inclin@)%
    \PSc@mment{End psarcellPADD}\resetc@ntr@l\et@tpsarcellPADD\fi\fi}}
\ctr@ld@f\def\psarcellPATD#1,#2,#3(#4,#5){{\ifcurr@ntPS\ifps@cri\s@uvc@ntr@l\et@tpsarcellPATD%
    \PSc@mment{psarcellPATD Center=#1,PtAxis1=#2,PtAxis2=#3 (Ang1=#4, Ang2=#5)}%
    \iffillm@de\Ps@rellPATD#1,#2,#3(#4,#5)%
    \f@gfill%
    \else\Ps@rellPATD#1,#2,#3(#4,#5)\f@gstroke\fi%
    \PSc@mment{End psarcellPATD}\resetc@ntr@l\et@tpsarcellPATD\fi\fi}}
\ctr@ld@f\def\Ps@rellPATD#1,#2,#3(#4,#5){\let\c@lprojSP=\relax%
    \setc@ntr@l{2}\figvectPTD-1[#1,#2]\figvectPTD-2[#1,#3]\c@lNbarcs{#4}{#5}%
    \v@leur=#4pt\c@lptellP{#1}{-1}{-2}\Figptpr@j-5:/-3/%
    \f@gnewpath\PSwrit@cmdS{-5}{\c@mmoveto}{\fwf@g}{\X@un}{\Y@un}%
    \edef\C@nt@r{#1}\s@mme=\z@\bcl@rellPATD}
\ctr@ld@f\def\bcl@rellPATD{\relax%
    \ifnum\s@mme<\p@rtent\advance\s@mme\@ne%
    \advance\v@leur\delt@\c@lptellP{\C@nt@r}{-1}{-2}\Figptpr@j-4:/-3/%
    \advance\v@leur\delt@\c@lptellP{\C@nt@r}{-1}{-2}\Figptpr@j-6:/-3/%
    \advance\v@leur\delt@\c@lptellP{\C@nt@r}{-1}{-2}\Figptpr@j-3:/-3/%
    \v@lX=\z@\v@lY=\z@\Figtr@nptDD{-5}{-5}\Figtr@nptDD{2}{-3}%
    \divide\v@lX\@vi\divide\v@lY\@vi%
    \Figtr@nptDD{3}{-4}\Figtr@nptDD{-1.5}{-6}\v@lmin=\v@lX\v@lmax=\v@lY%
    \v@lX=\z@\v@lY=\z@\Figtr@nptDD{2}{-5}\Figtr@nptDD{-5}{-3}%
    \divide\v@lX\@vi\divide\v@lY\@vi\Figtr@nptDD{-1.5}{-4}\Figtr@nptDD{3}{-6}%
    \BdingB@xfalse%
    \Figp@intregDD-4:(\v@lmin,\v@lmax)\PSwrit@cmdS{-4}{}{\fwf@g}{\X@de}{\Y@de}%
    \Figp@intregDD-4:(\v@lX,\v@lY)\PSwrit@cmdS{-4}{}{\fwf@g}{\X@tr}{\Y@tr}%
    \BdingB@xtrue\PSwrit@cmdS{-3}{\c@mcurveto}{\fwf@g}{\X@qu}{\Y@qu}%
    \B@zierBB@x{1}{\Y@un}(\X@un,\X@de,\X@tr,\X@qu)%
    \B@zierBB@x{2}{\X@un}(\Y@un,\Y@de,\Y@tr,\Y@qu)%
    \edef\X@un{\X@qu}\edef\Y@un{\Y@qu}\figptcopyDD-5:/-3/\bcl@rellPATD\fi}
\ctr@ld@f\def\c@lNbarcs#1#2{%
    \delt@=#2pt\advance\delt@-#1pt\maxim@m{\v@lmax}{\delt@}{-\delt@}%
    \v@leur=\v@lmax\divide\v@leur45 \p@rtentiere{\p@rtent}{\v@leur}\advance\p@rtent\@ne%
    \s@mme=\p@rtent\multiply\s@mme\thr@@\divide\delt@\s@mme}
\ctr@ld@f\def\psarcellPP#1,#2,#3[#4,#5]{{\ifcurr@ntPS\ifps@cri\s@uvc@ntr@l\et@tpsarcellPP%
    \PSc@mment{psarcellPP Center=#1,PtAxis1=#2,PtAxis2=#3 [Point1=#4, Point2=#5]}%
    \setc@ntr@l{2}\figvectP-2[#1,#3]\vecunit@{-2}{-2}\v@lmin=\result@t%
    \invers@{\v@lmax}{\v@lmin}%
    \figvectP-1[#1,#2]\vecunit@{-1}{-1}\v@leur=\result@t%
    \v@leur=\repdecn@mb{\v@lmax}\v@leur\edef\AsB@{\repdecn@mb{\v@leur}}% a/b
    \c@lAngle{#1}{#4}{\v@lmin}\edef\@ngdeb{\repdecn@mb{\v@lmin}}%
    \c@lAngle{#1}{#5}{\v@lmax}\ifdim\v@lmin>\v@lmax\advance\v@lmax\DePI@deg\fi%
    \edef\@ngfin{\repdecn@mb{\v@lmax}}\psarcellPA#1,#2,#3(\@ngdeb,\@ngfin)%
    \PSc@mment{End psarcellPP}\resetc@ntr@l\et@tpsarcellPP\fi\fi}}
\ctr@ld@f\def\c@lAngle#1#2#3{\figvectP-3[#1,#2]%
    \c@lproscal\delt@[-3,-1]\c@lproscal\v@leur[-3,-2]%
    \v@leur=\AsB@\v@leur\arct@n#3(\delt@,\v@leur)#3=\rdT@deg#3}
\ctr@ln@w{newif}\if@rrowratio\@rrowratiotrue
\ctr@ln@w{newif}\if@rrowhfill
\ctr@ln@w{newif}\if@rrowhout
\ctr@ld@f\def\Psset@rrowhe@d#1=#2|{\keln@mun#1|%
    \def\n@mref{a}\ifx\l@debut\n@mref\pssetarrowheadangle{#2}\else% angle
    \def\n@mref{f}\ifx\l@debut\n@mref\pssetarrowheadfill{#2}\else% fillmode
    \def\n@mref{l}\ifx\l@debut\n@mref\pssetarrowheadlength{#2}\else% length
    \def\n@mref{o}\ifx\l@debut\n@mref\pssetarrowheadout{#2}\else% out
    \def\n@mref{r}\ifx\l@debut\n@mref\pssetarrowheadratio{#2}\else% ratio
    \immediate\write16{*** Unknown attribute: \BS@ psset arrowhead(..., #1=...)}%
    \fi\fi\fi\fi\fi}
\ctr@ln@m\@rrowheadangle
\ctr@ln@m\C@AHANG \ctr@ln@m\S@AHANG \ctr@ln@m\UNSS@N
\ctr@ld@f\def\pssetarrowheadangle#1{\edef\@rrowheadangle{#1}{\c@ssin{\C@}{\S@}{#1}%
    \xdef\C@AHANG{\C@}\xdef\S@AHANG{\S@}\v@lmax=\S@ pt%
    \invers@{\v@leur}{\v@lmax}\maxim@m{\v@leur}{\v@leur}{-\v@leur}%
    \xdef\UNSS@N{\the\v@leur}}}
\ctr@ld@f\def\pssetarrowheadfill#1{\expandafter\set@rrowhfill#1:}
\ctr@ld@f\def\set@rrowhfill#1#2:{\if#1n\@rrowhfillfalse\else\@rrowhfilltrue\fi}
\ctr@ld@f\def\pssetarrowheadout#1{\expandafter\set@rrowhout#1:}
\ctr@ld@f\def\set@rrowhout#1#2:{\if#1n\@rrowhoutfalse\else\@rrowhouttrue\fi}
\ctr@ln@m\@rrowheadlength
\ctr@ld@f\def\pssetarrowheadlength#1{\edef\@rrowheadlength{#1}\@rrowratiofalse}
\ctr@ln@m\@rrowheadratio
\ctr@ld@f\def\pssetarrowheadratio#1{\edef\@rrowheadratio{#1}\@rrowratiotrue}
\ctr@ln@m\defaultarrowheadlength
\ctr@ld@f\def\psresetarrowhead{%
    \pssetarrowheadangle{\defaultarrowheadangle}%
    \pssetarrowheadfill{\defaultarrowheadfill}%
    \pssetarrowheadout{\defaultarrowheadout}%
    \pssetarrowheadratio{\defaultarrowheadratio}%
    \d@fm@cdim\defaultarrowheadlength{\defaulth@rdahlength}% Valeur par defaut...
    \pssetarrowheadlength{\defaultarrowheadlength}}
\ctr@ld@f\def\defaultarrowheadratio{0.1}
\ctr@ld@f\def\defaultarrowheadangle{20}
\ctr@ld@f\def\defaultarrowheadfill{no}
\ctr@ld@f\def\defaultarrowheadout{no}
\ctr@ld@f\def\defaulth@rdahlength{8pt}
\ctr@ln@m\psarrow
\ctr@ld@f\def\psarrowDD[#1,#2]{{\ifcurr@ntPS\ifps@cri\s@uvc@ntr@l\et@tpsarrow%
    \PSc@mment{psarrowDD [Pt1,Pt2]=[#1,#2]}\pssetfillmode{no}%
    \psarrowheadDD[#1,#2]\setc@ntr@l{2}\psline[#1,-3]%
    \PSc@mment{End psarrowDD}\resetc@ntr@l\et@tpsarrow\fi\fi}}
\ctr@ld@f\def\psarrowTD[#1,#2]{{\ifcurr@ntPS\ifps@cri\s@uvc@ntr@l\et@tpsarrowTD%
    \PSc@mment{psarrowTD [Pt1,Pt2]=[#1,#2]}\resetc@ntr@l{2}%
    \Figptpr@j-5:/#1/\Figptpr@j-6:/#2/\let\c@lprojSP=\relax\psarrowDD[-5,-6]%
    \PSc@mment{End psarrowTD}\resetc@ntr@l\et@tpsarrowTD\fi\fi}}
\ctr@ln@m\psarrowhead
\ctr@ld@f\def\psarrowheadDD[#1,#2]{{\ifcurr@ntPS\ifps@cri\s@uvc@ntr@l\et@tpsarrowheadDD%
    \if@rrowhfill\def\@hangle{-\@rrowheadangle}\else\def\@hangle{\@rrowheadangle}\fi%
    \if@rrowratio%
    \if@rrowhout\def\@hratio{-\@rrowheadratio}\else\def\@hratio{\@rrowheadratio}\fi%
    \PSc@mment{psarrowheadDD Ratio=\@hratio, Angle=\@hangle, [Pt1,Pt2]=[#1,#2]}%
    \Ps@rrowhead\@hratio,\@hangle[#1,#2]%
    \else%
    \if@rrowhout\def\@hlength{-\@rrowheadlength}\else\def\@hlength{\@rrowheadlength}\fi%
    \PSc@mment{psarrowheadDD Length=\@hlength, Angle=\@hangle, [Pt1,Pt2]=[#1,#2]}%
    \Ps@rrowheadfd\@hlength,\@hangle[#1,#2]%
    \fi%
    \PSc@mment{End psarrowheadDD}\resetc@ntr@l\et@tpsarrowheadDD\fi\fi}}
\ctr@ld@f\def\psarrowheadTD[#1,#2]{{\ifcurr@ntPS\ifps@cri\s@uvc@ntr@l\et@tpsarrowheadTD%
    \PSc@mment{psarrowheadTD [Pt1,Pt2]=[#1,#2]}\resetc@ntr@l{2}%
    \Figptpr@j-5:/#1/\Figptpr@j-6:/#2/\let\c@lprojSP=\relax\psarrowheadDD[-5,-6]%
    \PSc@mment{End psarrowheadTD}\resetc@ntr@l\et@tpsarrowheadTD\fi\fi}}
\ctr@ld@f\def\Ps@rrowhead#1,#2[#3,#4]{\v@leur=#1\p@\maxim@m{\v@leur}{\v@leur}{-\v@leur}%
    \ifdim\v@leur>\Cepsil@n{% Arrow is not degenerated
    \PSc@mment{ps@rrowhead Ratio=#1, Angle=#2, [Pt1,Pt2]=[#3,#4]}\v@leur=\UNSS@N%
    \v@leur=\curr@ntwidth\v@leur\v@leur=\ptpsT@pt\v@leur\delt@=.5\v@leur% = width / (2 sin(Angle))
    \setc@ntr@l{2}\figvectPDD-3[#4,#3]%
    \Figg@tXY{-3}\v@lX=#1\v@lX\v@lY=#1\v@lY\Figv@ctCreg-3(\v@lX,\v@lY)%
    \vecunit@{-4}{-3}\mili@u=\result@t%
    \ifdim#2pt>\z@\v@lXa=-\C@AHANG\delt@%
     \edef\c@ef{\repdecn@mb{\v@lXa}}\figpttraDD-3:=-3/\c@ef,-4/\fi%
    \edef\c@ef{\repdecn@mb{\delt@}}%
    \v@lXa=\mili@u\v@lXa=\C@AHANG\v@lXa%
    \v@lYa=\ptpsT@pt\p@\v@lYa=\curr@ntwidth\v@lYa\v@lYa=\sDcc@ngle\v@lYa%
    \advance\v@lXa-\v@lYa\gdef\sDcc@ngle{0}%
    \ifdim\v@lXa>\v@leur\edef\c@efendpt{\repdecn@mb{\v@leur}}%
    \else\edef\c@efendpt{\repdecn@mb{\v@lXa}}\fi%
    \Figg@tXY{-3}\v@lmin=\v@lX\v@lmax=\v@lY%
    \v@lXa=\C@AHANG\v@lmin\v@lYa=\S@AHANG\v@lmax\advance\v@lXa\v@lYa%
    \v@lYa=-\S@AHANG\v@lmin\v@lX=\C@AHANG\v@lmax\advance\v@lYa\v@lX%
    \setc@ntr@l{1}\Figg@tXY{#4}\advance\v@lX\v@lXa\advance\v@lY\v@lYa%
    \setc@ntr@l{2}\Figp@intregDD-2:(\v@lX,\v@lY)%
    \v@lXa=\C@AHANG\v@lmin\v@lYa=-\S@AHANG\v@lmax\advance\v@lXa\v@lYa%
    \v@lYa=\S@AHANG\v@lmin\v@lX=\C@AHANG\v@lmax\advance\v@lYa\v@lX%
    \setc@ntr@l{1}\Figg@tXY{#4}\advance\v@lX\v@lXa\advance\v@lY\v@lYa%
    \setc@ntr@l{2}\Figp@intregDD-1:(\v@lX,\v@lY)%
    \ifdim#2pt<\z@\fillm@detrue\psline[-2,#4,-1]% fill
    \else\figptstraDD-3=#4,-2,-1/\c@ef,-4/\psline[-2,-3,-1]\fi% no fill
    \ifdim#1pt>\z@\figpttraDD-3:=#4/\c@efendpt,-4/\else\figptcopyDD-3:/#4/\fi%
    \PSc@mment{End ps@rrowhead}}\fi}
\ctr@ld@f\def\sDcc@ngle{0}% Initialisation
\ctr@ld@f\def\Ps@rrowheadfd#1,#2[#3,#4]{{%
    \PSc@mment{ps@rrowheadfd Length=#1, Angle=#2, [Pt1,Pt2]=[#3,#4]}%
    \setc@ntr@l{2}\figvectPDD-1[#3,#4]\n@rmeucDD{\v@leur}{-1}\v@leur=\ptT@unit@\v@leur%
    \invers@{\v@leur}{\v@leur}\v@leur=#1\v@leur\edef\R@tio{\repdecn@mb{\v@leur}}%
    \Ps@rrowhead\R@tio,#2[#3,#4]\PSc@mment{End ps@rrowheadfd}}}
\ctr@ln@m\psarrowBezier
\ctr@ld@f\def\psarrowBezierDD[#1,#2,#3,#4]{{\ifcurr@ntPS\ifps@cri\s@uvc@ntr@l\et@tpsarrowBezierDD%
    \PSc@mment{psarrowBezierDD Control points=#1,#2,#3,#4}\setc@ntr@l{2}%
    \if@rrowratio\c@larclengthDD\v@leur,10[#1,#2,#3,#4]\else\v@leur=\z@\fi%
    \Ps@rrowB@zDD\v@leur[#1,#2,#3,#4]%
    \PSc@mment{End psarrowBezierDD}\resetc@ntr@l\et@tpsarrowBezierDD\fi\fi}}
\ctr@ld@f\def\psarrowBezierTD[#1,#2,#3,#4]{{\ifcurr@ntPS\ifps@cri\s@uvc@ntr@l\et@tpsarrowBezierTD%
    \PSc@mment{psarrowBezierTD Control points=#1,#2,#3,#4}\resetc@ntr@l{2}%
    \Figptpr@j-7:/#1/\Figptpr@j-8:/#2/\Figptpr@j-9:/#3/\Figptpr@j-10:/#4/%
    \let\c@lprojSP=\relax\ifnum\curr@ntproj<\tw@\psarrowBezierDD[-7,-8,-9,-10]%
    \else\f@gnewpath\PSwrit@cmd{-7}{\c@mmoveto}{\fwf@g}%
    \if@rrowratio\c@larclengthDD\mili@u,10[-7,-8,-9,-10]\else\mili@u=\z@\fi%
    \p@rtent=\NBz@rcs\advance\p@rtent\m@ne\subB@zierTD\p@rtent[#1,#2,#3,#4]%
    \f@gstroke%
    \advance\v@lmin\p@rtent\delt@% Initialized in \subB@zierTD
    \v@leur=\v@lmin\advance\v@leur0.33333 \delt@\edef\unti@rs{\repdecn@mb{\v@leur}}%
    \v@leur=\v@lmin\advance\v@leur0.66666 \delt@\edef\deti@rs{\repdecn@mb{\v@leur}}%
    \figptcopyDD-8:/-10/\c@lsubBzarc\unti@rs,\deti@rs[#1,#2,#3,#4]%
    \figptcopyDD-8:/-4/\figptcopyDD-9:/-3/\Ps@rrowB@zDD\mili@u[-7,-8,-9,-10]\fi%
    \PSc@mment{End psarrowBezierTD}\resetc@ntr@l\et@tpsarrowBezierTD\fi\fi}}
\ctr@ld@f\def\c@larclengthDD#1,#2[#3,#4,#5,#6]{{\p@rtent=#2\figptcopyDD-5:/#3/%
    \delt@=\p@\divide\delt@\p@rtent\c@rre=\z@\v@leur=\z@\s@mme=\z@%
    \loop\ifnum\s@mme<\p@rtent\advance\s@mme\@ne\advance\v@leur\delt@%
    \edef\T@{\repdecn@mb{\v@leur}}\figptBezierDD-6::\T@[#3,#4,#5,#6]%
    \figvectPDD-1[-5,-6]\n@rmeucDD{\mili@u}{-1}\advance\c@rre\mili@u%
    \figptcopyDD-5:/-6/\repeat\global\result@t=\ptT@unit@\c@rre}#1=\result@t}
\ctr@ld@f\def\Ps@rrowB@zDD#1[#2,#3,#4,#5]{{\pssetfillmode{no}%
    \if@rrowratio\delt@=\@rrowheadratio#1\else\delt@=\@rrowheadlength pt\fi%
    \v@leur=\C@AHANG\delt@\edef\R@dius{\repdecn@mb{\v@leur}}%
    \FigptintercircB@zDD-5::0,\R@dius[#5,#4,#3,#2]%
    \pssetarrowheadlength{\repdecn@mb{\delt@}}\psarrowheadDD[-5,#5]%
    \let\n@rmeuc=\n@rmeucDD\figgetdist\R@dius[#5,-3]%
    \FigptintercircB@zDD-6::0,\R@dius[#5,#4,#3,#2]%
    \figptBezierDD-5::0.33333[#5,#4,#3,#2]\figptBezierDD-3::0.66666[#5,#4,#3,#2]%
    \figptscontrolDD-5[-6,-5,-3,#2]\psBezierDD1[-6,-5,-4,#2]}}
\ctr@ln@m\psarrowcirc
\ctr@ld@f\def\psarrowcircDD#1;#2(#3,#4){{\ifcurr@ntPS\ifps@cri\s@uvc@ntr@l\et@tpsarrowcircDD%
    \PSc@mment{psarrowcircDD Center=#1 ; Radius=#2 (Ang1=#3,Ang2=#4)}%
    \pssetfillmode{no}\Pscirc@rrowhead#1;#2(#3,#4)%
    \setc@ntr@l{2}\figvectPDD -4[#1,-3]\vecunit@{-4}{-4}%
    \Figg@tXY{-4}\arct@n\v@lmin(\v@lX,\v@lY)%
    \v@lmin=\rdT@deg\v@lmin\v@leur=#4pt\advance\v@leur-\v@lmin%
    \maxim@m{\v@leur}{\v@leur}{-\v@leur}%
    \ifdim\v@leur>\DemiPI@deg\relax\ifdim\v@lmin<#4pt\advance\v@lmin\DePI@deg%
    \else\advance\v@lmin-\DePI@deg\fi\fi\edef\ar@ngle{\repdecn@mb{\v@lmin}}%
    \ifdim#3pt<#4pt\psarccirc#1;#2(#3,\ar@ngle)\else\psarccirc#1;#2(\ar@ngle,#3)\fi%
    \PSc@mment{End psarrowcircDD}\resetc@ntr@l\et@tpsarrowcircDD\fi\fi}}
\ctr@ld@f\def\psarrowcircTD#1,#2,#3;#4(#5,#6){{\ifcurr@ntPS\ifps@cri\s@uvc@ntr@l\et@tpsarrowcircTD%
    \PSc@mment{psarrowcircTD Center=#1,P1=#2,P2=#3 ; Radius=#4 (Ang1=#5, Ang2=#6)}%
    \resetc@ntr@l{2}\c@lExtAxes#1,#2,#3(#4)\let\c@lprojSP=\relax%
    \figvectPTD-11[#1,-4]\figvectPTD-12[#1,-5]\c@lNbarcs{#5}{#6}%
    \if@rrowratio\v@lmax=\degT@rd\v@lmax\edef\D@lpha{\repdecn@mb{\v@lmax}}\fi%
    \advance\p@rtent\m@ne\mili@u=\z@%
    \v@leur=#5pt\c@lptellP{#1}{-11}{-12}\Figptpr@j-9:/-3/%
    \f@gnewpath\PSwrit@cmdS{-9}{\c@mmoveto}{\fwf@g}{\X@un}{\Y@un}%
    \edef\C@nt@r{#1}\s@mme=\z@\bcl@rcircTD\f@gstroke%
    \advance\v@leur\delt@\c@lptellP{#1}{-11}{-12}\Figptpr@j-5:/-3/%
    \advance\v@leur\delt@\c@lptellP{#1}{-11}{-12}\Figptpr@j-6:/-3/%
    \advance\v@leur\delt@\c@lptellP{#1}{-11}{-12}\Figptpr@j-10:/-3/%
    \figptscontrolDD-8[-9,-5,-6,-10]%
    \if@rrowratio\c@lcurvradDD0.5[-9,-8,-7,-10]\advance\mili@u\result@t%
    \maxim@m{\mili@u}{\mili@u}{-\mili@u}\mili@u=\ptT@unit@\mili@u%
    \mili@u=\D@lpha\mili@u\advance\p@rtent\@ne\divide\mili@u\p@rtent\fi%
    \Ps@rrowB@zDD\mili@u[-9,-8,-7,-10]%
    \PSc@mment{End psarrowcircTD}\resetc@ntr@l\et@tpsarrowcircTD\fi\fi}}
\ctr@ld@f\def\bcl@rcircTD{\relax%
    \ifnum\s@mme<\p@rtent\advance\s@mme\@ne%
    \advance\v@leur\delt@\c@lptellP{\C@nt@r}{-11}{-12}\Figptpr@j-5:/-3/%
    \advance\v@leur\delt@\c@lptellP{\C@nt@r}{-11}{-12}\Figptpr@j-6:/-3/%
    \advance\v@leur\delt@\c@lptellP{\C@nt@r}{-11}{-12}\Figptpr@j-10:/-3/%
    \figptscontrolDD-8[-9,-5,-6,-10]\BdingB@xfalse%
    \PSwrit@cmdS{-8}{}{\fwf@g}{\X@de}{\Y@de}\PSwrit@cmdS{-7}{}{\fwf@g}{\X@tr}{\Y@tr}%
    \BdingB@xtrue\PSwrit@cmdS{-10}{\c@mcurveto}{\fwf@g}{\X@qu}{\Y@qu}%
    \if@rrowratio\c@lcurvradDD0.5[-9,-8,-7,-10]\advance\mili@u\result@t\fi%
    \B@zierBB@x{1}{\Y@un}(\X@un,\X@de,\X@tr,\X@qu)%
    \B@zierBB@x{2}{\X@un}(\Y@un,\Y@de,\Y@tr,\Y@qu)%
    \edef\X@un{\X@qu}\edef\Y@un{\Y@qu}\figptcopyDD-9:/-10/\bcl@rcircTD\fi}
\ctr@ld@f\def\Pscirc@rrowhead#1;#2(#3,#4){{%
    \PSc@mment{pscirc@rrowhead Center=#1 ; Radius=#2 (Ang1=#3,Ang2=#4)}%
    \v@leur=#2\unit@\edef\s@glen{\repdecn@mb{\v@leur}}\v@lY=\z@\v@lX=\v@leur%
    \resetc@ntr@l{2}\Figv@ctCreg-3(\v@lX,\v@lY)\figpttraDD-5:=#1/1,-3/%
    \figptrotDD-5:=-5/#1,#4/%
    \figvectPDD-3[#1,-5]\Figg@tXY{-3}\v@leur=\v@lX%
    \ifdim#3pt<#4pt\v@lX=\v@lY\v@lY=-\v@leur\else\v@lX=-\v@lY\v@lY=\v@leur\fi%
    \Figv@ctCreg-3(\v@lX,\v@lY)\vecunit@{-3}{-3}%
    \if@rrowratio\v@leur=#4pt\advance\v@leur-#3pt\maxim@m{\mili@u}{-\v@leur}{\v@leur}%
    \mili@u=\degT@rd\mili@u\v@leur=\s@glen\mili@u\edef\s@glen{\repdecn@mb{\v@leur}}%
    \mili@u=#2\mili@u\mili@u=\@rrowheadratio\mili@u\else\mili@u=\@rrowheadlength pt\fi%
    \figpttraDD-6:=-5/\s@glen,-3/\v@leur=#2pt\v@leur=2\v@leur%
    \invers@{\v@leur}{\v@leur}\c@rre=\repdecn@mb{\v@leur}\mili@u% = sin = L/(2R)
    \mili@u=\c@rre\mili@u=\repdecn@mb{\c@rre}\mili@u%
    \v@leur=\p@\advance\v@leur-\mili@u% \v@leur = cos*cos
    \invers@{\mili@u}{2\v@leur}\delt@=\c@rre\delt@=\repdecn@mb{\mili@u}\delt@%
    \xdef\sDcc@ngle{\repdecn@mb{\delt@}}% sin/(2*cos*cos) used in \Ps@rrowhead
    \sqrt@{\mili@u}{\v@leur}\arct@n\v@leur(\mili@u,\c@rre)%
    \v@leur=\rdT@deg\v@leur% \cor@ngle = atan(L/sqrt(4R*R-L*L))
    \ifdim#3pt<#4pt\v@leur=-\v@leur\fi%
    \if@rrowhout\v@leur=-\v@leur\fi\edef\cor@ngle{\repdecn@mb{\v@leur}}%
    \figptrotDD-6:=-6/-5,\cor@ngle/\psarrowheadDD[-6,-5]%
    \PSc@mment{End pscirc@rrowhead}}}
\ctr@ln@m\psarrowcircP
\ctr@ld@f\def\psarrowcircPDD#1;#2[#3,#4]{{\ifcurr@ntPS\ifps@cri%
    \PSc@mment{psarrowcircPDD Center=#1; Radius=#2, [P1=#3,P2=#4]}%
    \s@uvc@ntr@l\et@tpsarrowcircPDD\Ps@ngleparam#1;#2[#3,#4]%
    \ifdim\v@leur>\z@\ifdim\v@lmin>\v@lmax\advance\v@lmax\DePI@deg\fi%
    \else\ifdim\v@lmin<\v@lmax\advance\v@lmin\DePI@deg\fi\fi%
    \edef\@ngdeb{\repdecn@mb{\v@lmin}}\edef\@ngfin{\repdecn@mb{\v@lmax}}%
    \psarrowcirc#1;\r@dius(\@ngdeb,\@ngfin)%
    \PSc@mment{End psarrowcircPDD}\resetc@ntr@l\et@tpsarrowcircPDD\fi\fi}}
\ctr@ld@f\def\psarrowcircPTD#1;#2[#3,#4,#5]{{\ifcurr@ntPS\ifps@cri\s@uvc@ntr@l\et@tpsarrowcircPTD%
    \PSc@mment{psarrowcircPTD Center=#1; Radius=#2, [P1=#3,P2=#4,P3=#5]}%
    \figgetangleTD\@ngfin[#1,#3,#4,#5]\v@leur=#2pt%
    \maxim@m{\mili@u}{-\v@leur}{\v@leur}\edef\r@dius{\repdecn@mb{\mili@u}}%
    \ifdim\v@leur<\z@\v@lmax=\@ngfin pt\advance\v@lmax-\DePI@deg%
    \edef\@ngfin{\repdecn@mb{\v@lmax}}\fi\psarrowcircTD#1,#3,#5;\r@dius(0,\@ngfin)%
    \PSc@mment{End psarrowcircPTD}\resetc@ntr@l\et@tpsarrowcircPTD\fi\fi}}
\ctr@ld@f\def\psaxes#1(#2){{\ifcurr@ntPS\ifps@cri\s@uvc@ntr@l\et@tpsaxes%
    \PSc@mment{psaxes Origin=#1 Range=(#2)}\an@lys@xes#2,:\resetc@ntr@l{2}%
    \ifx\t@xt@\empty\ifTr@isDim\ps@xes#1(0,#2,0,#2,0,#2)\else\ps@xes#1(0,#2,0,#2)\fi%
    \else\ps@xes#1(#2)\fi\PSc@mment{End psaxes}\resetc@ntr@l\et@tpsaxes\fi\fi}}
\ctr@ld@f\def\an@lys@xes#1,#2:{\def\t@xt@{#2}}
\ctr@ln@m\ps@xes
\ctr@ld@f\def\ps@xesDD#1(#2,#3,#4,#5){%
    \figpttraC-5:=#1/#2,0/\figpttraC-6:=#1/#3,0/\psarrowDD[-5,-6]%
    \figpttraC-5:=#1/0,#4/\figpttraC-6:=#1/0,#5/\psarrowDD[-5,-6]}
\ctr@ld@f\def\ps@xesTD#1(#2,#3,#4,#5,#6,#7){%
    \figpttraC-7:=#1/#2,0,0/\figpttraC-8:=#1/#3,0,0/\psarrowTD[-7,-8]%
    \figpttraC-7:=#1/0,#4,0/\figpttraC-8:=#1/0,#5,0/\psarrowTD[-7,-8]%
    \figpttraC-7:=#1/0,0,#6/\figpttraC-8:=#1/0,0,#7/\psarrowTD[-7,-8]}
\ctr@ln@m\newGr@FN
\ctr@ld@f\def\newGr@FNPDF#1{\s@mme=\Gr@FNb\advance\s@mme\@ne\xdef\Gr@FNb{\number\s@mme}}
\ctr@ld@f\def\newGr@FNDVI#1{\newGr@FNPDF{}\xdef#1{\jobname GI\Gr@FNb.anx}}
\ctr@ld@f\def\psbeginfig#1{\newGr@FN\DefGIfilen@me\gdef\@utoFN{0}%
    \def\t@xt@{#1}\relax\ifx\t@xt@\empty\psupdatem@detrue%
    \gdef\@utoFN{1}\Psb@ginfig\DefGIfilen@me\else\expandafter\Psb@ginfigNu@#1 :\fi}
\ctr@ld@f\def\Psb@ginfigNu@#1 #2:{\def\t@xt@{#1}\relax\ifx\t@xt@\empty\def\t@xt@{#2}%
    \ifx\t@xt@\empty\psupdatem@detrue\gdef\@utoFN{1}\Psb@ginfig\DefGIfilen@me%
    \else\Psb@ginfigNu@#2:\fi\else\Psb@ginfig{#1}\fi}
\ctr@ln@m\PSfilen@me \ctr@ln@m\auxfilen@me
\ctr@ld@f\def\Psb@ginfig#1{\ifcurr@ntPS\else%
    \edef\PSfilen@me{#1}\edef\auxfilen@me{\jobname.anx}%
    \ifpsupdatem@de\ps@critrue\else\openin\frf@g=\PSfilen@me\relax%
    \ifeof\frf@g\ps@critrue\else\ps@crifalse\fi\closein\frf@g\fi%
    \curr@ntPStrue\c@ldefproj\expandafter\setupd@te\defaultupdate:%
    \ifps@cri\initb@undb@x%
    \immediate\openout\fwf@g=\auxfilen@me\initpss@ttings\fi%
    \fi}
\ctr@ld@f\def\Gr@FNb{0}
\ctr@ld@f\def\figforTeXFileno{\Gr@FNb}
\ctr@ld@f\def\figforTeXFigno{0 }
\ctr@ld@f\def\figforTeXnextFigno{1 }
\ctr@ld@f\edef\DefGIfilen@me{\jobname GI.anx}
\ctr@ld@f\def\initpss@ttings{\psreset{arrowhead,curve,first,flowchart,mesh,second,third}%
    \Use@llipsefalse}
\ctr@ld@f\def\B@zierBB@x#1#2(#3,#4,#5,#6){{\c@rre=\t@n\epsil@n% Do not reduce this value
    \v@lmax=#4\advance\v@lmax-#5\v@lmax=\thr@@\v@lmax\advance\v@lmax#6\advance\v@lmax-#3%
    \mili@u=#4\mili@u=-\tw@\mili@u\advance\mili@u#3\advance\mili@u#5%
    \v@lmin=#4\advance\v@lmin-#3\maxim@m{\v@leur}{-\v@lmax}{\v@lmax}%
    \maxim@m{\delt@}{-\mili@u}{\mili@u}\maxim@m{\v@leur}{\v@leur}{\delt@}%
    \maxim@m{\delt@}{-\v@lmin}{\v@lmin}\maxim@m{\v@leur}{\v@leur}{\delt@}%
    \ifdim\v@leur>\c@rre\invers@{\v@leur}{\v@leur}\edef\Uns@rM@x{\repdecn@mb{\v@leur}}%
    \v@lmax=\Uns@rM@x\v@lmax\mili@u=\Uns@rM@x\mili@u\v@lmin=\Uns@rM@x\v@lmin%
    \maxim@m{\v@leur}{-\v@lmax}{\v@lmax}\ifdim\v@leur<\c@rre%
    \maxim@m{\v@leur}{-\mili@u}{\mili@u}\ifdim\v@leur<\c@rre\else%
    \invers@{\mili@u}{\mili@u}\v@leur=-0.5\v@lmin%
    \v@leur=\repdecn@mb{\mili@u}\v@leur\m@jBBB@x{\v@leur}{#1}{#2}(#3,#4,#5,#6)\fi%
    \else\delt@=\repdecn@mb{\mili@u}\mili@u\v@leur=\repdecn@mb{\v@lmax}\v@lmin%
    \advance\delt@-\v@leur\ifdim\delt@<\z@\else\invers@{\v@lmax}{\v@lmax}%
    \edef\Uns@rAp{\repdecn@mb{\v@lmax}}\sqrt@{\delt@}{\delt@}%
    \v@leur=-\mili@u\advance\v@leur\delt@\v@leur=\Uns@rAp\v@leur%
    \m@jBBB@x{\v@leur}{#1}{#2}(#3,#4,#5,#6)%
    \v@leur=-\mili@u\advance\v@leur-\delt@\v@leur=\Uns@rAp\v@leur%
    \m@jBBB@x{\v@leur}{#1}{#2}(#3,#4,#5,#6)\fi\fi\fi}}
\ctr@ld@f\def\m@jBBB@x#1#2#3(#4,#5,#6,#7){{\relax\ifdim#1>\z@\ifdim#1<\p@%
    \edef\T@{\repdecn@mb{#1}}\v@lX=\p@\advance\v@lX-#1\edef\UNmT@{\repdecn@mb{\v@lX}}%
    \v@lX=#4\v@lY=#5\v@lZ=#6\v@lXa=#7\v@lX=\UNmT@\v@lX\advance\v@lX\T@\v@lY%
    \v@lY=\UNmT@\v@lY\advance\v@lY\T@\v@lZ\v@lZ=\UNmT@\v@lZ\advance\v@lZ\T@\v@lXa%
    \v@lX=\UNmT@\v@lX\advance\v@lX\T@\v@lY\v@lY=\UNmT@\v@lY\advance\v@lY\T@\v@lZ%
    \v@lX=\UNmT@\v@lX\advance\v@lX\T@\v@lY%
    \ifcase#2\or\v@lY=#3\or\v@lY=\v@lX\v@lX=#3\fi\b@undb@x{\v@lX}{\v@lY}\fi\fi}}
\ctr@ld@f\def\PsB@zier#1[#2]{{\f@gnewpath%
    \s@mme=\z@\def\list@num{#2,0}\extrairelepremi@r\p@int\de\list@num%
    \PSwrit@cmdS{\p@int}{\c@mmoveto}{\fwf@g}{\X@un}{\Y@un}\p@rtent=#1\bclB@zier}}
\ctr@ld@f\def\bclB@zier{\relax%
    \ifnum\s@mme<\p@rtent\advance\s@mme\@ne\BdingB@xfalse%
    \extrairelepremi@r\p@int\de\list@num\PSwrit@cmdS{\p@int}{}{\fwf@g}{\X@de}{\Y@de}%
    \extrairelepremi@r\p@int\de\list@num\PSwrit@cmdS{\p@int}{}{\fwf@g}{\X@tr}{\Y@tr}%
    \BdingB@xtrue%
    \extrairelepremi@r\p@int\de\list@num\PSwrit@cmdS{\p@int}{\c@mcurveto}{\fwf@g}{\X@qu}{\Y@qu}%
    \B@zierBB@x{1}{\Y@un}(\X@un,\X@de,\X@tr,\X@qu)%
    \B@zierBB@x{2}{\X@un}(\Y@un,\Y@de,\Y@tr,\Y@qu)%
    \edef\X@un{\X@qu}\edef\Y@un{\Y@qu}\bclB@zier\fi}
\ctr@ln@m\psBezier
\ctr@ld@f\def\psBezierDD#1[#2]{\ifcurr@ntPS\ifps@cri%
    \PSc@mment{psBezierDD N arcs=#1, Control points=#2}%
    \iffillm@de\PsB@zier#1[#2]%
    \f@gfill%
    \else\PsB@zier#1[#2]\f@gstroke\fi%
    \PSc@mment{End psBezierDD}\fi\fi}
\ctr@ln@m\et@tpsBezierTD% ou doubler les {}
\ctr@ld@f\def\psBezierTD#1[#2]{\ifcurr@ntPS\ifps@cri\s@uvc@ntr@l\et@tpsBezierTD%
    \PSc@mment{psBezierTD N arcs=#1, Control points=#2}%
    \iffillm@de\PsB@zierTD#1[#2]%
    \f@gfill%
    \else\PsB@zierTD#1[#2]\f@gstroke\fi%
    \PSc@mment{End psBezierTD}\resetc@ntr@l\et@tpsBezierTD\fi\fi}
\ctr@ld@f\def\PsB@zierTD#1[#2]{\ifnum\curr@ntproj<\tw@\PsB@zier#1[#2]\else\PsB@zier@TD#1[#2]\fi}
\ctr@ld@f\def\PsB@zier@TD#1[#2]{{\f@gnewpath%
    \s@mme=\z@\def\list@num{#2,0}\extrairelepremi@r\p@int\de\list@num%
    \let\c@lprojSP=\relax\setc@ntr@l{2}\Figptpr@j-7:/\p@int/%
    \PSwrit@cmd{-7}{\c@mmoveto}{\fwf@g}%
    \loop\ifnum\s@mme<#1\advance\s@mme\@ne\extrairelepremi@r\p@intun\de\list@num%
    \extrairelepremi@r\p@intde\de\list@num\extrairelepremi@r\p@inttr\de\list@num%
    \subB@zierTD\NBz@rcs[\p@int,\p@intun,\p@intde,\p@inttr]\edef\p@int{\p@inttr}\repeat}}
\ctr@ld@f\def\subB@zierTD#1[#2,#3,#4,#5]{\delt@=\p@\divide\delt@\NBz@rcs\v@lmin=\z@%
    {\Figg@tXY{-7}\edef\X@un{\the\v@lX}\edef\Y@un{\the\v@lY}%
    \s@mme=\z@\loop\ifnum\s@mme<#1\advance\s@mme\@ne%
    \v@leur=\v@lmin\advance\v@leur0.33333 \delt@\edef\unti@rs{\repdecn@mb{\v@leur}}%
    \v@leur=\v@lmin\advance\v@leur0.66666 \delt@\edef\deti@rs{\repdecn@mb{\v@leur}}%
    \advance\v@lmin\delt@\edef\trti@rs{\repdecn@mb{\v@lmin}}%
    \figptBezierTD-8::\trti@rs[#2,#3,#4,#5]\Figptpr@j-8:/-8/%
    \c@lsubBzarc\unti@rs,\deti@rs[#2,#3,#4,#5]\BdingB@xfalse%
    \PSwrit@cmdS{-4}{}{\fwf@g}{\X@de}{\Y@de}\PSwrit@cmdS{-3}{}{\fwf@g}{\X@tr}{\Y@tr}%
    \BdingB@xtrue\PSwrit@cmdS{-8}{\c@mcurveto}{\fwf@g}{\X@qu}{\Y@qu}%
    \B@zierBB@x{1}{\Y@un}(\X@un,\X@de,\X@tr,\X@qu)%
    \B@zierBB@x{2}{\X@un}(\Y@un,\Y@de,\Y@tr,\Y@qu)%
    \edef\X@un{\X@qu}\edef\Y@un{\Y@qu}\figptcopyDD-7:/-8/\repeat}}
\ctr@ld@f\def\NBz@rcs{2}
\ctr@ld@f\def\c@lsubBzarc#1,#2[#3,#4,#5,#6]{\figptBezierTD-5::#1[#3,#4,#5,#6]%
    \figptBezierTD-6::#2[#3,#4,#5,#6]\Figptpr@j-4:/-5/\Figptpr@j-5:/-6/%
    \figptscontrolDD-4[-7,-4,-5,-8]}
\ctr@ln@m\pscirc
\ctr@ld@f\def\pscircDD#1(#2){\ifcurr@ntPS\ifps@cri\PSc@mment{pscircDD Center=#1 (Radius=#2)}%
    \psarccircDD#1;#2(0,360)\PSc@mment{End pscircDD}\fi\fi}
\ctr@ld@f\def\pscircTD#1,#2,#3(#4){\ifcurr@ntPS\ifps@cri%
    \PSc@mment{pscircTD Center=#1,P1=#2,P2=#3 (Radius=#4)}%
    \psarccircTD#1,#2,#3;#4(0,360)\PSc@mment{End pscircTD}\fi\fi}
\ctr@ln@m\p@urcent
{\catcode`\%=12\gdef\p@urcent{%}}
\ctr@ld@f\def\PSc@mment#1{\ifpsdebugmode\immediate\write\fwf@g{\p@urcent\space#1}\fi}
\ctr@ln@m\acc@louv \ctr@ln@m\acc@lfer
{\catcode`\[=1\catcode`\{=12\gdef\acc@louv[{}}
{\catcode`\]=2\catcode`\}=12\gdef\acc@lfer{}]]
\ctr@ld@f\def\PSdict@{\ifUse@llipse%
    \immediate\write\fwf@g{/ellipsedict 9 dict def ellipsedict /mtrx matrix put}%
    \immediate\write\fwf@g{/ellipse \acc@louv ellipsedict begin}%
    \immediate\write\fwf@g{ /endangle exch def /startangle exch def}%
    \immediate\write\fwf@g{ /yrad exch def /xrad exch def}%
    \immediate\write\fwf@g{ /rotangle exch def /y exch def /x exch def}%
    \immediate\write\fwf@g{ /savematrix mtrx currentmatrix def}%
    \immediate\write\fwf@g{ x y translate rotangle rotate xrad yrad scale}%
    \immediate\write\fwf@g{ 0 0 1 startangle endangle arc}%
    \immediate\write\fwf@g{ savematrix setmatrix end\acc@lfer def}%
    \fi\PShe@der{EndProlog}}
\ctr@ld@f\def\Pssetc@rve#1=#2|{\keln@mun#1|%
    \def\n@mref{r}\ifx\l@debut\n@mref\pssetroundness{#2}\else% roundness
    \immediate\write16{*** Unknown attribute: \BS@ psset curve(..., #1=...)}%
    \fi}
\ctr@ln@m\curv@roundness
\ctr@ld@f\def\pssetroundness#1{\edef\curv@roundness{#1}}
\ctr@ld@f\def\defaultroundness{0.2} % Valeur par defaut
\ctr@ln@m\pscurve
\ctr@ld@f\def\pscurveDD[#1]{{\ifcurr@ntPS\ifps@cri\PSc@mment{pscurveDD Points=#1}%
    \s@uvc@ntr@l\et@tpscurveDD%
    \iffillm@de\Psc@rveDD\curv@roundness[#1]%
    \f@gfill%
    \else\Psc@rveDD\curv@roundness[#1]\f@gstroke\fi%
    \PSc@mment{End pscurveDD}\resetc@ntr@l\et@tpscurveDD\fi\fi}}
\ctr@ld@f\def\pscurveTD[#1]{{\ifcurr@ntPS\ifps@cri%
    \PSc@mment{pscurveTD Points=#1}\s@uvc@ntr@l\et@tpscurveTD\let\c@lprojSP=\relax%
    \iffillm@de\Psc@rveTD\curv@roundness[#1]%
    \f@gfill%
    \else\Psc@rveTD\curv@roundness[#1]\f@gstroke\fi%
    \PSc@mment{End pscurveTD}\resetc@ntr@l\et@tpscurveTD\fi\fi}}
\ctr@ld@f\def\Psc@rveDD#1[#2]{%
    \def\list@num{#2}\extrairelepremi@r\Ak@\de\list@num%
    \extrairelepremi@r\Ai@\de\list@num\extrairelepremi@r\Aj@\de\list@num%
    \f@gnewpath\PSwrit@cmdS{\Ai@}{\c@mmoveto}{\fwf@g}{\X@un}{\Y@un}%
    \setc@ntr@l{2}\figvectPDD -1[\Ak@,\Aj@]%
    \@ecfor\Ak@:=\list@num\do{\figpttraDD-2:=\Ai@/#1,-1/\BdingB@xfalse%
       \PSwrit@cmdS{-2}{}{\fwf@g}{\X@de}{\Y@de}%
       \figvectPDD -1[\Ai@,\Ak@]\figpttraDD-2:=\Aj@/-#1,-1/%
       \PSwrit@cmdS{-2}{}{\fwf@g}{\X@tr}{\Y@tr}\BdingB@xtrue%
       \PSwrit@cmdS{\Aj@}{\c@mcurveto}{\fwf@g}{\X@qu}{\Y@qu}%
       \B@zierBB@x{1}{\Y@un}(\X@un,\X@de,\X@tr,\X@qu)%
       \B@zierBB@x{2}{\X@un}(\Y@un,\Y@de,\Y@tr,\Y@qu)%
       \edef\X@un{\X@qu}\edef\Y@un{\Y@qu}\edef\Ai@{\Aj@}\edef\Aj@{\Ak@}}}
\ctr@ld@f\def\Psc@rveTD#1[#2]{\ifnum\curr@ntproj<\tw@\Psc@rvePPTD#1[#2]\else\Psc@rveCPTD#1[#2]\fi}
\ctr@ld@f\def\Psc@rvePPTD#1[#2]{\setc@ntr@l{2}%
    \def\list@num{#2}\extrairelepremi@r\Ak@\de\list@num\Figptpr@j-5:/\Ak@/%
    \extrairelepremi@r\Ai@\de\list@num\Figptpr@j-3:/\Ai@/%
    \extrairelepremi@r\Aj@\de\list@num\Figptpr@j-4:/\Aj@/%
    \f@gnewpath\PSwrit@cmdS{-3}{\c@mmoveto}{\fwf@g}{\X@un}{\Y@un}%
    \figvectPDD -1[-5,-4]%
    \@ecfor\Ak@:=\list@num\do{\Figptpr@j-5:/\Ak@/\figpttraDD-2:=-3/#1,-1/%
       \BdingB@xfalse\PSwrit@cmdS{-2}{}{\fwf@g}{\X@de}{\Y@de}%
       \figvectPDD -1[-3,-5]\figpttraDD-2:=-4/-#1,-1/%
       \PSwrit@cmdS{-2}{}{\fwf@g}{\X@tr}{\Y@tr}\BdingB@xtrue%
       \PSwrit@cmdS{-4}{\c@mcurveto}{\fwf@g}{\X@qu}{\Y@qu}%
       \B@zierBB@x{1}{\Y@un}(\X@un,\X@de,\X@tr,\X@qu)%
       \B@zierBB@x{2}{\X@un}(\Y@un,\Y@de,\Y@tr,\Y@qu)%
       \edef\X@un{\X@qu}\edef\Y@un{\Y@qu}\figptcopyDD-3:/-4/\figptcopyDD-4:/-5/}}
\ctr@ld@f\def\Psc@rveCPTD#1[#2]{\setc@ntr@l{2}%
    \def\list@num{#2}\extrairelepremi@r\Ak@\de\list@num%
    \extrairelepremi@r\Ai@\de\list@num\extrairelepremi@r\Aj@\de\list@num%
    \Figptpr@j-7:/\Ai@/%
    \f@gnewpath\PSwrit@cmd{-7}{\c@mmoveto}{\fwf@g}%
    \figvectPTD -9[\Ak@,\Aj@]%
    \@ecfor\Ak@:=\list@num\do{\figpttraTD-10:=\Ai@/#1,-9/%
       \figvectPTD -9[\Ai@,\Ak@]\figpttraTD-11:=\Aj@/-#1,-9/%
       \subB@zierTD\NBz@rcs[\Ai@,-10,-11,\Aj@]\edef\Ai@{\Aj@}\edef\Aj@{\Ak@}}}
\ctr@ld@f\def\psendfig{\ifcurr@ntPS\ifps@cri\immediate\closeout\fwf@g%
    \immediate\openout\fwf@g=\PSfilen@me\relax%
    \ifPDFm@ke\PSBdingB@x\else%
    \immediate\write\fwf@g{\p@urcent\string!PS-Adobe-2.0 EPSF-2.0}%
    \PShe@der{Creator\string: TeX (fig4tex.tex)}%
    \PShe@der{Title\string: \PSfilen@me}%
    \PShe@der{CreationDate\string: \the\day/\the\month/\the\year}%
    \PSBdingB@x%
    \PShe@der{EndComments}\PSdict@\fi%
    \immediate\write\fwf@g{\c@mgsave}%
    \openin\frf@g=\auxfilen@me\c@pypsfile\fwf@g\frf@g\closein\frf@g%
    \immediate\write\fwf@g{\c@mgrestore}%
    \PSc@mment{End of file.}\immediate\closeout\fwf@g%
    \immediate\openout\fwf@g=\auxfilen@me\immediate\closeout\fwf@g%
    \immediate\write16{File \PSfilen@me\space created.}\fi\fi\curr@ntPSfalse\ps@critrue}
\ctr@ld@f\def\PShe@der#1{\immediate\write\fwf@g{\p@urcent\p@urcent#1}}
\ctr@ld@f\def\PSBdingB@x{{\v@lX=\ptT@ptps\c@@rdXmin\v@lY=\ptT@ptps\c@@rdYmin%
     \v@lXa=\ptT@ptps\c@@rdXmax\v@lYa=\ptT@ptps\c@@rdYmax%
     \PShe@der{BoundingBox\string: \repdecn@mb{\v@lX}\space\repdecn@mb{\v@lY}%
     \space\repdecn@mb{\v@lXa}\space\repdecn@mb{\v@lYa}}}}
\ctr@ld@f\def\psfcconnect[#1]{{\ifcurr@ntPS\ifps@cri\PSc@mment{psfcconnect Points=#1}%
    \pssetfillmode{no}\s@uvc@ntr@l\et@tpsfcconnect\resetc@ntr@l{2}%
    \fcc@nnect@[#1]\resetc@ntr@l\et@tpsfcconnect\PSc@mment{End psfcconnect}\fi\fi}}
\ctr@ld@f\def\fcc@nnect@[#1]{\let\N@rm=\n@rmeucDD\def\list@num{#1}%
    \extrairelepremi@r\Ai@\de\list@num\edef\pr@m{\Ai@}\v@leur=\z@\p@rtent=\@ne\c@llgtot%
    \ifcase\fclin@typ@\edef\list@num{[\pr@m,#1,\Ai@}\expandafter\pscurve\list@num]%
    \else\ifdim\fclin@r@d\p@>\z@\Pslin@conge[#1]\else\psline[#1]\fi\fi%
    \v@leur=\@rrowp@s\v@leur\edef\list@num{#1,\Ai@,0}%
    \extrairelepremi@r\Ai@\de\list@num\mili@u=\epsil@n\c@llgpart%
    \advance\mili@u-\epsil@n\advance\mili@u-\delt@\advance\v@leur-\mili@u%
    \ifcase\fclin@typ@\invers@\mili@u\delt@%
    \ifnum\@rrowr@fpt>\z@\advance\delt@-\v@leur\v@leur=\delt@\fi%
    \v@leur=\repdecn@mb\v@leur\mili@u\edef\v@lt{\repdecn@mb\v@leur}%
    \extrairelepremi@r\Ak@\de\list@num%
    \figvectPDD-1[\pr@m,\Aj@]\figpttraDD-6:=\Ai@/\curv@roundness,-1/%
    \figvectPDD-1[\Ak@,\Ai@]\figpttraDD-7:=\Aj@/\curv@roundness,-1/%
    \delt@=\@rrowheadlength\p@\delt@=\C@AHANG\delt@\edef\R@dius{\repdecn@mb{\delt@}}%
    \ifcase\@rrowr@fpt%
    \FigptintercircB@zDD-8::\v@lt,\R@dius[\Ai@,-6,-7,\Aj@]\psarrowheadDD[-5,-8]\else%
    \FigptintercircB@zDD-8::\v@lt,\R@dius[\Aj@,-7,-6,\Ai@]\psarrowheadDD[-8,-5]\fi%
    \else\advance\delt@-\v@leur%
    \p@rtentiere{\p@rtent}{\delt@}\edef\C@efun{\the\p@rtent}%
    \p@rtentiere{\p@rtent}{\v@leur}\edef\C@efde{\the\p@rtent}%
    \figptbaryDD-5:[\Ai@,\Aj@;\C@efun,\C@efde]\ifcase\@rrowr@fpt%
    \delt@=\@rrowheadlength\unit@\delt@=\C@AHANG\delt@\edef\t@ille{\repdecn@mb{\delt@}}%
    \figvectPDD-2[\Ai@,\Aj@]\vecunit@{-2}{-2}\figpttraDD-5:=-5/\t@ille,-2/\fi%
    \psarrowheadDD[\Ai@,-5]\fi}
\ctr@ld@f\def\c@llgtot{\@ecfor\Aj@:=\list@num\do{\figvectP-1[\Ai@,\Aj@]\N@rm\delt@{-1}%
    \advance\v@leur\delt@\advance\p@rtent\@ne\edef\Ai@{\Aj@}}}
\ctr@ld@f\def\c@llgpart{\extrairelepremi@r\Aj@\de\list@num\figvectP-1[\Ai@,\Aj@]\N@rm\delt@{-1}%
    \advance\mili@u\delt@\ifdim\mili@u<\v@leur\edef\pr@m{\Ai@}\edef\Ai@{\Aj@}\c@llgpart\fi}
\ctr@ld@f\def\Pslin@conge[#1]{\ifnum\p@rtent>\tw@{\def\list@num{#1}%
    \extrairelepremi@r\Ai@\de\list@num\extrairelepremi@r\Aj@\de\list@num%
    \figptcopy-6:/\Ai@/\figvectP-3[\Ai@,\Aj@]\vecunit@{-3}{-3}\v@lmax=\result@t%
    \@ecfor\Ak@:=\list@num\do{\figvectP-4[\Aj@,\Ak@]\vecunit@{-4}{-4}%
    \minim@m\v@lmin\v@lmax\result@t\v@lmax=\result@t%
    \det@rm\delt@[-3,-4]\maxim@m\mili@u{\delt@}{-\delt@}\ifdim\mili@u>\Cepsil@n%
    \ifdim\delt@>\z@\figgetangleDD\Angl@[\Aj@,\Ak@,\Ai@]\else%
    \figgetangleDD\Angl@[\Aj@,\Ai@,\Ak@]\fi%
    \v@leur=\PI@deg\advance\v@leur-\Angl@\p@\divide\v@leur\tw@%
    \edef\Angl@{\repdecn@mb\v@leur}\c@ssin{\C@}{\S@}{\Angl@}\v@leur=\fclin@r@d\unit@%
    \v@leur=\S@\v@leur\mili@u=\C@\p@\invers@\mili@u\mili@u%
    \v@leur=\repdecn@mb{\mili@u}\v@leur%
    \minim@m\v@leur\v@leur\v@lmin\edef\t@ille{\repdecn@mb{\v@leur}}%
    \figpttra-5:=\Aj@/-\t@ille,-3/\psline[-6,-5]\figpttra-6:=\Aj@/\t@ille,-4/%
    \figvectNVDD-3[-3]\figvectNVDD-8[-4]\inters@cDD-7:[-5,-3;-6,-8]%
    \ifdim\delt@>\z@\psarccircP-7;\fclin@r@d[-5,-6]\else\psarccircP-7;\fclin@r@d[-6,-5]\fi%
    \else\psline[-6,\Aj@]\figptcopy-6:/\Aj@/\fi% Points alignes
    \edef\Ai@{\Aj@}\edef\Aj@{\Ak@}\figptcopy-3:/-4/}\psline[-6,\Aj@]}\else\psline[#1]\fi}
\ctr@ld@f\def\psfcnode[#1]#2{{\ifcurr@ntPS\ifps@cri\PSc@mment{psfcnode Points=#1}%
    \s@uvc@ntr@l\et@tpsfcnode\resetc@ntr@l{2}%
    \def\t@xt@{#2}\ifx\t@xt@\empty\def\g@tt@xt{\setbox\Gb@x=\hbox{\Figg@tT{\p@int}}}%
    \else\def\g@tt@xt{\setbox\Gb@x=\hbox{#2}}\fi%
    \v@lmin=\h@rdfcXp@dd\advance\v@lmin\Xp@dd\unit@\multiply\v@lmin\tw@%
    \v@lmax=\h@rdfcYp@dd\advance\v@lmax\Yp@dd\unit@\multiply\v@lmax\tw@%
    \Figv@ctCreg-8(\unit@,-\unit@)\def\list@num{#1}%
    \delt@=\curr@ntwidth bp\divide\delt@\tw@%
    \fcn@de\PSc@mment{End psfcnode}\resetc@ntr@l\et@tpsfcnode\fi\fi}}
\ctr@ld@f\def\d@butn@de{\g@tt@xt\v@lX=\wd\Gb@x%
    \v@lY=\ht\Gb@x\advance\v@lY\dp\Gb@x\advance\v@lX\v@lmin\advance\v@lY\v@lmax}
\ctr@ld@f\def\fcn@deE{%
    \@ecfor\p@int:=\list@num\do{\d@butn@de\v@lX=\unssqrttw@\v@lX\v@lY=\unssqrttw@\v@lY%
    \ifdim\thickn@ss\p@>\z@% Shadow
    \v@lXa=\v@lX\advance\v@lXa\delt@\v@lXa=\ptT@unit@\v@lXa\edef\XR@d{\repdecn@mb\v@lXa}%
    \v@lYa=\v@lY\advance\v@lYa\delt@\v@lYa=\ptT@unit@\v@lYa\edef\YR@d{\repdecn@mb\v@lYa}%
    \arct@n\v@leur(\v@lXa,\v@lYa)\v@leur=\rdT@deg\v@leur\edef\@nglde{\repdecn@mb\v@leur}%
    {\c@lptellDD-2::\p@int;\XR@d,\YR@d(\@nglde)}% \v@lmin & \v@lmax modified in \c@lptellDD
    \advance\v@leur-\PI@deg\edef\@nglun{\repdecn@mb\v@leur}%
    {\c@lptellDD-3::\p@int;\XR@d,\YR@d(\@nglun)}%
    \figptstra-6=-3,-2,\p@int/\thickn@ss,-8/\pssetfillmode{yes}\us@secondC@lor%
    \psline[-2,-3,-6,-5]\psarcell-4;\XR@d,\YR@d(\@nglun,\@nglde,0)\fi% End shadow
    \v@lX=\ptT@unit@\v@lX\v@lY=\ptT@unit@\v@lY%
    \edef\XR@d{\repdecn@mb\v@lX}\edef\YR@d{\repdecn@mb\v@lY}%
    \pssetfillmode{yes}\us@thirdC@lor\psarcell\p@int;\XR@d,\YR@d(0,360,0)%
    \pssetfillmode{no}\us@primarC@lor\psarcell\p@int;\XR@d,\YR@d(0,360,0)}}
\ctr@ld@f\def\fcn@deL{\delt@=\ptT@unit@\delt@\edef\t@ille{\repdecn@mb\delt@}%
    \@ecfor\p@int:=\list@num\do{\Figg@tXYa{\p@int}\d@butn@de%
    \ifdim\v@lX>\v@lY\itis@Ktrue\else\itis@Kfalse\fi%
    \advance\v@lXa-\v@lX\Figp@intreg-1:(\v@lXa,\v@lYa)%
    \advance\v@lXa\v@lX\advance\v@lYa-\v@lY\Figp@intreg-2:(\v@lXa,\v@lYa)%
    \advance\v@lXa\v@lX\advance\v@lYa\v@lY\Figp@intreg-3:(\v@lXa,\v@lYa)%
    \advance\v@lXa-\v@lX\advance\v@lYa\v@lY\Figp@intreg-4:(\v@lXa,\v@lYa)%
    \ifdim\thickn@ss\p@>\z@\Figg@tXYa{\p@int}\pssetfillmode{yes}\us@secondC@lor% Shadow
    \c@lpt@xt{-1}{-4}\c@lpt@xt@\v@lXa\v@lYa\v@lX\v@lY\c@rre\delt@%
    \Figp@intregDD-9:(\v@lZ,\v@lYa)\Figp@intregDD-11:(\v@lZa,\v@lYa)%
    \c@lpt@xt{-4}{-3}\c@lpt@xt@\v@lYa\v@lXa\v@lY\v@lX\delt@\c@rre%
    \Figp@intregDD-12:(\v@lXa,\v@lZ)\Figp@intregDD-10:(\v@lXa,\v@lZa)%
    \ifitis@K\figptstra-7=-9,-10,-11/\thickn@ss,-8/\psline[-9,-11,-5,-6,-7]\else%
    \figptstra-7=-10,-11,-12/\thickn@ss,-8/\psline[-10,-12,-5,-6,-7]\fi\fi% End shadow
    \pssetfillmode{yes}\us@thirdC@lor\psline[-1,-2,-3,-4]%
    \pssetfillmode{no}\us@primarC@lor\psline[-1,-2,-3,-4,-1]}}
\ctr@ld@f\def\c@lpt@xt#1#2{\figvectN-7[#1,#2]\vecunit@{-7}{-7}\figpttra-5:=#1/\t@ille,-7/%
    \figvectP-7[#1,#2]\Figg@tXY{-7}\c@rre=\v@lX\delt@=\v@lY\Figg@tXY{-5}}
\ctr@ld@f\def\c@lpt@xt@#1#2#3#4#5#6{\v@lZ=#6\invers@{\v@lZ}{\v@lZ}\v@leur=\repdecn@mb{#5}\v@lZ%
    \v@lZ=#2\advance\v@lZ-#4\mili@u=\repdecn@mb{\v@leur}\v@lZ%
    \v@lZ=#3\advance\v@lZ\mili@u\v@lZa=-\v@lZ\advance\v@lZa\tw@#1}
\ctr@ld@f\def\fcn@deR{\@ecfor\p@int:=\list@num\do{\Figg@tXYa{\p@int}\d@butn@de%
    \advance\v@lXa-0.5\v@lX\advance\v@lYa-0.5\v@lY\Figp@intreg-1:(\v@lXa,\v@lYa)%
    \advance\v@lXa\v@lX\Figp@intreg-2:(\v@lXa,\v@lYa)%
    \advance\v@lYa\v@lY\Figp@intreg-3:(\v@lXa,\v@lYa)%
    \advance\v@lXa-\v@lX\Figp@intreg-4:(\v@lXa,\v@lYa)%
    \ifdim\thickn@ss\p@>\z@\pssetfillmode{yes}\us@secondC@lor% Shadow
    \Figv@ctCreg-5(-\delt@,-\delt@)\figpttra-9:=-1/1,-5/%
    \Figv@ctCreg-5(\delt@,-\delt@)\figpttra-10:=-2/1,-5/%
    \Figv@ctCreg-5(\delt@,\delt@)\figpttra-11:=-3/1,-5/%
    \figptstra-7=-9,-10,-11/\thickn@ss,-8/\psline[-9,-11,-5,-6,-7]\fi% End shadow
    \pssetfillmode{yes}\us@thirdC@lor\psline[-1,-2,-3,-4]%
    \pssetfillmode{no}\us@primarC@lor\psline[-1,-2,-3,-4,-1]}}
\ctr@ln@m\@rrowp@s
\ctr@ln@m\Xp@dd     \ctr@ln@m\Yp@dd
\ctr@ln@m\fclin@r@d \ctr@ln@m\thickn@ss
\ctr@ld@f\def\Pssetfl@wchart#1=#2|{\keln@mtr#1|%
    \def\n@mref{arr}\ifx\l@debut\n@mref\expandafter\keln@mtr\l@suite|%
     \def\n@mref{owp}\ifx\l@debut\n@mref\edef\@rrowp@s{#2}\else% arrowposition
     \def\n@mref{owr}\ifx\l@debut\n@mref\setfcr@fpt#2|\else% arrowrefpt
     \immediate\write16{*** Unknown attribute: \BS@ psset flowchart(..., #1=...)}%
     \fi\fi\else%
    \def\n@mref{lin}\ifx\l@debut\n@mref\setfccurv@#2|\else% line
    \def\n@mref{pad}\ifx\l@debut\n@mref\edef\Xp@dd{#2}\edef\Yp@dd{#2}\else% padding
    \def\n@mref{rad}\ifx\l@debut\n@mref\edef\fclin@r@d{#2}\else% connection radius
    \def\n@mref{sha}\ifx\l@debut\n@mref\setfcshap@#2|\else% shape
    \def\n@mref{thi}\ifx\l@debut\n@mref\edef\thickn@ss{#2}\else% thickness
    \def\n@mref{xpa}\ifx\l@debut\n@mref\edef\Xp@dd{#2}\else% xpadding
    \def\n@mref{ypa}\ifx\l@debut\n@mref\edef\Yp@dd{#2}\else% ypadding
    \immediate\write16{*** Unknown attribute: \BS@ psset flowchart(..., #1=...)}%
    \fi\fi\fi\fi\fi\fi\fi\fi}
\ctr@ln@m\@rrowr@fpt \ctr@ln@m\fclin@typ@
\ctr@ld@f\def\setfcr@fpt#1#2|{\if#1e\def\@rrowr@fpt{1}\else\def\@rrowr@fpt{0}\fi}
\ctr@ld@f\def\setfccurv@#1#2|{\if#1c\def\fclin@typ@{0}\else\def\fclin@typ@{1}\fi}
\ctr@ln@m\h@rdfcXp@dd \ctr@ln@m\h@rdfcYp@dd
\ctr@ln@m\fcn@de \ctr@ln@m\fcsh@pe
\ctr@ld@f\def\setfcshap@#1#2|{%
    \if#1e\let\fcn@de=\fcn@deE\def\h@rdfcXp@dd{4pt}\def\h@rdfcYp@dd{4pt}%
     \edef\fcsh@pe{ellipse}\else%
    \if#1l\let\fcn@de=\fcn@deL\def\h@rdfcXp@dd{4pt}\def\h@rdfcYp@dd{4pt}%
     \edef\fcsh@pe{lozenge}\else%
          \let\fcn@de=\fcn@deR\def\h@rdfcXp@dd{6pt}\def\h@rdfcYp@dd{6pt}%
     \edef\fcsh@pe{rectangle}\fi\fi}
\ctr@ld@f\def\psline[#1]{{\ifcurr@ntPS\ifps@cri\PSc@mment{psline Points=#1}%
    \let\pslign@=\Pslign@P\Pslin@{#1}\PSc@mment{End psline}\fi\fi}}
\ctr@ld@f\def\pslineF#1{{\ifcurr@ntPS\ifps@cri\PSc@mment{pslineF Filename=#1}%
    \let\pslign@=\Pslign@F\Pslin@{#1}\PSc@mment{End pslineF}\fi\fi}}
\ctr@ld@f\def\pslineC(#1){{\ifcurr@ntPS\ifps@cri\PSc@mment{pslineC}%
    \let\pslign@=\Pslign@C\Pslin@{#1}\PSc@mment{End pslineC}\fi\fi}}
\ctr@ld@f\def\Pslin@#1{\iffillm@de\pslign@{#1}%
    \f@gfill%
    \else\pslign@{#1}\ifx\derp@int\premp@int%
    \f@gclosestroke%
    \else\f@gstroke\fi\fi}
\ctr@ld@f\def\Pslign@P#1{\def\list@num{#1}\extrairelepremi@r\p@int\de\list@num%
    \edef\premp@int{\p@int}\f@gnewpath%
    \PSwrit@cmd{\p@int}{\c@mmoveto}{\fwf@g}%
    \@ecfor\p@int:=\list@num\do{\PSwrit@cmd{\p@int}{\c@mlineto}{\fwf@g}%
    \edef\derp@int{\p@int}}}
\ctr@ld@f\def\Pslign@F#1{\s@uvc@ntr@l\et@tPslign@F\setc@ntr@l{2}\openin\frf@g=#1\relax%
    \ifeof\frf@g\message{*** File #1 not found !}\end\else%
    \read\frf@g to\tr@c\edef\premp@int{\tr@c}\expandafter\extr@ctCF\tr@c:%
    \f@gnewpath\PSwrit@cmd{-1}{\c@mmoveto}{\fwf@g}%
    \loop\read\frf@g to\tr@c\ifeof\frf@g\mored@tafalse\else\mored@tatrue\fi%
    \ifmored@ta\expandafter\extr@ctCF\tr@c:\PSwrit@cmd{-1}{\c@mlineto}{\fwf@g}%
    \edef\derp@int{\tr@c}\repeat\fi\closein\frf@g\resetc@ntr@l\et@tPslign@F}
\ctr@ln@m\extr@ctCF
\ctr@ld@f\def\extr@ctCFDD#1 #2:{\v@lX=#1\unit@\v@lY=#2\unit@\Figp@intregDD-1:(\v@lX,\v@lY)}
\ctr@ld@f\def\extr@ctCFTD#1 #2 #3:{\v@lX=#1\unit@\v@lY=#2\unit@\v@lZ=#3\unit@%
    \Figp@intregTD-1:(\v@lX,\v@lY,\v@lZ)}
\ctr@ld@f\def\Pslign@C#1{\s@uvc@ntr@l\et@tPslign@C\setc@ntr@l{2}%
    \def\list@num{#1}\extrairelepremi@r\p@int\de\list@num%
    \edef\premp@int{\p@int}\f@gnewpath%
    \expandafter\Pslign@C@\p@int:\PSwrit@cmd{-1}{\c@mmoveto}{\fwf@g}%
    \@ecfor\p@int:=\list@num\do{\expandafter\Pslign@C@\p@int:%
    \PSwrit@cmd{-1}{\c@mlineto}{\fwf@g}\edef\derp@int{\p@int}}%
    \resetc@ntr@l\et@tPslign@C}
\ctr@ld@f\def\Pslign@C@#1 #2:{{\def\t@xt@{#1}\ifx\t@xt@\empty\Pslign@C@#2:% Discard leading spaces
    \else\extr@ctCF#1 #2:\fi}}
\ctr@ln@m\c@ntrolmesh
\ctr@ld@f\def\Pssetm@sh#1=#2|{\keln@mun#1|%
    \def\n@mref{d}\ifx\l@debut\n@mref\pssetmeshdiag{#2}\else% diag
    \immediate\write16{*** Unknown attribute: \BS@ psset mesh(..., #1=...)}%
    \fi}
\ctr@ld@f\def\pssetmeshdiag#1{\edef\c@ntrolmesh{#1}}
\ctr@ld@f\def\defaultmeshdiag{0}    % Valeur par defaut
\ctr@ld@f\def\psmesh#1,#2[#3,#4,#5,#6]{{\ifcurr@ntPS\ifps@cri%
    \PSc@mment{psmesh N1=#1, N2=#2, Quadrangle=[#3,#4,#5,#6]}%
    \s@uvc@ntr@l\et@tpsmesh\Pss@tsecondSt\setc@ntr@l{2}%
    \ifnum#1>\@ne\Psmeshp@rt#1[#3,#4,#5,#6]\fi%
    \ifnum#2>\@ne\Psmeshp@rt#2[#4,#5,#6,#3]\fi%
    \ifnum\c@ntrolmesh>\z@\Psmeshdi@g#1,#2[#3,#4,#5,#6]\fi%
    \ifnum\c@ntrolmesh<\z@\Psmeshdi@g#2,#1[#4,#5,#6,#3]\fi\Psrest@reSt%
    \psline[#3,#4,#5,#6,#3]\PSc@mment{End psmesh}\resetc@ntr@l\et@tpsmesh\fi\fi}}
\ctr@ld@f\def\Psmeshp@rt#1[#2,#3,#4,#5]{{\l@mbd@un=\@ne\l@mbd@de=#1\loop%
    \ifnum\l@mbd@un<#1\advance\l@mbd@de\m@ne\figptbary-1:[#2,#3;\l@mbd@de,\l@mbd@un]%
    \figptbary-2:[#5,#4;\l@mbd@de,\l@mbd@un]\psline[-1,-2]\advance\l@mbd@un\@ne\repeat}}
\ctr@ld@f\def\Psmeshdi@g#1,#2[#3,#4,#5,#6]{\figptcopy-2:/#3/\figptcopy-3:/#6/%
    \l@mbd@un=\z@\l@mbd@de=#1\loop\ifnum\l@mbd@un<#1%
    \advance\l@mbd@un\@ne\advance\l@mbd@de\m@ne\figptcopy-1:/-2/\figptcopy-4:/-3/%
    \figptbary-2:[#3,#4;\l@mbd@de,\l@mbd@un]%
    \figptbary-3:[#6,#5;\l@mbd@de,\l@mbd@un]\Psmeshdi@gp@rt#2[-1,-2,-3,-4]\repeat}
\ctr@ld@f\def\Psmeshdi@gp@rt#1[#2,#3,#4,#5]{{\l@mbd@un=\z@\l@mbd@de=#1\loop%
    \ifnum\l@mbd@un<#1\figptbary-5:[#2,#5;\l@mbd@de,\l@mbd@un]%
    \advance\l@mbd@de\m@ne\advance\l@mbd@un\@ne%
    \figptbary-6:[#3,#4;\l@mbd@de,\l@mbd@un]\psline[-5,-6]\repeat}}
\ctr@ln@m\psnormal
\ctr@ld@f\def\psnormalDD#1,#2[#3,#4]{{\ifcurr@ntPS\ifps@cri%
    \PSc@mment{psnormal Length=#1, Lambda=#2 [Pt1,Pt2]=[#3,#4]}%
    \s@uvc@ntr@l\et@tpsnormal\resetc@ntr@l{2}\figptendnormal-6::#1,#2[#3,#4]%
    \figptcopyDD-5:/-1/\psarrow[-5,-6]%
    \PSc@mment{End psnormal}\resetc@ntr@l\et@tpsnormal\fi\fi}}
\ctr@ld@f\def\psreset#1{\trtlis@rg{#1}{\Psreset@}}
\ctr@ld@f\def\Psreset@#1|{\keln@mde#1|%
    \def\n@mref{ar}\ifx\l@debut\n@mref\psresetarrowhead\else% arrowhead
    \def\n@mref{cu}\ifx\l@debut\n@mref\psset curve(roundness=\defaultroundness)\else% curve
    \def\n@mref{fi}\ifx\l@debut\n@mref\psset (color=\defaultcolor,dash=\defaultdash,%
         fill=\defaultfill,join=\defaultjoin,width=\defaultwidth)\else% primary settings
    \def\n@mref{fl}\ifx\l@debut\n@mref\psset flowchart(arrowp=\defaultfcarrowposition,%
	arrowr=\defaultfcarrowrefpt,line=\defaultfcline,xpadd=\defaultfcxpadding,%
	ypadd=\defaultfcypadding,radius=\defaultfcradius,shape=\defaultfcshape,%
	thick=\defaultfcthickness)\else% flow chart
    \def\n@mref{me}\ifx\l@debut\n@mref\psset mesh(diag=\defaultmeshdiag)\else% mesh
    \def\n@mref{se}\ifx\l@debut\n@mref\psresetsecondsettings\else% secondary
    \def\n@mref{th}\ifx\l@debut\n@mref\psset third(color=\defaultthirdcolor)\else% ternary
    \immediate\write16{*** Unknown keyword #1 (\BS@ psreset).}%
    \fi\fi\fi\fi\fi\fi\fi}
\ctr@ld@f\def\psset#1(#2){\def\t@xt@{#1}\ifx\t@xt@\empty\trtlis@rg{#2}{\Pssetf@rst}% primary settings
    \else\keln@mde#1|%
    \def\n@mref{ar}\ifx\l@debut\n@mref\trtlis@rg{#2}{\Psset@rrowhe@d}\else% arrow-head
    \def\n@mref{cu}\ifx\l@debut\n@mref\trtlis@rg{#2}{\Pssetc@rve}\else% curve
    \def\n@mref{fi}\ifx\l@debut\n@mref\trtlis@rg{#2}{\Pssetf@rst}\else% primary settings
    \def\n@mref{fl}\ifx\l@debut\n@mref\trtlis@rg{#2}{\Pssetfl@wchart}\else% flow chart
    \def\n@mref{me}\ifx\l@debut\n@mref\trtlis@rg{#2}{\Pssetm@sh}\else% mesh
    \def\n@mref{se}\ifx\l@debut\n@mref\trtlis@rg{#2}{\Pssets@cond}\else% secondary settings
    \def\n@mref{th}\ifx\l@debut\n@mref\trtlis@rg{#2}{\Pssetth@rd}\else% ternary settings
    \immediate\write16{*** Unknown keyword: \BS@ psset #1(...)}%
    \fi\fi\fi\fi\fi\fi\fi\fi}
\ctr@ld@f\def\pssetdefault#1(#2){\ifcurr@ntPS\immediate\write16{*** \BS@ pssetdefault is ignored
    inside a \BS@ psbeginfig-\BS@ psendfig block.}%
    \immediate\write16{*** It must be called before \BS@ psbeginfig.}\else%
    \def\t@xt@{#1}\ifx\t@xt@\empty\trtlis@rg{#2}{\Pssd@f@rst}\else\keln@mde#1|%
    \def\n@mref{ar}\ifx\l@debut\n@mref\trtlis@rg{#2}{\Pssd@@rrowhe@d}\else% arrow-head
    \def\n@mref{cu}\ifx\l@debut\n@mref\trtlis@rg{#2}{\Pssd@c@rve}\else% curve
    \def\n@mref{fi}\ifx\l@debut\n@mref\trtlis@rg{#2}{\Pssd@f@rst}\else% primary settings
    \def\n@mref{fl}\ifx\l@debut\n@mref\trtlis@rg{#2}{\Pssd@fl@wchart}\else% flow chart
    \def\n@mref{me}\ifx\l@debut\n@mref\trtlis@rg{#2}{\Pssd@m@sh}\else% mesh
    \def\n@mref{se}\ifx\l@debut\n@mref\trtlis@rg{#2}{\Pssd@s@cond}\else% secondary settings
    \def\n@mref{th}\ifx\l@debut\n@mref\trtlis@rg{#2}{\Pssd@th@rd}\else% ternary settings
    \immediate\write16{*** Unknown keyword: \BS@ pssetdefault #1(...)}%
    \fi\fi\fi\fi\fi\fi\fi\fi\initpss@ttings\fi}
\ctr@ld@f\def\Pssd@f@rst#1=#2|{\keln@mun#1|%
    \def\n@mref{c}\ifx\l@debut\n@mref\edef\defaultcolor{#2}\else% color
    \def\n@mref{d}\ifx\l@debut\n@mref\edef\defaultdash{#2}\else% dash
    \def\n@mref{f}\ifx\l@debut\n@mref\edef\defaultfill{#2}\else% fillmode
    \def\n@mref{j}\ifx\l@debut\n@mref\edef\defaultjoin{#2}\else% line join
    \def\n@mref{u}\ifx\l@debut\n@mref\edef\defaultupdate{#2}\pssetupdate{#2}\else% update
    \def\n@mref{w}\ifx\l@debut\n@mref\edef\defaultwidth{#2}\else% line width
    \immediate\write16{*** Unknown attribute: \BS@ pssetdefault (..., #1=...)}%
    \fi\fi\fi\fi\fi\fi}
\ctr@ld@f\def\Pssd@@rrowhe@d#1=#2|{\keln@mun#1|%
    \def\n@mref{a}\ifx\l@debut\n@mref\edef\defaultarrowheadangle{#2}\else% angle
    \def\n@mref{f}\ifx\l@debut\n@mref\edef\defaultarrowheadangle{#2}\else% fillmode
    \def\n@mref{l}\ifx\l@debut\n@mref\y@tiunit{#2}\ifunitpr@sent%
     \edef\defaulth@rdahlength{#2}\else\edef\defaulth@rdahlength{#2pt}%
     \message{*** \BS@ pssetdefault (..., #1=#2, ...) : unit is missing, pt is assumed.}%
     \fi\else% length
    \def\n@mref{o}\ifx\l@debut\n@mref\edef\defaultarrowheadout{#2}\else% out
    \def\n@mref{r}\ifx\l@debut\n@mref\edef\defaultarrowheadratio{#2}\else% ratio
    \immediate\write16{*** Unknown attribute: \BS@ pssetdefault arrowhead(..., #1=...)}%
    \fi\fi\fi\fi\fi}
\ctr@ld@f\def\Pssd@c@rve#1=#2|{\keln@mun#1|%
    \def\n@mref{r}\ifx\l@debut\n@mref\edef\defaultroundness{#2}\else%
    \immediate\write16{*** Unknown attribute: \BS@ pssetdefault curve(..., #1=...)}%
    \fi}
\ctr@ld@f\def\Pssd@fl@wchart#1=#2|{\keln@mtr#1|%
    \def\n@mref{arr}\ifx\l@debut\n@mref\expandafter\keln@mtr\l@suite|%
     \def\n@mref{owp}\ifx\l@debut\n@mref\edef\defaultfcarrowposition{#2}\else% arrowposition
     \def\n@mref{owr}\ifx\l@debut\n@mref\edef\defaultfcarrowrefpt{#2}\else% arrowrefpt
     \immediate\write16{*** Unknown attribute: \BS@ pssetdefault flowchart(..., #1=...)}%
     \fi\fi\else%
    \def\n@mref{lin}\ifx\l@debut\n@mref\edef\defaultfcline{#2}\else% line
    \def\n@mref{pad}\ifx\l@debut\n@mref\edef\defaultfcxpadding{#2}%
                    \edef\defaultfcypadding{#2}\else% padding
    \def\n@mref{rad}\ifx\l@debut\n@mref\edef\defaultfcradius{#2}\else% connection radius
    \def\n@mref{sha}\ifx\l@debut\n@mref\edef\defaultfcshape{#2}\else% shape
    \def\n@mref{thi}\ifx\l@debut\n@mref\edef\defaultfcthickness{#2}\else% thickness
    \def\n@mref{xpa}\ifx\l@debut\n@mref\edef\defaultfcxpadding{#2}\else% xpadding
    \def\n@mref{ypa}\ifx\l@debut\n@mref\edef\defaultfcypadding{#2}\else% ypadding
    \immediate\write16{*** Unknown attribute: \BS@ pssetdefault flowchart(..., #1=...)}%
    \fi\fi\fi\fi\fi\fi\fi\fi}
\ctr@ld@f\def\defaultfcarrowposition{0.5}%\ctr@ld@f\let\defaultfcarrowpos=\defaultfcarrowposition
\ctr@ld@f\def\defaultfcarrowrefpt{start}
\ctr@ld@f\def\defaultfcline{polygon}
\ctr@ld@f\def\defaultfcradius{0}
\ctr@ld@f\def\defaultfcshape{rectangle}
\ctr@ld@f\def\defaultfcthickness{0}%\ctr@ld@f\let\defaultfcthick=\defaultfcthickness
\ctr@ld@f\def\defaultfcxpadding{0}%\ctr@ld@f\let\defaultfcxpad=\defaultfcxpadding
\ctr@ld@f\def\defaultfcypadding{0}%\ctr@ld@f\let\defaultfcypad=\defaultfcypadding
\ctr@ld@f\def\Pssd@m@sh#1=#2|{\keln@mun#1|%
    \def\n@mref{d}\ifx\l@debut\n@mref\edef\defaultmeshdiag{#2}\else%
    \immediate\write16{*** Unknown attribute: \BS@ pssetdefault mesh(..., #1=...)}%
    \fi}
\ctr@ld@f\def\Pssd@s@cond#1=#2|{\keln@mun#1|%
    \def\n@mref{c}\ifx\l@debut\n@mref\edef\defaultsecondcolor{#2}\else%
    \def\n@mref{d}\ifx\l@debut\n@mref\edef\defaultseconddash{#2}\else%
    \def\n@mref{w}\ifx\l@debut\n@mref\edef\defaultsecondwidth{#2}\else%
    \immediate\write16{*** Unknown attribute: \BS@ pssetdefault second(..., #1=...)}%
    \fi\fi\fi}
\ctr@ld@f\def\Pssd@th@rd#1=#2|{\keln@mun#1|%
    \def\n@mref{c}\ifx\l@debut\n@mref\edef\defaultthirdcolor{#2}\else%
    \immediate\write16{*** Unknown attribute: \BS@ pssetdefault third(..., #1=...)}%
    \fi}
\ctr@ln@w{newif}\iffillm@de
\ctr@ld@f\def\pssetfillmode#1{\expandafter\setfillm@de#1:}
\ctr@ld@f\def\setfillm@de#1#2:{\if#1n\fillm@defalse\else\fillm@detrue\fi}
\ctr@ld@f\def\defaultfill{no}     % Valeur par defaut
\ctr@ln@w{newif}\ifpsupdatem@de
\ctr@ld@f\def\pssetupdate#1{\ifcurr@ntPS\immediate\write16{*** \BS@ pssetupdate is ignored inside a
     \BS@ psbeginfig-\BS@ psendfig block.}%
    \immediate\write16{*** It must be called before \BS@ psbeginfig.}%
    \else\expandafter\setupd@te#1:\fi}
\ctr@ld@f\def\setupd@te#1#2:{\if#1n\psupdatem@defalse\else\psupdatem@detrue\fi}
\ctr@ld@f\def\defaultupdate{no}     % Valeur par defaut
\ctr@ln@m\curr@ntcolor \ctr@ln@m\curr@ntcolorc@md
\ctr@ld@f\def\Pssetc@lor#1{\ifps@cri\result@tent=\@ne\expandafter\c@lnbV@l#1 :%
    \def\curr@ntcolor{}\def\curr@ntcolorc@md{}%
    \ifcase\result@tent\or\pssetgray{#1}\or\or\pssetrgb{#1}\or\pssetcmyk{#1}\fi\fi}
\ctr@ln@m\curr@ntcolorc@mdStroke
\ctr@ld@f\def\pssetcmyk#1{\ifps@cri\def\curr@ntcolor{#1}\def\curr@ntcolorc@md{\c@msetcmykcolor}%
    \def\curr@ntcolorc@mdStroke{\c@msetcmykcolorStroke}%
    \ifcurr@ntPS\PSc@mment{pssetcmyk Color=#1}\us@primarC@lor\fi\fi}
\ctr@ld@f\def\pssetrgb#1{\ifps@cri\def\curr@ntcolor{#1}\def\curr@ntcolorc@md{\c@msetrgbcolor}%
    \def\curr@ntcolorc@mdStroke{\c@msetrgbcolorStroke}%
    \ifcurr@ntPS\PSc@mment{pssetrgb Color=#1}\us@primarC@lor\fi\fi}
\ctr@ld@f\def\pssetgray#1{\ifps@cri\def\curr@ntcolor{#1}\def\curr@ntcolorc@md{\c@msetgray}%
    \def\curr@ntcolorc@mdStroke{\c@msetgrayStroke}%
    \ifcurr@ntPS\PSc@mment{pssetgray Gray level=#1}\us@primarC@lor\fi\fi}
\ctr@ln@m\fillc@md
\ctr@ld@f\def\us@primarC@lor{\immediate\write\fwf@g{\d@fprimarC@lor}%
    \let\fillc@md=\prfillc@md}
\ctr@ld@f\def\prfillc@md{\d@fprimarC@lor\space\c@mfill}
\ctr@ld@f\def\defaultcolor{0}       % Valeur par defaut
\ctr@ld@f\def\c@lnbV@l#1 #2:{\def\t@xt@{#1}\relax\ifx\t@xt@\empty\c@lnbV@l#2:% Discard leading spaces
    \else\c@lnbV@l@#1 #2:\fi}
\ctr@ld@f\def\c@lnbV@l@#1 #2:{\def\t@xt@{#2}\ifx\t@xt@\empty%
    \def\t@xt@{#1}\ifx\t@xt@\empty\advance\result@tent\m@ne\fi% Discard trailing spaces
    \else\advance\result@tent\@ne\c@lnbV@l@#2:\fi}
\ctr@ld@f\def\Blackcmyk{0 0 0 1}
\ctr@ld@f\def\Whitecmyk{0 0 0 0}
\ctr@ld@f\def\Cyancmyk{1 0 0 0}
\ctr@ld@f\def\Magentacmyk{0 1 0 0}
\ctr@ld@f\def\Yellowcmyk{0 0 1 0}
\ctr@ld@f\def\Redcmyk{0 1 1 0}
\ctr@ld@f\def\Greencmyk{1 0 1 0}
\ctr@ld@f\def\Bluecmyk{1 1 0 0}
\ctr@ld@f\def\Graycmyk{0 0 0 0.50}
\ctr@ld@f\def\BrickRedcmyk{0 0.89 0.94 0.28} % PANTONE 1805
\ctr@ld@f\def\Browncmyk{0 0.81 1 0.60} % PANTONE 1615
\ctr@ld@f\def\ForestGreencmyk{0.91 0 0.88 0.12} % PANTONE 349
\ctr@ld@f\def\Goldenrodcmyk{ 0 0.10 0.84 0} % PANTONE 109
\ctr@ld@f\def\Marooncmyk{0 0.87 0.68 0.32} % PANTONE 201
\ctr@ld@f\def\Orangecmyk{0 0.61 0.87 0} % PANTONE ORANGE-021
\ctr@ld@f\def\Purplecmyk{0.45 0.86 0 0} % PANTONE PURPLE
\ctr@ld@f\def\RoyalBluecmyk{1. 0.50 0 0} % No PANTONE match
\ctr@ld@f\def\Violetcmyk{0.79 0.88 0 0} % PANTONE VIOLET
\ctr@ld@f\def\Blackrgb{0 0 0}
\ctr@ld@f\def\Whitergb{1 1 1}
\ctr@ld@f\def\Redrgb{1 0 0}
\ctr@ld@f\def\Greenrgb{0 1 0}
\ctr@ld@f\def\Bluergb{0 0 1}
\ctr@ld@f\def\Cyanrgb{0 1 1}
\ctr@ld@f\def\Magentargb{1 0 1}
\ctr@ld@f\def\Yellowrgb{1 1 0}
\ctr@ld@f\def\Grayrgb{0.5 0.5 0.5}
\ctr@ld@f\def\Chocolatergb{0.824 0.412 0.118}
\ctr@ld@f\def\DarkGoldenrodrgb{0.722 0.525 0.043}
\ctr@ld@f\def\DarkOrangergb{1 0.549 0}
\ctr@ld@f\def\Firebrickrgb{0.698 0.133 0.133}
\ctr@ld@f\def\ForestGreenrgb{0.133 0.545 0.133}
\ctr@ld@f\def\Goldrgb{1 0.843 0}
\ctr@ld@f\def\HotPinkrgb{1 0.412 0.706}
\ctr@ld@f\def\Maroonrgb{0.690 0.188 0.376}
\ctr@ld@f\def\Pinkrgb{1 0.753 0.796}
\ctr@ld@f\def\RoyalBluergb{0.255 0.412 0.882}
\ctr@ld@f\def\Pssetf@rst#1=#2|{\keln@mun#1|%
    \def\n@mref{c}\ifx\l@debut\n@mref\Pssetc@lor{#2}\else% color
    \def\n@mref{d}\ifx\l@debut\n@mref\pssetdash{#2}\else% dash
    \def\n@mref{f}\ifx\l@debut\n@mref\pssetfillmode{#2}\else% fillmode
    \def\n@mref{j}\ifx\l@debut\n@mref\pssetjoin{#2}\else% line join
    \def\n@mref{u}\ifx\l@debut\n@mref\pssetupdate{#2}\else% update
    \def\n@mref{w}\ifx\l@debut\n@mref\pssetwidth{#2}\else% line width
    \immediate\write16{*** Unknown attribute: \BS@ psset (..., #1=...)}%
    \fi\fi\fi\fi\fi\fi}
\ctr@ln@m\curr@ntdash
\ctr@ld@f\def\s@uvdash#1{\edef#1{\curr@ntdash}}
\ctr@ld@f\def\defaultdash{1}        % Valeur par defaut (numero sans espace)
\ctr@ld@f\def\pssetdash#1{\ifps@cri\edef\curr@ntdash{#1}\ifcurr@ntPS\expandafter\Pssetd@sh#1 :\fi\fi}
\ctr@ld@f\def\Pssetd@shI#1{\PSc@mment{pssetdash Index=#1}\ifcase#1%
    \or\immediate\write\fwf@g{[] 0 \c@msetdash}%         Index=1
    \or\immediate\write\fwf@g{[6 2] 0 \c@msetdash}%      Index=2
    \or\immediate\write\fwf@g{[4 2] 0 \c@msetdash}%      Index=3
    \or\immediate\write\fwf@g{[2 2] 0 \c@msetdash}%      Index=4
    \or\immediate\write\fwf@g{[1 2] 0 \c@msetdash}%      Index=5
    \or\immediate\write\fwf@g{[2 4] 0 \c@msetdash}%      Index=6
    \or\immediate\write\fwf@g{[3 5] 0 \c@msetdash}%      Index=7
    \or\immediate\write\fwf@g{[3 3] 0 \c@msetdash}%      Index=8
    \or\immediate\write\fwf@g{[3 5 1 5] 0 \c@msetdash}%  Index=9
    \or\immediate\write\fwf@g{[6 4 2 4] 0 \c@msetdash}%  Index=10
    \fi}
\ctr@ld@f\def\Pssetd@sh#1 #2:{{\def\t@xt@{#1}\ifx\t@xt@\empty\Pssetd@sh#2:% Discard leading spaces
    \else\def\t@xt@{#2}\ifx\t@xt@\empty\Pssetd@shI{#1}\else\s@mme=\@ne\def\debutp@t{#1}%
    \an@lysd@sh#2:\ifodd\s@mme\edef\debutp@t{\debutp@t\space\finp@t}\def\finp@t{0}\fi%
    \PSc@mment{pssetdash Pattern=#1 #2}%
    \immediate\write\fwf@g{[\debutp@t] \finp@t\space\c@msetdash}\fi\fi}}
\ctr@ld@f\def\an@lysd@sh#1 #2:{\def\t@xt@{#2}\ifx\t@xt@\empty\def\finp@t{#1}\else%
    \edef\debutp@t{\debutp@t\space#1}\advance\s@mme\@ne\an@lysd@sh#2:\fi}
\ctr@ln@m\curr@ntwidth
\ctr@ld@f\def\s@uvwidth#1{\edef#1{\curr@ntwidth}}
\ctr@ld@f\def\defaultwidth{0.4}     % Valeur par defaut
\ctr@ld@f\def\pssetwidth#1{\ifps@cri\edef\curr@ntwidth{#1}\ifcurr@ntPS%
    \PSc@mment{pssetwidth Width=#1}\immediate\write\fwf@g{#1 \c@msetlinewidth}\fi\fi}
\ctr@ln@m\curr@ntjoin
\ctr@ld@f\def\pssetjoin#1{\ifps@cri\edef\curr@ntjoin{#1}\ifcurr@ntPS\expandafter\Pssetj@in#1:\fi\fi}
\ctr@ld@f\def\Pssetj@in#1#2:{\PSc@mment{pssetjoin join=#1}%
    \if#1r\def\t@xt@{1}\else\if#1b\def\t@xt@{2}\else\def\t@xt@{0}\fi\fi%
    \immediate\write\fwf@g{\t@xt@\space\c@msetlinejoin}}
\ctr@ld@f\def\defaultjoin{miter}   % Valeur par defaut
\ctr@ld@f\def\Pssets@cond#1=#2|{\keln@mun#1|%
    \def\n@mref{c}\ifx\l@debut\n@mref\Pssets@condcolor{#2}\else%
    \def\n@mref{d}\ifx\l@debut\n@mref\pssetseconddash{#2}\else%
    \def\n@mref{w}\ifx\l@debut\n@mref\pssetsecondwidth{#2}\else%
    \immediate\write16{*** Unknown attribute: \BS@ psset second(..., #1=...)}%
    \fi\fi\fi}
\ctr@ln@m\curr@ntseconddash
\ctr@ld@f\def\pssetseconddash#1{\edef\curr@ntseconddash{#1}}
\ctr@ld@f\def\defaultseconddash{4}  % Valeur par defaut (numero sans espace)
\ctr@ln@m\curr@ntsecondwidth
\ctr@ld@f\def\pssetsecondwidth#1{\edef\curr@ntsecondwidth{#1}}
\ctr@ld@f\edef\defaultsecondwidth{\defaultwidth} % Valeur par defaut
\ctr@ld@f\def\psresetsecondsettings{%
    \pssetseconddash{\defaultseconddash}\pssetsecondwidth{\defaultsecondwidth}%
    \Pssets@condcolor{\defaultsecondcolor}}
\ctr@ln@m\sec@ndcolor \ctr@ln@m\sec@ndcolorc@md
\ctr@ld@f\def\Pssets@condcolor#1{\ifps@cri\result@tent=\@ne\expandafter\c@lnbV@l#1 :%
    \def\sec@ndcolor{}\def\sec@ndcolorc@md{}%
    \ifcase\result@tent\or\pssetsecondgray{#1}\or\or\pssetsecondrgb{#1}%
    \or\pssetsecondcmyk{#1}\fi\fi}
\ctr@ln@m\sec@ndcolorc@mdStroke
\ctr@ld@f\def\pssetsecondcmyk#1{\def\sec@ndcolor{#1}\def\sec@ndcolorc@md{\c@msetcmykcolor}%
    \def\sec@ndcolorc@mdStroke{\c@msetcmykcolorStroke}}
\ctr@ld@f\def\pssetsecondrgb#1{\def\sec@ndcolor{#1}\def\sec@ndcolorc@md{\c@msetrgbcolor}%
    \def\sec@ndcolorc@mdStroke{\c@msetrgbcolorStroke}}
\ctr@ld@f\def\pssetsecondgray#1{\def\sec@ndcolor{#1}\def\sec@ndcolorc@md{\c@msetgray}%
    \def\sec@ndcolorc@mdStroke{\c@msetgrayStroke}}
\ctr@ld@f\def\us@secondC@lor{\immediate\write\fwf@g{\d@fsecondC@lor}%
    \let\fillc@md=\sdfillc@md}
\ctr@ld@f\def\sdfillc@md{\d@fsecondC@lor\space\c@mfill}
\ctr@ld@f\edef\defaultsecondcolor{\defaultcolor} % Valeur par defaut
\ctr@ld@f\def\Pss@tsecondSt{%
    \s@uvdash{\typ@dash}\pssetdash{\curr@ntseconddash}%
    \s@uvwidth{\typ@width}\pssetwidth{\curr@ntsecondwidth}\us@secondC@lor}
\ctr@ld@f\def\Psrest@reSt{\pssetwidth{\typ@width}\pssetdash{\typ@dash}\us@primarC@lor}
\ctr@ld@f\def\Pssetth@rd#1=#2|{\keln@mun#1|%
    \def\n@mref{c}\ifx\l@debut\n@mref\Pssetth@rdcolor{#2}\else%
    \immediate\write16{*** Unknown attribute: \BS@ psset third(..., #1=...)}%
    \fi}
\ctr@ln@m\th@rdcolor \ctr@ln@m\th@rdcolorc@md
\ctr@ld@f\def\Pssetth@rdcolor#1{\ifps@cri\result@tent=\@ne\expandafter\c@lnbV@l#1 :%
    \def\th@rdcolor{}\def\th@rdcolorc@md{}%
    \ifcase\result@tent\or\Pssetth@rdgray{#1}\or\or\Pssetth@rdrgb{#1}%
    \or\Pssetth@rdcmyk{#1}\fi\fi}
\ctr@ln@m\th@rdcolorc@mdStroke
\ctr@ld@f\def\Pssetth@rdcmyk#1{\def\th@rdcolor{#1}\def\th@rdcolorc@md{\c@msetcmykcolor}%
    \def\th@rdcolorc@mdStroke{\c@msetcmykcolorStroke}}
\ctr@ld@f\def\Pssetth@rdrgb#1{\def\th@rdcolor{#1}\def\th@rdcolorc@md{\c@msetrgbcolor}%
    \def\th@rdcolorc@mdStroke{\c@msetrgbcolorStroke}}
\ctr@ld@f\def\Pssetth@rdgray#1{\def\th@rdcolor{#1}\def\th@rdcolorc@md{\c@msetgray}%
    \def\th@rdcolorc@mdStroke{\c@msetgrayStroke}}
\ctr@ld@f\def\us@thirdC@lor{\immediate\write\fwf@g{\d@fthirdC@lor}%
    \let\fillc@md=\thfillc@md}
\ctr@ld@f\def\thfillc@md{\d@fthirdC@lor\space\c@mfill}
\ctr@ld@f\def\defaultthirdcolor{1}  % Valeur par defaut
\ctr@ld@f\def\pstrimesh#1[#2,#3,#4]{{\ifcurr@ntPS\ifps@cri%
    \PSc@mment{pstrimesh Type=#1, Triangle=[#2,#3,#4]}%
    \s@uvc@ntr@l\et@tpstrimesh\ifnum#1>\@ne\Pss@tsecondSt\setc@ntr@l{2}%
    \Pstrimeshp@rt#1[#2,#3,#4]\Pstrimeshp@rt#1[#3,#4,#2]%
    \Pstrimeshp@rt#1[#4,#2,#3]\Psrest@reSt\fi\psline[#2,#3,#4,#2]%
    \PSc@mment{End pstrimesh}\resetc@ntr@l\et@tpstrimesh\fi\fi}}
\ctr@ld@f\def\Pstrimeshp@rt#1[#2,#3,#4]{{\l@mbd@un=\@ne\l@mbd@de=#1\loop\ifnum\l@mbd@de>\@ne%
    \advance\l@mbd@de\m@ne\figptbary-1:[#2,#3;\l@mbd@de,\l@mbd@un]%
    \figptbary-2:[#2,#4;\l@mbd@de,\l@mbd@un]\psline[-1,-2]%
    \advance\l@mbd@un\@ne\repeat}}
\initpr@lim\initpss@ttings\initPDF@rDVI% Initialisation preliminaire
\ctr@ln@w{newbox}\figBoxA
\ctr@ln@w{newbox}\figBoxB
\ctr@ln@w{newbox}\figBoxC
\catcode`\@=12

\pssetdefault(update=yes)
\newbox\figbox
\def\courb#1#2#3#4#5{
    \figptssym 102 = #2 /#1, #3 /
    \figgetdist\rayon[#1,#2]
    \figvectN 111 [#2,#1]
    \figvectU 112 [111]
    \figptstra 104 = #2 /#4, 112/
    \figptssym 105 = 104 /#1, #3 /
    \psset(color=#5)
    \psset(width=1.2)
    \psarccircP #1 ; \rayon [#2,102]   
    \psline[#2,104]
    \psline[102,105]
}

\def\CoulF{0.5 .25 1.}
\def\CoulA{0. 0. 0.}
\def\CoulB{1. 0. 0.}
\def\CoulC{0. 0. 1.}
\def\CoulD{1. 1. 1.}
%%%%%%%%%%%%%%%--fig4tex--%%%%%%%%%%%%%%%

%%-----------------------------
%%      the top matter
%%-----------------------------
\title{Quantum waveguides with corners}

\author{Monique Dauge, \ Yvon Lafranche, \ Nicolas Raymond}
\address{Labortatoire IRMAR, UMR 6625 du CNRS,
Campus de Beaulieu
35042 Rennes cedex, France}

\begin{abstract} 
The simplest modeling of planar quantum waveguides is the Dirichlet eigenproblem for the Laplace operator in unbounded open sets which are uniformly thin in one direction. Here we consider V-shaped guides. Their spectral properties depend essentially on a sole parameter, the opening of the V. The free energy band is a semi-infinite interval bounded from below. As soon as the V is not flat, there are bound states below the free energy band. There are a finite number of them, depending on the opening. This number tends to infinity as the opening tends to $0$ (sharply bent V). In this situation, the eigenfunctions concentrate and become self-similar. In contrast, when the opening gets large (almost flat V), the eigenfunctions spread and enjoy a different self-similar structure. We explain all these facts and illustrate them by numerical simulations.
\end{abstract}

\maketitle

%%-----------------------------
%%      your text
%%-----------------------------
\section*{Introduction}
A quantum waveguide refer to nanoscale electronic device with a wire or thin surface shape. In the first case, one speaks of a quantum wire. The electronic density is low enough to allow a modeling of the system by a simple one-body Schr\"odinger operator with potential
\[
   \psi\ \longmapsto\ -\Delta\psi + V\psi\quad\mbox{in}\quad\xR^3.
\]
The structure of the device causes the potential to be very large outside and very small inside the device.  As a relevant approximation, we can consider that the potential is zero in the device and infinite outside ; this can be described by a Dirichlet operator
\[
   \psi\ \longmapsto\ -\Delta\psi \ \ \mbox{in}\ \ \Omega 
   \quad\mbox{and}\quad\psi=0 \ \ \mbox{on}\ \ \partial\Omega 
\]
where $\Omega$ is the open set filled  by the device. We refer to \cite{BDPR} where we can see, at least on the numerical simulations, the analogy between a problem with a confining potential and a Dirichlet condition.

These kinds of device are intended to drive electronic fluxes. But their shape may capture some bound states, \ie{} eigenpairs of the Dirichlet problem:
\[
   -\Delta\psi=\lambda\psi \ \ \mbox{in}\ \ \Omega 
   \quad\mbox{and}\quad\psi=0 \ \ \mbox{on}\ \ \partial\Omega .
\]
The topic of this paper is two-dimensional wire shaped structures, \ie{} structures which coincide with strips of the form $\xR_+\times(0,\alpha)$ outside a ball of center ${\bf 0}$ and radius $R$ large enough. These structures can be called \emph{planar waveguides}. More specifically, there are the \emph{bent waveguides} and the \emph{broken waveguides}: Bent waveguides have a constant width around some central smooth curve, see Figure \ref{F:0a}, and the central curve of a broken waveguide is a broken line,  see Figure \ref{F:0b}.

\begin{figure}
\begin{minipage}{0.4\textwidth}
% 1. Definition of characteristic points
    \figinit{1mm}
    \figpt 1:( 0,  0)
    \figpt 2:( -5, 10)
    \figpt 3:( 50, 0)
    \figptshom 12 = 2 /1, 0.75/
    \figptshom 22 = 2 /1, 1.25/
% 2. Creation of the postscript file
    \def\MyPSfile{}
    \psbeginfig{\MyPSfile}
    \courb{1}{2}{3}{20}{\CoulB}
    \courb{1}{12}{3}{20}{\CoulA}
    \courb{1}{22}{3}{20}{\CoulA}
    \psendfig

% 3. Writing text on the figure
    \figvisu{\figbox}{}{%
    \figinsert{\MyPSfile}
    \figsetmark{$\bullet$}
    \figwritee 1 :{${\bf 0}$}(2)
        }
    \centerline{\box\figbox}
%    \caption{Guide courbe}
    \caption{Curved guide}
\label{F:0a}
\end{minipage}
\begin{minipage}{0.4\textwidth}
% 1. Definition of characteristic points
    \figinit{1mm}
    \figpt 1:( 0,  0)
    \figpt 2:( -0.05, 0.1)
    \figpt 3:( 50, 0)
    \figpt 4:( 25, 0)
    \figvectC 30 (1,0)
    \figptstra 11 = 1,2,3 /6, 30/
    \figptstra 21 = 1,2,3 /12, 30/
% 2. Creation of the postscript file
    \def\MyPSfile{}
    \psbeginfig{\MyPSfile}
    \courb{1}{2}{3}{48}{\CoulA}
    \courb{11}{12}{13}{48}{\CoulB}
    \courb{21}{22}{23}{48}{\CoulA}
    \psendfig

% 3. Writing text on the figure
    \figvisu{\figbox}{}{%
    \figinsert{\MyPSfile}
    \figsetmark{$\bullet$}
    \figwritee 4 :{${\bf 0}$}(2)
    }
    \centerline{\box\figbox}
    \caption{Broken guide}
\label{F:0b}
\end{minipage}
\end{figure}
Due to the semi-infinite strips contained in such waveguide, the spectrum of the Laplacian $-\Delta$ with Dirichlet conditions is not discrete: It contains a semi-infinite interval of the form $[\mu,+\infty)$ which is the energy band where electronic transport can occur. The presence of discrete spectrum at lower energy levels is not obvious, but nevertheless, frequent.

A remarkable result by Duclos and Exner \cite{Duclos95} (and generalized in \cite{Duclos05}) tells us that if the mid-line of a planar waveguide is smooth and straight outside a compact set, then there exists bound states as soon as the line is not straight everywhere. For broken guides, a similar result holds, \cite{Exner89,ABGM91}: There exist bound states as soon as the guide is not a straight strip. The question of the increasing number of such bound states for sharply bent broken waveguides has been answered in \cite{CLMM93}.

Before stating the main results of this paper, let us mention a few other works involving broken waveguides such as  \cite{FriSolo08, FriSolo09}. It also turns out that our analysis leads to the investigation of triangles with a sharp angle as in \cite{Fre07} (see also \cite{BoFre10}). The generalization  to dimension 3 of the broken strip arises in \cite{ExTa10} where a conical waveguide is studied, whereas a Born-Oppenheimer approach is in progress in \cite{Ourm11}.

In this paper, we revisit the results of \cite{ABGM91, CLMM93} and prove several other quantitative or qualitative properties of the eigenpairs of planar broken waveguides. Here are the contents of the present work: In section~\ref{s:1} we recall from the literature notions of unbounded self-adjoint operators, discrete and essential spectrum, Rayleigh quotients. After proving that the Rayleigh quotients are increasing functions of the opening angle $\theta$ of the guide (section~\ref{s:3}), we adapt the technique of \cite{Duclos05} to give a self-contained proof of the existence of bound states (section~\ref{s:4}). We prove that the number of bound states is always finite, though depending on the opening angle $\theta$ (section~\ref{s:5}), that this number tends to infinity like the inverse $\theta^{-1}$ of the opening angle when $\theta\to0$ (section~\ref{s:7}).
Concerning eigenvectors we prove that they satisfy an even symmetry property with respect to the symmetry axis of the guide (section~\ref{s:2}). Their decay along the semi-infinite straight parts of the guide can be precisely evaluated (section~\ref{s:6}) and is stronger and stronger when $\theta$ decreases. We perform numerical computations by the finite element method, which clearly illustrate these decay properties. When the opening gets small, concentration and self-similarity appears, which can be explained by a semi-classical analysis: We give some overview of the asymptotic expansions established in our other paper \cite{DauRay11}. We end this work by evaluations of the numerical convergence of the algorithms used for our finite element computations (section~\ref{s:9}).

\medskip\noindent{\sc Notation.} \
The $\xLtwo$ norm on an open set $\mathcal{U}$ will be denoted by $\|\cdot\|_{\mathcal{U}}$.

%%-----------------------------
%%      section 1
%%-----------------------------

\section{Unbounded self-adjoint operators} 
\label{s:1}
In this section we recall from the literature some definitions and fundamental facts on unbounded self-adjoint operators and their spectrum. We quote the standard book of Reed and Simon \cite[Chapter VIII]{ReSi1} and, for the readers who can read french, the book of L\'evy-Bruhl \cite{L-B03} (see in particular Chapter 10 on unbounded operators).

Let $H$ be a separable Hilbert space with scalar product $\langle\cdot,\cdot\rangle_H$. We will consider operators $A$ defined on a dense subspace $\Dom(A)$ of $H$ called the \emph{domain of $A$}. The adjoint $A^*$ of $A$ is the operator defined as follows:
\begin{enumerate}
\item[\iti1] The domain $\Dom(A^*)$ is the space of the elements $u$ of $H$ such that the form
\[
   \Dom(A)\ni v\mapsto \langle u,Av\rangle_H
\] 
can be extended to a continuous form on $H$.
\item[\iti2] For any $u\in\Dom(A^*)$, $A^*u$ is the unique element of $H$ provided by the Riesz theorem such that
\begin{equation}\label{DomA*}
   \forall v\in\Dom(A),\quad
   \langle A^*u,v\rangle_H = \langle u,Av\rangle_H
\end{equation}
\end{enumerate}

\begin{dfntn}
The operator $A$ with domain $\Dom(A)$ is said \emph{self-adjoint} if $A=A^*$, which means that
\begin{equation}\label{Sym}
   \Dom(A^*)=\Dom(A) \quad\mbox{and}\quad
   \forall u,v\in\Dom(A),\quad
   \langle Au,v\rangle_H = \langle u,Av\rangle_H.
\end{equation}
\end{dfntn}

\subsection{Operators in variational form}
The operators that we will consider can be defined by variational formulation. Let us introduce the general framework first. Let be given two separable Hilbert spaces $H$ and $V$ with continuous embedding of $V$ into $H$ and such that $V$ is dense in $H$. Let $b$ be an hermitian sesquilinear form on $V$
\[
   b: V\times V \ni (u,v) \mapsto b(u,v)\in\xC
\]
which is assumed to be continuous and coercive: This means that there exist three real numbers $c$, $C$ and $\Lambda$ such that
\begin{equation}
\label{eq:1-3}
   \forall u\in V,\quad
   c\|u\|^2 _V\le b(u,u) + \Lambda\langle u,u\rangle_H \le C\|u\|^2_V\,.
\end{equation}

Let $\unA$ be the operator defined from $V$ into its dual $V'$ by the natural expression
\[
   \forall v\in V,\quad \langle \unA u,v\rangle_H = b(u,v).
\]
In other words, for all $u\in V$, $\unA u$ is the linear form $v\mapsto b(u,v)$. 
Note that the operator $\unA+\Lambda\,\Id$ is associated in the same way to the sesquilinear form
\[
   b + \Lambda \langle \cdot,\cdot\rangle_H : V\times V \ni (u,v) \mapsto 
   b(u,v) + \Lambda\langle u,v\rangle_H\in\xC 
\]
which is strongly elliptic by \eqref{eq:1-3}. As a consequence of the Riesz theorem (or the more general Lax-Milgram theorem) there holds 
\begin{equation}
\label{eq:1-4}
   \unA+\Lambda\,\Id \quad\mbox{is an isomorphism from}\ \ V \ \ \mbox{onto}\ \ V'.
\end{equation}

This situation provides many examples of (unbounded) self-adjoint operators.  The following lemma is related to the Friedrichs's lemma \cf{} \cite[Section 10.7]{L-B03}.

\begin{lmm}
\label{lem:1-1}
Let $\unA$ be the operator associated with an hermitian sesquilinear form $b$ coercive on $V$. Let $A$ be the restriction of the operator $\unA$ on the domain
\[
   \Dom(A) = \{u\in V:\quad \unA u \in H\}.
\]
Then $A$ is self-adjoint.
\end{lmm}

\begin{proof}
The operator $A$ is symmetric because $b$ is hermitian. Moreover we check immediately that
\[
   \forall u,v\in \Dom(A),\quad \langle A u,v\rangle_H = \langle u,Av\rangle_H = b(u,v).
\] 
In particular, we deduce   $\Dom(A)\subset\Dom(A^*)$.
Let us prove that $\Dom(A^*)\subset\Dom(A)$.
Since the domain of $A$ and $A^*$ are unchanged by the addition of $\Lambda\,\Id$, and since
\[
   (A+\Lambda\,\Id)^*=A^*+\Lambda\,\Id
\]
we can consider $A+\Lambda\,\Id$ instead of $A$, or, in other words, using \eqref{eq:1-4}, assume that $\unA$ is bijective.

Then we deduce that $A$ is an isomorphism from $\Dom(A)$ onto $H$: Indeed, $A$ is injective because $\unA$ is injective; If $f\in H$, there exists $u\in V$ such that $\unA u = f$, and $u\in\Dom(A)$ by the very definition of $\Dom(A)$.

Let $w$ belong to $\Dom(A^*)$. This means, \cf{} \eqref{Sym}, that $w\in H$ and $A^*w=f\in H$. Let us prove that $w$ belongs to $\Dom(A)$.

Since $A$ is bijective there exists $u\in \Dom(A)$ such that $A u = f$. We have
\[
   \forall v\in V,\quad b(u,v) = \langle f,v\rangle_H  
\] 
and therefore 
\[
   \forall v\in \Dom(A),\quad \langle u,Av\rangle_H = \langle w,Av\rangle_H.
\]
Hence, as $A$ is bijective, for all $g\in H$, $\langle u,g\rangle_H = \langle w,g\rangle_H$. Finally $u=w$, which ends the proof.
\end{proof}

\begin{xmpl}
\label{ex:1-1}
Let $\Omega$ be an open set in $\xR^n$ and $\partial_\Dir\Omega$ a part of its boundary. On
\[
   V = \{\psi\in \rH^1(\Omega):\quad \psi=0 \ \ \mbox{on}\ \ \partial_\Dir\Omega\}
\]
we consider the bilinear form
\[
   b(\psi,\psi') = \int_\Omega \nabla\psi(\bx) \cdot \nabla\psi'(\bx)\, \xdif\bx.
\]
The operator $A$ is equal to $-\Delta$ and it is self-adjoint on $H=\xLtwo(\Omega)$ with domain 
\[
   \Dom(A) = \{\psi\in V:\quad \Delta \psi\in \xLtwo(\Omega)\quad\mbox{and}\quad 
   \partial_n\psi=0 \ \ \mbox{on}\ \ \partial\Omega\setminus\partial_\Dir\Omega\}.
\]
\end{xmpl}

\subsection{Discrete and essential spectrum}

Let $A$ be an unbounded self-adjoint operator on $H$ with domain $\Dom(A)$.
We recall the following characterizations of its spectrum $\sigma(A)$, its essential spectrum $\sigma_\ess(A)$ and its discrete spectrum $\sigma_\dis(A)$:
\begin{itemize}
\item Spectrum: $\lambda\in\sigma(A)$ if and only if $(A-\lambda \,\Id)$ is not invertible from $\Dom(A)$ onto $H$, 
\item Essential spectrum: $\lambda\in\sigma_\ess(A)$  if and only if  $(A-\lambda \,\Id)$ is not Fredholm\footnote{We recall that an operator is said to be \emph{Fredholm} if its kernel is finite dimensional, its range is closed and with finite codimension.} from $\Dom(A)$ into $H$ (see \cite[Chapter VI]{ReSi1} and \cite[Chapter 3]{L-B03}),
\item Discrete spectrum: $\sigma_\dis(A) := \sigma(A)\setminus\sigma_\ess(A)$.
\end{itemize}

\medskip
We list now several fundamental properties of essential and discrete spectrum.

%\begin{lmm}[{\cite[Theorem VII.4]{ReSi1}}]
%The essential spectrum is the union of three sets of points:
%\begin{enumerate}
%\item Eigenvalues of infinite multiplicity,
%\item Accumulation points of eigenvalues,
%\item Continuous spectrum.
%\end{enumerate}
%\end{lmm}

\begin{lmm}[Weyl criterion]\label{Weyl}
We have $\la\in\sigma_{\ess}(A)$ if and only if there exists a sequence $(u_{n})\in \Dom(A)$ such that $\|u_{n}\|_{H}=1$, $(u_{n})$ has no subsequence converging in $H$ and $(A-\la \,\Id)u_{n}\underset{n\to+\infty}{\to }0$ in $H$.
\end{lmm}

From this lemma, one can deduce (see \cite[Proposition 2.21 and Proposition 3.11]{L-B03}):
\begin{lmm}
The discrete spectrum is formed by isolated eigenvalues of finite multiplicity.
\end{lmm}
\begin{lmm}\label{stab-ess}
The essential spectrum is stable under any perturbation which is compact from $\Dom(A)$ into $H$.
\end{lmm}

%Let us finally quote a criterion of stability of the essential spectrum \cf{} \cite[Corollary 1]{ReSi1}:
%\begin{lmm}\label{crit-stab-ess}
%If $A$ and $B$ are two self-adjoint operators such that the difference of resolvents $(A+i)^{-1}-(B+i)^{-1}$ is compact from $H$ into itself, then $\sigma_{\ess}(A)=\sigma_{\ess}(B).$
%\end{lmm}

\begin{xmpl}
\label{ex:1-5}
Let $A$ be the self-adjoint operator on $H$ associated with an hermitian sesquilinear form $b$ coercive on $V$, \cf{} Lemma \ref{lem:1-1}. Let us assume that $V$ is compactly embedded in $H$. Then the spectrum of $A$ is discrete and formed by a non-decreasing sequence $\nu_k$ of eigenvalues which tend to $+\infty$ as $k\to+\infty$ (see \cite[Chapter 13]{L-B03}). Let $(v_k)_{k\ge1}$ be an associated orthonormal basis of eigenvectors:
\[
   Av_k = \nu_k v_k,\quad \forall k\ge1.
\]
Then we have the following identities
\begin{gather}
\label{eq:nu0}
   \forall u\in H,\quad \|u\|^2_H = \sum_{k\ge1} \big|\big\langle u, v_k\big\rangle_H\big|^2, \\
\label{eq:nu1}
   \forall u\in V,\quad b(u,u) = \sum_{k\ge1} \nu_k \big|\big\langle u, v_k\big\rangle_H\big|^2, \\
\label{eq:nu2}
   \forall u\in \Dom(A),\quad \|Au\|^2_H = 
   \sum_{k\ge1} \nu_k^2\big|\big\langle u, v_k\big\rangle_H\big|^2 .
\end{gather}
\end{xmpl}

\begin{xmpl}
\label{ex:1-2}
Let us define $\Delta^\Dir_\Omega$ as the Dirichlet problem for the Laplace operator $-\Delta$ on an open set $\Omega\subset\xR^n$, \cf{} Example \ref{ex:1-1} with $\partial_\Dir\Omega = \partial\Omega$. 

\emph{(i)} If $\Omega$ is bounded, $\Delta^\Dir_\Omega$ has a purely discrete spectrum, which is an increasing sequence of positive numbers.

\emph{(ii)} Let us assume that there is a compact set $K$ such that
\[
   \Omega\setminus K = \bigcup_{j\ \rm finite} \Omega_j \quad\mbox{(disjoint union)}
\]
where $\Omega_j$ is isometrically affine to a half-tube $\Sigma_j = (0,+\infty)\times \omega_j$, with $\omega_j$ bounded open set in $\xR^{n-1}$. Let $\mu_j$ be the first eigenvalue of the Dirichlet problem for the Laplace operator $-\Delta$ on $\omega_j$. Then, we have:
\[
   \sigma_\ess(\Delta^\Dir_\Omega) = \cup_j [\mu_j,+\infty) = [\min_j \mu_j, +\infty).
\]
The proof can be organized in two main steps. Firstly, for each $j$ we construct Weyl sequences supported in $\Sigma_j$ associated with any $\lambda>\mu_{j}$, which proves that $\sigma_{\ess}(\Delta^\Dir_\Omega)\subset[\min_j \mu_j, +\infty)$. Secondly we apply Lemma \ref{stab-ess} with $A=B-C$ where $A=\Delta^\Dir_\Omega$ and $B=\Delta^\Dir_{\Omega}+W$, where $W$ is a non negative and smooth potential which is compactly supported and such that $W\geq \min_{j}\mu_{j}$ on $K$. On one hand, since $C$ is compact, we get that $\sigma_{\ess}(A)=\sigma_{\ess}(B)$. On the other hand, we notice that:
$$\int_{\Omega} \|\nabla\psi\|^2\, \xdif\bx+\int_{\Omega}W |\psi|^2\, \xdif\bx\geq \int_{\Omega} \|\nabla\psi\|^2\, \xdif\bx+\min_{j}\mu_{j}\int_{K} |\psi|^2\, \xdif\bx$$
and, using the Poincar\'e inequality with respect to the transversal variable in each strip $\Sigma_j$:
$$\int_{\Omega} \|\nabla\psi\|^2\, \xdif\bx\geq\sum_{j}\int_{\Sigma{j}} \|\nabla\psi\|^2\, \xdif\bx\geq\sum_{j}\int_{\Sigma{j}} \mu_{j} |\psi|^2\, \xdif\bx\geq\min_{j}\mu_{j}\int_{\Omega\setminus K}|\psi|^2.$$
We infer that:
$$\int_{\Omega} \|\nabla\psi\|^2\, \xdif\bx+\int_{\Omega}W |\psi|^2\, \xdif\bx\geq \min_{j}\mu_{j}\int_{\Omega}|\psi|^2 \, \xdif\bx.$$
The min-max principle provides that $\inf\sigma (B)\geq \min_{j}\mu_{j}$ so that $\inf\sigma_{\ess} (B)\geq \min_{j}\mu_{j}$, and finally we get $\displaystyle{\inf\sigma_{\ess} (A)\geq \min_{j}\mu_{j}}$.

For an example of this technique, we refer for instance to \cite[Section 3.1]{Duclos05}. 
Let us notice that the Persson's theorem provides a direct proof (see \cite{Persson60} and \cite[Appendix B]{FouHel10}).

The same formula holds even if the boundary conditions on $\partial\Omega\cap K$ are mixed Dirichlet-Neumann.
\end{xmpl}

\subsection{Rayleigh quotients}
We recall now the definition of the Rayleigh quotients of a self-adjoint operator $A$ (see \cite[Proposition 6.17 and 13.1]{L-B03}). 

\begin{dfntn}
The Rayleigh quotients associated to the self-adjoint operator $A$ on $H$ of domain $\Dom(A)$ are defined for all positive natural number $j$ by
\[
   \lambda_j = \inff_{
   \substack {u_1,\ldots,u_j\in \Dom(A)\\ 
   \mbox{\footnotesize independent}}} \ \ \sup_{u\in[u_1,\ldots,u_j]}
   \frac{\langle Au,u\rangle_H} {\langle u,u\rangle_H} .
\]
Here $[u_1,\ldots,u_j]$ denotes the subspace generated by the $j$ independent vectors $u_1,\ldots,u_j$.
\end{dfntn}

The following statement gives the relation between Rayleigh quotients and eigenvalues.

\begin{thrm}
\label{th:1-1}
Let $A$ be a self-adjoint operator of domain $\Dom(A)$. We assume that $A$ is semi-bounded from below, i.e., there exists $\Lambda\in\xR$ such that
\[
   \forall u\in\Dom(A),\quad \langle Au,u\rangle_H + \Lambda\langle u,u\rangle_H \ge 0.
\]
We set $\gamma=\min\sigma_\ess(A)$.
Then the Rayleigh quotients $\lambda_j$ of $A$ form a non-decreasing sequence and there holds
\begin{enumerate}
\item[\iti1] If $\lambda_j<\gamma$, it is an eigenvalue of $A$,
\item[\iti2] If $\lambda_j\ge\gamma$, then $\lambda_j=\gamma$,
\item[\iti3] The $j$-th eigenvalue $<\gamma$  of $A$ (if exists) coincides with $\lambda_j$.
\end{enumerate}
\end{thrm}

\begin{lmm}
\label{lem:1-2}
Let $A$ be the self-adjoint operator on $H$ associated with an hermitian sesquilinear form $b$ coercive on $V$, \cf{} Lemma \ref{lem:1-1}. Then the Rayleigh quotients of $A$ are equal to
\[
   \lambda_j = \inff_{
   \substack {u_1,\ldots,u_j\in V\\ 
   \mbox{\footnotesize independent}}} \ \ \sup_{u\in[u_1,\ldots,u_j]}
   \frac{b(u,u)} {\langle u,u\rangle_H} .
\]
\end{lmm}

\begin{crllr}
\label{co:lem1-2}
Let $A$ and $\hat A$ be the self-adjoint operators on $H$ and $\hat H$ associated with the hermitian sesquilinear forms $b$ and $\hat b$ coercive on $V$ and $\hat V$, respectively. We assume that
\[
   \hat H\subset H,\quad
   \hat V\subset V,\quad
   \hat b(u,u) \ge b(u,u) \ \forall u\in\hat V.
\]
Let $\lambda_j$ and $\hat \lambda_j$ be the Rayleigh quotients associated to $A$ and $\hat A$, respectively. Then
\[
   \forall j\ge1,\quad \hat \lambda_j \ge \lambda_j \,.
\]
\end{crllr}

\begin{xmpl}[Conforming Galerkin projection]
\label{ex:1-3}
This consists in choosing a finite dimensional subspace $\hat V$ of $V$, which also defines a subspace of $H$, and $\hat b =b$. The Rayleigh quotients $\hat \lambda_j$ are the eigenvalues of the discrete operator $\hat A$, and they are larger than the Rayleigh quotients $\lambda_j$ of $A$.
\end{xmpl}

\begin{xmpl}[Monotonicity of Dirichlet eigenvalues]
\label{ex:1-4}
Let $\Omega$ be an open subset of $\xR^n$ and $\widehat \Omega\subset\Omega$. The extension by $0$ from $\widehat\Omega$ to $\Omega$ realizes a natural embedding of $\rH^1_0(\widehat\Omega)$ into $\rH^1_0(\Omega)$, and  of $\xLtwo(\widehat\Omega)$ into $\xLtwo(\Omega)$. Then, with obvious notations:
\[
   \forall j\ge1,\quad \lambda_j(\Delta^\Dir_{\widehat\Omega}) \ge \lambda_j(\Delta^\Dir_{\Omega}) .
\]
\end{xmpl}

\begin{rmrk}[Non-monotonicity of Neumann eigenvalues]
The Neumann problem consists in taking $H^1(\Omega)$ as variational space. The argument above does not work because there is no canonical embedding of $H^1(\widehat\Omega)$ into $H^1(\Omega)$. Moreover, the monotonicity with respect to the domain is wrong: 

\noindent \emph{(i)} Let us choose $\Omega$ bounded and connected, and $\widehat \Omega$ a subset of $\Omega$ with two connected components. Then
\[
   \lambda_1(\Delta^\Neu_{\widehat\Omega}) = \lambda_2(\Delta^\Neu_{\widehat\Omega}) = 0
   \quad\mbox{and}\quad
   \lambda_1(\Delta^\Neu_{\Omega}) = 0,\ \lambda_2(\Delta^\Neu_{\Omega}) > 0.
\] 
\noindent \emph{(ii)} If we take two embedded intervals for $\widehat\Omega$ and $\Omega$, then explicit calculations show that $\lambda_j(\Delta^\Neu_{\widehat\Omega}) \ge \lambda_j(\Delta^\Neu_{\Omega})$.

\smallskip
Another nice and non trivial counter-example can be found with the de Gennes operator appearing in the superconductivity theory, \cf{} \cite{DauHel}.
\end{rmrk}

%%-----------------------------
%%      section 2
%%-----------------------------

\section{The broken guide}
\label{s:2}
Let us denote the Cartesian coordinates in $\xR^2$ by $\bx=(x_1,x_2)$.
The open sets $\Omega$ that we consider are unbounded plane V-shaped sets. The question of interest is the presence and the properties of bound states for the Laplace operator $\Delta=\partial^2_1+\partial^2_2$ with Dirichlet boundary conditions in such $\Omega$.
We can assume without loss of generality that our set $\Omega$ is normalized so that
\begin{itemize}
\item it has its non-convex corner at the origin ${\bf 0}=(0,0)$,
\item it is symmetric with respect to the $x_1$ axis,
\item its thickness is equal to $\pi$.
\end{itemize}
The sole remaining parameter is the opening of the V: We denote by $\theta\in(0,\frac\pi2)$ the half-opening and by $\Omega_\theta$ the associated broken guide, see Figure \ref{F:1}. We have
\begin{equation}
\label{eq:Ot}
   \Omega_{\theta}=\left\{(x_{1},x_{2})\in\xR^2 : \quad
   x_{1}\tan\theta<|x_{2}|<\left(x_{1}+\frac{\pi}{\sin\theta}\right)\tan\theta\right\}.
\end{equation}
Finally, like in Example \ref{ex:1-2}, we specify the \emph{positive} Dirichlet Laplacian by the notation $\Delta^\Dir_{\Omega_{\theta}}$. This operator is an unbounded self-adjoint operator with domain
\[
   \Dom(\Delta^\Dir_{\Omega_{\theta}}) = \{\psi\in\rH^1_0(\Omega_\theta):\quad
   \Delta\psi\in\xLtwo(\Omega_\theta)\}.
\]

\begin{figure}[ht]
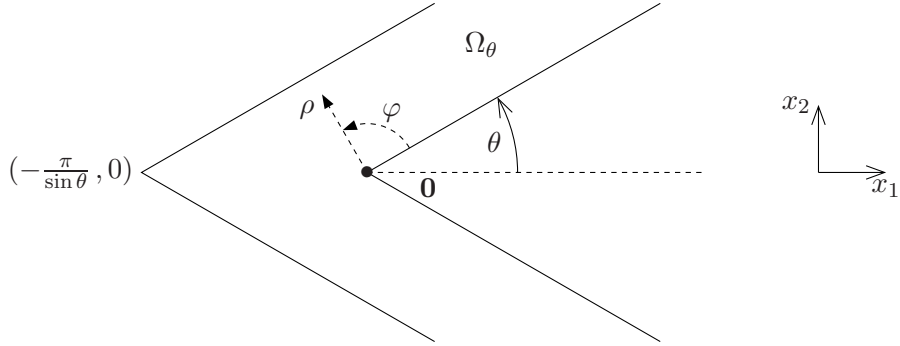

% 1. Definition of characteristic points
    \figinit{1mm}
    \figpt 1:(  0,  0)
    \figpt 2:( 45,  0)
    \figptrot 3: = 2 /1, 30/
    \figpt 4:(-30,  0)
    \figpt 5:(12, 0)
    \figptrot 6: = 5 /1, 120/
    \figvectP 101 [1,4]
    \figptstra 14 = 3/1, 101/
    \figptssym 23 = 3, 14 /1, 2/
    \figpt 10:(60,  0)

% 2. Creation of the postscript file
    \def\MyPSfile{}
    \psbeginfig{\MyPSfile}
    \psaxes 10(9)
    \figptsaxes 11 : 10(9)
    \psline[23,1,3]
    \psline[24,4,14]
    \psarrowcirc 1 ; 20 (0,30)
    \psset arrowhead (fillmode=yes, length=2)
    \psset (dash=4)
    \psarrow[1,6]
    \psarrowcirc 1 ; 6.5 (30,120)
    \psline[1,2]
  \psendfig

% 3. Writing text on the figure
    \figvisu{\figbox}{}{%
    \figinsert{\MyPSfile}
    \figwrites 11 :{$x_1$}(1)
    \figwritew 12 :{$x_2$}(1)
    \figwritegcw 4 :{$(-\frac{\pi}{\sin\theta}\,,0)$}(1,0)
    \figwritegce 1 :{$\Omega_{\theta}$}(13,17)
    \figwritegce 1 :{$\varphi$}(2,8)
    \figwritegce 1 :{$\theta$}(16,4)
    \figwritegcw 6 :{$\rho$}(1,-2)
    \figsetmark{$\bullet$}
    \figwritegce 1 :{${\bf 0}$}(7,-2.1)
    }
\centerline{\box\figbox}
\caption{The broken guide $\Omega_{\theta}$ (here $\theta=\frac\pi6$).}
\label{F:1}
\end{figure}

The boundary of $\Omega_\theta$ is not smooth, it is polygonal. The presence of the non-convex corner with vertex at the origin is the reason for the domain $\Dom(\Delta^\Dir_{\Omega_{\theta}})$ to be distinct from $\rH^2\cap\rH^1_0(\Omega_\theta)$. Nevertheless this domain can be precisely characterized as follows. Let us introduce polar coordinates $(\rho,\varphi)$ centered at the origin, with $\varphi=0$ coinciding with the upper part $x_2=x_1\tan\theta$ of the boundary of $\Omega_\theta$. Let $\chi$ be a smooth radial cutoff function with support in the region $x_{1}\tan\theta<|x_{2}|$ and $\chi\equiv1$ in a neighborhood of the origin. We introduce the explicit \emph{singular function}
\begin{equation}
\label{eq:sing}
   \psi_\sing(x_1,x_2) = \chi(\rho)\, \rho^{\pi/\omega} \sin\frac{\pi\varphi}{\omega},\quad
   \mbox{with}\quad \omega = 2(\pi-\theta).
\end{equation}
Then there holds, \cf{} the classical references \cite{Kondratev67,Grisvard85}:
\begin{equation}
\label{eq:dom}
   \Dom(\Delta^\Dir_{\Omega_{\theta}}) = \rH^2\cap\rH^1_0(\Omega_\theta) \oplus [\psi_\sing]
\end{equation}
where $[\psi_\sing]$ denotes the space generated by $\psi_\sing$.

\subsection{Essential spectrum}
\begin{prpstn}
\label{P:ess}
For any $\theta\in(0,\frac\pi2)$ the essential spectrum of the operator $\Delta^{\Dir}_{\Omega_{\theta}}$ coincides with $[1,+\infty)$. 
\end{prpstn}

\begin{proof}
This proposition is a consequence of Example \ref{ex:1-2} \emph{(ii)}: Outside a compact set, $\Omega_\theta$ is the union of two strips isometric to $(0,+\infty)\times(0,\pi)$. Since the first eigenvalue of $-\partial_y^2$ on $\rH^1_0(0,\pi)$ is $1$, the proposition is proved.
\end{proof}

\subsection{Symmetry}
In our quest of bounded states $(\lambda,\psi)$ of $\Delta^{\Dir}_{\Omega_{\theta}}$, i.e.
\begin{equation}
\label{eq:1}
   \begin{cases}
   \ -\Delta \psi = \lambda \psi &\quad \mbox{in}\quad \Omega_\theta, \\
   \qquad \psi = 0 &\quad \mbox{on}\quad \partial\Omega_\theta
   \end{cases}
   \qquad\mbox{with}\quad \lambda<1,\quad\psi\neq 0,
\end{equation}
we can reduce to the half-guide $\Omega^+_\theta$ defined as (see Figure \ref{F:2})
\[
   \Omega^+_\theta = \{(x_{1},x_{2})\in\Omega_\theta : x_2>0\},
   \qquad\mbox{with the Dirichlet part of its boundary}\ \ 
   \partial_\Dir\Omega^+_{\theta} = 
   \partial\Omega_{\theta} \cap \partial\Omega^+_{\theta},
\]
as we are going to explain now.
Let us introduce $\Delta^\Mix_{\Omega_{\theta}^+}$ as the positive Laplacian with mixed Dirichlet-Neumann conditions on $\Omega_{\theta}^+$ with domain, \cf{} Example \ref{ex:1-1}:
\[
  \Dom(\Delta^\Mix_{\Omega_{\theta}^+}) =
  \big\{ \psi\in H^1(\Omega_{\theta}^+) : \quad\Delta\psi\in L^2(\Omega_{\theta}^+),\quad
  \psi=0 \ \mbox{ on } \ \partial_\Dir\Omega^+_{\theta} 
  \ \ \mbox{and}\ \  \partial_{2}\psi=0 \ \mbox{ on }\ x_{2}=0 \big\}.
\]
Then $\sigma_\ess(\Delta^\Mix_{\Omega_{\theta}^+})$ coincides with $\sigma_\ess(\Delta^\Dir_{\Omega_{\theta}})$. Concerning the discrete spectrum we have:

\begin{figure}[ht]
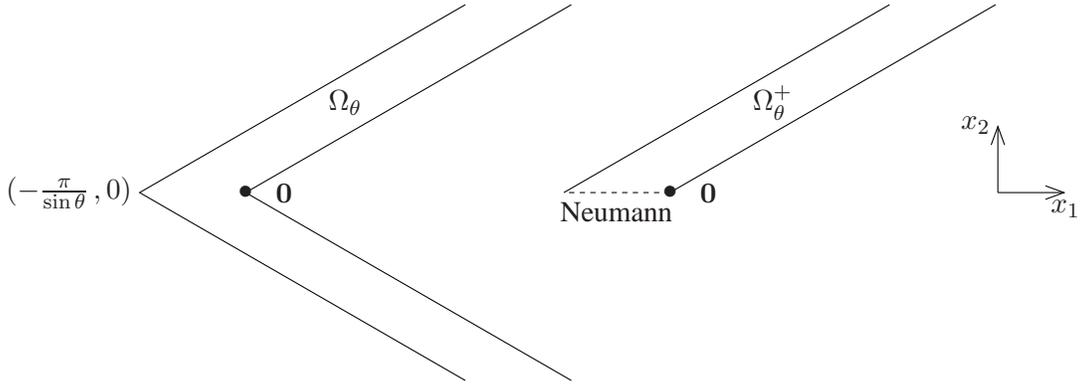

% 1. Definition of characteristic points
    \figinit{1mm}
    \figpt 1:(  0,  0)
    \figpt 2:( 50,  0)
    \figptrot 3: = 2 /1, 30/
    \figpt 4:(-14.1,  0)
    \figpt 10:(100,  0)
    \figvectP 101 [1,4]
    \figptstra 14 = 3/1, 101/
    \figptssym 23 = 3, 14 /1, 2/
    \figptstra 50 = 1, 3, 4, 14/-4, 101/

% 2. Creation of the postscript file
    \def\MyPSfile{}
    \psbeginfig{\MyPSfile}
    \psaxes 10(9)
    \figptsaxes 11 : 10(9)
    \psline[23,1,3]
    \psline[24,4,14]
    \psline[50,51]
    \psline[52,53]
    \psset (dash=4)
    \psline[50,52]
  \psendfig

% 3. Writing text on the figure
    \figvisu{\figbox}{}{%
    \figinsert{\MyPSfile}
    \figwrites 11 :{$x_1$}(1)
    \figwritew 12 :{$x_2$}(1)
    \figwritegcw 4 :{$(-\frac{\pi}{\sin\theta}\,,0)$}(1,0)
    \figwritegce 1 :{$\Omega_{\theta}$}(11,12)
    \figwritegce 50 :{$\Omega_{\theta}^+$}(11,12)
    \figwritegce 52 :{Neumann}(-0.5,-2.5)
    \figsetmark{$\bullet$}
    \figwritee 1 :{${\bf 0}$}(4)
    \figwritee 50 :{${\bf 0}$}(4)
    }
\centerline{\box\figbox}
\caption{The waveguide $\Omega_{\theta}$ and the half-guide $\Omega_{\theta}^+$ (here $\theta=\frac\pi6$).}
\label{F:2}
\end{figure}

\begin{prpstn}
\label{P:dis}
For any $\theta\in(0,\frac\pi2)$,  $\sigma_\dis(\Delta^{\Dir}_{\Omega_{\theta}})$ coincides with $\sigma_\dis(\Delta^\Mix_{\Omega_{\theta}^+})$.
\end{prpstn} 

\begin{proof}
The proof relies on the fact that $\Delta^\Dir_{\Omega_{\theta}}$ commutes with the symmetry 
$S : (x_1,x_2)\mapsto(x_1,-x_2)$.

\noindent
{\em(i)} If $(\la,\psi)$ is an eigenpair of $\Delta^\Mix_{\Omega_{\theta}^+}$, the even extension of $u_\la$ to $\Omega_\theta$ defines an eigenfunction of $\Delta^\Dir_{\Omega_{\theta}}$ associated  with the same eigenvalue $\lambda$. Therefore we have the inclusion
$$\sigma_\dis(\Delta^\Mix_{\Omega_{\theta}^+})\subset\sigma_\dis(\Delta^\Dir_{\Omega_{\theta}}).$$

\noindent
{\em(ii)} Conversely, let $(\la,\psi)$ be an eigenpair of $\Delta^\Dir_{\Omega_{\theta}}$ with $\la<1$. Splitting $\psi$ into its odd and even parts $\psi^{\odd}$ and $\psi^{\even}$ with respect to $x_{2}$, we obtain:
$$
   \psi = \psi^{\odd} + \psi^{\even},\quad
   \Delta^\Dir_{\Omega_{\theta}}\psi^{\odd}=\la \psi^{\odd}\quad \mbox{and}\quad
   \Delta^\Dir_{\Omega_{\theta}}\psi^{\even}=\la \psi^{\even}.
$$
We note that $\psi^{\odd}$ satisfies the Dirichlet condition on the line $x_{2}=0$, and $\psi^{\even}$ the Neumann condition on the same line.
Let us check that $\psi^{\odd}=0.$ If it is not the case, this would mean that $\la$ is an eigenvalue for the Dirichlet Laplacian on the half-waveguide $\Omega_{\theta}^+$. But, by monotonicity of the Dirichlet spectrum with the respect to the domain, \cf{} Example \ref{ex:1-4}, we obtain that the spectrum on $\Omega_{\theta}^+$ is higher than the spectrum on the infinite strip which coincides with $\Omega_{\theta}^+$ when $x_1>0$. This latter spectrum is equal to $[1,+\infty)$. Therefore the Dirichlet Laplacian on the half-waveguide $\Omega_{\theta}^+$ cannot have an eigenvalue below $1$.
Thus, we have necessarily: $\psi^{\odd}=0$ and $\psi=\psi^{\even}$ which is an eigenfunction of $\Delta^\Mix_{\Omega_{\theta}^+}$ associated with $\la$.
\end{proof}

We take advantage of Proposition \ref{P:dis} for further proofs and for numerical simulations.

%%-----------------------------
%%      section 3
%%-----------------------------

\section{Monotonicity of Rayleigh quotients with respect to the opening}
\label{s:3}
From now on we consider the Rayleigh quotients associated with the operator $\Delta^\Mix_{\Omega_{\theta}^+}$, which we write in the form, \cf{} Lemma \ref{lem:1-2}
\begin{equation}
\label{eq:3-1}
   \lambda_j(\theta) = \inff_{
   \substack {\psi_1,\ldots,\psi_j\ \mbox{\footnotesize independent in}\ 
   \rH^1(\Omega^+_\theta),\\ 
   \psi_1,\ldots,\psi_j=0\ \mbox{\footnotesize on}\ \partial_\Dir\Omega^+_\theta 
   }} \ \ \sup_{\psi\in[\psi_1,\ldots,\psi_j]}
   \frac{\|\nabla\psi\|^2_{\Omega^+_\theta}} {\|\psi\|^2_{\Omega^+_\theta}} .
\end{equation}
Those $\lambda_j(\theta)$ which are $<1$ are all the eigenvalues of $\Delta^\Dir_{\Omega_\theta}$ sitting below its essential spectrum.

\begin{prpstn}
\label{prop:3-1}
For any integer $j\ge1$, the function $\theta\mapsto\lambda_j(\theta)$ defined in \eqref{eq:3-1} is non-decreasing from $(0,\frac\pi2)$ into $\xR_+$.
\end{prpstn}

\begin{proof}
We cannot use directly Corollary \ref{co:lem1-2} because of the part of the boundary where Neumann conditions are prescribed. Instead we introduce the open set $\widetilde{\Omega}_{\theta}$ isometric to $\Omega_{\theta}^+$, see Figure \ref{F:3},
$$
   \widetilde{\Omega}_{\theta} = 
   \left\{(\tilde x,\tilde y)\in \left(-\frac{\pi}{\tan\theta},+\infty\right)\times\left(0,\pi\right) : \quad
   \tilde y<\tilde x\tan\theta + \pi\ 
   \mbox{ if }\ \tilde x\in\left(-\frac{\pi}{\tan\theta},0\right)\right\}.
$$

\begin{figure}[ht]
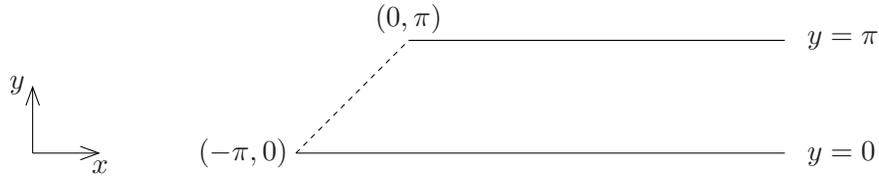

% 1. Definition of characteristic points
    \figinit{1mm}
    \figpt 1:(  0,  0)
    \figpt 2:( 50,  0)
    \figpt 3: (50, 15)
    \figpt 4: (0,  15)
    \figpt 5:(-15,  0)
    \figpt 10:(-50,  0)

% 2. Creation of the postscript file
    \def\MyPSfile{}
    \psbeginfig{\MyPSfile}
    \psaxes 10(9)
    \figptsaxes 11 : 10(9)
    \psline[5,2]
    \psline[3,4]
    \psset (dash=4)
    \psline[5,4]
  \psendfig

% 3. Writing text on the figure
    \figvisu{\figbox}{}{%
    \figinsert{\MyPSfile}
    \figwritee  2 :$y=0$(3)
    \figwritee  3 :$y=\pi$(3)
    \figwriten  4 :{$(0,\pi)$}(1)
    \figwritew  5 :{$(-\pi,0)$}(1)
    \figwrites 11 :{$x$}(1)
    \figwritew 12 :{$y$}(1)
    }
\centerline{\box\figbox}
\caption{The reference half-guide $\widetilde\Omega:=\widetilde{\Omega}_{\pi/4}$.}
\label{F:3}
\end{figure}

The part $\partial_\Dir\widetilde{\Omega}_{\theta}$ of the boundary carrying the Dirichlet condition is the union of its horizontal parts. The numbers $\lambda_j(\theta)$ can be equivalently defined by the Rayleigh quotients \eqref{eq:3-1} on $\widetilde{\Omega}_{\theta}$.

Let us now perform the change of variable: 
\[
   x=\tilde x\tan\theta, \quad y=\tilde y,
\]
so that the new integration domain $\widetilde\Omega:=\widetilde{\Omega}_{\pi/4}$ is independent of $\theta$. The bilinear gradient form $b$ on $\widetilde{\Omega}_{\theta}$ is transformed into
the anisotropic form $b_\theta$ on the fixed set $\widetilde\Omega$:
\begin{equation}
\label{eq:3-2}
   b_\theta(\psi,\psi') = \int_{\widetilde\Omega}
   \tan^2\!\theta \,(\partial_x\psi\,\partial_x\psi') + (\partial_y\psi\,\partial_y\psi') \xdif x\xdif y,
\end{equation}
with associated form domain 
\begin{equation}
\label{eq:3-3}
   V:=\{\psi\in \rH^1(\widetilde\Omega) :  \psi=0 \mbox{ on } \partial_{\Dir}\widetilde\Omega \}
\end{equation}
independent of $\theta$.

The function $\theta\mapsto\tan^2\theta$ being increasing on $(0,\frac\pi2)$, we have
\[
   \forall\psi\in V,\quad \theta\mapsto 
   b_\theta(\psi,\psi) \ \mbox{non-decreasing on $(0,\frac\pi2)$.}
\]
We conclude thanks to Corollary \ref{co:lem1-2}.
\end{proof}

\begin{rmrk}
Using the perturbation theory, we also know that, for all $j\geq 1$, the function $\theta\mapsto\la_{j}(\theta)$ is continuous with respect to $\theta\in\left(0,\frac{\pi}{2}\right)$. Moreover, $\la_{1}$ being simple (as the first eigenvalue of a Laplace-Dirichlet problem), it is analytic because we are in the situation of an analytic family of type $(B)$ (see \cite[p. 387 and 395]{Kato80}). In fact the numerical simulations lead to think that all the eigenvalues below $1$ are simple and thus analytic.
\end{rmrk}

%%-----------------------------
%%      section 4
%%-----------------------------

\section{Existence of discrete spectrum}
\label{s:4}
We recall that the lower bound of the essential spectrum of the operator $\Delta^\Mix_{\Omega^+_\theta}$ is $1$. Its first Rayleigh quotient is given by, \cf{} \eqref{eq:3-1},
\begin{equation}
\label{eq:4-1}
   \lambda_1(\theta) = \inff_{
   \psi\in \rH^1(\Omega^+_\theta),\ \psi=0\ \mbox{\footnotesize on}\ \partial_\Dir\Omega^+_\theta } 
   \frac{\|\nabla\psi\|^2_{\Omega^+_\theta}} {\|\psi\|^2_{\Omega^+_\theta}} .
\end{equation}
In this section, we prove the following proposition:

\begin{prpstn}
\label{prop:4-1}
For any $\theta\in(0,\frac\pi2)$, the first Rayleigh quotient $\lambda_1(\theta)$ is $<1$.
\end{prpstn}

This statement implies that $\lambda_1(\theta)$ is an eigenvalue of $\Delta^\Mix_{\Omega^+_\theta}$ (application of Theorem \ref{th:1-1}), hence of the Dirichlet Laplacian $\Delta^\Dir_{\Omega_\theta}$ on the broken guide (Proposition \ref{P:dis}).

This statement was first established in \cite{ABGM91}. Here we present a distinct, more synthetic proof, using the method of \cite[p. 104-105]{Duclos05}.

\begin{proof}
For convenience, it is easier to work in the reference set $\widetilde\Omega=\widetilde{\Omega}_{\pi/4}$ introduced in the previous section, with the bilinear form $b_\theta$ \eqref{eq:3-2} and the form domain $V$ \eqref{eq:3-3}. We are going to work with the shifted bilinear form
\[
   \tilde b_\theta(\psi,\psi') = b_\theta(\psi,\psi') - 
   \int_{\widetilde\Omega}
   \psi\,\psi' \xdif x\xdif y
\]
which is associated with the quadratic form:
$$
  \widetilde Q_{\theta}(\psi) = \tilde b_\theta(\psi,\psi).
$$
Then the first Rayleigh quotient of $\widetilde Q_\theta$ is equal to $\lambda_1(\theta) - 1$.
To prove our statement, this is enough to construct a function $\psi\in V$ such that:
$$\widetilde Q_{\theta}(\psi)<0.$$
This will be done by the construction of
\begin{enumerate}
\item A sequence $\psi_n\in V$ such that $\widetilde Q_{\theta}(\psi_n)\to 0$ as $n\to\infty$,
\item An element $\phi$ of $V$ such that $\tilde b_\theta(\psi_n,\phi)$ is nonzero and independent of $n$.
\end{enumerate}
The desired function will then be obtained as a suitable combination $\psi_n+\varepsilon\phi$. Let us give details now.

\medskip\noindent
{\sc Step 1.} In order to do that, we consider the Weyl sequence defined as follows. Let $\chi$ be a smooth cutoff function equal to $1$ for $x\leq 0$ and $0$ for $x\geq 1$. We let, for $n\in\mathbb{N}\setminus\{0\}$:
$$
   \chi_{n}(x)=\chi\left(\frac{x}{n}\right)
   \quad\mbox{and}\quad\psi_{n}(x,y)=\chi_{n}(x)\sin y.
$$
Using the support of $\chi_{n}$, we find that $\widetilde Q_{\theta}(\psi_{n})$ is equal to
$$
   \int_{-\pi}^0\int_{0}^{x+\pi} \!\! (\cos^2 y-\sin^2 y) \xdif y \xdif x +
   \int_0^\infty\!\!\int_{0}^\pi\!  \big(
   \tan^2\theta(\chi'_{n})^2\sin^2 y+\chi_{n}^2(\cos^2 y-\sin^2 y)\big) \xdif y \xdif x.
$$
Then, elementary computations provide:
$$
   \int_{0}^\pi (\cos^2y-\sin^2y) \xdif y =0
   \quad\mbox{and}\quad
   \int_{-\pi}^0\int_{0}^{x+\pi} (\cos^2y-\sin^2 y)  \xdif y \xdif x=0.
$$
Moreover, we have:
$$
   \int_0^\infty\!\!\int_{0}^\pi\tan^2\theta(\chi'_{n})^2\sin^2 y \xdif y \xdif x\leq 
   \left(\int_{0}^1|\chi'(u)|^2 \xdif u\right) \frac{\pi\tan^2 \theta}{2n}.
$$
Hence we have proved that $\widetilde Q_{\theta}(\psi_{n})$ tends to $0$ as $n\to\infty$:
\begin{equation}
\label{eq:4-2}
   \widetilde Q_{\theta}(\psi_{n}) \le \frac{K_\theta}{2n}
   \quad\mbox{with}\quad K_\theta =
   \left(\int_{0}^1|\chi'(u)|^2 \xdif u\right) \pi\tan^2 \theta\,.
\end{equation}
{\sc Step 2.} We introduce a smooth cutoff function $\eta$ of $x$ supported in $(-\pi,0)$. We consider a function $f$ of $y\in[0,\pi]$ to be determined later and satisfying $f(0)=0$. We define $\phi(x,y)=\eta(x)f(y)$.
For $\eps>0$ to be chosen small enough, we introduce:
$$\psi_{n,\eps}(x,y)=\psi_{n}(x,y)+\eps\phi(x,y).$$
We have:
$$
   \widetilde Q_{\theta}(\psi_{n,\eps}) =
   \widetilde Q_{\theta}(\psi_{n})+2\eps \,\tilde b_{\theta}(\psi_{n},\phi)
   +\eps^2\widetilde Q_{\theta}(\phi).
$$
Let us compute $\tilde b_{\theta}(\psi_{n},\phi)$. We can write, thanks to considerations of support:
$$
   \tilde b_{\theta}(\psi_{n},\phi) =
   \int_{-\pi}^0\int_{0}^{x+\pi} \hskip-2.0ex\eta(x)\bigl(\cos y f'(y)-\sin y f(y)\big) \xdif y \xdif x=
   \int_{-\pi}^0\int_{0}^{x+\pi} \hskip-2.0ex\eta(x)\bigl(\cos y f(y)\big)'  \xdif y \xdif x.
$$
Using $f(0)=0$, this leads to:
$$
   \tilde b_{\theta}(\psi_{n},\phi)=\int_{-\pi}^0 \eta(x)\cos(x+\pi)f(x+\pi) \xdif x.
$$
We choose $f(y)=\eta(y-\pi)\cos(y-\pi)$ and we find:
$$
   \tilde b_{\theta}(\psi_{n},\phi)=
   -\int_{-\pi}^0 \eta^2(x)\cos^2(x) \xdif x = -\Gamma<0.
$$
This implies, using \eqref{eq:4-2}:
$$
   \widetilde Q_{\theta}(\psi_{n,\eps})\leq
   \frac{K_\theta}{2n}-2\Gamma\eps+D\eps^2,
$$
where $D=\widetilde Q_{\theta}(\phi)$ is independent of $\varepsilon$ and $n$. 
There exists $\eps>0$ such that:
$$-2\Gamma\eps+D\eps^2\leq- \Gamma\eps.$$
The angle
$\theta$ being fixed, we can take $N$ large enough so that
$$
    \frac{K_\theta}{2N}\leq
   \frac{\Gamma}{2}\,\eps ,
$$
from which we deduce that
$\widetilde Q_{\theta}(\psi_{N,\eps})\leq -\varepsilon\Gamma/2 <0$,
which ends the proof of Proposition \ref{prop:4-1}.
\end{proof}

%%-----------------------------
%%      section 5
%%-----------------------------

\section{Finite number of eigenvalues below the essential spectrum}
\label{s:5}
For a self-adjoint operator $A$ and a chosen real number $\lambda$ we denote by $\cN(A,\lambda)$ the maximal index $j$ such that the $j$-th Rayleigh quotient of $A$ is $<\lambda$. By extension of notation, if the operator $A$ is defined by a coercive hermitian form $b$ on a form domain $V$, and if $Q$ denotes the associated quadratic form $Q(u)=b(u,u)$, we also denote by $\cN(Q,\lambda)$ the number $\cN(A,\lambda)$. This is coherent with the fact that in this case the Rayleigh quotients can be defined directly by $Q$, \cf{} Lemma \ref{lem:1-2}:
\[
   \lambda_j = \inff_{
   \substack {u_1,\ldots,u_j\in V\\ 
   \mbox{\footnotesize independent}}} \ \ \sup_{u\in[u_1,\ldots,u_j]}
   \frac{Q(u)} {\langle u,u\rangle_H} \,.
\]
This section is devoted to the proof of the following proposition:

\begin{prpstn}
\label{prop:5-1}
For any $\theta\in(0,\frac{\pi}{2})$, $\cN(\Delta^\Dir_{\Omega_\theta},1)$ is finite.
\end{prpstn}

Thus in any case $\Delta^\Dir_{\Omega_\theta}$ has a nonzero finite number of eigenvalues under its essential spectrum.

\begin{proof}
For the proof of Proposition \ref{prop:5-1} we use a similar method as \cite[Theorem 2.1]{MoTr05}.

Like for the proof of Proposition \ref{prop:4-1}, it is easier to work in the reference set $\widetilde\Omega$ introduced in section \ref{s:3}, with the bilinear form $b_\theta$ \eqref{eq:3-2} and the form domain $V$ \eqref{eq:3-3}. The opening $\theta$ being fixed, we drop the index $\theta$ in the notation of quadratic forms and write simply as $Q$ the quadratic form associated with $b_\theta$:
\[
   Q(\psi) = b_\theta(\psi,\psi) = \int_{\widetilde\Omega} 
   \tan^2\!\theta\,|\partial_{x}\phi|^2+|\partial_{y}\phi|^2 \xdif x\xdif y.
\]
We recall that the form domain $V$ is the subspace of $\psi\in H^1(\widetilde\Omega)$ which satisfy the Dirichlet condition  on $\partial_{\Dir}\widetilde\Omega$. We want to prove that
\[
   \cN(Q,1)\quad\mbox{is finite}.
\]
We consider a $\xCone$ partition of unity $(\chi_{0},\chi_{1})$ such that 
\[
   \chi_{0}(x)^2+\chi_{1}(x)^2=1
\]
with $\chi_{0}(x)=1$ for $x<1$ and $\chi_{0}(x)=0$ for $x>2$. 
For $R>0$ and $\ell\in\{0,1\}$, we introduce:
$$
   \chi_{\ell,R}(x)=\chi_{\ell}(R^{-1}x).
$$
Thanks to the IMS formula (see for instance \cite{Cycon}), we can split the quadratic form as:
\begin{equation}
\label{eq:IMS}
   Q(\psi)=Q(\chi_{0,R}\psi)+Q(\chi_{1,R}\psi)
   -\|\chi'_{0,R}\psi\|^2_{\widetilde\Omega}-\|\chi'_{1,R}\psi\|^2_{\widetilde\Omega}\,.
\end{equation}
We can write
\[
   |\chi'_{0,R}(x)|^2+|\chi'_{1,R}(x)|^2 = R^{-2}W_R(x)\quad
   \mbox{with}\quad W_R(x) = |\chi'_{0}(R^{-1}x)|^2+|\chi'_{1}(R^{-1}x)|^2 \,.
\]
Then
\begin{align}
\label{eq:5-1}
   \|\chi'_{0,R}\psi\|^2_{\widetilde\Omega}+\|\chi'_{1,R}\psi\|^2_{\widetilde\Omega} &=
   \int_{\widetilde\Omega} R^{-2} W_{R}(x)|\psi|^2 \xdif x\xdif y 
   \nonumber\\ &= 
   \int_{\widetilde\Omega} R^{-2} W_{R}(x) 
   \big(|\chi_{0,R}\psi|^2 + |\chi_{1,R}\psi|^2 \big)\xdif x\xdif y.
\end{align}
Let us introduce the subsets of $\widetilde\Omega$:
$$
   \mathcal{O}_{0,R} = \{(x,y)\in\widetilde\Omega : x<2R\}
   \quad\mbox{ and }\quad
   \mathcal{O}_{1,R}=\{(x,y)\in\widetilde\Omega : x>R\}
$$
and the associated form domains
\[
\begin{gathered}
   V_0 = \left\{\phi\in H^1(\mathcal{O}_{0,R}) :  \quad
   \phi=0 \ \mbox{ on }\ \partial_{\Dir}\widetilde\Omega \cap \partial \mathcal{O}_{0,R}
   \ \mbox{ and on }\ \{2R\}\times(0,\pi)\right\} \\
   V_1 = \rH^1_0(\mathcal{O}_{1,R}).
\end{gathered}
\]
We define the two quadratic forms $Q_{0,R}$ and $Q_{1,R}$ by
\begin{equation}
\label{eq:Qell}
   Q_{\ell,R}(\phi) = \int_{\mathcal{O}_{\ell,R}} 
   \tan^2\theta|\partial_{x}\phi|^2+|\partial_{y}\phi|^2-R^{-2}W_{R}(x)|\phi|^2 \xdif x\xdif y
   \quad\mbox{for}\quad \psi\in V_\ell,\ \ \ell=0,1.
\end{equation}
As a consequence of \eqref{eq:IMS} and \eqref{eq:5-1} we find
\begin{equation}
\label{Q=Q0+Q1}
   Q(\psi) = Q_{0,R}(\chi_{0,R}\psi) + Q_{1,R}(\chi_{1,R}\psi)\quad \forall\psi\in V.
\end{equation}
Let us prove

\begin{lmm}
\label{lem:5-1}
We have: 
$$
   \cN(Q,1) \leq \cN(Q_{0,R},1) + \cN(Q_{1,R},1).
$$
\end{lmm}

\begin{proof}
We recall the formula for the $j$-th Rayleigh quotient of $Q$:
\[
   \lambda_j =
   \inff_{\substack{E\subset V\\ \dim E=j}} \ \sup_{\substack{
    \psi\in E}}   \frac{Q(\psi)}{\|\psi\|^2_{\widetilde\Omega}} \,.
\]
The idea is now to give a lower bound for $\lambda_j$. Let us introduce:
$$
\mathcal{J} : \left\{\begin{array}{ccc}
V &\to& V_0 \times V_1 \\
\psi&\mapsto&(\chi_{0,R}\psi\ ,\,\chi_{1,R}\psi) \,.
\end{array}\right.
$$
As $(\chi_{0,R},\chi_{1,R})$ is a partition of the unity, $\mathcal{J}$ is injective. In particular, we notice that $\mathcal{J} : V \to\mathcal{J}(V)$ is bijective so that we have:
\begin{align*}
   \lambda_j &=
   \inf_{\substack{F\subset \mathcal{J}(V)\\ \dim F=j}} \ \sup_{\substack{
    \psi\in \mathcal{J}^{-1}(F)}}   \frac{Q(\psi)}{\|\psi\|^2_{\widetilde\Omega}}\\
    &= \inf_{\substack{F\subset \mathcal{J}(V)\\ \dim F=j}} \ \sup_{\substack{
   \psi\in \mathcal{J}^{-1}(F)}}   \frac{Q_{0,R}(\chi_{0,R}\psi)+Q_{1,R}(\chi_{1,R}\psi)}
   {\|\chi_{0,R}\psi\|^2_{\widetilde\Omega} + \|\chi_{1,R}\psi\|^2_{\widetilde\Omega}}\\
    &= \inf_{\substack{F\subset \mathcal{J}(V)\\ \dim F=j}} \ \sup_{\substack{
    (\psi_{0},\psi_{1})\in F}}   \frac{Q_{0,R}(\psi_{0})+Q_{1,R}(\psi_{1})}
    {\|\psi_{0}\|^2_{\mathcal{O}_{0,R}}+\|\psi_{1}\|^2_{\mathcal{O}_{1,R}}}.
\end{align*}
As  $\mathcal{J}(V) \subset V_0 \times V_1$, we deduce by an application of Corollary \ref{co:lem1-2}:
\begin{align}\label{eq:5-2}
   \lambda_j &\geq \inf_{\substack{F\subset V_0\times V_1\\ \dim F=j}} \ \sup_{\substack{
    (\psi_{0},\psi_{1})\in F}}   \frac{Q_{0,R}(\psi_{0})+Q_{1,R}(\psi_{1})}
    {\|\psi_{0}\|^2_{\mathcal{O}_{0,R}} + \|\psi_{1}\|^2_{\mathcal{O}_{1,R}}}
    =: \nu_{j},
\end{align}
Let $A_{\ell,R}$ be the self-adjoint operator with domain $\Dom(A_{\ell,R})$ associated with the coercive bilinear form corresponding to the quadratic form $Q_{\ell,R}$ on $V_\ell$.
We see that $\nu_{j}$ in \eqref{eq:5-2} is the $j$-th Rayleigh quotient of the diagonal self-adjoint operator $A_R$
\[
   \begin{pmatrix}
   A_{0,R} & 0\\
   0 & A_{1,R}
   \end{pmatrix} \quad\mbox{with domain}\quad \Dom(A_{0,R}) \times  \Dom(A_{1,R})\,.
\]
The Rayleigh quotients of $A_{\ell,R}$ are associated with the quadratic form $Q_{\ell,R}$ for $\ell=0,1$.
Thus $\nu_j$ is the $j$-th element of the ordered set
\[
   \{\lambda_k(Q_{0,R}), \ k\ge1\} \cup \{\lambda_k(Q_{1,R}), \ k\ge1\}.
\]
Lemma \ref{lem:5-1} follows.
\end{proof}

The operator $A_{0,R}$ is elliptic on a bounded open set, hence has a compact resolvent. Therefore we get:

\begin{lmm}
For all $R>0$, $\cN(Q_{0,R},1)$ is finite.
\end{lmm}

To achieve the proof of Proposition \ref{prop:5-1}, it remains to establish the following lemma:

\begin{lmm}
There exists $R_{0}>0$  such that, for $R\geq R_{0}$, $\cN(Q_{1,R},1)$ is finite.
\end{lmm}

\begin{proof}
For all $\phi\in V_1$, we write: 
$$\phi=\Pi_{0}\phi+\Pi_{1}\phi,$$
where 
\begin{equation}
\label{eq:Pi0}
   \Pi_{0}\phi(x,y) = \Phi(x) \sin y\quad\mbox{with}\quad
   \Phi(x) = \int_0^\pi \phi(x,y)\sin{y}\xdif y
\end{equation}
is the projection on the first eigenvector of $-\partial^2_y$ on $\rH^1_0(0,\pi)$, and $\Pi_{1}=\Id-\Pi_{0}$.
We have, for all $\eps>0$:
\begin{align}
   Q_{1,R}(\phi)
   &=Q_{1,R}(\Pi_{0}\phi)+Q_{1,R}(\Pi_{1}\phi)
    -2\int_{\mathcal{O}_{1,R}} R^{-2}W_{R}(x)\Pi_{0}\phi\, \Pi_{1}\phi\xdif x\xdif y
    \nonumber\\
   &\geq Q_{1,R}(\Pi_{0}\phi)+Q_{1,R}(\Pi_{1}\phi)
    -\eps^{-1}\int_{\mathcal{O}_{1,R}} R^{-2}W_{R}(x)|\Pi_{0}\phi|^2\xdif x\xdif y
    -\eps\int_{\mathcal{O}_{1,R}} R^{-2}W_{R}(x)|\Pi_{1}\phi|^2\xdif x\xdif y
%    \nonumber\\
    \label{eq:Q1R}
%   &\geq Q_{1,R}(\Pi_{0}\phi)+(1-\eps)Q_{1,R}(\Pi_{1}\phi)
%   -\eps^{-1}\int_{\mathcal{O}_{1,R}} R^{-2}W_{R}(x)|\Pi_{0}\phi|^2\xdif x\xdif y.
\end{align}
Since the second eigenvalue of $-\partial^2_y$ on $\rH^1_0(0,\pi)$ is $4$, we have:
\[
   \int_{\mathcal{O}_{1,R}} 
   |\partial_{y}\Pi_{1}\phi|^2 \xdif x\xdif y \ge 4 \|\Pi_{1}\phi\|^2_{\mathcal{O}_{1,R}}\,.
\]
Denoting by $M$ the maximum of $W_R$ (which is independent of $R$), and using \eqref{eq:Qell} we deduce
$$
   Q_{1,R}(\Pi_{1}\phi)\geq (4-MR^{-2})\|\Pi_{1}\phi\|^2_{\mathcal{O}_{1,R}}\,.
$$
Combining this with \eqref{eq:Q1R} where we take $\varepsilon=1$, and with the definition \eqref{eq:Pi0} of $\Pi_0$, we find
$$
   Q_{1,R}(\phi) \geq q_{R}(\Phi) +(4-2MR^{-2})\|\Pi_{1}\phi\|^2_{\mathcal{O}_{1,R}},
$$
where 
\begin{align*}
   q_{R}(\Phi) &= \int_{R}^\infty \tan^2\theta|\partial_{x}\Phi|^2+|\Phi|^2
  -R^{-2}W_{R}(x)|\Phi|^2 \xdif x \\
  &\ge \int_{R}^\infty \tan^2\theta|\partial_{x}\Phi|^2+|\Phi|^2
  -R^{-2}M\mathds{1}_{[R,2R]}|\Phi|^2 \xdif x.
\end{align*}
We choose $R=\sqrt{M}$ so that $(4-2MR^{-2})=2$, and then
\begin{equation}
\label{eq:5-3}
   Q_{1,R}(\phi) \geq \tilde q_{R}(\Phi) + 2\|\Pi_{1}\phi\|^2_{\mathcal{O}_{1,R}},
\end{equation}
where now
\begin{equation}
\label{eq:5-4}
   \tilde q_{R}(\Phi) = \int_{R}^\infty \tan^2\theta|\partial_{x}\Phi|^2+
   (1-\mathds{1}_{[R,2R]})|\Phi|^2 \xdif x.
\end{equation}
Let $\tilde a_R$ denote the 1D operator associated with the quadratic form $\tilde q_R$.
From \eqref{eq:5-3}-\eqref{eq:5-4}, we deduce that the $j$-th Rayleigh quotient of $A_{1,R}$ admits as lower bound the $j$-th Rayleigh quotient of the diagonal operator:
$$
  \begin{pmatrix}
  \tilde a_R & 0\\
  0 & 2\,\Id
  \end{pmatrix}
$$
so that we find:
$$
   \cN(Q_{1,R},1)\leq \cN(\tilde q_{R},1).
$$
Finally, the eigenvalues $<1$ of $\tilde a_R$ can be computed explicitly and this is an elementary exercise to deduce that $\cN(\tilde q_{R},1)$ is finite.
\end{proof}

This concludes the proof of Proposition \ref{prop:5-1}.
\end{proof}

%%-----------------------------
%%      section 6
%%-----------------------------

\section{Decay of eigenvectors at infinity -- Computations for large angles}
\label{s:6}
\subsection{Decay at infinity}
In order to study theoretical properties of eigenvectors of the operator $\Delta^\Dir_{\Omega_\theta}$ corresponding to eigenvalues below $1$, we use the equivalent configuration on $\widetilde\Omega_\theta$ introduced in section \ref{s:3}, see also Figure \ref{F:4}. The eigenvalues $<1$ of $\Delta^\Dir_{\Omega_\theta}$ are the same as those of $\Delta^\Mix_{\widetilde\Omega_\theta}$ (with Dirichlet conditions on the horizontal parts of the boundary of $\widetilde\Omega_\theta$) and the eigenvectors are isometric. The main result of this section is a quasi-optimal decay in the straight part $(0,+\infty)\times(0,\pi)$ of the set $\widetilde\Omega_\theta$ as $x\to\infty$.

\begin{figure}[ht]
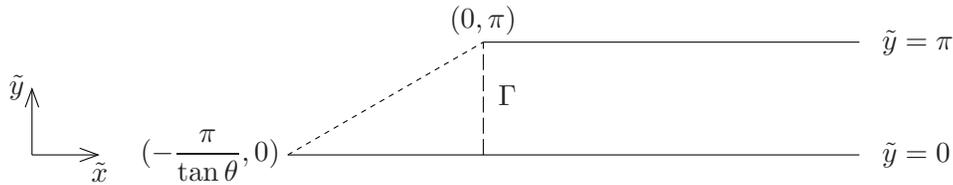

% 1. Definition of characteristic points
    \figinit{1mm}
    \figpt 1:(  0,  0)
    \figpt 2:( 50,  0)
    \figpt 3: (50, 15)
    \figpt 4: (0,  15)
    \figpt 5:(-26,  0)
    \figpt 10:(-60,  0)

% 2. Creation of the postscript file
    \def\MyPSfile{}
    \psbeginfig{\MyPSfile}
    \psaxes 10(9)
    \figptsaxes 11 : 10(9)
    \psline[5,2]
    \psline[3,4]
    \psset (dash=4)
    \psline[5,4]
    \psset (dash=2)
    \psline[1,4]
  \psendfig

% 3. Writing text on the figure
    \figvisu{\figbox}{}{%
    \figinsert{\MyPSfile}
    \figwritee  2 :$\tilde y=0$(3)
    \figwritee  3 :$\tilde y=\pi$(3)
    \figwriten  4 :{$(0,\pi)$}(1)
    \figwritew  5 :{$(-\displaystyle{\pi\over\tan\theta},0)$}(1)
    \figwrites 11 :{$\tilde x$}(1)
    \figwritew 12 :{$\tilde y$}(1)
    \figwritegce 4 :{$\Gamma$}(2,-7)
    }
\centerline{\box\figbox}
\caption{The half-guide $\widetilde\Omega_\theta$ (here $\theta=\frac\pi6$).}
\label{F:4}
\end{figure}

\begin{prpstn}
\label{prop:6-1}
Let $\theta\in(0,\frac\pi2)$. Let $\psi$ be an eigenvector of $\Delta^\Mix_{\widetilde\Omega_\theta}$ associated with an eigenvalue $\lambda<1$. Then for all $\varepsilon>0$ the following integral is finite:
\begin{equation}
\label{eq:6-1}
   \int_0^\infty \int_0^\pi \re^{2 \tilde x(\sqrt{1-\lambda}-\varepsilon)} \left(
   |\psi(\tilde x, \tilde y)|^2 + |\nabla\psi(\tilde x, \tilde y)|^2 \right) \xdif \tilde x\xdif \tilde y 
   < \infty.
\end{equation}
\end{prpstn}

\begin{proof}
We give here an elementary proof based on the representation of $\psi$ as solution of the Dirichlet problem in the half-strip $\Sigma := \xR_+\times(0,\pi)$
\begin{equation}
\label{eq:6-2}
   \begin{cases}
   \ -\Delta \psi = \lambda \psi &\quad \mbox{in}\quad \Sigma, \\
   \ \psi(\tilde x,0) = 0,\  \psi(\tilde x,\pi) = 0&\quad \forall \tilde x>0,\\
   \ \psi(0,\tilde y) = g &\quad \forall  \tilde y \in (0,\pi) \\
   \end{cases}
\end{equation}
where $g$ is the trace of $\psi$ on the segment $\Gamma$, see Figure \ref{F:4}. Since $\psi$ belongs to $\xHone(\Sigma)$, its trace on $\partial\Sigma$ belongs to $\rH^{1/2}$. Because of the zero trace on the lines $\tilde y=0$ and $\tilde y=\pi$, we find that $g$ belongs to the smaller space\footnote{
The space $\rH^{1/2}_{00}(\Gamma)$ is the subspace of $\rH^{1/2}(\Gamma)$ spanned by the functions $v$ such that
\[
   \int_0^\pi \frac{v^2(\tilde y)} {\tilde y(\pi-\tilde y)} \,\xdif\tilde y <\infty\,.   
\]
} 
$\rH^{1/2}_{00}(\Gamma)$, which is the interpolation space of index $\frac12$ between $\xHone_0(\Gamma)$ and $\xLtwo(\Gamma)$, \cf{} \cite{LionsMagenes68}.

We expand $\psi$ on the eigenvector basis of the operator $-\partial^2_y$, self-adjoint on $\rH^2\cap\rH^1_0(0,\pi)$. Its normalized eigenvectors are
\[
   v_k(\tilde y) = \sqrt{\frac2\pi}\, \sin k\tilde y\quad
   \mbox{with eigenvalue}\quad \nu_k = k^2.
\]
We expand $g$ in this basis:
$$
   g(\tilde y)=\sum_{k\geq 1} g_{k}\, v_{k}(\tilde y),\quad \mbox{ where } 
   g_k = \int_0^\pi g(\tilde y)v_k(\tilde y) \xdif\tilde y.
$$
Interpolating between \eqref{eq:nu0} and \eqref{eq:nu1}, we find
\[
   \|g\|^2_{\rH^{1/2}_{00}(\Gamma)} \simeq \sum_{k\ge1} k \,g^2_k.
\]
We can easily solve \eqref{eq:6-2} by separation of variables. We find
\begin{equation}
\label{eq:6-3}
   \psi(\tilde x, \tilde y) = \sum_{k\ge 1}\re^{-\tilde x \sqrt{k^2-\lambda}} \, g_{k}\, v_k(\tilde y).
\end{equation}
The estimate of $\psi$ in \eqref{eq:6-1} is then trivial.
Let us prove now the estimate of $\partial_{\tilde x}\psi$.
We use that $\sum_{k\ge1} k \,g^2_k$ is finite so that we can write:
$$
   \partial_{\tilde x}\psi(\tilde x, \tilde y) =
   -\sum_{k\ge 1} \sqrt{k^2-\lambda}\ \re^{-\tilde x \sqrt{k^2-\lambda}} \, g_{k}\,v_k(\tilde y).
$$
We have:
$$
   \re^{\tilde x(\sqrt{1-\lambda}-\eps)}\, \partial_{\tilde x}\psi(\tilde x, \tilde y) =
   -\sum_{k\geq 1} \sqrt{k^2-\lambda}\ \re^{\tilde x 
   \left(\sqrt{1-\lambda}-\sqrt{k^2-\lambda}\right)}\ \re^{-\eps\tilde x} \,g_{k}\, v_k(\tilde y)
$$
leading to the $\xLtwo$ estimate:
\begin{align*}
   \int_0^\infty \int_0^\pi    
   \big| \re^{\tilde x(\sqrt{1-\lambda}-\eps)}\,\partial_{\tilde x}\psi(\tilde x, \tilde y) \big|^2
   \xdif \tilde x\xdif \tilde y 
   &=  
   \sum_{k\geq 1}  \int_0^\infty  (k^2-\lambda)\ \re^{2\tilde x 
   \left(\sqrt{1-\lambda}-\sqrt{k^2-\lambda}\right)}\ \re^{-2\eps\tilde x} \,g^2_{k} \xdif \tilde x\\
   &\le
   \Big(\sum_{k\geq 1}k\,g^2_{k}\Big) \,
   \sup_{k\ge1}  \int_0^\infty 
   k \,\re^{-2\gamma\tilde x\sqrt{k^2-1}}\ \re^{-2\eps\tilde x} \xdif \tilde x,
\end{align*}
where $\gamma=\gamma(\la)>0$ is a constant, uniform with respect to $k\geq 1$. Using the change of variables $\tilde x \mapsto \sqrt k\,\tilde x$, we can see that the integrals
\[
   \int_0^{+\infty} k \re^{-2\tilde x(\varepsilon + \gamma\sqrt{k^2-1})}
\] 
are uniformly bounded as $k\to\infty$, which ends the proof of the estimate of $\partial_{\tilde x}\psi$ in \eqref{eq:6-1}. The estimate of $\partial_{\tilde y}\psi$ is similar.
\end{proof}

\begin{rmrk}
Under the conditions of Proposition \ref{prop:6-1}, we have also a sharp global $\xLinfty$ estimate
\[
   \| \re^{\tilde x\sqrt{1-\lambda}} \, \psi \|_{\xLinfty(\widetilde\Omega_\theta)} < \infty.
\]
To prove this we use again the representation \eqref{eq:6-3} and the fact that  the trace $g$ is more regular than $\rH^{1/2}_{00}(\Gamma)$. In fact, as a consequence of the characterization \eqref{eq:dom} of the domain of the operator $\Delta^\Dir_{\Omega_\theta}$, the trace $g$ belongs to $\rH^1_0(\Gamma)$ and, therefore, the sum $\sum_{k\ge1} k^2 \,g^2_k$ is finite by \eqref{eq:nu1}. 
Then, we have:
$$
   \re^{\tilde x\sqrt{1-\la}} \,\psi(\tilde x,\tilde y)=
   \sum_{k\geq 1}\re^{\tilde x \left(\sqrt{1-\lambda}-\sqrt{k^2-\lambda}\right)} \, g_{k}\,v_k(\tilde y).
$$
Taking absolute values, we get:
$$
   \re^{\tilde x\sqrt{1-\la}}|\psi(\tilde x,\tilde y)|\leq 
   \sqrt{\frac2\pi} \sum_{k\geq 1}|g_{k}|\leq  
   \sqrt{\frac2\pi}\Big(\sum_{k\geq 1}k^{-2}\Big)^{1/2}\Big(\sum_{k\geq 1}k^2|g_{k}|^2\Big)^{1/2}.
$$
\end{rmrk}

\begin{rmrk}
The estimate \eqref{eq:6-1} can also be proved with a general method due to Agmon (see for instance \cite{Agmon82}).
\end{rmrk}

\subsection{Computations for large angles}
\label{ss:CompLA}
When $\theta\to\frac\pi2$, $\lambda_1(\theta)$ tends to $1$, see \cite{ABGM91}. The representation \eqref{eq:6-3} shows that in such a situation, the behavior of the associated eigenvector $\psi$ is dominated by its first term, proportional to
\[
   \re^{-\tilde x \sqrt{1-\lambda}} \, \sin(\tilde y).
\]

\begin{figure}[ht]
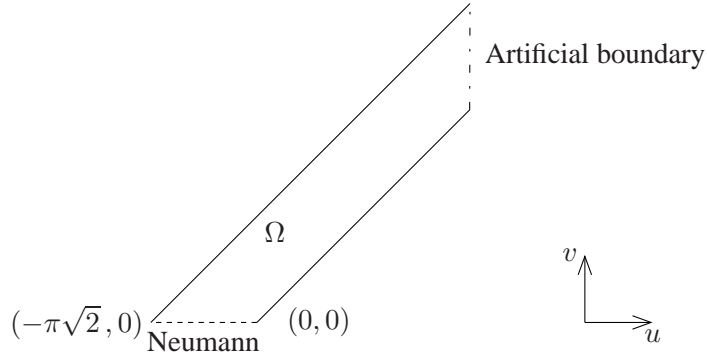

% 1. Definition of characteristic points
    \figinit{1mm}
    \figpt 1:(  0,  0)
    \figpt 2:( 40,  0)
    \figpt 5:( 45.9,  0)
    \figptrot 3: = 2 /1, 45/
    \figpt 4:(-14.1,  0)
    \figptrot 6: = 5 /4, 45/
    \figpt 10:(100,  0)
    \figvectP 101 [1,4]
    \figptstra 50 = 1, 3, 4, 6/-4, 101/

% 2. Creation of the postscript file
    \def\MyPSfile{}
    \psbeginfig{\MyPSfile}
    \psaxes 10(9)
    \figptsaxes 11 : 10(9)
    \psline[50,51]
    \psline[52,53]
    \psset (dash=4)
    \psline[50,52]
    \psset (dash=9)
    \psline[51,53]
  \psendfig

% 3. Writing text on the figure
    \figvisu{\figbox}{}{%
    \figinsert{\MyPSfile}
    \figwrites 11 :{$\uu$}(1)
    \figwritew 12 :{$\vv$}(1)
    \figwritee 50 :{$(0,0)$}(4)
    \figwritegcw 52 :{$(-\pi\sqrt2\,,0)$}(1,0)
    \figwritegce 50 :{$\Omega$}(1,12)
    \figwritegce 52 :{Neumann}(-0.5,-2.5)
    \figwritegce 53 :{Artificial boundary}(2,-7)
    }
\centerline{\box\figbox}
\caption{The model half-guide $\Omega:=\Omega_{\pi/4}^+$.}
\label{F:6}
\end{figure}

In order to compute such an eigenpair by a finite element method, we have to be careful and take large enough domains --- we simply put Dirichlet conditions on an artificial boundary far enough from the corners of the guide.
Our computations are performed in the model half-guide $\Omega:=\Omega^+_{\pi/4}$ for the scaled operator
\begin{equation}
\label{eq.Ltheta}
   \cL_\theta := -2\sin^2\!\theta\,\partial_{\uu}^2-2\cos^2\!\theta\,\partial_{\vv}^2 
\end{equation}
equivalent to $-\Delta$ in $\Omega^+_\theta$ through the variable change
\[
   \uu = x_1\sqrt2\sin\theta\quad\mbox{and}\quad
   \vv = x_2\sqrt2\cos\theta.
\]

\begin{figure}
\centerline{\includegraphics[scale=0.45]{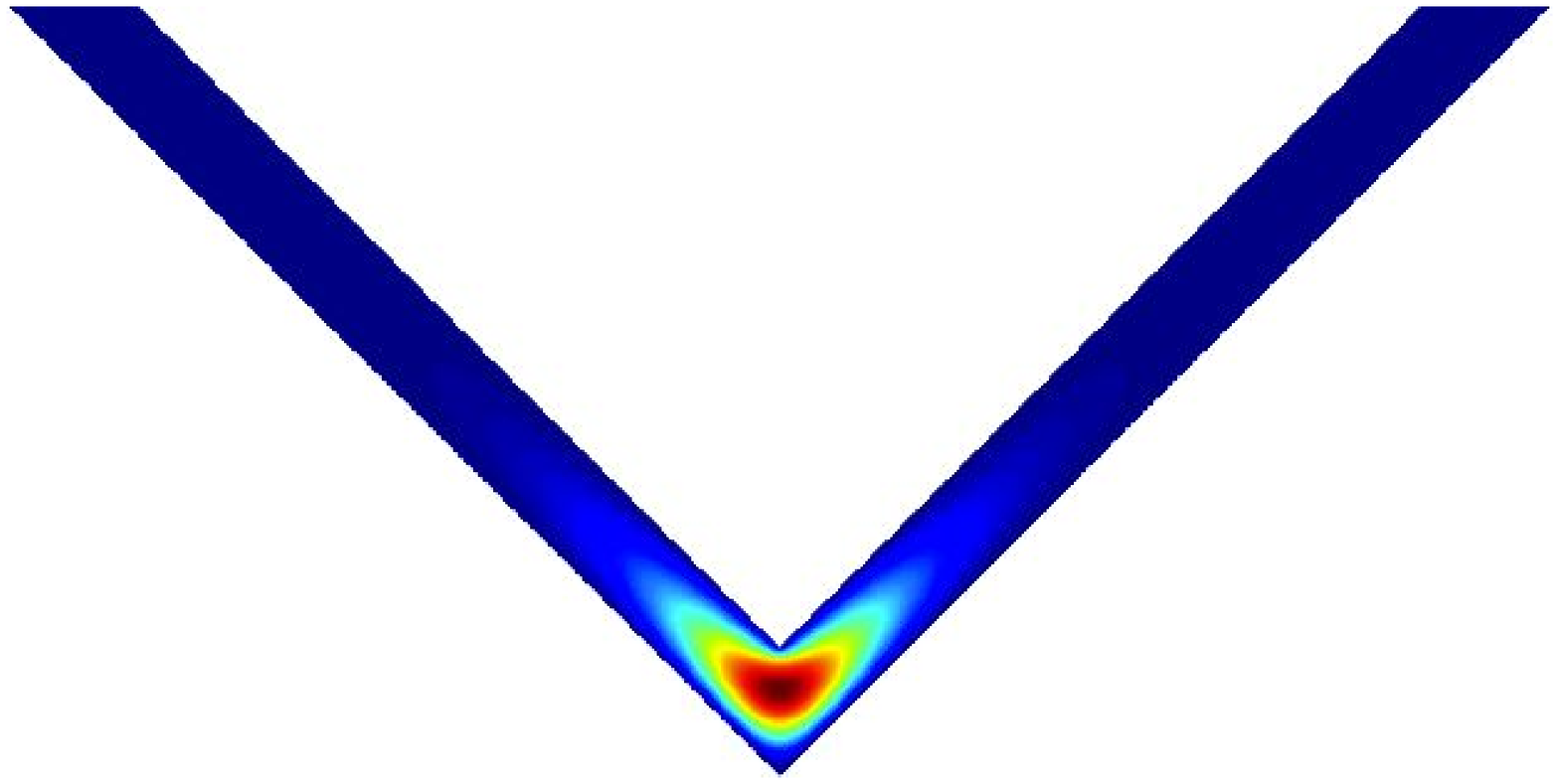}}
\centerline{$\theta = 0.5000*\pi/2$\quad\quad $\lambda_1^\comp(\theta)=0.92934$}

\centerline{\includegraphics[scale=0.45]{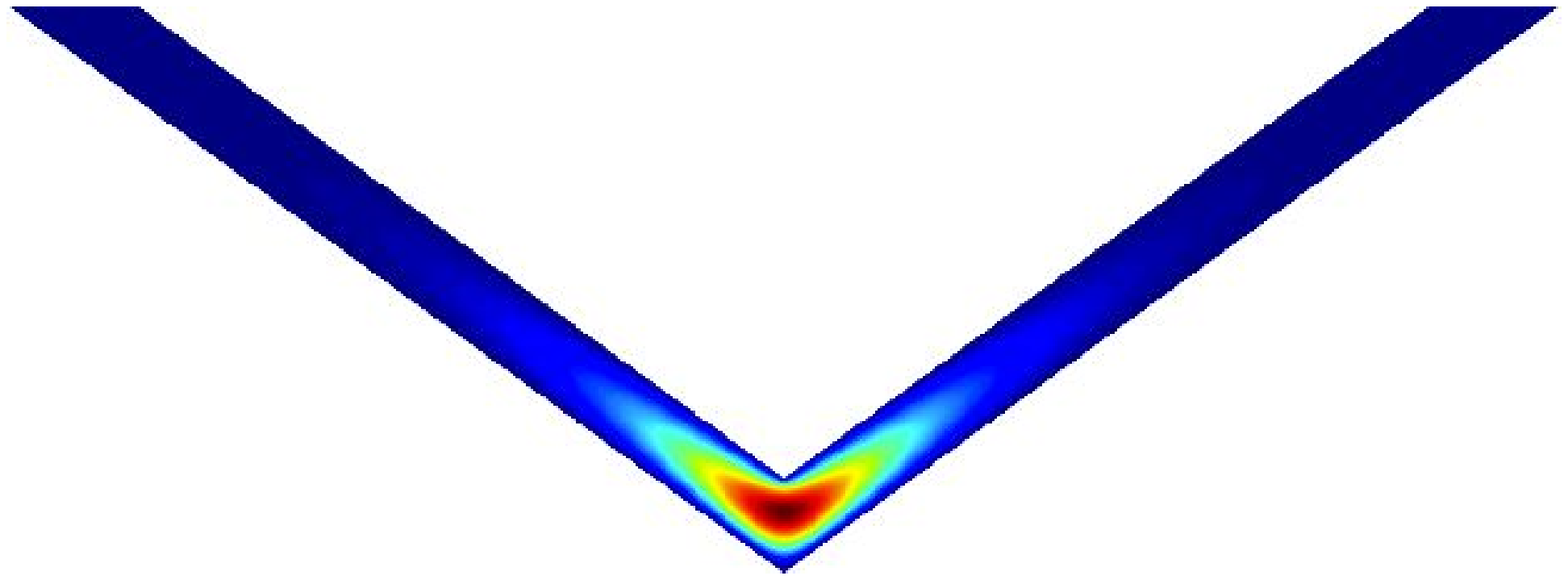}}
\centerline{$\theta = 0.5983*\pi/2$\quad\quad $\lambda_1^\comp(\theta)=0.96897$}

\centerline{\includegraphics[scale=0.45]{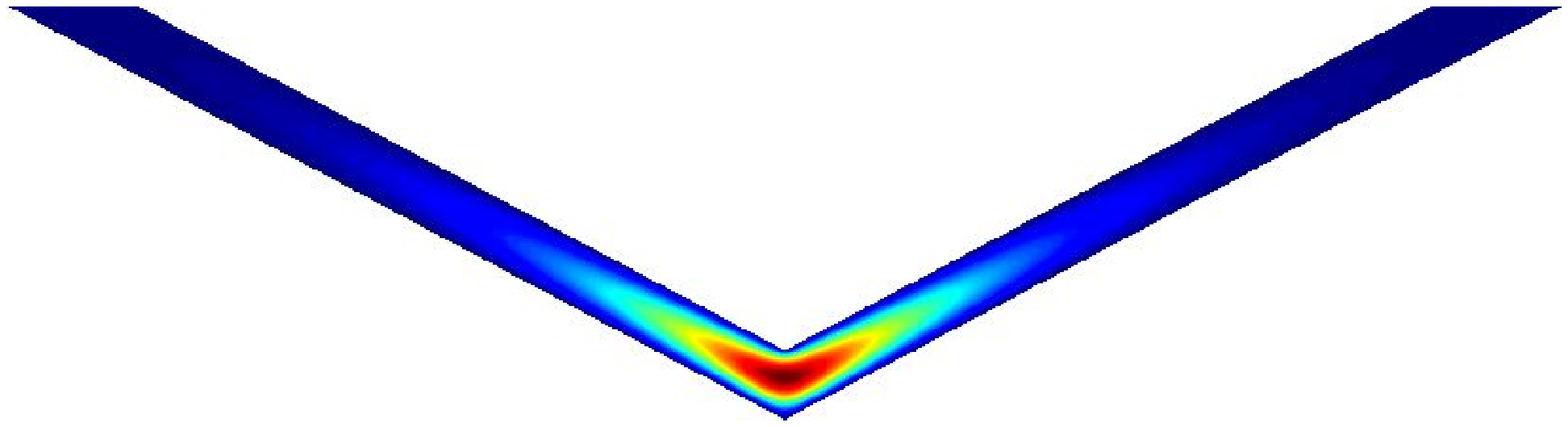}}
\centerline{$\theta = 0.6889*\pi/2$\quad\quad $\lambda_1^\comp(\theta)=0.98844$}
\caption{Computations for moderately large angles. Plots of the first eigenvectors in the physical domain (rotated by $\frac\pi2$). Numerical value of the corresponding eigenvalue $\lambda_1(\theta)$.}
\label{F:5}
\end{figure}

We use a Galerkin discretization by finite elements in a truncated subset of $\Omega$ with Dirichlet condition on the artificial boundary, see Figure \ref{F:6}. According to Corollary \ref{co:lem1-2}, \cf{} also Examples \ref{ex:1-3} and \ref{ex:1-4}, the eigenvalues $\lambda^\comp_j(\theta)$ of the discretized problem are larger than the Rayleigh quotients $\lambda_j(\theta)$ of $\cL_\theta$. When the discretization gets finer and the computational domain, larger, $\lambda^\comp_j(\theta)$ tends to $\lambda_j(\theta)$ for $j=1,2,\ldots$

For the values of $\theta$ considered in this section ($\theta\ge\frac\pi4$), the numerical evidence is that the discrete spectrum of $\cL_\theta$ has only one element $\lambda_1(\theta)$. The numerical effect of this is the convergence to $1$ of all other computational eigenvalues $\lambda^\comp_j(\theta)$ for $j=2,3,\ldots$

The computations represented in Figure \ref{F:5} are performed with the artificial boundary set at the abscissa $\uu=5\pi\sqrt2$. The plots are mapped back to the corresponding physical domain by a postprocessing of the numerical results.

\begin{figure}
\begin{tabular}{c@{\,}c@{\,}c}
\includegraphics[scale=0.25]{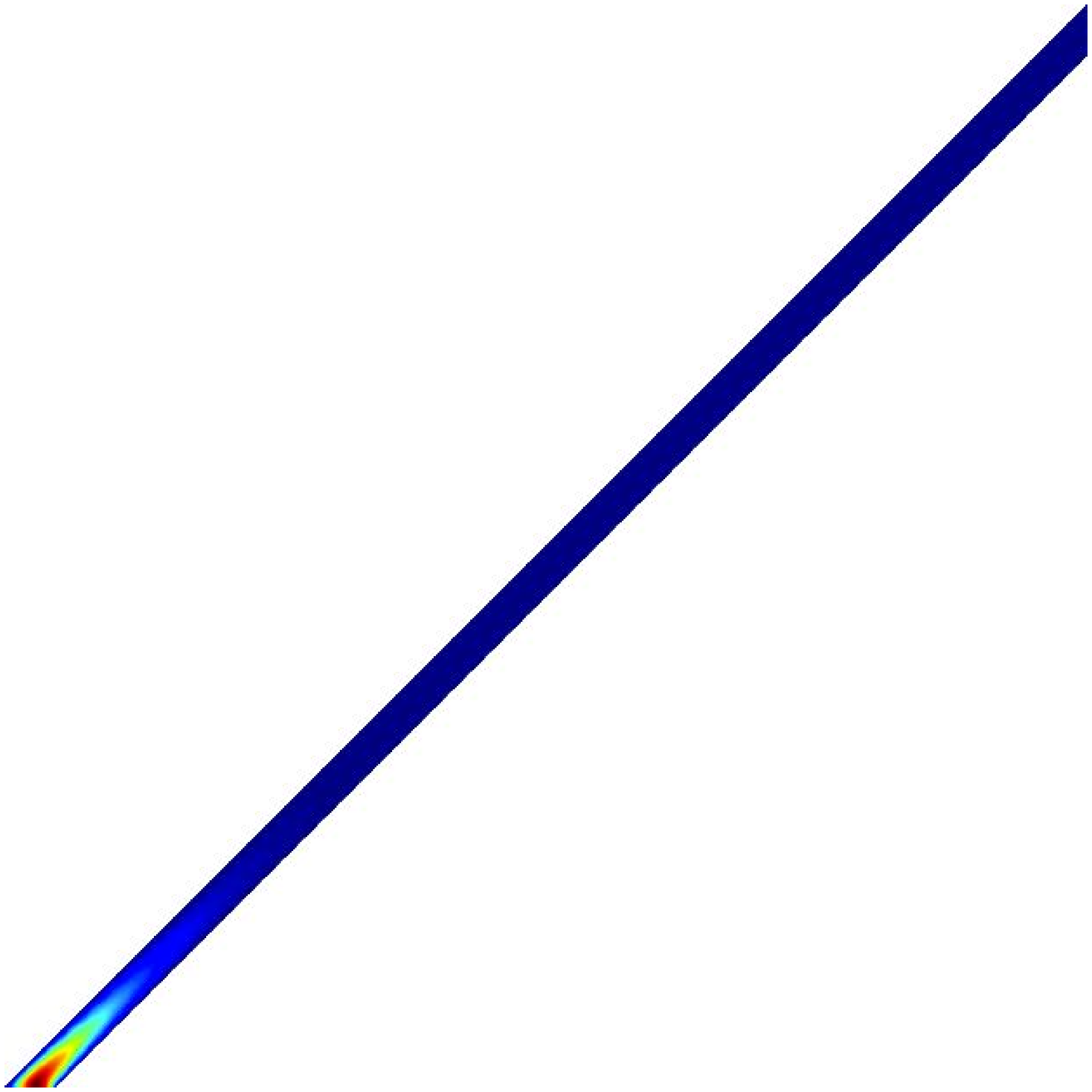}&
\includegraphics[scale=0.25]{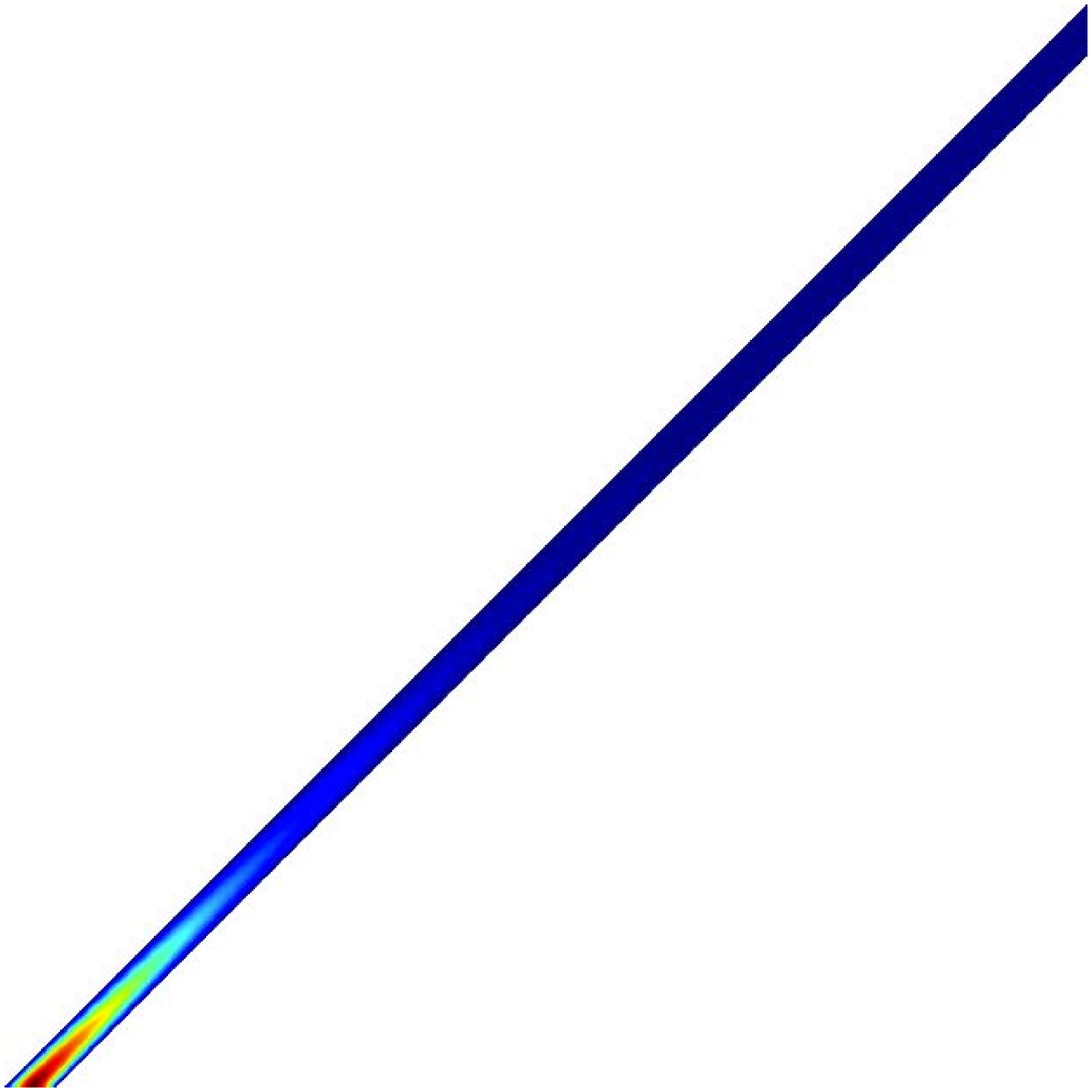}&
\includegraphics[scale=0.25]{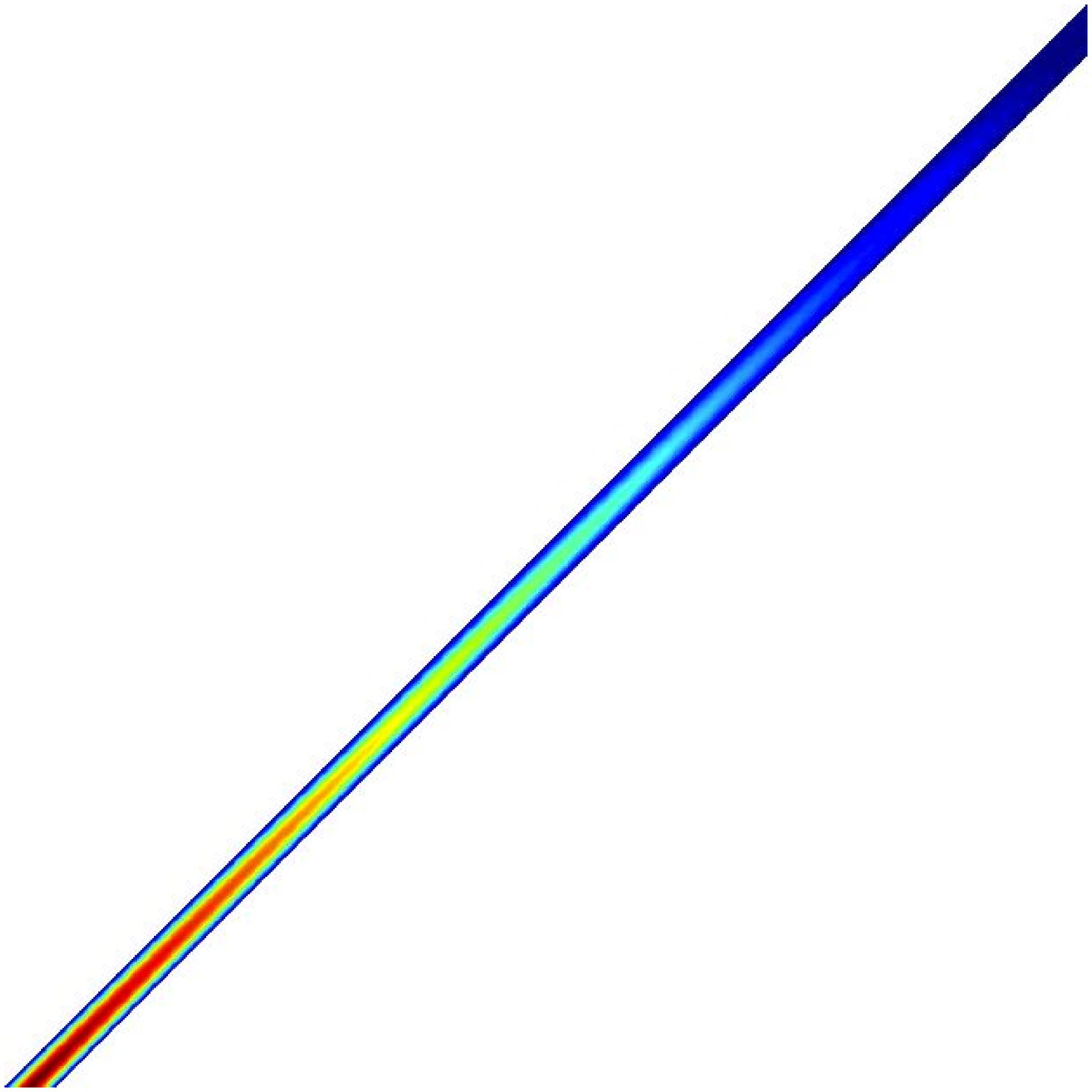}
\\
$\theta = 0.7022*\pi/2$ &
$\theta = 0.8538*\pi/2$ &
$\theta = 0.9702*\pi/2$ 
\\
$\lambda_1^\comp(\theta)=0.9903037$&
$\lambda_1^\comp(\theta)=0.9994215$&
$\lambda_1^\comp(\theta)=0.9999998$
\end{tabular}
\caption{Computations for very large angles with the mesh M4L (see Figure \ref{F:mesh4L}). Plots in the computational domain $\Omega$.}
\label{F:7}
\end{figure}

The computations in Figure \ref{F:7} are performed with the artificial boundary set at the abscissa $\uu=10\pi\sqrt2$. The plots use the computational domain because the corresponding physical domains would be too much elongated to be represented.

%%-----------------------------
%%      section 7
%%-----------------------------

\section{Accumulation of eigenpairs for small angles}
\label{s:7}
When $\theta$ tends to $0$, there is more and more room for eigenvectors between the two corners of the guide.

\begin{figure}[ht]
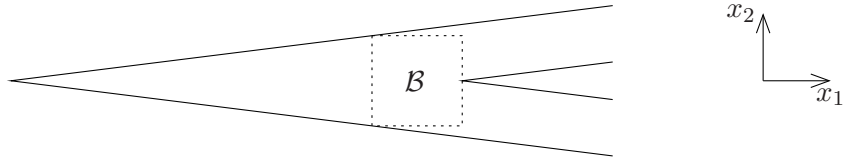

% 1. Definition of characteristic points
    \figinit{1mm}
    \figpt 1:(  0,  0)
    \figpt 2:( 80,  0)
    \figpt 3:( 80,  10)
    \figpt 4:( 100,  0)
    \figptshom 5 = 1, 3 /2, 0.25/
    \figpthom 7 := 3 /1, 0.75/
    \figpthom 8 := 3 /1, 0.6/
    \figptorthoprojline  9 := 8 /5, 7/
    \figptorthoprojline 10 := 8 /5, 1/
    \figptssym 23 = 3, 4, 5, 6, 7, 8, 9, 10 /1, 2/

% 2. Creation of the postscript file
    \def\MyPSfile{}
    \psbeginfig{\MyPSfile}
    \psaxes 4(9)
    \figptsaxes 11 : 4(9)
    \psline[23,1,3]
    \psline[26,5,6]
    \psset (dash=5)
    \psline[8,28,29,9,8]
  \psendfig

% 3. Writing text on the figure
    \figvisu{\figbox}{}{%
    \figinsert{\MyPSfile}
    \figwrites 11 :{$x_1$}(1)
    \figwritew 12 :{$x_2$}(1)
    \figwritew  5 :{$\cB$}(5)
    }
\centerline{\box\figbox}
\caption{The waveguide $\Omega_{\theta}$ and a Dirichlet box $\cB$ inside.}
\label{F:8}
\end{figure}

For any rectangular box $\cB$ contained in $\Omega_{\theta}$ like in Figure \ref{F:8}, by the monotonicity of Dirichlet eigenvalues (Example \ref{ex:1-3}), we know that for any $j$
\[
   \lambda_j(\theta) \le \lambda_j(\cB)
\]
where $\lambda_j(\theta)$ are the Rayleigh quotients of $\Delta^\Dir_{\Omega_\theta}$ and $\lambda_j(\cB)$ the Dirichlet eigenvalues on $\cB$. We choose $\cB$ in the form of a rectangle bounded by the vertical lines $x_1=-\alpha\pi$ and $x_1=0$, and the horizontal lines $x_2=\pm\beta\pi$. Thus $\alpha$ and $\beta$ satisfy, \cf{} \eqref{eq:Ot}
\[
   \alpha\in(0,\frac1{\sin\theta}),\quad \beta\pi = (-\alpha\sin\theta + 1)\frac\pi{\cos\theta}\,.
\]
The eigenvalues of the Dirichlet problem in $\cB$ are
\[
   \frac{k^2}{4\beta^2} + \frac{\ell^2}{\alpha^2},\quad k,\ell\in\{1,2,\ldots\}\,.
\]
We look for the eigenvalues $\lambda_j(\cB)$ less than $1$. Therefore $\alpha$ has to be chosen $>1$. Thus $\beta<(1-\sin\theta)/\cos\theta\le1$ for any $\theta$. As a consequence $k=1$ and the eigenvalues $\lambda_j(\cB)$ less than $1$ are necessarily of the form
\begin{align*}
   \lambda_j(\cB) &= \frac{1}{4\beta^2} + \frac{j^2}{\alpha^2} \\
   & = \frac{\cos^2\theta}{4(1-\alpha\sin\theta)^2} + \frac{j^2}{\alpha^2} \,.
\end{align*}
We optimize $\cB$: The minimum of $\lambda_j(\cB)$ is obtained for $\alpha$ such that
\[
   \frac{\sin\theta\cos^2\theta}{2(1-\alpha\sin\theta)^3} - \frac{2j^2}{\alpha^3} = 0
\]

\begin{figure}
\centerline{\includegraphics[scale=0.5]{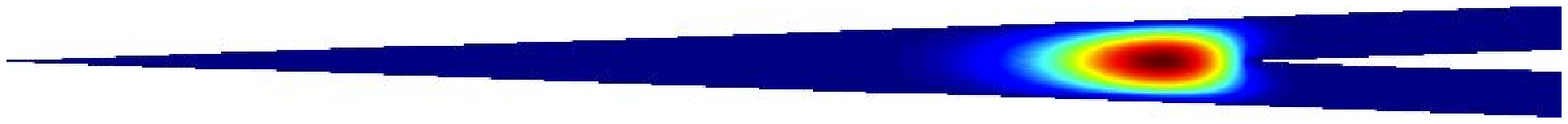}}
\centerline{$\lambda_1^\comp(\theta)= 0.32783 $}

\centerline{\includegraphics[scale=0.5]{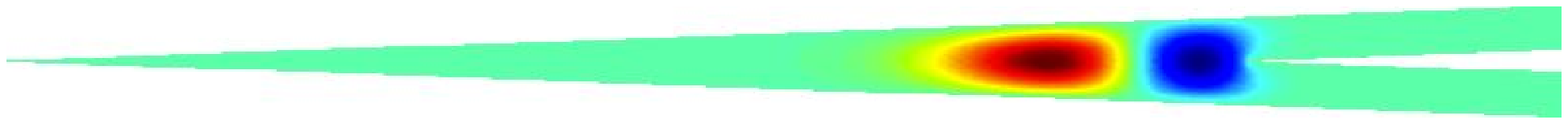}}
\centerline{$\lambda_2^\comp(\theta)= 0.40217 $}

\centerline{\includegraphics[scale=0.5]{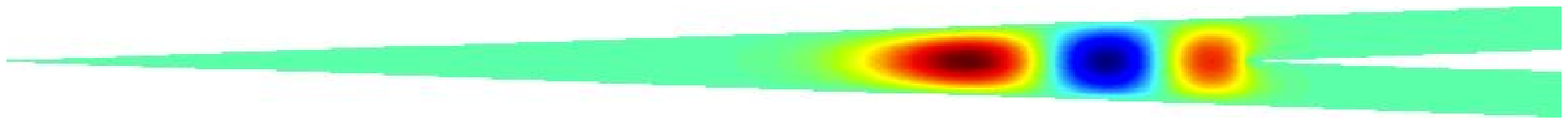}}
\centerline{$\lambda_3^\comp(\theta)= 0.47230 $}

\centerline{\includegraphics[scale=0.5]{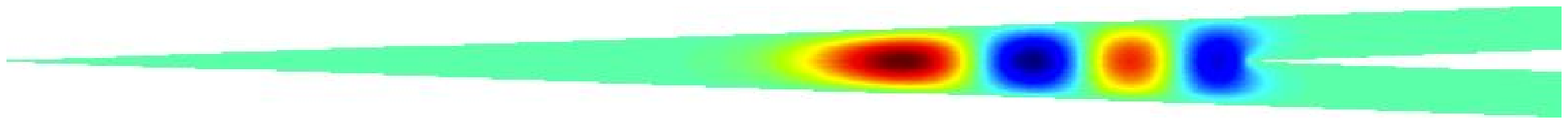}}
\centerline{$\lambda_4^\comp(\theta)= 0.54181 $}

\centerline{\includegraphics[scale=0.5]{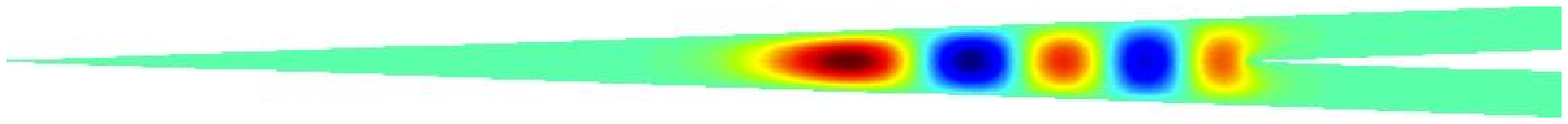}}
\centerline{$\lambda_5^\comp(\theta)= 0.61194 $}

\centerline{\includegraphics[scale=0.5]{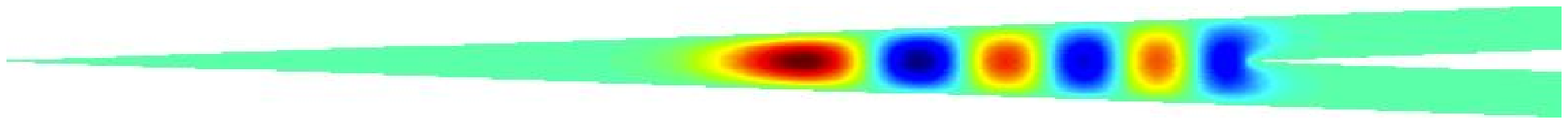}}
\centerline{$\lambda_6^\comp(\theta)= 0.68328 $}

\centerline{\includegraphics[scale=0.5]{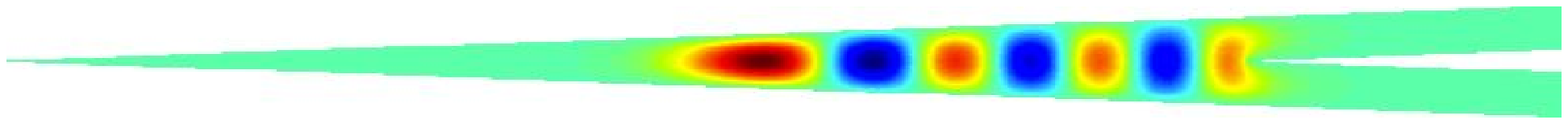}}
\centerline{$\lambda_7^\comp(\theta)= 0.75607 $}

\centerline{\includegraphics[scale=0.5]{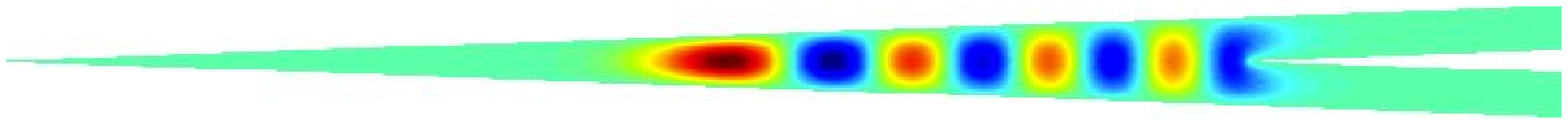}}
\centerline{$\lambda_8^\comp(\theta)= 0.83040 $}

\centerline{\includegraphics[scale=0.5]{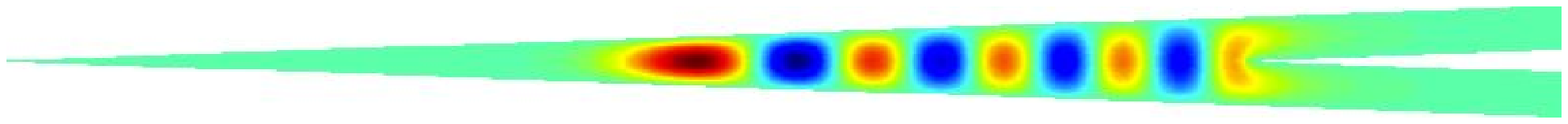}}
\centerline{$\lambda_9^\comp(\theta)= 0.90610 $}

\centerline{\includegraphics[scale=0.5]{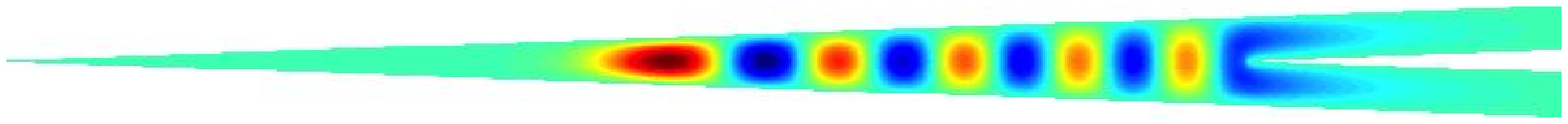}}
\centerline{$\lambda_{10}^\comp(\theta)= 0.98195 $}
\caption{Computations for $\theta=0.0226*\pi/2\sim2^\circ$ with the mesh M64S (see Figure \ref{F:mesh64S}). Numerical values of the 10 eigenvalues $\lambda_j(\theta)<1$. Plots of the associated eigenvectors in the physical domain. }
\label{F:11}
\end{figure}

\noindent
Since we are interested in the behavior as $\theta\to0$, we take without asymptotic loss
\[
   \alpha = 4^{1/3}j^{2/3}\sin^{-1/3}\theta
\]
which provides 
\begin{align*}
   \lambda_j(\cB) 
%   & = \frac{1}{4} + \frac{3}{4^{2/3}}\,j^{2/3}\sin^{2/3}\theta 
%   + \cO(j^{4/3}\sin^{4/3}\theta)\,.
   & = \frac14\Big( \frac{\cos^2\theta}{(1-4^{1/3}j^{2/3}\sin^{2/3}\theta)^2} 
   + 4^{1/3}j^{2/3}\sin^{2/3}\theta \Big)\,.
\end{align*}
As a consequence, as soon as the quantity $Z:=4^{1/3}j^{2/3}\sin^{2/3}\theta$ is less that the first root of the equation
\[
   \frac14\Big(\frac1{(1-Z)^2} + Z\Big) = 1,
\]
\ie{} for $Z\le 0.4679$, we have $\lambda_j(\cB)<1$ and, hence, $\lambda_j(\theta)<1$. This implies that the maximal number $J$ such that $\lambda_J(\cB)<1$  is greater than
\[
   0.4679^{3/2}\cdot 0.5 \,\sin^{-1}\theta \simeq 0.1601\,\sin^{-1}\theta\,.
\]
Therefore the number of eigenvalues of $\Delta^\Dir_{\Omega_\theta}$ less than $1$ tends to infinity (at least) like $\theta^{-1}$ as $\theta$ tends to $0$.

We present in Figure \ref{F:11} the computations of the all eigenpairs of $\Delta^\Dir_{\Omega_\theta}$ for the angle $\theta=0.0226*\pi/2$. We find 10 eigenvalues $<1$. Note that for this angle $\theta$, the numerical value of $0.1601\,\sin^{-1}\theta$ is $4.5103$.

%%-----------------------------
%%      section 8
%%-----------------------------

\section{Asymptotic behavior of eigenpairs for small angles}
\label{s:8}
We present in Figures \ref{F:9} and \ref{F:10} computations of the first eigenvector for smaller and smaller values of the angle $\theta$. We notice that the eigenvectors look similar, with the appearance of a short scale in the horizontal variable.

\begin{figure}[h]
\begin{tabular}{c@{\,}c@{\,}c}
\includegraphics[scale=0.25]{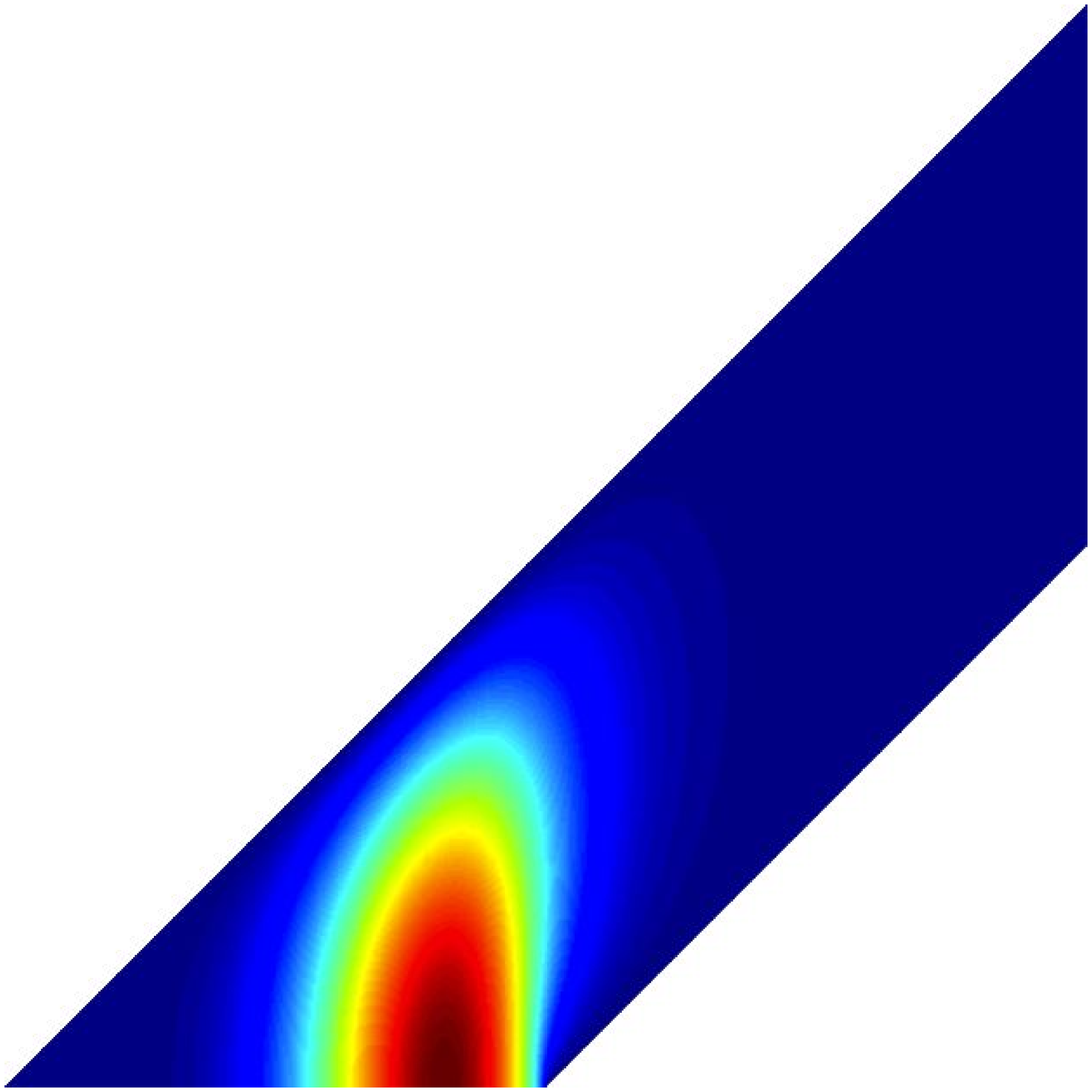}&
\includegraphics[scale=0.25]{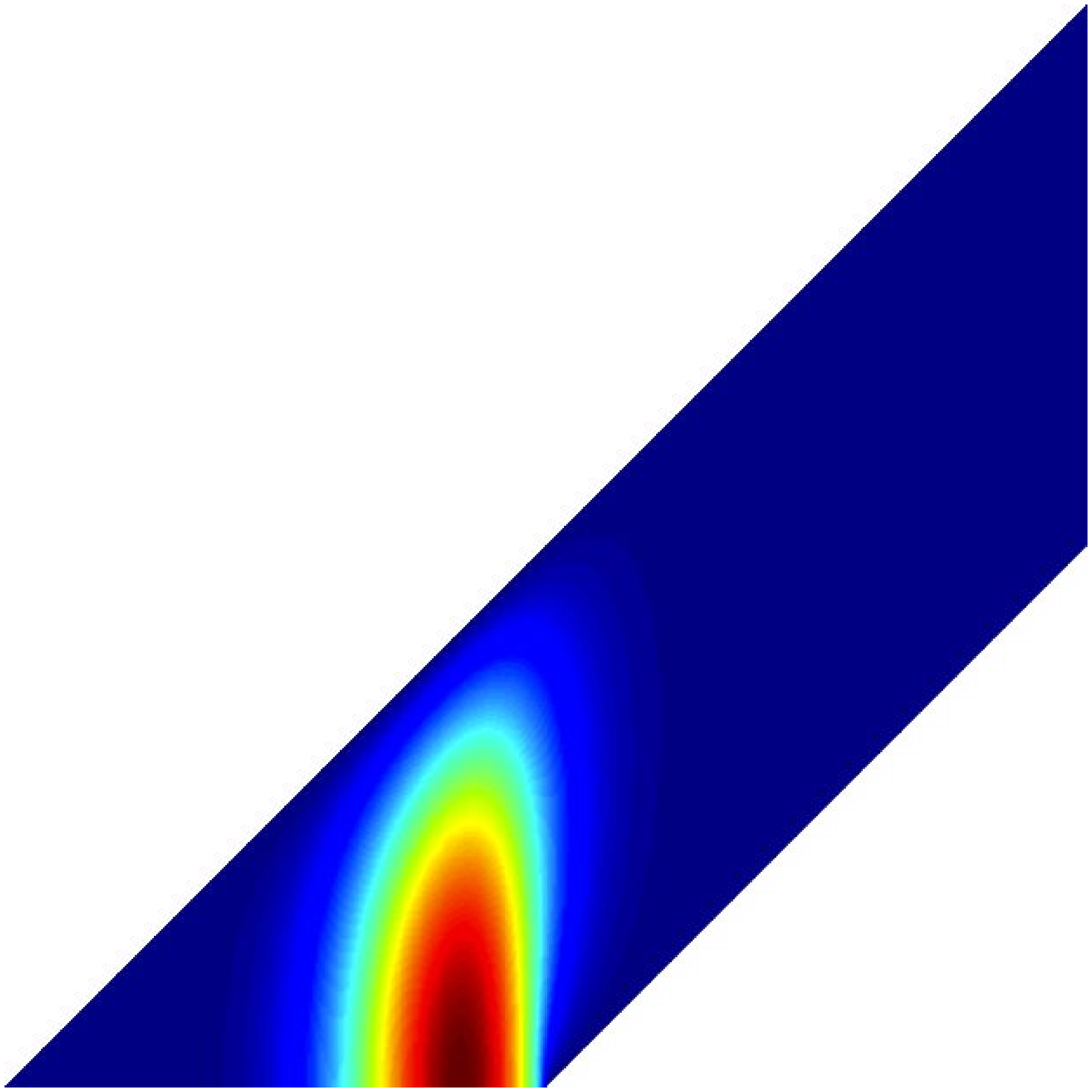}&
\includegraphics[scale=0.25]{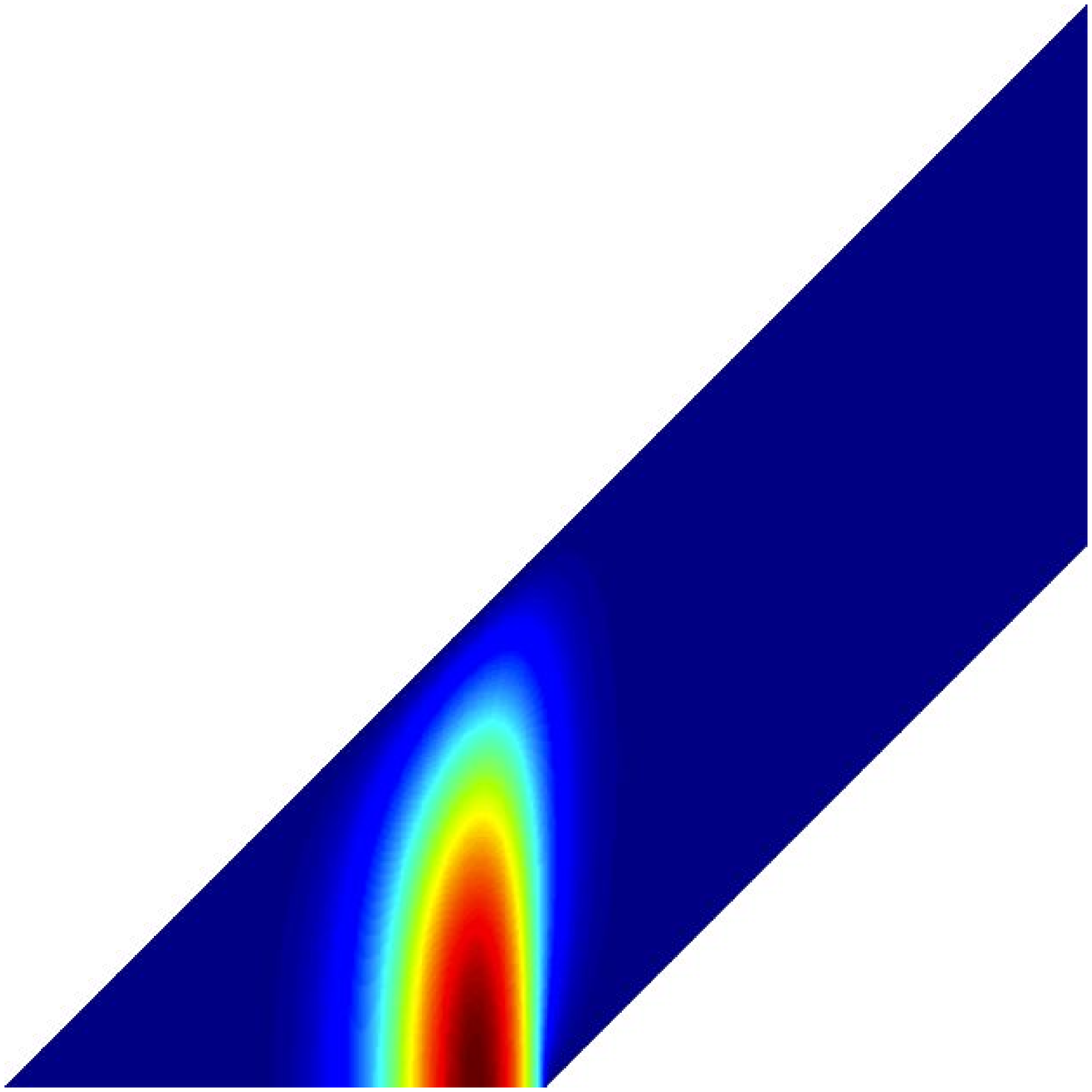}
\\
$\theta = 0.1482*\pi/2$ &
$\theta = 0.1032*\pi/2$ &
$\theta = 0.0701*\pi/2$ 
\\
$\lambda_1^\comp(\theta)= 0.56209 $&
$\lambda_1^\comp(\theta)= 0.48754 $&
$\lambda_1^\comp(\theta)= 0.42763 $
\end{tabular}
\caption{Computations for small angles. Plots in the computational domain $\Omega$.}
\label{F:9}
\end{figure}

\begin{figure}
\begin{tabular}{c@{\,}c@{\,}c}
\includegraphics[scale=0.25]{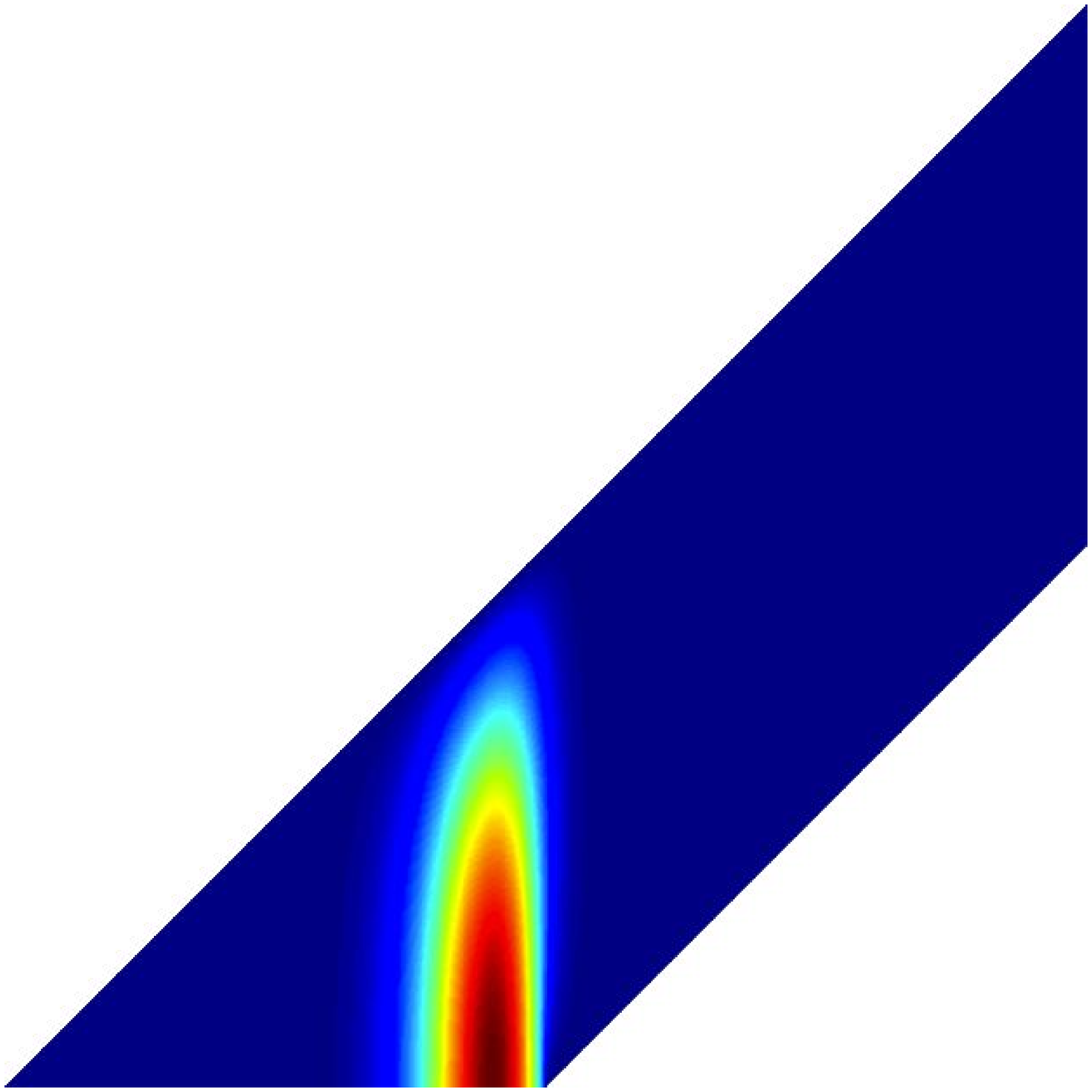}&
\includegraphics[scale=0.25]{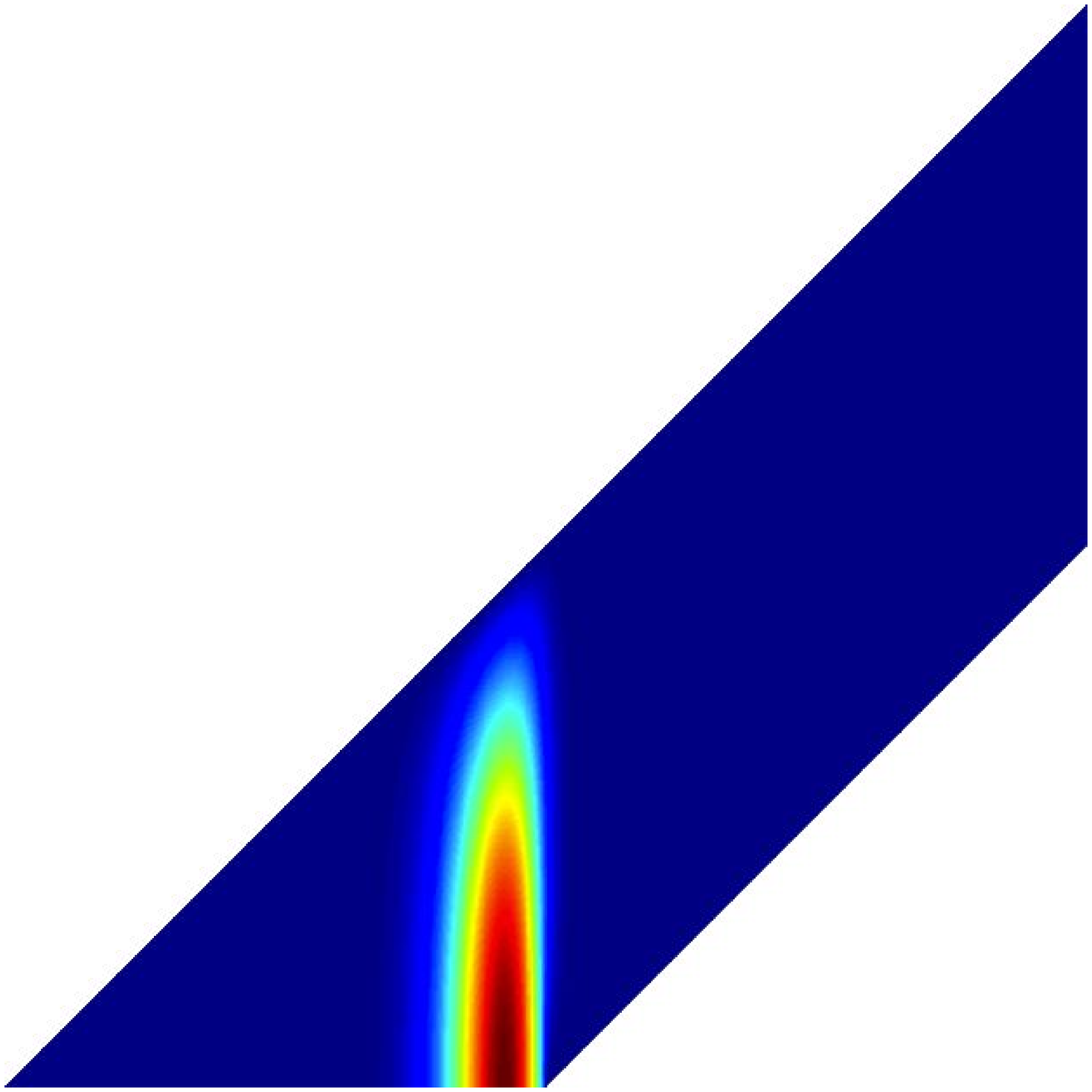}&
\includegraphics[scale=0.25]{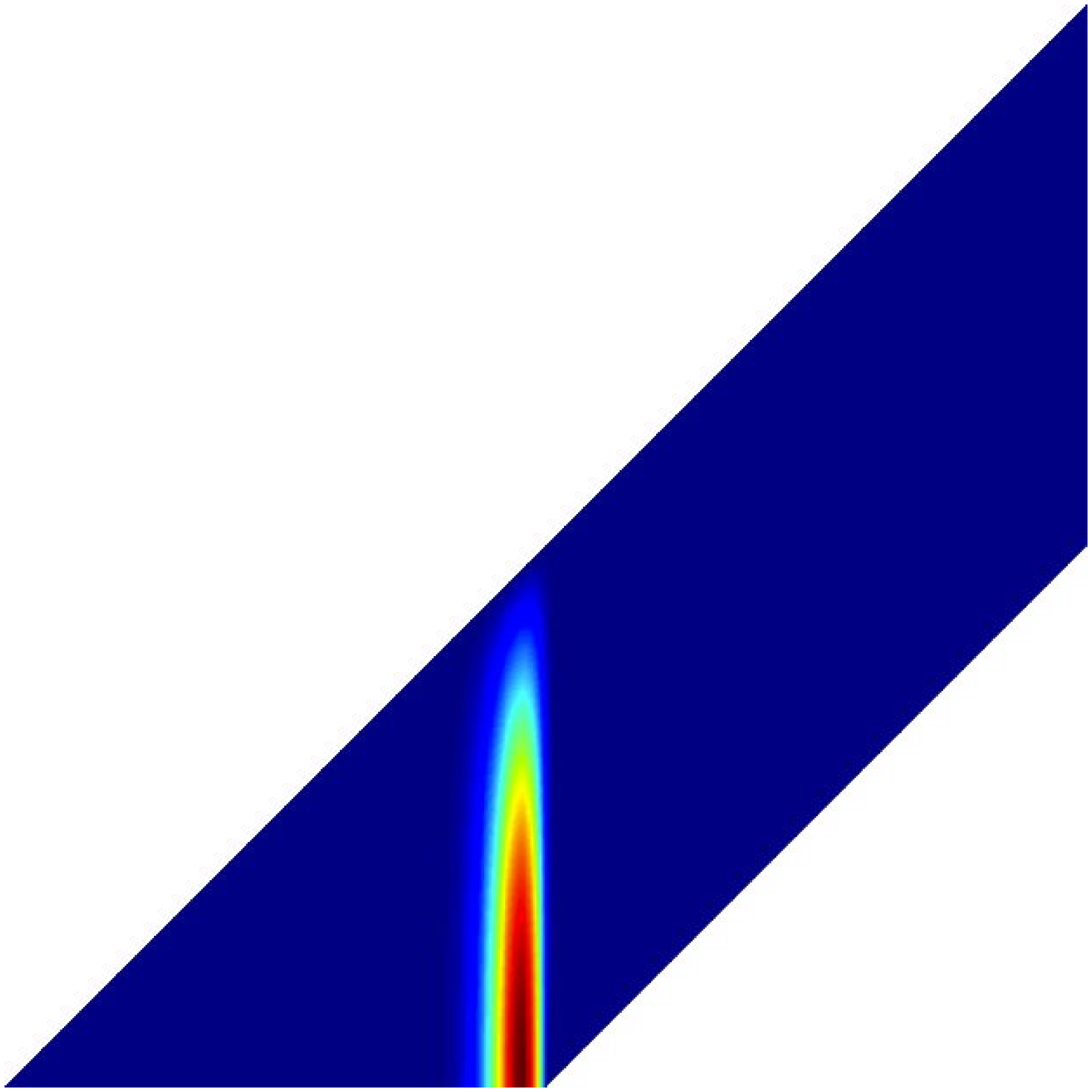}
\\
$\theta = 0.0416*\pi/2$ &
$\theta = 0.0270*\pi/2$ &
$\theta = 0.0112*\pi/2$ 
\\
$\lambda_1^\comp(\theta)= 0.37085 $&
$\lambda_1^\comp(\theta)= 0.33845 $&
$\lambda_1^\comp(\theta)= 0.29766 $
\end{tabular}
\caption{Computations for very small angles. Plots in the computational domain $\Omega$.}
\label{F:10}
\end{figure}

We perform in \cite{DauRay11} asymptotic expansions of the first eigenpairs at any order as $\theta\to0$, using techniques of semi-classical analysis (see for instance \cite{HelSj84, Hel88} and also \cite{Martinez89}). 
We briefly describe now some of the results proved there.

Let us recall that the eigenvalues $\lambda_j(\theta)<1$ of our operator $\Delta^\Dir_{\Omega_\theta}$ coincide with those of the scaled operator
\[
   \cL_\theta := -2\sin^2\!\theta\,\partial_{\uu}^2-2\cos^2\!\theta\,\partial_{\vv}^2 
\]
in the model half-guide $\Omega$. The construction and validation of asymptotic expansions for the eigenpairs of $\cL_\theta$ rely on a Born-Oppenheimer approximation. Roughly, this consists of four steps:
\begin{enumerate}
\item Definition of an operator $\cL_\theta^\BO$ in the horizontal variable $\uu$ by replacing the operator in the vertical variable $\cM_\uu := -2\cos^2\!\theta\,\partial_{\vv}^2$ in each slice $\uu$ = const.\ by its first eigenvalue $\Lambda(\uu)$ 
\item Semi-classical analysis of the eigenpairs of $\cL_\theta^\BO$ as $\theta\to0$.
\item Determination of a change of variables $(\uu,\vv)\mapsto(u,t)$ on $\Omega$ in order to exhibit a tensor product structure. Here appears the role of the limit operator $\cM_{0-}$ as $\uu \nearrow 0$. Its first eigenvector is $\vv\mapsto\cos(\vv/2)$.
\item Construction of expansions of eigenvectors in the new variables.
\end{enumerate}

\medskip\noindent
{\sc Step 1:} Definition of $\cL_\theta^\BO$.\\
Let $\cI(\uu)$ be the intersection of $\Omega$ with the vertical line of abscissa $\uu$.
\begin{itemize}
\item For $\uu\in(-\pi\sqrt2,0)$, $\cI(\uu) = (0,\uu+\pi\sqrt2)$. The operator $\cM_\uu$ has Neumann condition at $0$ and Dirichlet at $\uu+\pi\sqrt2$. Thus
\[
   \Lambda(\uu) = 2\cos^2\!\theta\,\frac{\pi^2}{4(\uu+\pi\sqrt2)^2}
\]
\item For $\uu\in(0,+\infty)$, $\cI(\uu) = (\uu,\uu+\pi\sqrt2)$. The operator $\cM_\uu$ has Dirichlet conditions at both ends. Thus
\[
   \Lambda(\uu) = \cos^2\!\theta.
\]
\end{itemize}
The Born-Oppenheimer operator is
\[
   \cL^\BO_\theta = -2\sin^2\!\theta\,\partial_{\uu}^2 + \Lambda(\uu).
\]

\medskip\noindent
{\sc Step 2:} Semi-classical analysis of $\cL_\theta^\BO$.\\
The operator $\cL_\theta^\BO$ can be viewed as a 1D Schr\"odinger operator with potential $\Lambda$. The potential has a well with bottom at $\uu=0$. The well is not smooth and has a triangular shape, see Figure \ref{F:12}.

\begin{figure}
\centerline{\includegraphics[scale=0.5]{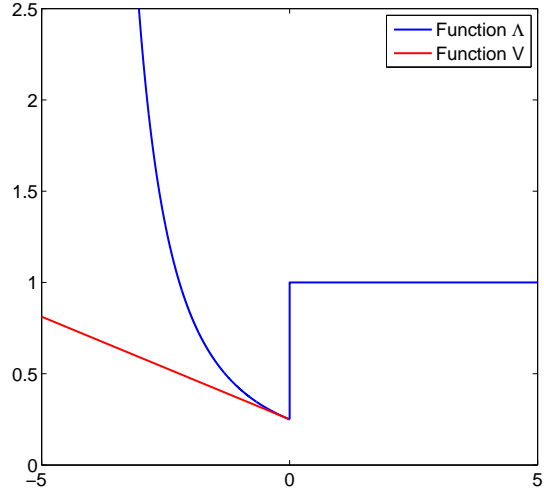}}
\caption{The Born-Oppenheimer potential $\Lambda$ and its left tangent at $\uu=0$ (here $\cos^2\!\theta$ is set to $1$).}
\label{F:12}
\end{figure}

The behavior of the eigenpairs of $\cL_\theta^\BO$ as $\theta\to0$ is governed by the Taylor expansion of the potential $\Lambda$ at the well bottom $\uu=0$, \ie{} by the tangent potential $V$ defined by
\[
   V(\uu) =  \cos^2\!\theta
   \begin{cases}
   \,\di\frac14 -\frac{\uu}{2\pi\sqrt{2}}, &\mbox{if $\uu<0$,} \\
   \,1, &\mbox{if $\uu>0$.}
\end{cases}
\]
The corresponding model problem is the problem of the behavior as $h\to0$ of the eigenpairs of the operator
\[
   \cH_h = -h^2\partial^2_\uu + \begin{cases}
   -\uu, &\mbox{if $\uu<0$,} \\
   \,1, &\mbox{if $\uu>0$.}
\end{cases}
\] 
The potential barrier on $(0,+\infty)$ produces a Dirichlet condition at $\uu=0$ for the leading terms of the asymptotics: We are led to the Airy-type eigenvalue problem
\begin{equation}
\label{eq:airyh}
   -h^2\partial^2_\uu \psi_h - \uu\psi_h = E_h\psi_h \ \ \mbox{on}\ \  (-\infty,0),\quad
   \mbox{with}\quad \psi(0)=0.
\end{equation}
The change of variables $\uu\mapsto X = \uu h^{-2/3}$ transforms the above equation into the reverse Airy equation
\[
   -\partial^2_X \Psi - X\Psi = \mu\Psi \ \ \mbox{on}\ \ (0,+\infty),\quad
   \mbox{with}\quad\Psi(0)=0, \quad (\mbox{where } \mu=h^{-2/3}E_h)
\]
whose solutions can be easily exhibited: As we will see, the eigenvalues $\mu$ are the zeros of the reverse Airy function $\A(X):=\Ai(-X)$, where $\Ai$ is the standard Airy function. The zeros of $\A$ form an increasing sequence of positive numbers, which we denote by $z_\A(j)$, $j\ge1$.

Since $-\A''(X)-X\A(X) = 0$ we have for any $E\in\xR$
\[
   -\A''(X)-(X-E)\A(X) = E\A(X) \quad\mbox{ \ie}\quad
   -\A''(X+E)-X\A(X+E) = E\A(X+E)\,.
\]
If $E=z_\A(j)$, then $\A(X+E)$ vanishes at $X=0$, hence
\[
   \Big(z_\A(j), \A\big(X+z_\A(j)\big) \Big) \quad\mbox{is an eigenpair}.
\]
Conversely, all eigenpairs are of this form.

We deduce that the eigenvalues of problem \eqref{eq:airyh} are $E_h=h^{2/3}z_\A(j)$ and the associated eigenvectors are $\psi(\uu)=\A(\uu h^{-2/3} + z_\A(j))$. 

Coming back to our operator $\cL^\BO_\theta$, we prove in\cite{DauRay11} that its eigenvalues have asymptotic expansions of the form
\begin{equation}
\label{eq:8-1}
   \lambda^\BO_j(\theta) \underset{\theta\to0}{\simeq} \frac14 + \frac{ 2\theta^{2/3}  z_\A(j)}
   {(4\pi\sqrt{2})^{2/3}} + \sum_{n\ge3} \theta^{n/3}\gamma^\BO_{j,n}\,,
\end{equation}
where $\gamma^\BO_{j,n}$ are some real coefficients.
The eigenvectors have expansions in powers of $\theta^{1/3}$, using the scale $\uu h^{-2/3}$ for $\uu<0$ and $\uu h^{-1}$ for $\uu>0$.

\medskip\noindent
{\sc Step 3:} Tensorial structure of $\cL_\theta$.\\
On the left part of $\Omega$, \ie{} its triangular part $\Tri$ in the half-plane $\uu<0$, we perform a change of variables to transform $\Tri$ into a square: 
\begin{equation}
\label{eq:Eut-}
   \Tri\ni(\uu,\vv) \mapsto (\uu,t)\in(-\pi\sqrt{2},0)\times(0,\pi\sqrt{2}) ,
   \quad\mbox{with}\quad t=\frac{\vv\,\pi\sqrt{2}}{\uu+\pi\sqrt{2}}.
\end{equation}
On the right part of $\Omega$, \ie{} its strip part contained in the half-plane $\uu>0$, we perform a change of variables to transform it into a horizontal strip:
\begin{equation}
\label{eq:Eut+}
   (\uu,\vv) \mapsto (\uu,\tau)\in(0,+\infty)\times(0,\pi\sqrt{2}) ,
   \quad\mbox{with}\quad \tau=\vv-\uu.
\end{equation}
\begin{figure}[ht]
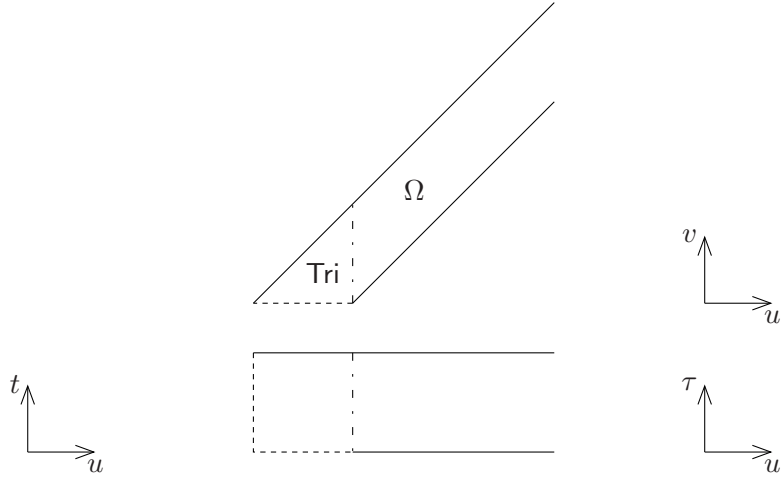

% 1. Definition of characteristic points
    \figinit{1mm}
    \figpt 1:(  0,  0)
    \figpt 2:( 40,  0)
    \figpt 3:( 40,  40)
    \figpt 4:( 60,  0)
    \figpt 8:( -30,  0)
    \figptshom 5 = 1, 3 /2, 0.67/
    \figvectP 100 [6,3]
    \figptstra 7 = 5/1, 100/
    \figptstra 21 = 1, 5, 2, 4, 8/-0.5, 100/
    \figptstra 31 = 21, 22, 23, 24, 25/-1, 100/

% 2. Creation of the postscript file
    \def\MyPSfile{}
    \psbeginfig{\MyPSfile}
    \psaxes 4(9)
    \figptsaxes 11 : 4(9)
    \psaxes 34(9)
    \figptsaxes 13 : 34(9)
    \psaxes 35(9)
    \figptsaxes 15 : 35(9)
    \psline[1,3]
    \psline[5,6]
    \psline[32,33] 
    \psline[23,22,21]
    \psset (dash=4)
    \psline[5,1]
    \psline[21,31,32]
    \psset (dash=9)
    \psline[5,7]
    \psline[22,32]
  \psendfig

% 3. Writing text on the figure
    \figvisu{\figbox}{}{%
    \figinsert{\MyPSfile}
    \figwrites 11 :{$\uu$}(1)
    \figwritew 12 :{$\vv$}(1)
    \figwrites 13 :{$\uu$}(1)
    \figwritew 14 :{$\tau$}(1)
    \figwrites 15 :{$\uu$}(1)
    \figwritew 16 :{$t$}(1)
    \figwritegce 1 :{$\Tri$}(7,4)
    \figwritegce 1 :{$\Omega$}(20,15)
    }
\centerline{\box\figbox}
\caption{The model waveguide $\Omega$ and the change of variables.}
\label{F:13}
\end{figure}
This allows to work with the Taylor expansions of the transformed operator around $\uu=0$, on the left and on the right. The operator $\cM_{0-} = -2\partial^2_t$ appears on the left with Dirichlet condition at $t=1$ and Neumann at $t=0$. Its first eigenvector is $\cos\frac t 2$, associated with the eigenvalue $\frac14$. The operator $\cM_{0+} = -2\partial^2_\tau$ appears on the right with Dirichlet condition at $t=0,1$.

\medskip\noindent
{\sc Step 4:} Asymptotics of the eigenpairs of $\cL_\theta$.\\
The projection on the eigenvector $t\mapsto\cos\frac t 2$ appears naturally at the first step of the expansion (this projection is sometimes called Feshbach or Grushin projection) and lets appear $\cL^\BO_\theta$. A complete asymptotics for eigenpairs can be constructed, and in a further step, validated. As a result $\lambda_j(\theta)$ has an expansion similar to $\lambda^\BO_j(\theta)$, with different coefficients for $n\ge3$:
\begin{equation}
\label{eq:8-2}
   \lambda_j(\theta) \underset{\theta\to0}{\simeq} \frac14 + \frac{ 2\theta^{2/3}  z_\A(j)}
   {(4\pi\sqrt{2})^{2/3}} + \sum_{n\ge3} \theta^{n/3}\gamma_{j,n} \,
\end{equation}
where $\gamma_{j,n}$ are some real coefficients. In Figure \ref{F:EVtg} we display the functions $(\theta*2/\pi)^{2/3}\mapsto \lambda_j(\theta)$ for $j=1,2$ (computed by finite elements) and their linear approximation as $\theta\to0$, corresponding to the two-term approximation in \eqref{eq:8-2}, \ie{} $(\theta*2/\pi)^{2/3}\mapsto \frac14 + \frac{ 2\theta^{2/3}  z_\A(j)}
   {(4\pi\sqrt{2})^{2/3}} $.
\begin{figure}
\centerline{\includegraphics[scale=0.66]{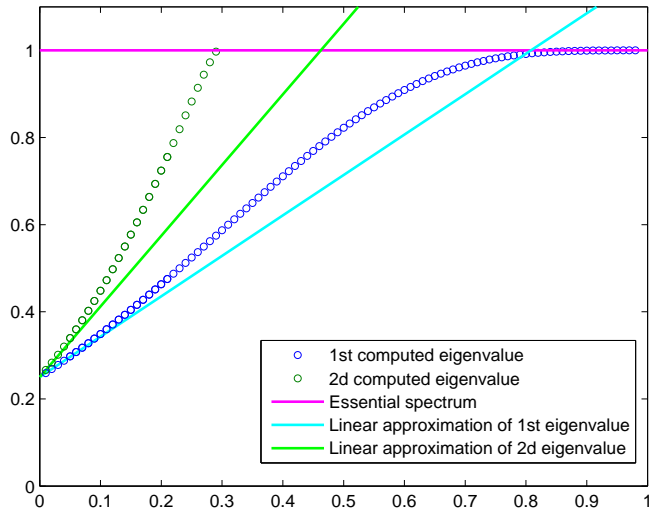}}
\vskip-3.5ex
\caption{The eigenvalues $\lambda_1$ and $\lambda_2$ as functions of $(\theta*2/\pi)^{2/3}$.}
\label{F:EVtg}
\end{figure}

In the reference domain $\Omega$, the eigenvectors of $\cL_\theta$ have a multiscale expansion in powers of $\theta^{1/3}$. On the left side, there are two scales, the ultra-short scale $(\uu \theta^{-1},\vv)$ and the short scale $(\uu\theta^{-2/3},\vv)$, on the right side there is only the ultra-short scale. The asymptotics is dominated by its first term, which is
\[
   \A\Big(\frac\uu{\theta^{2/3}}-z_\A(1)\Big)\cos\frac{t}2\ .
\]
This structure explains the results seen in Figures \ref{F:9}-\ref{F:10}.

%%-----------------------------
%%      section 9
%%-----------------------------

\section{Convergence of the finite element method}
\label{s:9}
We now present some aspects of the computation process that led to the numerical results shown in the previous sections.

As described in the section \ref{ss:CompLA}, the eigenvalue problem writes
$\cL_\theta \psi = \lambda \psi$ in the domain $\Omega$ with homogeneous Dirichlet conditions on its boundary, except on the horizontal segment where Neumann conditions are set, see identity \eqref{eq.Ltheta} and Figure \ref{F:6}. The associate bilinear form is
\[
   b(\psi,\psi') = \int_\Omega 2\sin^2\theta (\partial_u\psi\,\partial_u\psi' )
   + 2\cos^2\theta (\partial_v\psi\,\partial_v\psi') \xdif u\xdif v
\] 
defined on the corresponding form domain 
\[
   V = \{\psi\in\rH^1(\Omega):\ \psi=0\ \mbox{on}\ \partial_\Dir\Omega\}.
\]
The eigenvalue problem writes in variational form: find non-zero $\psi\in V$ and $\lambda$ such that
\[
   \forall\psi'\in V,\quad b(\psi,\psi') = \lambda\big\langle\psi,\psi'\big\rangle_{\xLtwo(\Omega)}.
\]
By Galerkin projection on a finite dimension subspace $V_{\sf fds}$ of $V$, this problem can be rewritten as the generalized eigenvalue problem: find the eigenpairs $(\lambda, w)$ such that $Sw = \lambda Mw$, where $S$ and $M$ are the stiffness and mass matrices associated with a basis $(\Psi_1,\ldots,\Psi_N)$ of $V_{\sf fds}$:
\[
   S = \Big( b(\Psi_j,\Psi_k) \Big)_{1\le j,k\le N} \quad\mbox{and}\quad
   M = \Big( \big\langle\Psi_j,\Psi_k\big\rangle_{\xLtwo(\Omega)} \Big)_{1\le j,k\le N} .
\]

The computation process consists of two main steps: first, using a finite element method that leads to the two matrices $S$ and $M$, and second, using an algorithm to compute the eigenpairs. In the following two sections, we focus on the algorithm for the computation of the eigenpairs and then on the influence of the choice of meshes and polynomial degrees in the finite element method. 
\par\smallskip
All the computations have been done in double precision arithmetic, on a iMac computer (4 GB memory, 3.06 GHz Intel Core 2 Duo processor).

%-------
\subsection{Computation of the eigenpairs}
We have first written a program using the Fortran 77 finite element library ``Melina" \cite{Melina}, which also provides a routine to compute eigenpairs of real symmetric eigenvalue problems. This routine uses an algorithm based on subspaces iterations. We call {\sf Mel} this method.

We have also written another program using the C++ version of the previous library, called ``Melina++". The computation of the eigenpairs have been made using the well-known package ARPACK++ \cite{ARPACK++}. We call {\sf Arp} this method. We thus were able to reproduce the results obtained with the {\sf Mel} method, which is a cross validation of both methods.

Here, we compare the computation of the eigenpairs with the two methods. For this purpose, in both cases, we fix the parameters governing the finite element part: interpolation degree 6 at Gauss-Lobatto points,
quadrature rule of degree 13.
\par\smallskip
We have selected two test configurations:
\begin{itemize}
\item ``large angle" configuration: $\theta = 0.9702*\pi/2$, \ 4 eigenvalues computed, mesh M4L with 656 triangles (see Figure \ref{F:mesh4L}), leading to matrices of size 12325$\times$12325;
\item ``small angle" configuration: $\theta = 0.0226*\pi/2$, \ 12 eigenvalues computed, mesh M64S with 6144 triangles (see Figure \ref{F:mesh64S}), leading to matrices of size 111265$\times$111265.
\end{itemize}

% Maillage grand angle
\begin{figure}[p]
\includegraphics[scale=0.8]{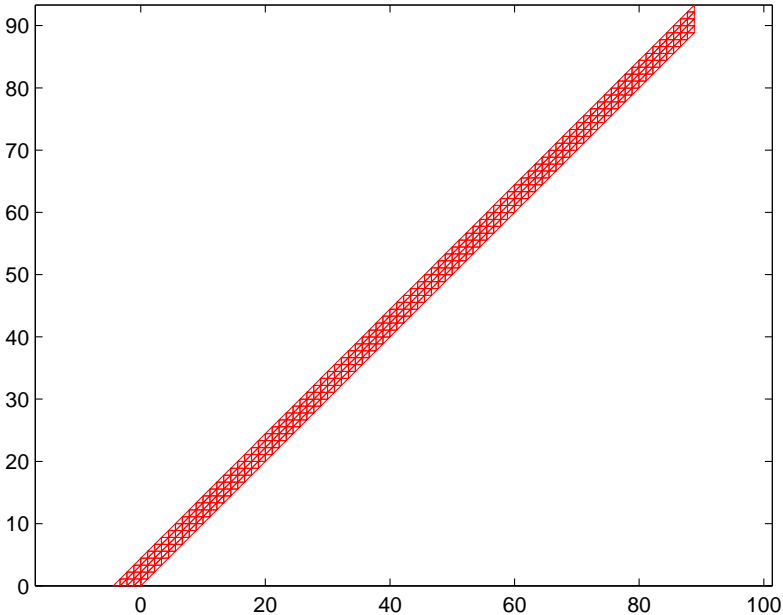}
\vskip-4ex
\caption{Mesh M4L for large angle, 656 triangles.}
\label{F:mesh4L}
\end{figure}
% Maillage petit angle
\begin{figure}[p]
\includegraphics[scale=0.8]{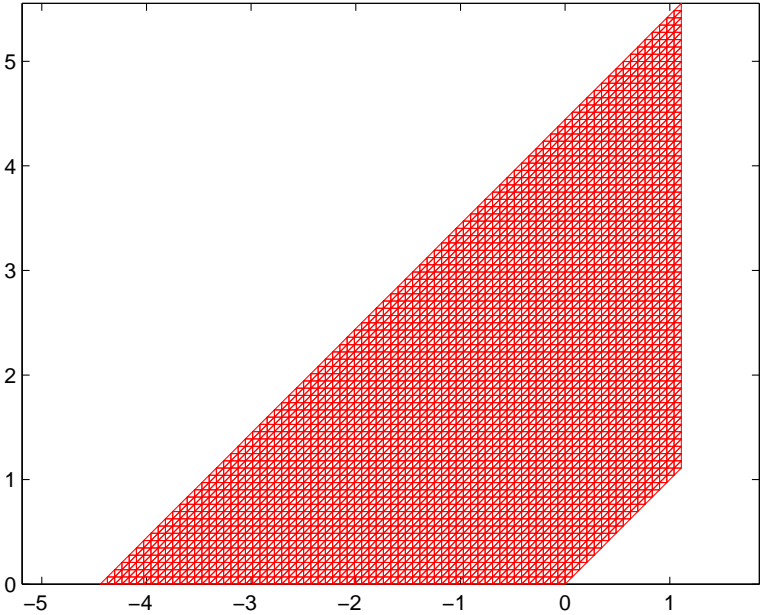}
\vskip-4ex
\caption{Mesh M64S for small angle, 6144 triangles.}
\label{F:mesh64S}
\end{figure}

Although the internal algorithms of the two methods are different, the parameters governing the computation of the eigenvalues are the same: a tolerance $\varepsilon$ that controls the end of the iteration process, and the dimension $N_{\sf sub}$ of the subspaces involved in the subspace iteration --- $N_{\sf sub}$ is at least equal to $N_{\sf val}+1$, where $N_{\sf val}$ is the number of desired eigenvalues. We let $N_{\sf sub}$ vary between $10$ and $70$ and ran the programs for $\varepsilon=10^{-n}$, $n=4,5,6,7$, while recording the number of iterations and the CPU time needed. The CPU time of the finite element part, i.e. the computation of the matrices $S$ and $M$, does not depend on these parameters. It appears on the graphs as {\sf CPUef}.

The results are gathered on Figure \ref{F:CalGdAng} for the ``large angle" configuration and on Figure \ref{F:CalPtAng} for the ``small angle" configuration. We can observe that the number of iterations decreases as $N_{\sf sub}$ increases. A horizontal line at the beginning of the first two graphs indicates a failed computation (no convergence after the chosen maximal number of iterations). For the values tested, the parameter $\varepsilon$ does not play a significant role on the number of iterations, nor really on the CPU time (although it does on the residual). Since the algorithms used in the two methods are not the same, the number of iterations cannot be compared directly. They are mainly a good indicator of the computation process behavior. 

It is also worth to notice that the memory requirements increase as the subspace dimension $N_{\sf sub}$ increases. This parameter is difficult to handle since it is problem dependent. These graphs suggest that there is no critical value: it can be chosen large enough to ensure computation to succeed, but not too large, mainly because of memory considerations.
\par
The CPU graphs of the {\sf Arp} method seem a bit chaotic: both algorithms need an initial vector which is chosen as a random vector by default. From our experience, we can say that this method is more sensible to this initial vector than the {\sf Mel} method. Finally, these graphs show clearly that the {\sf Arp} method is more efficient than the {\sf Mel} method.
% Calculs grand angle
\begin{figure}[h]
\begin{tabular}{cc}
\includegraphics[scale=0.48]{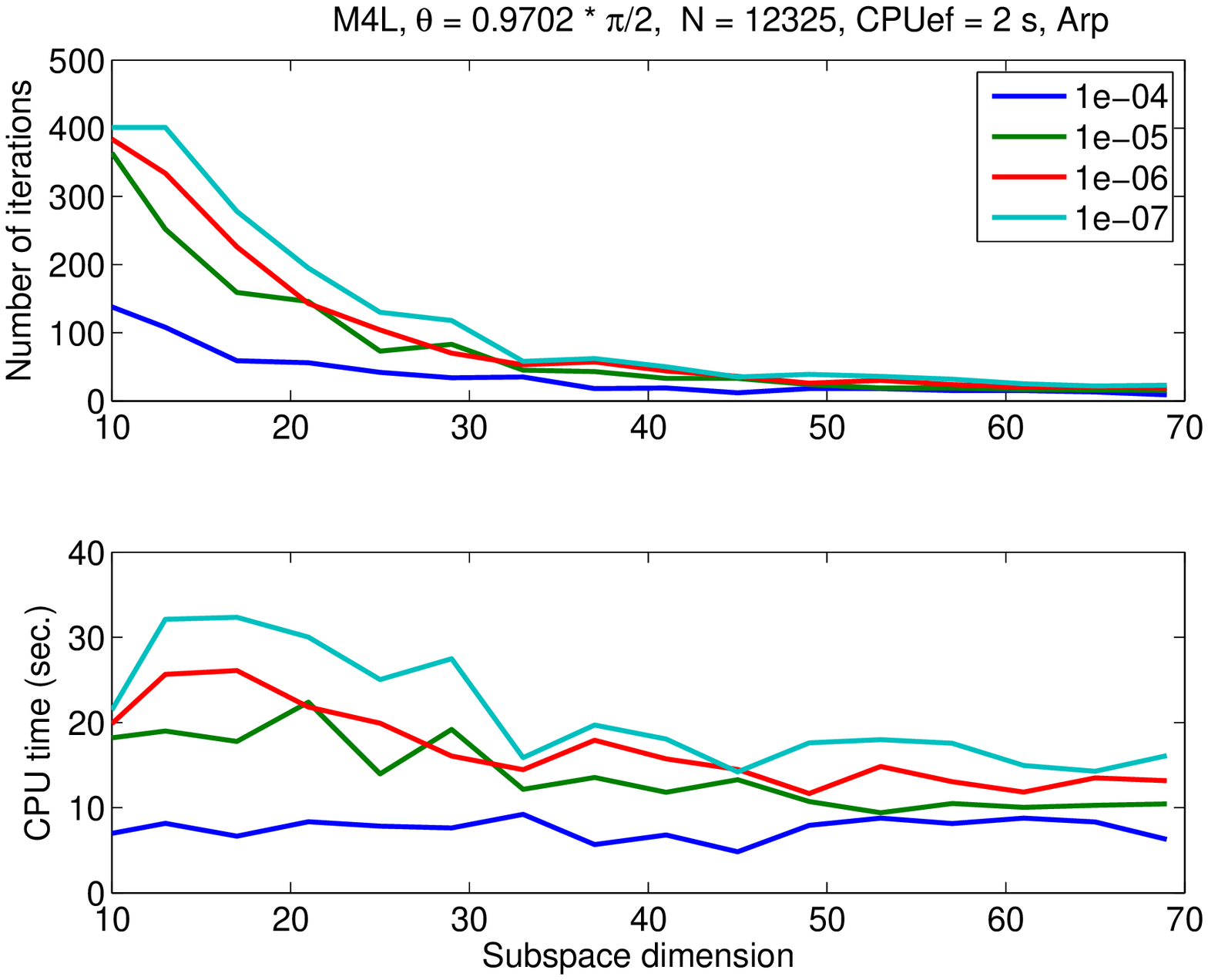}
\includegraphics[scale=0.48]{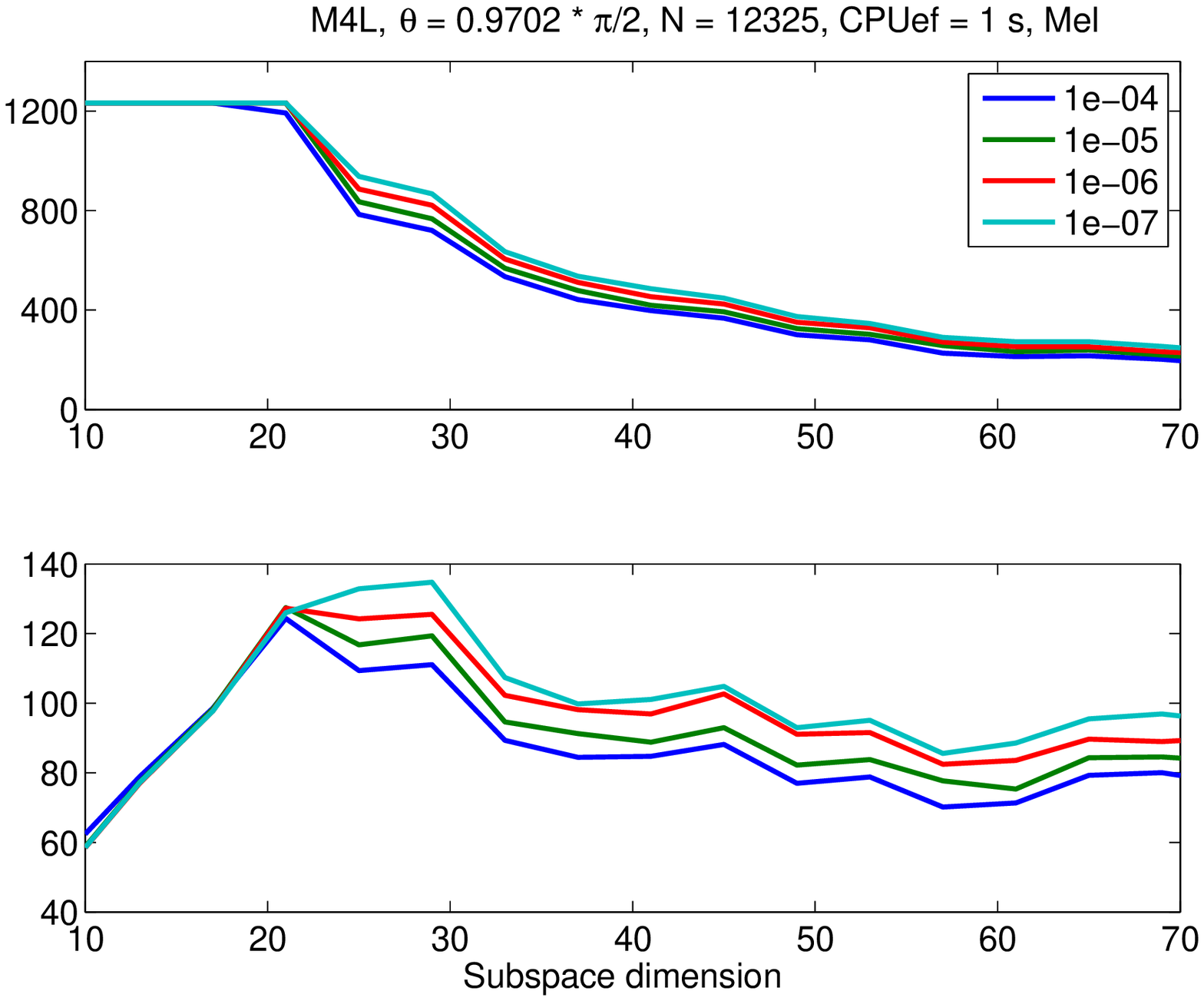}
\end{tabular}
\caption{Comparison of the behavior of the methods for the large angle $\theta=0.9702*\pi/2$.}
\label{F:CalGdAng}
\end{figure}
% Calculs petit angle
\begin{figure}[h]
\begin{tabular}{cc}
\includegraphics[scale=0.48]{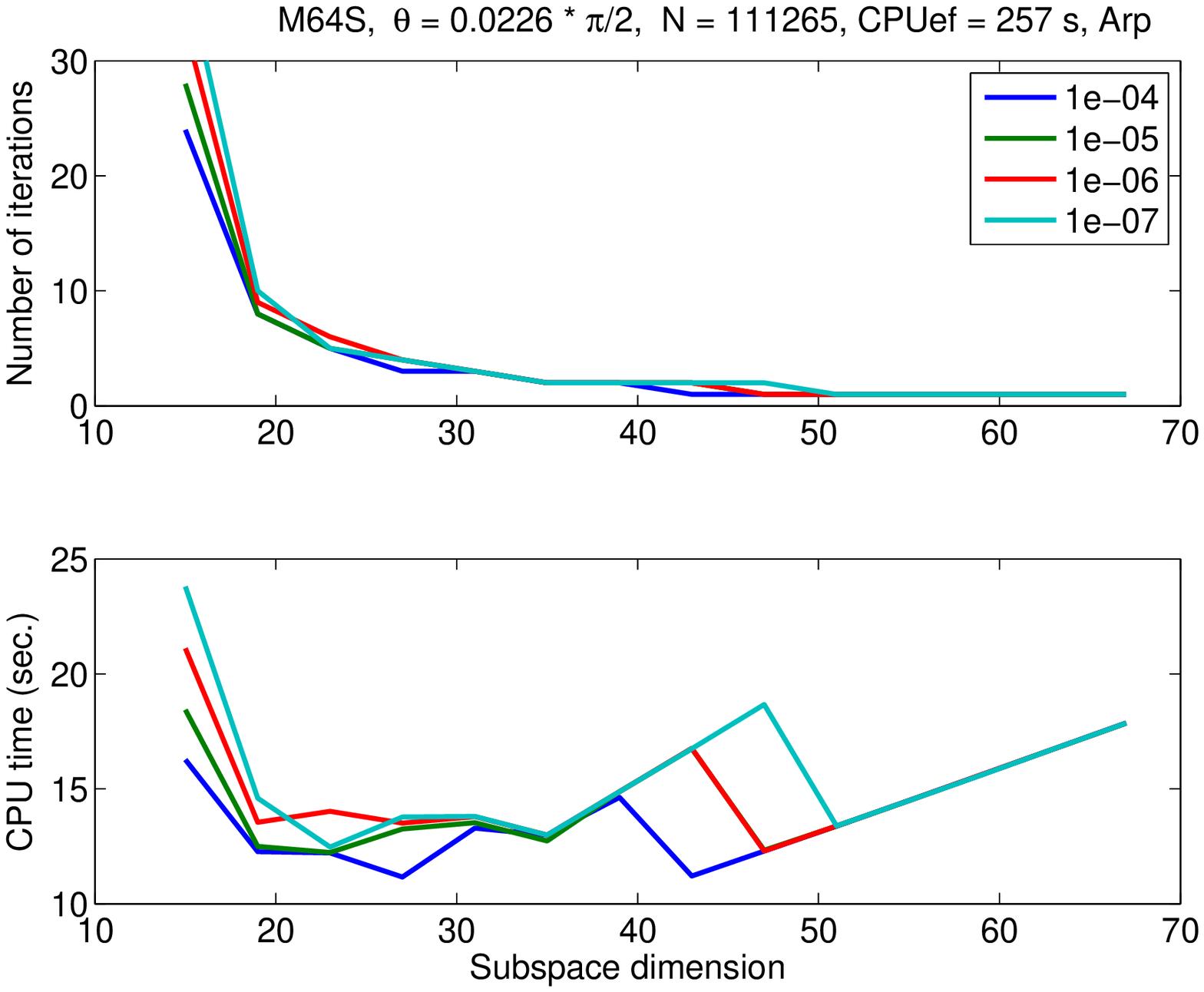}
\includegraphics[scale=0.48]{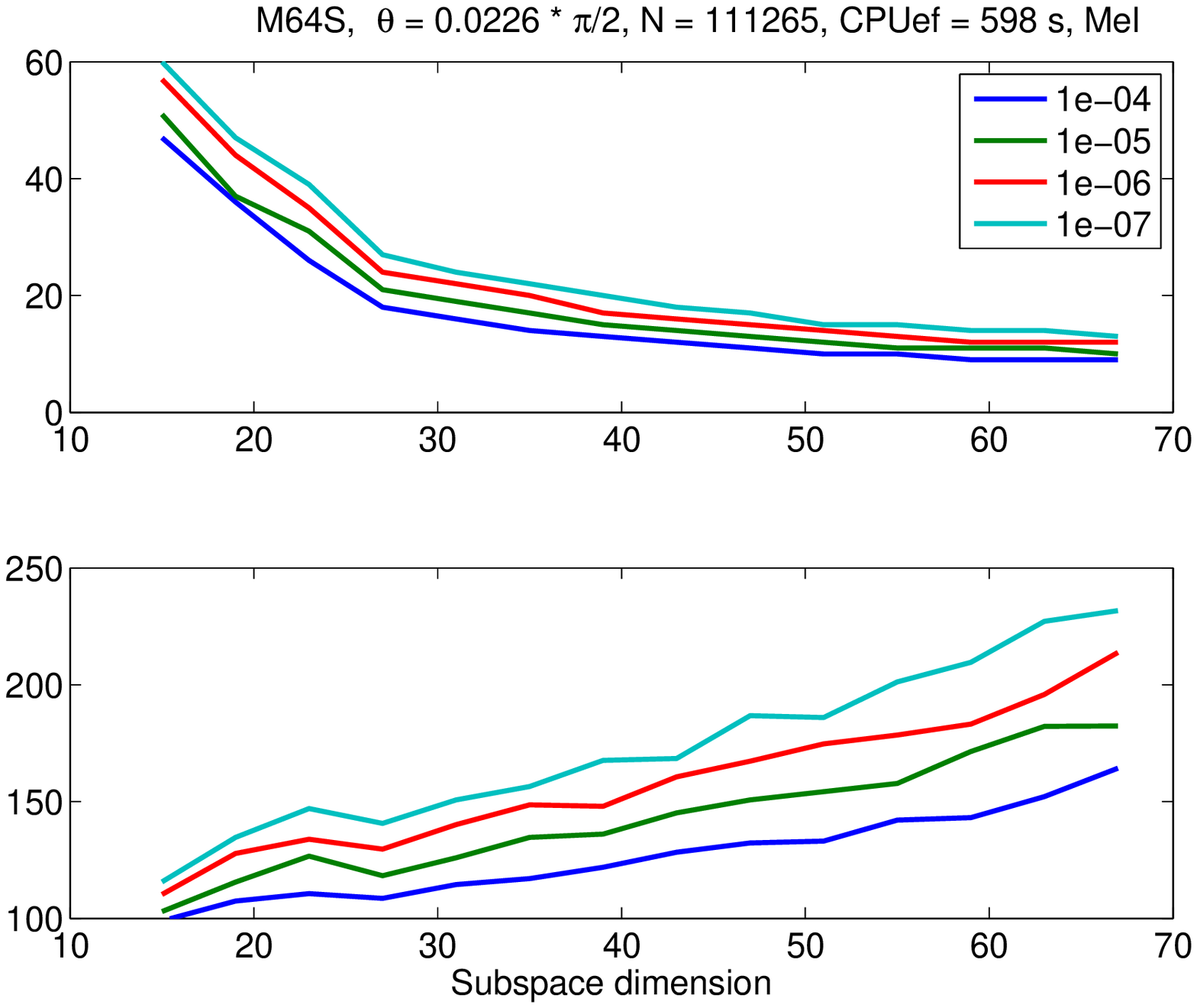}
\end{tabular}
\caption{Comparison of the behavior of the methods for the small angle $\theta=0.0226*\pi/2$.}
\label{F:CalPtAng}
\end{figure}

\begin{rmrk}
We are interested in the smallest eigenvalues of the problem $Sw = \lambda Mw$. Both methods {\sf Arp} and {\sf Mel} provide an option to choose in which end of the spectrum the wanted eigenvalues are to be searched. Let us mention that these algorithms perform better at computing the largest eigenvalues in general (this is due to the fact that they are ultimately based on the power method). In our case, this is typical and computation nearly always fails if the algorithms are asked to compute the smallest eigenvalues of this problem.
For example, with ARPACK, we can observe the following behavior:
\begin{itemize}
\item in the ``large angle" configuration, for $\varepsilon=10^{-4}$ and $N_{\sf sub}=17$ or $45$, the computation fails ;
\item in the ``small angle" configuration, for $\varepsilon=10^{-4}$ and $N_{\sf sub}=27$ or $43$, the computation fails ; if we change the mesh to a coarser one (M16S) with 384 triangles, the computation fails for $N_{\sf sub}=27$ and succeeds for $N_{\sf sub}=43$ with a CPU time of 236 sec.
\end{itemize}
Thus, the correct strategy is to compute the largest eigenvalues of the eigenvalue problem $Mw = \nu Sw$ and retrieve the wanted eigenvalues $\lambda$ by inversion ($\lambda = 1/\nu$ and the eigenvectors are the same). All the computations presented here have been carried on using this strategy. To conclude this remark, let us mention that the computation that took 236 sec. with the wrong strategy, takes 0.47 sec. with the right one.
\end{rmrk}

%-------
\subsection{Influence of meshes and polynomial degrees in the finite element method}
We now choose the small angle $\theta=0.0226*\pi/2$ and fix the parameters governing the computation of the eigenvalues: $\varepsilon=10^{-6}$, $N_{\sf sub}=25$, $N_{\sf val}=10$, since there are $10$ eigenvalues $<1$ as shown on Figure \ref{F:11}.

We have built five nested meshes of the same computational domain $\Omega$, called M4S, M8S (see Figure \ref{F:mesh4S_8S}), M16S, M32S (see Figure \ref{F:mesh16S_32S}) and M64S (see Figure \ref{F:mesh64S}). The number $n$ in the name (M$n$S) is the number of segments of the subdivision of the horizontal (and vertical) boundary of $\Omega$. It indicates the characteristic diameter $h$ of the triangles, which is halved from a mesh to the next one in the list. Thus, the number of triangles is multiplied by 4 from a mesh to the next one.

We have computed the 10 eigenvalues for $\varepsilon=10^{-8}$ using the finest mesh M64S and considered the first and last eigenvalues obtained as reference values. We denote them by $\lambda_1^{\sf ref}$ and $\lambda_{10}^{\sf ref}$.

For each mesh, we let the interpolation degree $k$ vary from 1 to 6, while recording the differences $|\lambda_1 - \lambda_1^{\sf ref}|$ and $|\lambda_{10} - \lambda_{10}^{\sf ref}|$ obtained for each degree. The results are gathered on Figure \ref{F:converg}. Each point on the graphs correspond to the result of a computation. Points corresponding to the same mesh have been linked together by a line to make the graphs more readable ; the corresponding degree is written below. The graphs show the convergence of the first and last computed eigenvalues:
\begin{itemize}
\item with respect to the interpolation degree for a given mesh ;
\item with respect to the mesh, for a given degree.
\end{itemize}
The number $N$ of degrees of freedom (d.o.f.) of the problem, which is the dimension of the matrices, is roughly proportional to $(kn)^2$, which explains the choice of the abscissa $\log_{10}(N)/2$. The precision attained is about one order of magnitude better for the first eigenvalue $\lambda_1$ than for the last one $\lambda_{10}$. For the first eigenvalue and the coarser meshes, we observe a kind of super convergence for the small degrees, then a linear convergence. The convergence tends to be linear as the mesh becomes finer. 
For the last eigenvalue $\lambda_{10}$, the convergence is mainly linear.
% Maillages
\begin{figure}[h]
\begin{tabular}{cc}
\includegraphics[scale=0.55]{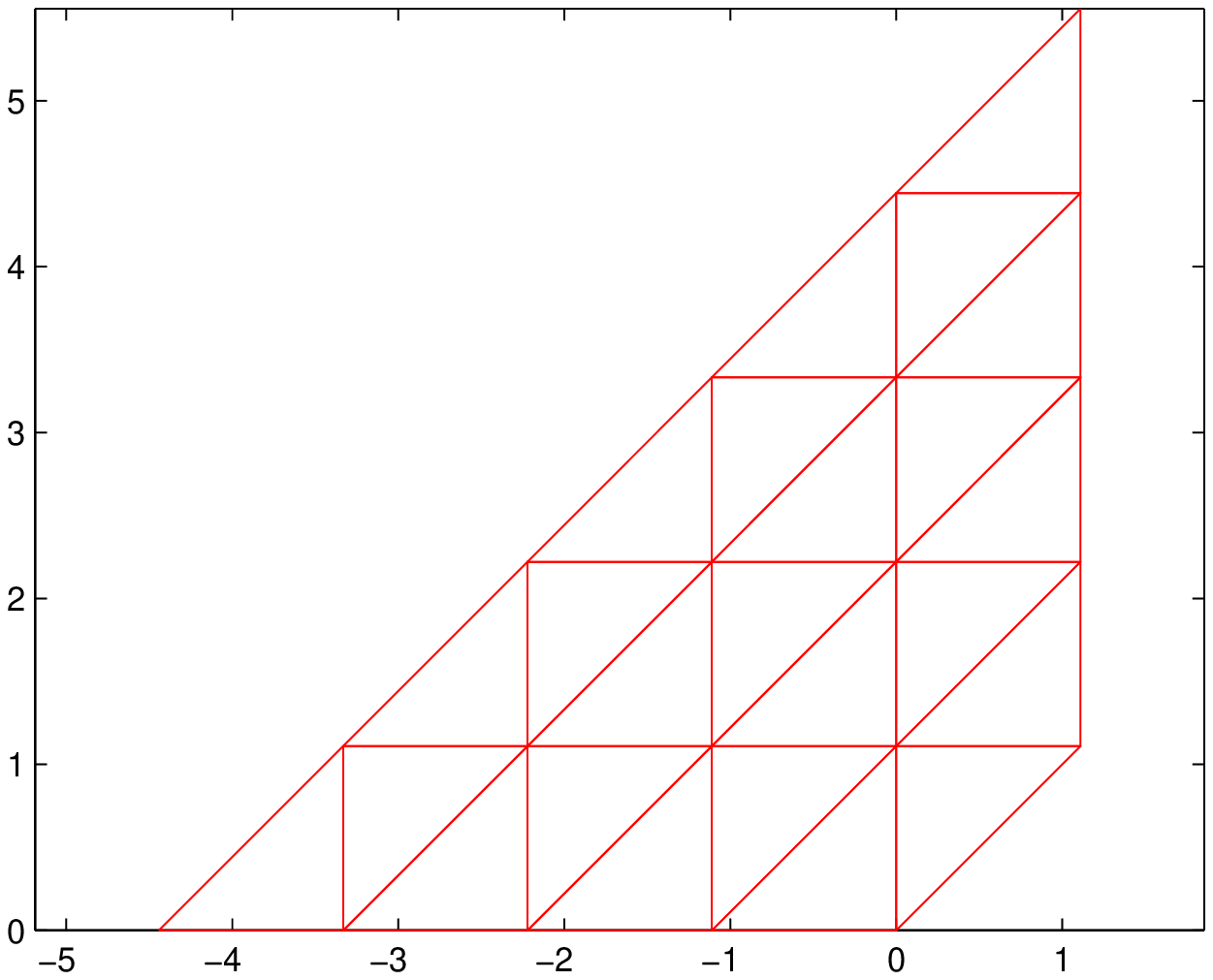}
\includegraphics[scale=0.55]{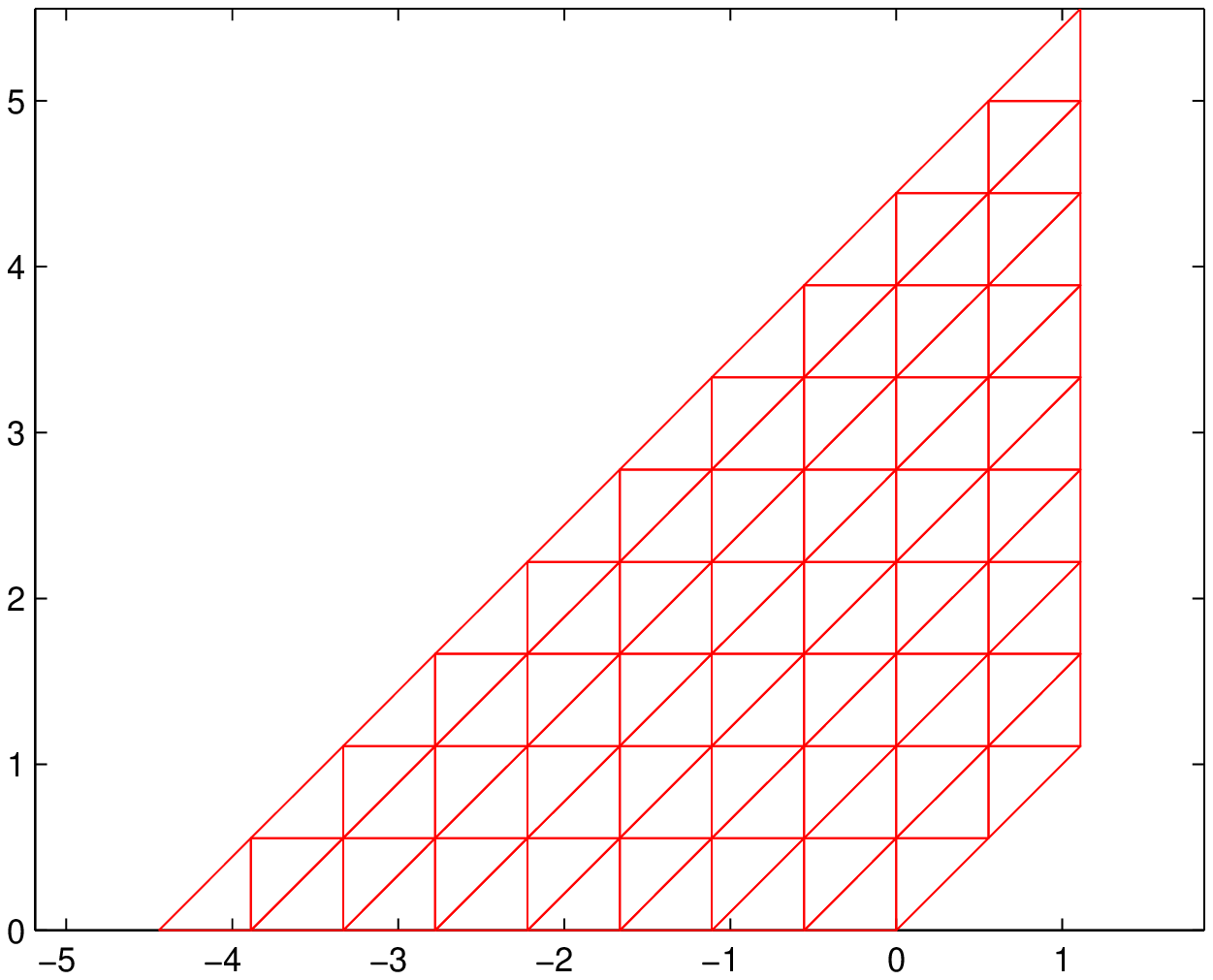}
\end{tabular}
\caption{Meshes M4S (24 triangles) and M8S (96 triangles).}
\label{F:mesh4S_8S}
\end{figure}
\begin{figure}[h]
\begin{tabular}{cc}
\includegraphics[scale=0.55]{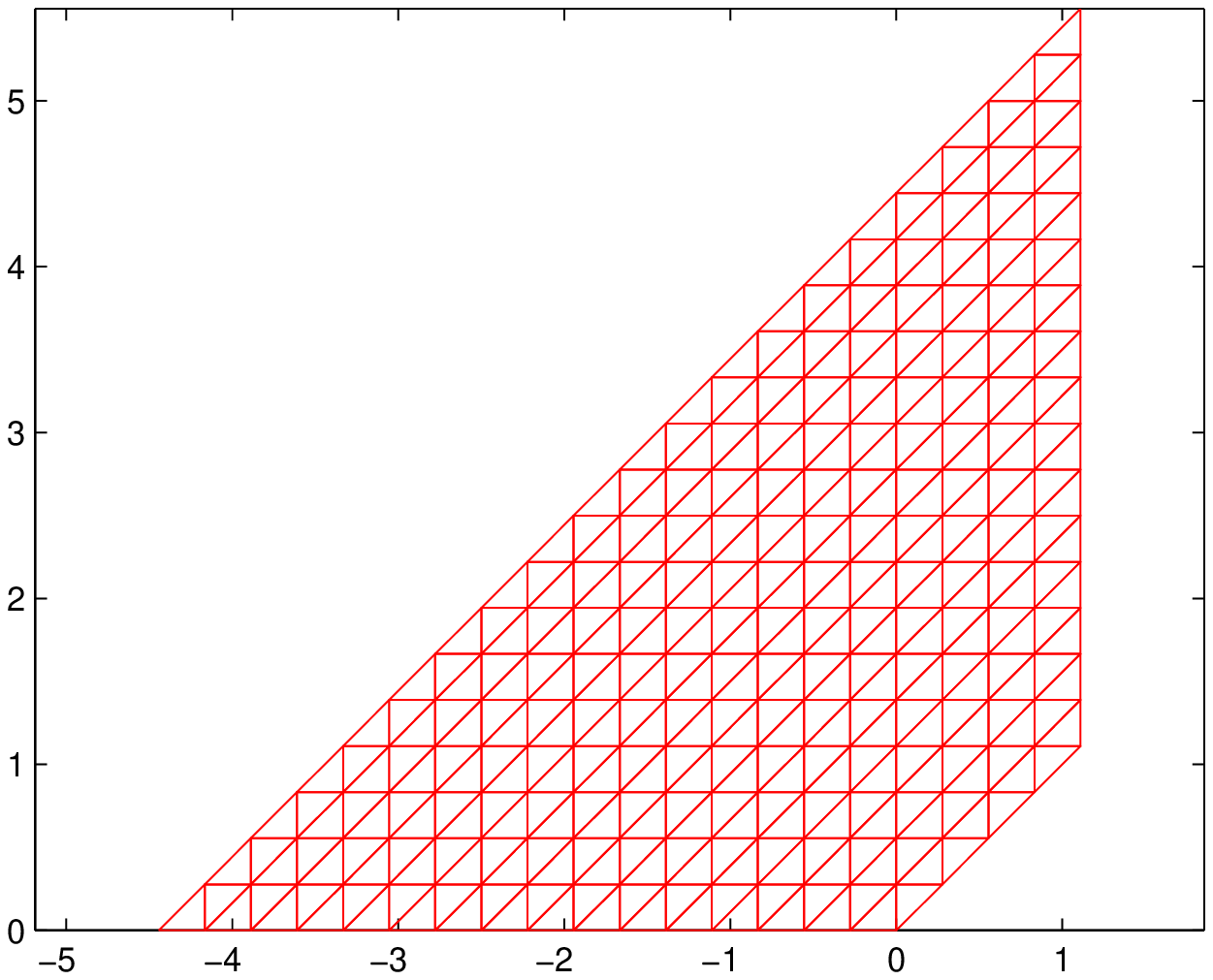}
\includegraphics[scale=0.55]{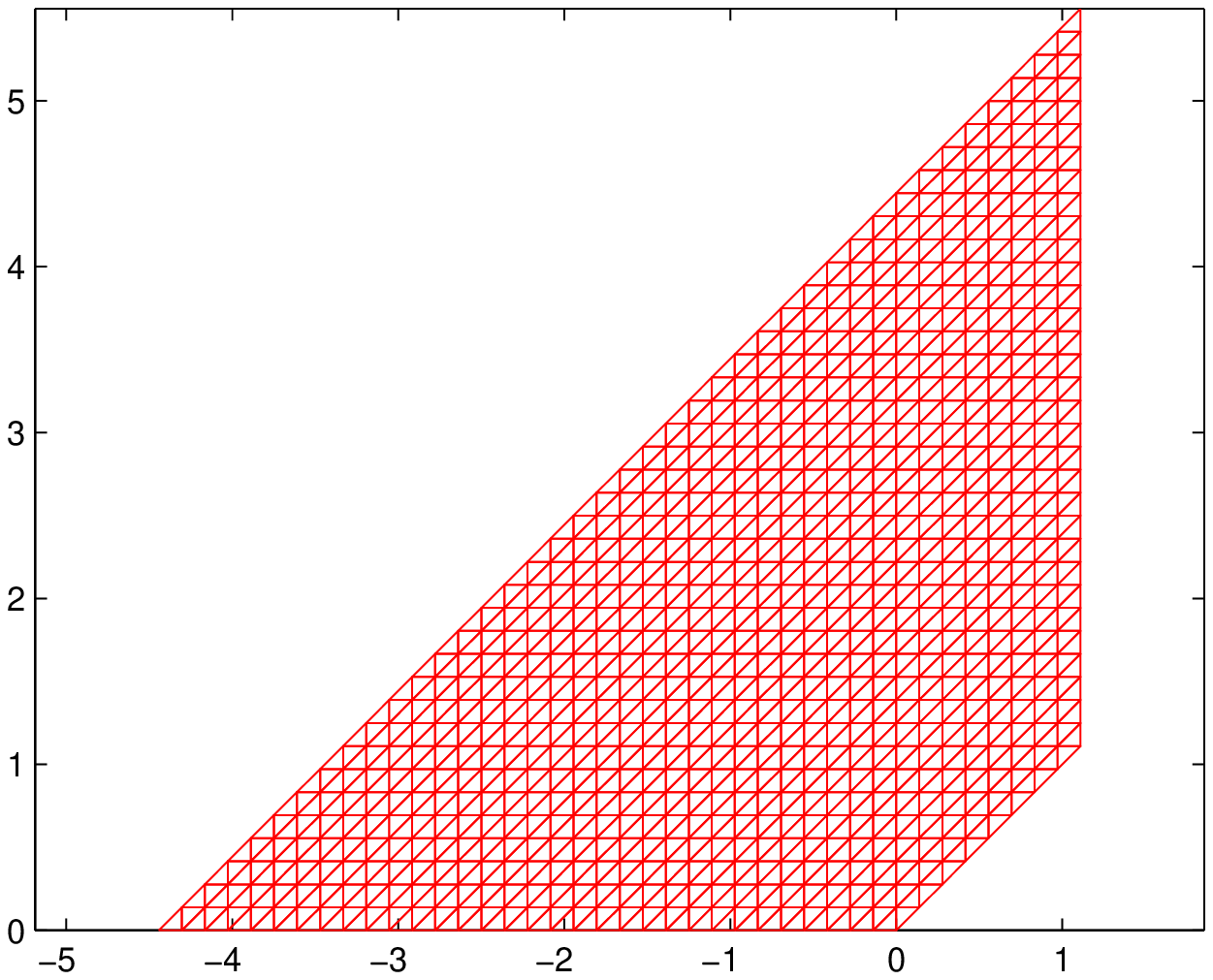}
\end{tabular}
\caption{Meshes M16S (384 triangles) and M32S (1536 triangles).}
\label{F:mesh16S_32S}
\end{figure}

% Graphiques
\begin{figure}[h]
\begin{tabular}{cc}
\includegraphics[scale=0.48]{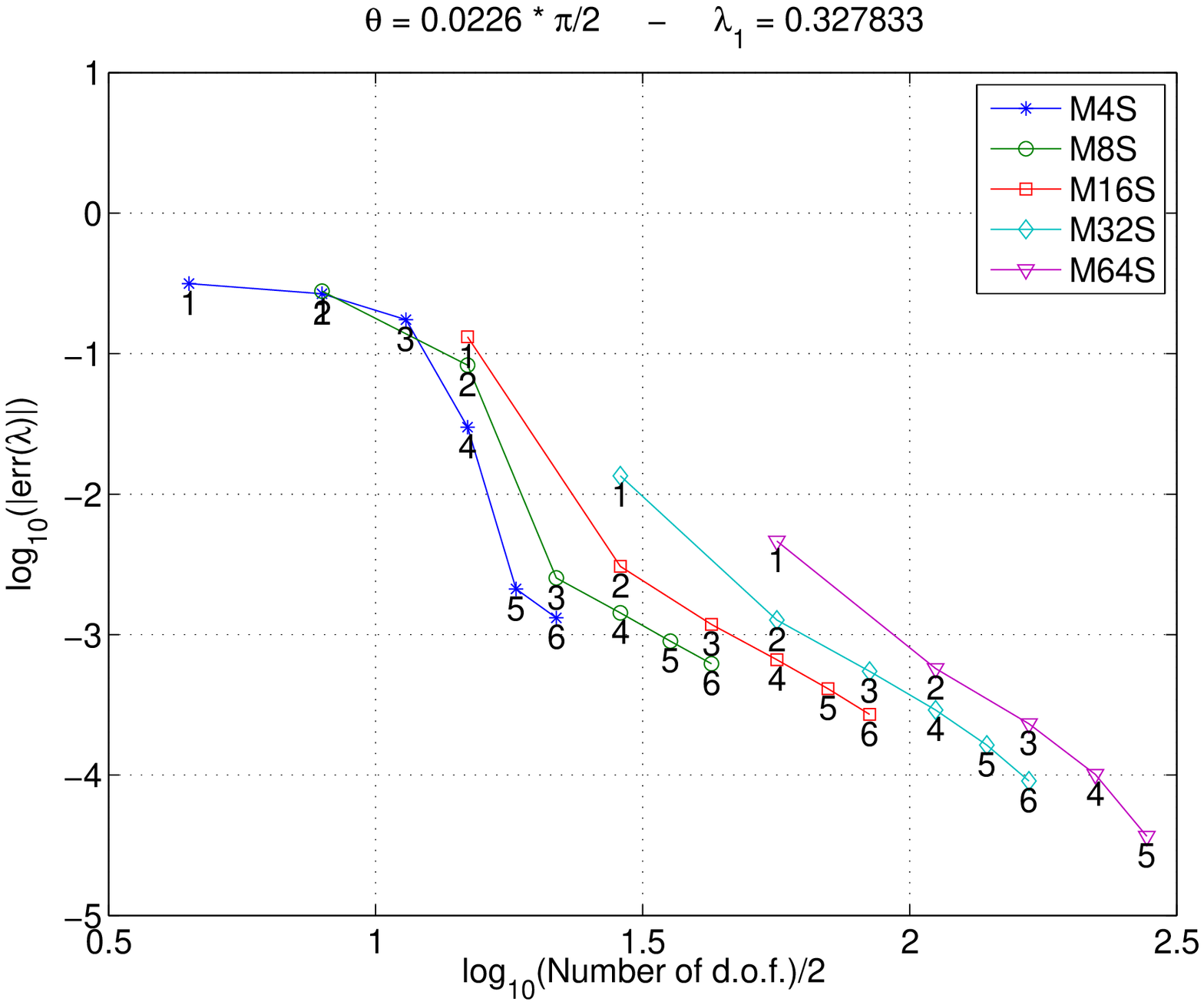}
\includegraphics[scale=0.48]{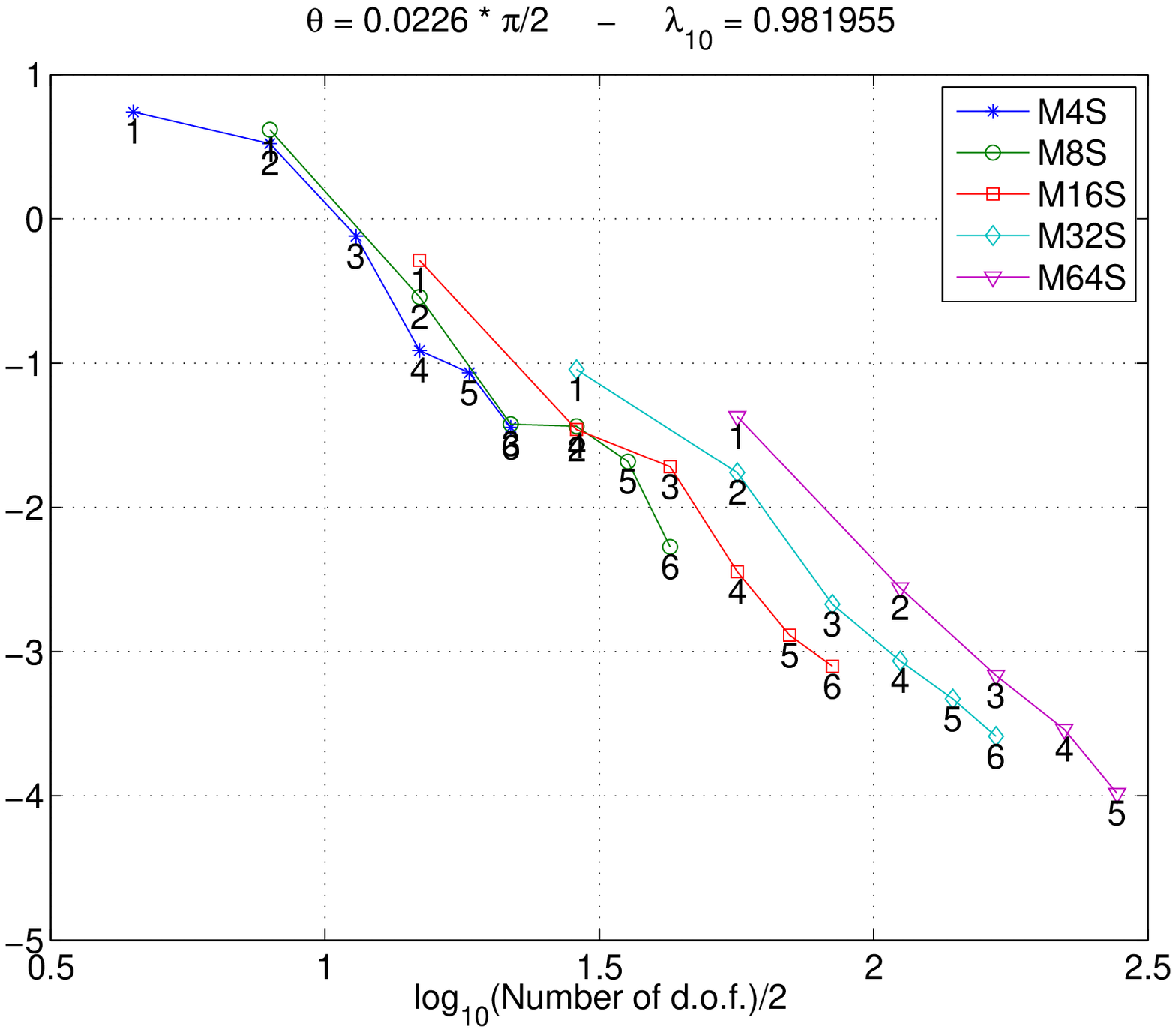}
\end{tabular}
\caption{Convergence towards $\lambda_1\simeq 0.32783$ and $\lambda_{10}\simeq 0.98195$.}
\label{F:converg}
\end{figure}

This phenomenon of super convergence is due to a sort of locking which takes place in coarser meshes: such meshes with low degree $k$ are not able to capture the fine scale structure of the eigenvectors (see in Figure \ref{F:10} representations of the first eigenvector in the computational domain for small angles).

The average rate of convergence with respect to $k^{-1}$ is twice the rate of convergence with respect to $h$, in accordance with a well-known convergence result in finite elements \cite{Schwab98}. Typically the rate is $1$ in $h$ and $2$ in $k^{-1}$. These somewhat low rates are due to the singularity at the reentrant corner of $\Omega$. This singularity comes from the Laplace singularity \eqref{eq:sing}, and for $\theta = 0.0226*\pi/2$, the singularity exponent $\frac\pi\omega$ is $\simeq0.5057$.

Nevertheless, all our computations are accurate enough to display clearly the asymptotic behavior of the eigenpairs in the small angle and large angle limits, see Figure \ref{F:EVtg}.

%\clearpage

\end{document}